\documentclass[11pt]{article}

\usepackage[T1]{fontenc}
\usepackage[utf8]{inputenc}
\usepackage{amsmath,amssymb,amsthm,mathtools,mathrsfs}
\usepackage{aliascnt}
\usepackage[margin=1in]{geometry}
\usepackage{microtype}
\usepackage{enumitem}
\usepackage{booktabs,tabularx}
\usepackage{etoolbox}
\usepackage{xcolor}
\usepackage[most]{tcolorbox}
\usepackage{placeins}
\usepackage{needspace}
\usepackage{tikz}
\usetikzlibrary{arrows.meta,calc,positioning}
\usepackage{pgfplots}
\pgfplotsset{compat=1.18}
\usepackage[intoc,norefeq,norefpage]{nomencl}
\usepackage[
  colorlinks=true,
  linkcolor=blue!55!black,
  citecolor=blue!55!black,
  urlcolor=blue!55!black,
  pdftitle={The Gromov--Hausdorff Distance Between Consecutive Spheres},
  pdfauthor={Donghan Kim, Sunhyuk Lim, Facundo M\'emoli}
]{hyperref}
\usepackage[nameinlink,capitalise]{cleveref}
\usepackage[normalem]{ulem}

\newtheorem{letteredtheorem}{Theorem}

\newenvironment{maintheorem}[1][]{\begin{letteredtheorem}[#1]}{\end{letteredtheorem}}
\crefname{letteredtheorem}{Theorem}{Theorems}
\Crefname{letteredtheorem}{Theorem}{Theorems}
\newaliascnt{lemma}{theorem}
\newtheorem{lemma}[lemma]{Lemma}
\aliascntresetthe{lemma}
\newaliascnt{proposition}{theorem}
\newtheorem{proposition}[proposition]{Proposition}
\aliascntresetthe{proposition}
\newaliascnt{corollary}{theorem}
\newtheorem{corollary}[corollary]{Corollary}
\aliascntresetthe{corollary}
\newaliascnt{remark}{theorem}
\newtheorem{remark}[remark]{Remark}
\aliascntresetthe{remark}
\newaliascnt{definition}{theorem}
\newtheorem{definition}[definition]{Definition}
\aliascntresetthe{definition}

\newcommand{\R}{\mathbb{R}}
\newcommand{\Sp}{\mathbb{S}}
\newcommand{\dis}{\operatorname{dis}}

\newcommand{\dgh}{d_{\mathrm{GH}}}

\makenomenclature

\renewcommand{\nomgroup}[1]{\item[\bfseries
  \ifstrequal{#1}{A}{Metric spaces and correspondences}{\ifstrequal{#1}{B}{The anchored--chord construction}{\ifstrequal{#1}{C}{Pairwise comparisons and the admissible region}{\ifstrequal{#1}{D}{Spherical caps and the upper boundary}{Auxiliary notation}}}}]}

\DeclareRobustCommand{\nomdefpage}[1]{\textit{Defined on }\hyperref[#1]{p.~\pageref*{#1}}}
\DeclareRobustCommand{\nomdefpages}[2]{\textit{Defined on }\hyperref[#1]{p.~\pageref*{#1}}\textit{ and }\hyperref[#2]{p.~\pageref*{#2}}\textit{, respectively}}
\DeclareRobustCommand{\nomresult}[1]{\textit{See }\hyperref[#1]{\Cref*{#1}, p.~\pageref*{#1}}}
\DeclareRobustCommand{\nomdefresults}[2]{\textit{Defined in }\hyperref[#1]{\Cref*{#1}, p.~\pageref*{#1}}\textit{ and }\hyperref[#2]{\Cref*{#2}, p.~\pageref*{#2}}\textit{, respectively}}

\newcommand{\acConstructionHeading}{\section{Geometry of the canonical gain}\label{sec:canonical-anchored-chord-all-n}}
\newcommand{\acEstimatesHeading}{\section{Same-anchor comparisons and the lower
  \texorpdfstring{\((Q,R)\)-boundary}{(Q,R)-boundary}}\label{sec:anchored-chord-estimates}}
\newcommand{\acUpperBoundaryHeading}{\section{Spherical-cap reduction and the upper
  \texorpdfstring{\((Q,R)\)-boundary}{(Q,R)-boundary}}\label{sec:anchored-chord-upper-boundary}}
\newcommand{\acAssemblyHeading}{\section{Remaining cases and proof of
  \texorpdfstring{\Cref{thm:main}}{the consecutive-sphere theorem}}\label{sec:anchored-chord-assembly}}
\newcommand{\acSharpDistortionEquation}{\eqref{eq:main-distortion}}
\newcommand{\acMainResultEquation}{\eqref{eq:main-result}}
\newcommand{\acLowerBoundArgument}{The reverse inequality is \Cref{prop:adjacent-lower-bound}.}

\title{The Gromov--Hausdorff Distance Between Consecutive Spheres}
\author{Donghan Kim\thanks{Department of Mathematical Sciences, Korea Advanced
  Institute of Science and Technology (KAIST), Daejeon, Republic of Korea.
  Email: \href{mailto:patrick6231@kaist.ac.kr}{\texttt{patrick6231@kaist.ac.kr}}.}
  \and
  Sunhyuk Lim\thanks{Department of Mathematics, Sungkyunkwan University
  (SKKU), Suwon, Republic of Korea. Email:
  \href{mailto:lsh3109@skku.edu}{\texttt{lsh3109@skku.edu}}.}
  \and
  Facundo M\'emoli\thanks{Department of Mathematics, Rutgers, The State
  University of New Jersey, Piscataway, NJ, USA. Email:
  \href{mailto:facundo.memoli@gmail.com}{\texttt{facundo.memoli@gmail.com}}.}}

\date{\today}

\begin{document}

\maketitle

\begin{abstract}

We determine the Gromov--Hausdorff distance between consecutive unit round
spheres equipped with their geodesic metrics. Put $\zeta_n:=\arccos(-\tfrac{1}{n+1}),$
the common geodesic distance between distinct vertices of a regular simplex
with \(n+2\) vertices inscribed in \(\Sp^n\). 

\smallskip
We prove that
\[
 d_{\mathrm{GH}}(\Sp^n,\Sp^{n+1})=\frac{\zeta_n}{2}
 \qquad(n\geq1),
\]
resolving a conjecture of Lim, Mémoli, and Smith. All cases \(n\geq4\)
were previously open. 

\smallskip
This equality is established by explicitly constructing  a family of correspondences
\(\mathcal R_n\subseteq \Sp^{n+1}\times \Sp^n\), whose distortion matches the known quantitative Borsuk--Ulam lower bound $\zeta_n$.

\smallskip
We also introduce synchronized spherical joins and suspensions of
correspondences and prove that the distortion of a join is exactly the
maximum of the distortions of its factors. In particular, suspension
preserves distortion. Applying these join and suspension operations to
the optimal correspondences \(\mathcal R_n\) yields new
bounds for spheres of nonconsecutive dimensions, including
\[
 \lim_{m\to\infty}
 d_{\mathrm{GH}}\bigl(\Sp^m,\Sp^{m+d(m)}\bigr)
 =\frac{\pi}{4}
 \qquad\text{whenever } d(m)\geq1,\ \text{and }d(m)=o(m).
\]

\end{abstract}



\newpage
\tableofcontents

\newpage
\section{Main results and proof strategy}
\label{sec:introduction}

Recall that a correspondence between compact metric spaces \(X\) and \(Y\)
is a relation \(\mathcal R\subseteq X\times Y\) whose two coordinate
projections are surjective.  Its distortion is
\phantomsection\label{def:correspondence-distortion}
\[
 \dis(\mathcal R)
 :=\sup_{(x,y),(x',y')\in\mathcal R}
 \left|d_X(x,x')-d_Y(y,y')\right|.
\]
For \((x,y),(x',y')\in\mathcal R\), we call
\(\left|d_X(x,x')-d_Y(y,y')\right|\) their \emph{metric defect}.  We
usually drop the qualifier and call this quantity their \emph{defect}.
Moreover,
\begin{equation}\label{eq:gh-correspondence-formula}
 \dgh(X,Y)=\frac12\inf_{\mathcal R}\dis(\mathcal R).
\end{equation}
For a function \(f\colon X\to Y\), we write \(\dis(f)\) for the
distortion of its graph.

\paragraph{{The consecutive-sphere problem.}}
For each \(n\geq1\), we equip the unit round sphere
\(\Sp^n\subset\R^{n+1}\) with its geodesic metric
\phantomsection\label{def:geodesic-metric}
\[
 d_n(x,x'):=\arccos(x\cdot x')
\]
and  put
\begin{equation}\label{eq:intro-zeta}
 \zeta_n:=\arccos\left(-\frac1{n+1}\right);
\end{equation}
this is the geodesic distance between any two distinct vertices of a
regular simplex with \(n+2\) vertices inscribed in \(\Sp^n\).

Via a \emph{quantitative} Borsuk-Ulam principle, Lim, M\'emoli and Smith proved
\(\dgh(\Sp^m,\Sp^{n})\geq\tfrac{\zeta_m}{2}\) for all $m<n$
\cite[Theorem~B]{lim2021gromov} and conjectured that equality always holds when $n = m+1$
\cite[Conjecture~1]{lim2021gromov}.  Our principal result proves their
conjecture.

\begin{maintheorem}\label{thm:main}
 For every integer \(n\geq1\),
\begin{equation}\label{eq:main-result}
 \dgh(\Sp^n,\Sp^{n+1})
 =\frac{\zeta_n}{2}
 =\frac12\arccos\left(-\frac1{n+1}\right).
\end{equation}
Furthermore, for every integer \(n\geq1\), the relation \(\mathcal R_n\) constructed in
\Cref{def:anchored-chord-relation} is a correspondence between
\(\Sp^{n+1}\) and \(\Sp^n\) satisfying
\begin{equation}\label{eq:main-distortion}
 \dis(\mathcal R_n)=\zeta_n.
\end{equation}
\end{maintheorem}

All cases \(n\geq4\) are new. The cases \(n=1\) and \(n=2\) of \Cref{eq:main-result} were established
by Lim, M\'emoli and Smith
\cite[Propositions~1.16 and~1.19]{lim2021gromov}.  The case \(n=3\) was
subsequently proved by Rodr\'iguez Mart\'in
\cite[Theorem~1.5]{rodriguezmartin2024novel}.  Our construction works
uniformly for every \(n\geq1\).  It therefore supplies a single formula
for an optimal correspondence in all dimensions, including alternative
proofs in the three previously known cases \(n=1,2,3\).

\paragraph{Spherical joins and bounds for nonconsecutive spheres.}
The optimal correspondences in \Cref{thm:main} can also be combined to
obtain bounds for spheres of nonconsecutive dimensions.  The relevant operation is the
following join of correspondences, whose exact distortion formula is our
second main result.  We use the isometric realization
\[
 \Sp^{m_1}*\Sp^{m_2}
 =
 \left\{
  (x_1\cos t,x_2\sin t):
  x_i\in\Sp^{m_i}\ (i=1,2),\ 0\leq t\leq\tfrac{\pi}{2}
 \right\}
 \cong\Sp^{m_1+m_2+1}.
\]

\begin{definition}[Synchronized spherical join]
\label{def:synchronized-spherical-join}
Let
\[
 \mathcal R^{(i)}\subseteq\Sp^{m_i}\times\Sp^{n_i},
 \qquad i=1,2,
\]
be correspondences.  Their \emph{synchronized spherical join} is
\begin{equation}\label{eq:synchronized-join-definition}
 \mathcal R^{(1)}*\mathcal R^{(2)}
 :=\bigl\{((x_1\cos t,x_2\sin t),(y_1\cos t,y_2\sin t)):
 0\leq t\leq\tfrac{\pi}{2},\ (x_i,y_i)\in\mathcal R^{(i)}\ (i=1,2)\bigr\}.
\end{equation}
We call \(t\) the \emph{join parameter} of the displayed element.  Thus
\(\mathcal R^{(1)}*\mathcal R^{(2)}\) is a relation between
\(\Sp^{m_1+m_2+1}\) and \(\Sp^{n_1+n_2+1}\).
\end{definition}

The use of the same join parameter in the source and target is essential
for the exact distortion formula below.

A correspondence \(\mathcal R\subseteq\Sp^m\times\Sp^n\) is
\emph{antipodal} if \((x,y)\in\mathcal R\) implies
\((-x,-y)\in\mathcal R\).

\begin{maintheorem}[Exact distortion of spherical joins]
\label{thm:exact-join-distortion}
Let \(k\geq2\), and for \(1\leq i\leq k\), let
\(\mathcal R^{(i)}\subseteq\Sp^{m_i}\times\Sp^{n_i}\) be a
correspondence.  Then their iterated synchronized spherical join,
formed by successive application of
\Cref{def:synchronized-spherical-join}, is a correspondence and
\begin{equation}\label{eq:iterated-join-distortion}
 \dis(\mathcal R^{(1)}*\cdots*\mathcal R^{(k)})
 =\max_{1\leq i\leq k}\dis(\mathcal R^{(i)}).
\end{equation}
If every \(\mathcal R^{(i)}\) is antipodal, then so is their iterated
join.
\end{maintheorem}

Thus spherical join acts as a maximum operation on distortion, whereas
composition of correspondences generally provides only the sum of the
distortion bounds.  Joining a correspondence with the identity
correspondence on \(\Sp^0\) defines its synchronized spherical suspension;
\Cref{thm:exact-join-distortion} shows that suspension preserves distortion.

\begin{remark}[Fixed gaps and infinite dimension]
\Cref{thm:main} gives the fixed-gap sequence for \(d=1\) exactly:
\[
 m\longmapsto \dgh(\Sp^m,\Sp^{m+1})
\]
is strictly decreasing, from
\(\dgh(\Sp^1,\Sp^2)=\tfrac{\pi}{3}\) to
\[
 \lim_{m\to\infty}\dgh(\Sp^m,\Sp^{m+1})=\frac{\pi}{4}.
\]
\Cref{thm:main,thm:exact-join-distortion} place this in a more general pattern.  For every fixed
\(d\geq1\), the sequence
\[
 m\longmapsto \dgh(\Sp^m,\Sp^{m+d})
\]
is nonincreasing and converges to \(\tfrac{\pi}{4}\); the same limit holds
when \(d=d(m)=o(m)\); see \Cref{cor:balanced-join-asymptotics}.

This finite-dimensional stabilization contrasts with the fact that every
finite-dimensional sphere remains at the maximal possible distance from the
infinite-dimensional sphere:
\[
 \dgh(\Sp^m,\Sp^\infty)=\frac{\pi}{2}
 \qquad(m<\infty),
\]
where \(\Sp^\infty\) denotes the unit sphere in \(\ell^2\) with its geodesic
metric \cite[Proposition~1.6]{lim2021gromov}.
\end{remark}

The proof and further consequences of
\Cref{thm:exact-join-distortion} are given in
\Cref{sec:spherical-join-consequences}.

\paragraph{Context and motivation.} The Gromov--Hausdorff distance, introduced by Edwards
\cite{edwards1975structure} and developed systematically by Gromov
\cite{gromov1999metric}, is one of the basic comparison tools of metric
geometry.  The topology that it induces underlies compactness and
convergence notions for families of Riemannian manifolds; see
\cite{gromov1999metric,burago2022course}. Exact
values are nevertheless rare, even for highly symmetric compact metric
spaces, which provide a natural testing
ground for exact calculations and for the interaction between metric
and topological obstructions: in the sphere problem studied here, proving
an exact value requires a quantitative topological obstruction and an
explicit geometric construction to meet sharply.

The Gromov--Hausdorff distance
also supplies a suitable notion of stability for persistence-based shape
signatures \cite{carlson-memoli-clustering,chazal2009gromov,lim2020vietoris,memoli2019persistent,zhou-cup-product}, while its computational aspects have
been studied in connection with non-rigid shape matching
\cite{memoli-sapiro,memoli-properties,schmiedl2017computational}.
Through its correspondence formulation, the Gromov--Hausdorff distance
is a prototypical \emph{minimax quadratic matching problem}: the cost is attached not
to an individual match \(x\leftrightarrow y\), but to pairs of matches
through the discrepancy
\(\bigl|d_X(x,x')-d_Y(y,y')\bigr|\).  Schmiedl proved that the problem of
computing the Gromov--Hausdorff distance between finite metric spaces
is NP-hard \cite{schmiedl2017computational}.  This provides another
motivation for the general research program: the matching problem
encoded by the Gromov--Hausdorff distance serves as
a basic model for many methods used in practice to compare metric data
and shapes: the quality of a matching is measured not pointwise, but by
how well it preserves pairwise distances.  Exact solutions for symmetric
spaces such as spheres can therefore provide ground-truth
\emph{benchmarks} for practical heuristics, relaxations, and
reference examples for related
pairwise-cost comparison distances such as the Gromov--Wasserstein
distance \cite{memoli-2011-GW,sturm-GW,arya-GW}.

Unit spheres with their geodesic metrics form a canonical symmetric
family in which this optimization problem combines topology with
extremal geometry.  
Lim, M\'emoli and Smith initiated the systematic study of the pairwise
Gromov--Hausdorff distances between spheres. Their lower-bound
obstruction mentioned above is closely related to the quantitative
Borsuk--Ulam theorem of Dubins and Schwarz
\cite{dubins1981equidiscontinuity}: topology forces every candidate
comparison to incur a definite metric error.  The same mechanism links
the sphere-distance problem to the topology of Vietoris--Rips
complexes, as developed further in the Polymath work of Adams et al.\
\cite{adams2022gromov} in order to provide improved lower bounds.

On the upper-bound side, Lim, M\'emoli and Smith introduced the
regular-simplex and hemisphere framework for constructing
correspondences, through which they established the cases \(n=1\) and
\(n=2\) of the conjecture
\cite[Propositions~1.16 and~1.19]{lim2021gromov}.  The Polymath
collaboration subsequently obtained the dimension-independent bound
\(\dgh(\Sp^n,\Sp^{n+1})\leq\tfrac{\pi}{3}\)
\cite[Theorem~1.2]{adams2022gromov}.  Within the same framework,
Rodr\'iguez Mart\'in introduced continuous radial motion and used it to
prove the case \(n=3\)
\cite[Theorem~1.5]{rodriguezmartin2024novel}.  Embedding--projection
correspondences provide a complementary approach
\cite{memoli2024embedding}; building partly on these ideas, Harrison
and Jeffs determined the exact distance from the circle to every
higher-dimensional sphere and obtained upper bounds for general pairs
of dimensions \cite{harrison2023quantitative}.
\Cref{sec:previous-work} compares these correspondence constructions in
greater detail.

The lower and upper bounds establishing \Cref{thm:main} arise from rather different sources.  The
lower bound is topological, whereas the upper bound is realized by an
explicit correspondence governed by regular-simplex geometry.  Among the
principal ingredients are a vertex-dependent isometric
embedding for pairs associated with the same simplex vertex, and an exact
dimension reduction that converts an optimization over two spherical caps
into a one-variable concave maximization. Within the normalized-chord ansatz, the interpolation is not chosen ad hoc: it is determined by sharpness on an extremal family of comparisons together with a reciprocal symmetry between radial levels.

\paragraph{Organization.}
\Cref{sec:anchored-chord-definition} introduces the normalized-chord
correspondence ansatz, explains how an extremal comparison selects its gain, and
defines the correspondence \(\mathcal R_n\).
\Cref{sec:proof-strategy} describes the strategy for proving that this
correspondence is optimal, and
\Cref{sec:self-contained-lower-bound} gives a self-contained concise proof of
the known lower bound.
\Cref{sec:spherical-join-consequences} proves the exact distortion formula
for synchronized spherical joins and derives bounds for spheres of arbitrary dimensions
from the sharp relations \(\mathcal R_n\) and from exact circle-to-sphere
correspondences.
\Cref{sec:previous-work} places the normalized-chord ansatz in the context
of earlier correspondence constructions.
\Cref{sec:canonical-anchored-chord-all-n} develops the geometry of the
construction, proves that its gain is uniquely determined by the stated
sharpness conditions, and establishes the properties needed in the
distortion estimates.
\Cref{sec:anchored-chord-estimates,sec:anchored-chord-upper-boundary}
prove that
\(\operatorname{dis}(\mathcal R_n)\leq\zeta_n\): the two sections
control the two possible ways in which the source and target distances
can differ.
Finally, \Cref{sec:anchored-chord-assembly} treats the remaining
endpoint configurations and the case \(n=1\), and completes the proof
of the main theorem.
The subsidiary algebraic calculations are collected in the appendices.

\subsection{The anchored--chord correspondence}
\label{sec:anchored-chord-definition}

Put
\phantomsection\label{def:zero-sum-model}
\phantomsection\label{def:unit-spheres}
\[
 \begin{aligned}
 W_{n+2}
 :=\left\{x\in\R^{n+2}:\sum_{k=1}^{n+2}x_k=0\right\} \quad\text{and}\quad
 \Sp(W_{n+2})
 :=\{x\in W_{n+2}:\lVert x\rVert=1\}.
 \end{aligned}
\]
We identify \(\Sp^n\) isometrically with \(\Sp(W_{n+2})\).
Define the regular-simplex vertices
\phantomsection\label{def:simplex-vertices}
\[
 v_i:=\sqrt{\frac{n+2}{n+1}}
 \left(e_i-\frac1{n+2}\mathbf 1\right),
 \qquad 1\leq i\leq n+2.
\]
We call \(v_1,\ldots,v_{n+2}\) the \emph{anchors} of the construction.
Put
\phantomsection\label{def:simplex-angular-constants}
\[
 \boxed{
 \rho_n:=\frac1{n+1}
 \qquad\text{and}\qquad
 \widehat{\zeta}_n:=\pi-\zeta_n=\arccos\rho_n.}
\]
Thus \(v_i\cdot v_j=-\rho_n\) whenever \(i\neq j\), and
for distinct \(i,j\),
\[
 d_n(v_i,-v_j)=\arccos\!\bigl(v_i\cdot(-v_j)\bigr)
 =\arccos\rho_n=\widehat{\zeta}_n.
\]
Thus \(\widehat{\zeta}_n\) is both the complement of the simplex edge
length and the distance from \(v_i\) to the point \(-v_j\) used in the
``swapped-anchor" configuration below.

For \(x\in\Sp(W_{n+2})\), define the set of \emph{nearest simplex vertices} and
the corresponding \emph{closed spherical Voronoi cells} by
\phantomsection\label{def:voronoi-cells}
\[
 \begin{aligned}
 I(x)&:=\{i:x_i=\max_kx_k\},\\
 M_i&:=\{x\in\Sp(W_{n+2}):i\in I(x)\}
      =\{x:x_i=\max_kx_k\}.
 \end{aligned}
\]
Indeed, the zero-sum condition gives
\(x\cdot v_i=\sqrt{(n+2)/(n+1)}\,x_i\), so maximizing \(x_i\) is
equivalent to minimizing \(d_n(x,v_i)\).  The sets \(M_i\) are closed
and cover \(\Sp(W_{n+2})\); a point on a Voronoi boundary belongs to
every cell whose index lies in \(I(x)\).

\phantomsection\label{def:north-pole}
Let \(e_0:=(0,\ldots,0,1)\in\R^{n+3}\), and put \(N:=e_0\), the north
pole.  We regard \(W_{n+2}\) as the subspace
\(W_{n+2}\times\{0\}\subset\R^{n+3}\), so that
\(e_0\perp W_{n+2}\), and identify \(\Sp^{n+1}\) isometrically with
\(\Sp(W_{n+2}\oplus\R e_0)\).

The upper hemisphere of \(\Sp^{n+1}\) is parametrized by
\phantomsection\label{def:hemisphere-parametrization}
\[
 \begin{aligned}
 p_\alpha\colon\Sp(W_{n+2})&\longrightarrow
 \Sp(W_{n+2}\oplus\R e_0)=\Sp^{n+1}\subset\R^{n+3},\\
 p_\alpha(x)&:=x\sin\alpha+e_0\cos\alpha\in\R^{n+3},
 \qquad 0\leq\alpha\leq\frac\pi2.
 \end{aligned}
\]
Here \(\alpha=d_{n+1}(e_0,p_\alpha(x))\) is the geodesic distance from
the north pole along the meridian through \(x\); thus \(\alpha=0\)
gives the north pole and \(\alpha=\tfrac{\pi}{2}\) gives the equatorial point
\(x\).

\paragraph{The normalized-chord construction determined by a gain.}
For the moment, let
\[
 g\colon[0,\widehat{\zeta}_n)\longrightarrow[0,\infty)
\]
be continuous and satisfy \(g(0)=0\).
For \(1\leq i\leq n+2\), \(x\in M_i\), and
\(0\leq\alpha<\widehat{\zeta}_n\), define
\[
 \begin{aligned}
 H_{\alpha,i}^{g}\colon M_i&\longrightarrow\Sp(W_{n+2}),\\
 H_{\alpha,i}^{g}(x)
 &:=
 \frac{v_i+g(\alpha)x}
      {\lVert v_i+g(\alpha)x\rVert}.
 \end{aligned}
\]
At the right endpoint, put
\(H_{\widehat{\zeta}_n,i}^{g}(x):=x\).
The denominator is nonzero.  Indeed, its vanishing would force
\(g(\alpha)=1\) and \(x=-v_i\), whereas \(-v_i\notin M_i\): for every
\(j\neq i\),
 $v_i\cdot(-v_i)=-1<\rho_n=v_j\cdot(-v_i).$
Moreover,
\[
 \frac{v_i+g(\alpha)x}{1+g(\alpha)}
 =
 \frac1{1+g(\alpha)}v_i
 +\frac{g(\alpha)}{1+g(\alpha)}x
\]
lies on the Euclidean chord \([v_i,x]\), and
\(H_{\alpha,i}^{g}(x)\) is its radial projection to
\(\Sp(W_{n+2})\).  When \(x\neq v_i\), the two parts of the chord have
length ratio
\[
 \frac{\left\lVert
       \frac{v_i+g(\alpha)x}{1+g(\alpha)}-v_i
       \right\rVert}
      {\left\lVert
       x-\frac{v_i+g(\alpha)x}{1+g(\alpha)}
       \right\rVert}
 =g(\alpha).
\]
Thus \(g(\alpha)\) is the \emph{gain}: the values \(0\) and \(1\)
correspond to the anchor \(v_i\) and the midpoint of \([v_i,x]\),
respectively. If \(g(\alpha)\to+\infty\) as
\(\alpha\uparrow\widehat{\zeta}_n\), then the chord point tends to the
endpoint \(x\).  The canonical gain defined below has this limiting
behavior.

\begin{definition}[Normalized-chord relation determined by a gain]
\label{def:gain-dependent-relation}
For a gain \(g\) as above, define the upper-hemisphere relation
\begingroup
\small
\[
 \mathcal R_n^{+}(g):={}
 \underbrace{\bigcup_{i=1}^{n+2}
 \{(p_\alpha(x),H_{\alpha,i}^{g}(x)):
 0\leq\alpha\leq\widehat{\zeta}_n,\ x\in M_i\}}
 _{\text{anchored--chord part}}
 \ \cup\
 \underbrace{\{(p_\alpha(x),x):
 \widehat{\zeta}_n\leq\alpha\leq\tfrac{\pi}{2},\ x\in\Sp(W_{n+2})\}}
 _{\text{identity band}},
\]
\endgroup
and extend it antipodally by
\[
 \mathcal R_n(g):=\mathcal R_n^{+}(g)
 \cup\{(-q,-z):(q,z)\in\mathcal R_n^{+}(g)\}.
\]
\end{definition}

The cells \(M_i\) cover \(\Sp(W_{n+2})\), so the first coordinate
projection of \(\mathcal R_n(g)\) is all of \(\Sp^{n+1}\).  The identity
band contains every \(x\in\Sp(W_{n+2})\) as a second coordinate, so the
second coordinate projection is all of \(\Sp^n\).  Hence \(\mathcal R_n(g)\) is a correspondence for every gain \(g\) as above.

For a pair \((q,z)\in\mathcal R_n(g)\), we call \(q\) its \emph{source
point} and \(z\) its \emph{target point}; these terms refer only to the
two coordinates and do not assert that the relation is a map.
In
\((p_\alpha(x),H_{\alpha,i}^{g}(x))\), the index \(i\) is the
\emph{anchor index} and \(v_i\) is the \emph{anchor}; for each
\(\alpha<\widehat{\zeta}_n\), the target \(H_{\alpha,i}^{g}(x)\) is the
radial projection of a point of \([v_i,x]\), and
\(H_{0,i}^{g}(x)=v_i\).  The remaining terminology is read directly from
\Cref{def:gain-dependent-relation}: the second part of
\(\mathcal R_n^+(g)\) is the \emph{identity band}, and the passage from
\(\mathcal R_n^+(g)\) to \(\mathcal R_n(g)\) is the \emph{antipodal
extension}.  Retaining every index in \(I(x)\) at a Voronoi tie makes
\(\mathcal R_n(g)\) a relation rather than, in general, the graph of a
function.  The construction is illustrated in
\Cref{fig:anchored-chord-schematic} after the gain is selected below.

The construction above leaves one essential choice: the gain \(g\).
To determine this choice, fix distinct indices
\(i,j\in\{1,\ldots,n+2\}\), and take the equatorial point
\(-v_j\) in the \(i\)-th cell and the equatorial point \(-v_i\) in the
\(j\)-th cell.
We call these two anchor--equatorial pairings the
\emph{swapped-anchor configuration}.

\paragraph{Equal radial levels and the switching level.}
At the same radial level \(\alpha\), the two source points
\[
 p_\alpha(-v_j),\qquad p_\alpha(-v_i)
\]
have inner product
\[
 \cos^2\alpha-\rho_n\sin^2\alpha.
\]
Introduce the \emph{half-angle constants}
\phantomsection\label{def:half-angle-constants}
\[
 \boxed{
 \begin{gathered}
  c_n:=\cos\frac{\widehat{\zeta}_n}{2}
      =\sqrt{\frac{1+\rho_n}{2}},\qquad
  s_n:=\sin\frac{\widehat{\zeta}_n}{2}
      =\sqrt{\frac{1-\rho_n}{2}}.
 \end{gathered}}
\]
The function
\phantomsection\label{def:theta-function}
\[
 \begin{aligned}
 \Theta_n\colon[0,\widehat{\zeta}_n]&\longrightarrow[0,\frac{\pi}{2}],\\
 \Theta_n(\alpha)&:=\arcsin\!\bigl(c_n\sin\alpha\bigr),
 \end{aligned}
\]
is chosen so that
\[
 \cos\bigl(2\Theta_n(\alpha)\bigr)
 =\cos^2\alpha-\rho_n\sin^2\alpha.
\]
Consequently,
\phantomsection\label{def:equal-level-source-distance}
\[
 D_n(\alpha)
 :=d_{n+1}\bigl(p_\alpha(-v_j),p_\alpha(-v_i)\bigr)
 =2\Theta_n(\alpha).
\]
These are the equal-level source points in this configuration.
Requiring this
comparison to attain the desired distortion bound prescribes the target
distance \(\zeta_n+2\Theta_n(\alpha)\). Regarded as an equation for the still-undetermined value
\(g(\alpha)\), this sharpness condition uniquely determines
\(g(\alpha)\) for \(0\leq\alpha\leq\alpha_n^\ast\).  These values form
the first branch of the canonical gain in \eqref{eq:ac-1}.

As \(\alpha\) increases, the prescribed target distance first reaches
\(\pi\) when \(2\Theta_n(\alpha)=\widehat{\zeta}_n\).  Since
\(\Theta_n\) is strictly increasing, this first permissible antipodal
event determines the \emph{switching level} \(\alpha_n^\ast\) as the
unique point satisfying
\phantomsection\label{def:switching-level}
\[
 \Theta_n(\alpha_n^\ast)=\frac{\widehat{\zeta}_n}{2}.
\]
Equivalently,
\[
 \alpha_n^\ast
 :=\arcsin\left(\frac{s_n}{c_n}\right)
 =\arcsin\sqrt{\frac{1-\rho_n}{1+\rho_n}}.
\]
At this level the source distance is \(\widehat{\zeta}_n\) and the
prescribed target distance is \(\pi\).

\paragraph{Unequal radial levels and the critical involution.}
For the equatorial points \(-v_j\in M_i\) and \(-v_i\in M_j\), the
normalized targets \(H_{\alpha,i}^{g}(-v_j)\) and
\(H_{\beta,j}^{g}(-v_i)\) are antipodal exactly when
\(g(\alpha)g(\beta)=1\); this elementary calculation is given in
\Cref{sec:critical-involution}.  The corresponding source points
\(p_\alpha(-v_j)\) and \(p_\beta(-v_i)\) are at distance
\(\widehat{\zeta}_n\) exactly when
\begin{equation}
 \cos\alpha\cos\beta
 -\rho_n\sin\alpha\sin\beta=\rho_n.
 \label{eq:ac-2}
\end{equation}
For each \(\alpha\in[0,\widehat{\zeta}_n]\), define
\(F_n(\alpha)\) to be the unique
\(\beta\in[0,\widehat{\zeta}_n]\) satisfying this equation.  Thus
\(F_n(\alpha)\) is the unique radial level for which these two source
points are at distance \(\widehat{\zeta}_n\).
Existence, uniqueness, and the involution property are proved in
\Cref{prop:ac-critical-involution}.  Solving the defining equation gives
the explicit formula
\begin{equation}
\label{eq:ac-F-explicit}
 \begin{aligned}
 F_n\colon[0,\widehat{\zeta}_n]&\longrightarrow
 [0,\widehat{\zeta}_n],\\
 F_n(\alpha)&:=\arctan\!\left(
 \frac{\sin\widehat{\zeta}_n-\sin\alpha}
      {\rho_n\cos\alpha}
 \right),
 \end{aligned}
\end{equation}
where \(\arctan\) takes values in \([0,\tfrac{\pi}{2})\).  The displayed implicit
equation \eqref{eq:ac-2} gives the geometric characterization of
\(F_n\), whereas \eqref{eq:ac-F-explicit} is the explicit formula used
to define the gain below.

A direct substitution in \eqref{eq:ac-F-explicit} gives
\[
 F_n(\alpha_n^\ast)=\alpha_n^\ast.
\]
Once the sharpness condition has determined \(g\) on
\([0,\alpha_n^\ast]\), reciprocal compatibility determines it on
\((\alpha_n^\ast,\widehat{\zeta}_n)\).  Indeed,
\Cref{prop:ac-critical-involution} shows that \(F_n\) maps this latter
interval into \((0,\alpha_n^\ast)\), and the reciprocal identity requires
\[
 g(\alpha)=\frac{1}{g(F_n(\alpha))}
 \qquad
 (\alpha_n^\ast<\alpha<\widehat{\zeta}_n).
\]

\paragraph{The two requirements on the gain.}
The preceding discussion leads to the following two requirements.
\begin{enumerate}[label=\textup{(\roman*)},nosep]
 \item For \(0\leq\alpha\leq\alpha_n^\ast\), the equal-level
 swapped-anchor comparison is sharp:
 \[
  d_n\bigl(H_{\alpha,i}^{g}(-v_j),
           H_{\alpha,j}^{g}(-v_i)\bigr)
  =\zeta_n+2\Theta_n(\alpha).
 \]
 \item For \(0<\alpha<\widehat{\zeta}_n\), the gain is compatible with
 the critical involution through the reciprocal identity
 \[
  g(\alpha)\,g(F_n(\alpha))=1.
 \]
\end{enumerate}
The detailed derivation of these requirements and the proof that they
characterize the gain are given in
\Cref{sec:swapped-anchor-configuration,sec:critical-involution}.

\paragraph{Definition of the \emph{canonical} gain.}

Define the \emph{canonical gain}
\[
 r_n\colon[0,\widehat{\zeta}_n]\longrightarrow[0,+\infty]
\]
by
\begin{equation}
 \boxed{
 \begin{aligned}
 r_n(\alpha)&:=
 \begin{cases}
 \dfrac{\sin\Theta_n(\alpha)}
       {\sin\!\bigl(\widehat{\zeta}_n-\Theta_n(\alpha)\bigr)},
 &0\leq\alpha\leq\alpha_n^\ast,\\[3mm]
       \dfrac{\sin\!\bigl(\widehat{\zeta}_n-
                 \Theta_n(F_n(\alpha))\bigr)}
       {\sin\Theta_n(F_n(\alpha))},
 &\alpha_n^\ast<\alpha<\widehat{\zeta}_n,
 \end{cases}\\[-1mm]
 r_n(\widehat{\zeta}_n)&:=+\infty.
 \end{aligned}}
 \label{eq:ac-1}
\end{equation}
\Cref{prop:canonical-gain-characterization} shows that \(r_n\) is the
unique continuous candidate
\(g\colon[0,\widehat{\zeta}_n)\to[0,\infty)\), with \(g(0)=0\), that
satisfies conditions \textup{(i)}--\textup{(ii)} above.  This is the
precise sense in which the gain is \emph{canonical}.
Both branch expressions are continuous on their respective intervals.
Since
\[
 \Theta_n(\alpha_n^\ast)
 =\Theta_n(F_n(\alpha_n^\ast))
 =\frac{\widehat{\zeta}_n}{2},
\]
both branch expressions tend to \(1\) as
\(\alpha\to\alpha_n^\ast\).  Thus \(r_n\) is continuous at
\(\alpha_n^\ast\).
For a first reading, the essential behavior of \(r_n\) is summarized by
\[
 r_n(0)=0,\qquad
 r_n(\alpha_n^\ast)=1,\qquad
 r_n(\alpha)\longrightarrow+\infty
 \quad\text{as }\alpha\uparrow\widehat{\zeta}_n.
\]
Accordingly, for \(0\leq\alpha<\widehat{\zeta}_n\), the Euclidean chord
point \((v_i+r_n(\alpha)x)/(1+r_n(\alpha))\) moves from \(v_i\), through
the midpoint of \([v_i,x]\), and toward \(x\).  A bounded visualization
of this motion appears in \Cref{fig:canonical-gain-profiles}.

\begin{definition}[Anchored--chord correspondence]
\label{def:anchored-chord-relation}
Specialize \Cref{def:gain-dependent-relation} to
\(r_n|_{[0,\widehat{\zeta}_n)}\), and abbreviate
\(H_{\alpha,i}^{r_n}\), \(\mathcal R_n^{+}(r_n)\), and
\(\mathcal R_n(r_n)\) by \(H_{\alpha,i}\), \(\mathcal R_n^{+}\), and
\(\mathcal R_n\), respectively.  Thus, for
\(1\leq i\leq n+2\) and
\(0\leq\alpha<\widehat{\zeta}_n\), define
\begin{equation}
 \begin{aligned}
 H_{\alpha,i}\colon M_i&\longrightarrow\Sp(W_{n+2}),\\
 H_{\alpha,i}(x)&:=H_{\alpha,i}^{r_n}(x)
 =
 \frac{v_i+r_n(\alpha)\,x}{\lVert v_i+r_n(\alpha)\,x\rVert}.
 \end{aligned}
 \label{eq:ac-7}
\end{equation}
At the right endpoint, put
\(H_{\widehat{\zeta}_n,i}(x):=x\).
Set \(\mathcal R_n^{+}:=\mathcal R_n^{+}(r_n)\).  Explicitly, the
\emph{upper-hemisphere relation} is
\begingroup
\small
\begin{equation}
 \mathcal R_n^{+}:={}
 \underbrace{\bigcup_{i=1}^{n+2}
 \{(p_\alpha(x),H_{\alpha,i}(x)):
 0\leq\alpha\leq\widehat{\zeta}_n,\ x\in M_i\}}
 _{\text{anchored--chord part}}
 \ \cup\
 \underbrace{\{(p_\alpha(x),x):
 \widehat{\zeta}_n\leq\alpha\leq\tfrac{\pi}{2},\ x\in\Sp(W_{n+2})\}}
 _{\text{identity band}}.
 \label{eq:ac-8}
\end{equation}
\endgroup
Finally, extend this relation antipodally:
\begin{equation}
 \mathcal R_n:=\mathcal R_n(r_n)
 =\mathcal R_n^{+}
 \cup\{(-q,-z):(q,z)\in\mathcal R_n^{+}\}.
 \label{eq:ac-9}
\end{equation}
\end{definition}

\Cref{fig:anchored-chord-schematic} separates the two operations in
\Cref{def:anchored-chord-relation}: for fixed \(i\) and \(x\in M_i\),
panel~\textup{(a)} records the target paired with \(p_\alpha(x)\) as
\(\alpha\) varies from \(0\) to \(\tfrac{\pi}{2}\), while panel~\textup{(b)} shows
the normalized Euclidean-chord interpolation defining
\(H_{\alpha,i}\).

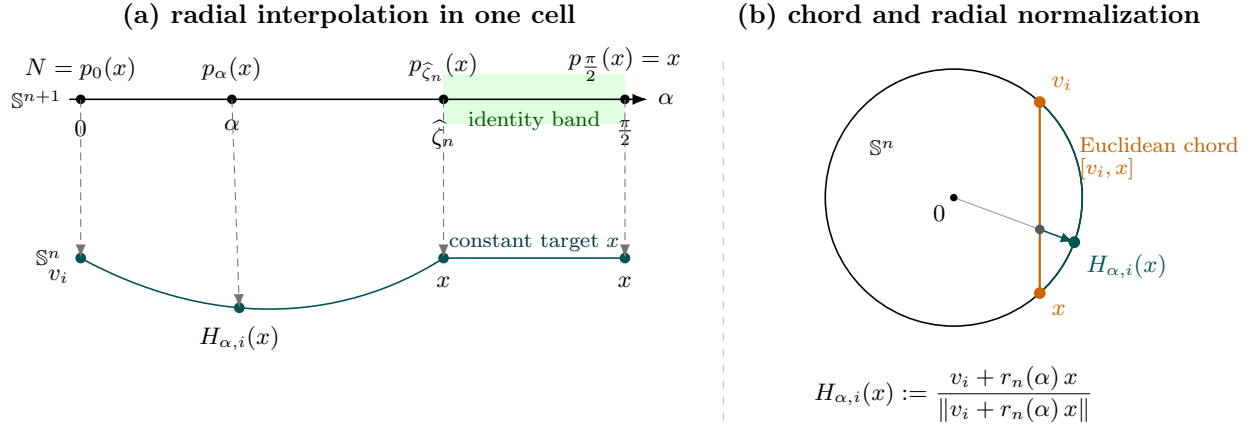
\begin{figure}[htbp]
\centering
\begin{tikzpicture}[
  every node/.style={font=\footnotesize},
  coupling/.style={-{Latex[length=1.8mm]},black!55,densely dashed},
  schedule/.style={teal!65!black,semithick}
]
\begin{scope}[xshift=-1.0cm]
  \node[font=\small\bfseries] at (3.55,2.55)
    {(a) radial interpolation in one cell};
  \fill[green!12] (4.8,1.12) rectangle (7.2,1.78);
  \draw[-{Latex[length=2mm]},semithick] (-0.15,1.45)--(7.5,1.45)
    node[right] {$\alpha$};
  \node[font=\scriptsize\bfseries,anchor=east] at (-0.12,1.45)
    {$\Sp^{n+1}$};
  \coordinate (szero) at (0,1.45);
  \coordinate (salpha) at (2.0,1.45);
  \coordinate (sA) at (4.8,1.45);
  \coordinate (sequator) at (7.2,1.45);
  \foreach \p in {szero,salpha,sA,sequator}
    \fill (\p) circle[radius=1.8pt];
  \node[above=3pt] at (szero) {$N=p_0(x)$};
  \node[above=3pt] at (salpha) {$p_\alpha(x)$};
  \node[above=3pt] at (sA) {$p_{\widehat{\zeta}_n}(x)$};
  \node[above=3pt] at (sequator) {$p_{\tfrac{\pi}{2}}(x)=x$};
  \node[below=4pt] at (szero) {$0$};
  \node[below=4pt] at (salpha) {$\alpha$};
  \node[below=4pt] at (sA) {$\widehat{\zeta}_n$};
  \node[below=4pt] at (sequator) {$\tfrac{\pi}{2}$};
  \node[font=\scriptsize,green!35!black] at (6.0,1.17)
    {identity band};

  \coordinate (tzero) at (0,-0.65);
  \coordinate (tA) at (4.8,-0.65);
  \coordinate (tequator) at (7.2,-0.65);
  \path (tzero) .. controls (1.65,-1.55) and (3.4,-1.55) .. (tA)
    coordinate[pos=.42] (talpha);
  \draw[schedule]
    (tzero) .. controls (1.65,-1.55) and (3.4,-1.55) .. (tA)
    --(tequator);
  \node[font=\scriptsize\bfseries,anchor=east] at (-0.12,-0.65)
    {$\Sp^n$};
  \foreach \p in {tzero,talpha,tA,tequator}
    \fill[teal!65!black] (\p) circle[radius=2pt];
  \node[below left=1pt] at (tzero) {$v_i$};
  \node[below=3pt] at (talpha) {$H_{\alpha,i}(x)$};
  \node[below=3pt] at (tA) {$x$};
  \node[below=3pt] at (tequator) {$x$};
  \node[font=\scriptsize,teal!45!black] at (6.0,-0.45)
    {constant target $x$};
  \draw[coupling] (szero)--(tzero);
  \draw[coupling] (salpha)--(talpha);
  \draw[coupling] (sA)--(tA);
  \draw[coupling] (sequator)--(tequator);
  \end{scope}

\draw[black!25,dashed] (7.5,-2.80)--(7.5,2.0);
  \node[font=\small\bfseries] at (10.75,2.55)
    {(b) chord and radial normalization};
  \coordinate (O) at (10.55,0.15);
  \coordinate (vi) at ($(O)+(48:1.70)$);
  \coordinate (xx) at ($(O)+(-48:1.70)$);
  \coordinate (cc) at ($(vi)!.6666667!(xx)$);
  \coordinate (HH) at ($(O)!1.70cm!(cc)$);
  \draw[semithick] (O) circle[radius=1.70];
  \node[font=\scriptsize\bfseries,fill=white,inner sep=1pt]
    at ($(O)+(145:1.15)$) {$\Sp^n$};
  \draw[orange!75!black,thick] (vi)--(xx);
  \draw[schedule]
    ($(O)+(48:1.70)$)
    arc[start angle=48,end angle=-48,radius=1.70];
  \draw[black!40] (O)--(cc);
  \draw[-{Latex[length=1.8mm]},teal!65!black,semithick] (cc)--(HH);
  \fill (O) circle[radius=1.4pt] node[below left=-1pt] {$0$};
  \fill[orange!80!black] (vi) circle[radius=2.2pt]
    node[above right=1pt] {$v_i$};
  \fill[orange!80!black] (xx) circle[radius=2.2pt]
    node[below right=1pt] {$x$};
  \fill[black!65] (cc) circle[radius=1.8pt];
  \fill[teal!65!black] (HH) circle[radius=2.2pt]
    node[below right=1pt] {$H_{\alpha,i}(x)$};
  \node[font=\scriptsize,orange!75!black,align=left,anchor=west]
    at (12.08,0.67) {Euclidean chord\\[-1pt]$[v_i,x]$};
  \node[align=center] at (10.55,-2.45)
    {$H_{\alpha,i}(x):=\dfrac{v_i+r_n(\alpha)\,x}
      {\lVert v_i+r_n(\alpha)\,x\rVert}$};
\end{tikzpicture}
\caption{The two operations defining the anchored--chord correspondence
\(\mathcal R_n\) along a fixed closed spherical Voronoi cell \(M_i\).
For fixed \(x\in M_i\), panel~(a) records the target paired with
\(p_\alpha(x)\) as \(\alpha\) varies from \(0\) to \(\tfrac{\pi}{2}\): the target
moves from the anchor \(v_i\) to \(x\), reaches \(x\) at
\(\widehat{\zeta}_n\), and remains there throughout the identity band.
Panel~(b) shows the normalized-chord operation defining
\(H_{\alpha,i}\): a point on the Euclidean chord from \(v_i\) to \(x\)
is projected radially onto \(\Sp^n\).}
\label{fig:anchored-chord-schematic}
\end{figure}

For \(n=1\), \Cref{fig:anchored-chord-n-one} displays the entire
relation defined by \eqref{eq:ac-8}--\eqref{eq:ac-9}.

\begin{figure}[!ht]
\centering
\resizebox{1.00\textwidth}{!}{\begin{tikzpicture}[
  every node/.style={font=\footnotesize},
  hidden/.style={black!27,densely dashed},
  boundary/.style={black!48,dashed},
  coupling/.style={densely dashed,semithick},
  chordpath/.style={
    teal!70!black,
    very thick,
    line cap=round
  },
  identitypath/.style={
    purple!62!black,
    very thick,
    line cap=round
  }
]

  \def\Rs{2.60}
  \def\tilt{25}
  \def\Aone{60}
  \def\samplealpha{46.5149815}
  \def\xang{-40}
  \def\Hang{-56.1648801}
  \def\Rt{2.08}

  \pgfmathdeclarefunction{spx}{2}{\pgfmathparse{
      \Rs*sin(#1)*cos(#2)
    }}

  \pgfmathdeclarefunction{spy}{2}{\pgfmathparse{
      \Rs*(
        sin(\tilt)*sin(#1)*sin(#2)
        + cos(\tilt)*cos(#1)
      )
    }}

\pgfmathdeclarefunction{vistheta}{1}{\pgfmathparse{
      atan(tan(\tilt)/sin(#1))
    }}

\pgfmathdeclarefunction{vissouth}{1}{\pgfmathparse{
      180 + atan(tan(\tilt)/sin(#1))
    }}

\pgfmathsetmacro{\phicut}{
    asin(tan(\tilt)/tan(\Aone))
  }
  \pgfmathsetmacro{\phicutright}{
    180-\phicut
  }
  \pgfmathsetmacro{\phisouthleft}{
    -180+\phicut
  }
  \pgfmathsetmacro{\phisouthright}{
    -\phicut
  }
  \pgfmathsetmacro{\thetasouthcut}{
    180-\tilt
  }

  \newcommand{\MakeSpt}[3]{\coordinate (#1) at
      (
        {\Rs*sin(#2)*cos(#3)},
        {
          \Rs*(
            sin(\tilt)*sin(#2)*sin(#3)
            + cos(\tilt)*cos(#2)
          )
        }
      );
  }

  \coordinate (SO) at (0,0);

  \MakeSpt{spNorth}{0}{0}
  \MakeSpt{spSouth}{180}{0}
  \MakeSpt{Sq}{\samplealpha}{\xang}
  \MakeSpt{SqA}{\Aone}{\xang}
  \MakeSpt{Sqbeta}{78}{\xang}
  \MakeSpt{Sx}{90}{\xang}

\shade[
    inner color=white,
    outer color=blue!11
  ]
    (SO) circle[radius=\Rs];

  \path[
    fill=purple!7,
    draw=none
  ]
    plot[
      domain=-180:0,
      samples=61,
      variable=\p
    ]
      ({spx(\Aone,\p)},{spy(\Aone,\p)})
    --
    plot[
      domain=\Aone:120,
      samples=31,
      variable=\a
    ]
      ({spx(\a,0)},{spy(\a,0)})
    --
    plot[
      domain=0:-180,
      samples=61,
      variable=\p
    ]
      ({spx(120,\p)},{spy(120,\p)})
    --
    plot[
      domain=120:\Aone,
      samples=31,
      variable=\a
    ]
      ({spx(\a,-180)},{spy(\a,-180)})
    -- cycle;

  \path[
    fill=green!20,
    fill opacity=.25,
    draw=none
  ]
    plot[
      domain=\phicut:90,
      samples=40,
      variable=\p
    ]
      ({spx(\Aone,\p)},{spy(\Aone,\p)})
    --
    plot[
      domain=\Aone:\tilt,
      samples=25,
      variable=\a
    ]
      ({spx(\a,90)},{spy(\a,90)})
    --
    plot[
      domain=90:\phicut,
      samples=40,
      variable=\p
    ]
      (
        {spx(vistheta(\p),\p)},
        {spy(vistheta(\p),\p)}
      )
    -- cycle;

  \path[
    fill=blue!20,
    fill opacity=.25,
    draw=none
  ]
    plot[
      domain=\tilt:\Aone,
      samples=25,
      variable=\a
    ]
      ({spx(\a,90)},{spy(\a,90)})
    --
    plot[
      domain=90:\phicutright,
      samples=40,
      variable=\p
    ]
      ({spx(\Aone,\p)},{spy(\Aone,\p)})
    --
    plot[
      domain=\phicutright:90,
      samples=40,
      variable=\p
    ]
      (
        {spx(vistheta(\p),\p)},
        {spy(vistheta(\p),\p)}
      )
    -- cycle;

  \path[
    fill=orange!20,
    draw=none
  ]
    plot[
      domain=0:\Aone,
      samples=25,
      variable=\a
    ]
      ({spx(\a,-150)},{spy(\a,-150)})
    --
    plot[
      domain=-150:-30,
      samples=50,
      variable=\p
    ]
      ({spx(\Aone,\p)},{spy(\Aone,\p)})
    --
    plot[
      domain=\Aone:0,
      samples=25,
      variable=\a
    ]
      ({spx(\a,-30)},{spy(\a,-30)})
    -- cycle;

  \path[
    fill=green!20,
    draw=none
  ]
    plot[
      domain=0:\Aone,
      samples=25,
      variable=\a
    ]
      ({spx(\a,-30)},{spy(\a,-30)})
    --
    plot[
      domain=-30:\phicut,
      samples=30,
      variable=\p
    ]
      ({spx(\Aone,\p)},{spy(\Aone,\p)})
    --
    plot[
      domain=\phicut:90,
      samples=45,
      variable=\p
    ]
      (
        {spx(vistheta(\p),\p)},
        {spy(vistheta(\p),\p)}
      )
    --
    plot[
      domain=\tilt:0,
      samples=20,
      variable=\a
    ]
      ({spx(\a,90)},{spy(\a,90)})
    -- cycle;

  \path[
    fill=blue!20,
    draw=none
  ]
    plot[
      domain=0:\tilt,
      samples=20,
      variable=\a
    ]
      ({spx(\a,90)},{spy(\a,90)})
    --
    plot[
      domain=90:\phicutright,
      samples=45,
      variable=\p
    ]
      (
        {spx(vistheta(\p),\p)},
        {spy(vistheta(\p),\p)}
      )
    --
    plot[
      domain=\phicutright:210,
      samples=35,
      variable=\p
    ]
      ({spx(\Aone,\p)},{spy(\Aone,\p)})
    --
    plot[
      domain=\Aone:0,
      samples=25,
      variable=\a
    ]
      ({spx(\a,210)},{spy(\a,210)})
    -- cycle;

  \path[
    fill=orange!15,
    fill opacity=.22,
    draw=none
  ]
    plot[
      domain=180:120,
      samples=25,
      variable=\a
    ]
      ({spx(\a,30)},{spy(\a,30)})
    --
    plot[
      domain=30:150,
      samples=50,
      variable=\p
    ]
      ({spx(120,\p)},{spy(120,\p)})
    --
    plot[
      domain=120:180,
      samples=25,
      variable=\a
    ]
      ({spx(\a,150)},{spy(\a,150)})
    -- cycle;

  \path[
    fill=green!15,
    fill opacity=.22,
    draw=none
  ]
    plot[
      domain=180:120,
      samples=25,
      variable=\a
    ]
      ({spx(\a,150)},{spy(\a,150)})
    --
    plot[
      domain=150:180,
      samples=20,
      variable=\p
    ]
      ({spx(120,\p)},{spy(120,\p)})
    --
    plot[
      domain=120:180,
      samples=25,
      variable=\a
    ]
      ({spx(\a,180)},{spy(\a,180)})
    -- cycle;

  \path[
    fill=green!15,
    fill opacity=.22,
    draw=none
  ]
    plot[
      domain=180:120,
      samples=25,
      variable=\a
    ]
      ({spx(\a,-180)},{spy(\a,-180)})
    --
    plot[
      domain=-180:-90,
      samples=35,
      variable=\p
    ]
      ({spx(120,\p)},{spy(120,\p)})
    --
    plot[
      domain=120:180,
      samples=25,
      variable=\a
    ]
      ({spx(\a,-90)},{spy(\a,-90)})
    -- cycle;

  \path[
    fill=blue!15,
    fill opacity=.22,
    draw=none
  ]
    plot[
      domain=180:120,
      samples=25,
      variable=\a
    ]
      ({spx(\a,-90)},{spy(\a,-90)})
    --
    plot[
      domain=-90:30,
      samples=50,
      variable=\p
    ]
      ({spx(120,\p)},{spy(120,\p)})
    --
    plot[
      domain=120:180,
      samples=25,
      variable=\a
    ]
      ({spx(\a,30)},{spy(\a,30)})
    -- cycle;

  \path[
    fill=green!15,
    draw=none
  ]
    plot[
      domain=\phisouthleft:-90,
      samples=40,
      variable=\p
    ]
      ({spx(120,\p)},{spy(120,\p)})
    --
    plot[
      domain=120:\thetasouthcut,
      samples=25,
      variable=\a
    ]
      ({spx(\a,-90)},{spy(\a,-90)})
    --
    plot[
      domain=-90:\phisouthleft,
      samples=40,
      variable=\p
    ]
      (
        {spx(vissouth(\p),\p)},
        {spy(vissouth(\p),\p)}
      )
    -- cycle;

  \path[
    fill=blue!15,
    draw=none
  ]
    plot[
      domain=-90:\phisouthright,
      samples=40,
      variable=\p
    ]
      ({spx(120,\p)},{spy(120,\p)})
    --
    plot[
      domain=\phisouthright:-90,
      samples=40,
      variable=\p
    ]
      (
        {spx(vissouth(\p),\p)},
        {spy(vissouth(\p),\p)}
      )
    --
    plot[
      domain=\thetasouthcut:120,
      samples=25,
      variable=\a
    ]
      ({spx(\a,-90)},{spy(\a,-90)})
    -- cycle;

  \draw[hidden]
    plot[
      domain=\phicut:\phicutright,
      samples=61,
      variable=\p
    ]
      ({spx(60,\p)},{spy(60,\p)});

  \draw[
    black!58,
    semithick
  ]
    plot[
      domain=-180:\phicut,
      samples=61,
      variable=\p
    ]
      ({spx(60,\p)},{spy(60,\p)});

  \draw[
    black!58,
    semithick
  ]
    plot[
      domain=\phicutright:180,
      samples=31,
      variable=\p
    ]
      ({spx(60,\p)},{spy(60,\p)});

  \draw[hidden]
    plot[
      domain=0:180,
      samples=61,
      variable=\p
    ]
      ({spx(90,\p)},{spy(90,\p)});

  \draw[
    black!58,
    semithick
  ]
    plot[
      domain=-180:0,
      samples=61,
      variable=\p
    ]
      ({spx(90,\p)},{spy(90,\p)});

  \draw[hidden]
    plot[
      domain=-180:\phisouthleft,
      samples=20,
      variable=\p
    ]
      ({spx(120,\p)},{spy(120,\p)});

  \draw[
    black!58,
    semithick
  ]
    plot[
      domain=\phisouthleft:\phisouthright,
      samples=61,
      variable=\p
    ]
      ({spx(120,\p)},{spy(120,\p)});

  \draw[hidden]
    plot[
      domain=\phisouthright:180,
      samples=61,
      variable=\p
    ]
      ({spx(120,\p)},{spy(120,\p)});

\foreach \p in {-150,-30}{
    \draw[boundary]
      plot[
        domain=0:\Aone,
        samples=25,
        variable=\a
      ]
        ({spx(\a,\p)},{spy(\a,\p)});
  }

\draw[boundary]
    plot[
      domain=0:\tilt,
      samples=15,
      variable=\a
    ]
      ({spx(\a,90)},{spy(\a,90)});

  \draw[hidden]
    plot[
      domain=\tilt:\Aone,
      samples=20,
      variable=\a
    ]
      ({spx(\a,90)},{spy(\a,90)});

\draw[boundary]
    plot[
      domain=120:\thetasouthcut,
      samples=20,
      variable=\a
    ]
      ({spx(\a,-90)},{spy(\a,-90)});

  \draw[hidden]
    plot[
      domain=\thetasouthcut:180,
      samples=15,
      variable=\a
    ]
      ({spx(\a,-90)},{spy(\a,-90)});

\foreach \p in {30,150}{
    \draw[hidden]
      plot[
        domain=120:180,
        samples=25,
        variable=\a
      ]
        ({spx(\a,\p)},{spy(\a,\p)});
  }

  \draw[chordpath]
    plot[
      domain=0:\Aone,
      samples=35,
      variable=\a
    ]
      ({spx(\a,\xang)},{spy(\a,\xang)});

  \draw[identitypath]
    plot[
      domain=\Aone:120,
      samples=35,
      variable=\a
    ]
      ({spx(\a,\xang)},{spy(\a,\xang)});

\draw[
    black!72,
    semithick
  ]
    (SO) circle[radius=\Rs];

  \fill
    (spNorth)
    circle[radius=1.8pt]
    node[above left=3pt,fill=white,inner sep=1pt] {$N=p_0(x)$};

  \fill
    (spSouth)
    circle[radius=1.5pt]
    node[below left=3pt,fill=white,inner sep=1pt] {$S$};

  \fill[teal!70!black]
    (Sq)
    circle[radius=2.3pt];

  \draw[
    teal!70!black,
    fill=white
  ]
    (SqA)
    circle[radius=2pt]
    node[left=5pt,fill=white,inner sep=1pt]
      {$p_{\widehat{\zeta}_1}(x)$};

  \fill[purple!62!black]
    (Sqbeta)
    circle[radius=1.8pt];

  \fill[purple!62!black]
    (Sx)
    circle[radius=2.2pt]
    node[left=6pt,yshift=-5pt,fill=white,inner sep=1pt]
      {$p_{\tfrac{\pi}{2}}(x)$};

  \node[
    font=\small\bfseries
  ]
    at (0,3.3)
    {source $\Sp^2$};

  \node[
    orange!65!black
  ]
    at (-0.20,1.38)
    {Voronoi cell $M_2$};

  \node[
    purple!60!black,
    align=center,
    fill=white,
    inner sep=1pt
  ]
    at (-0.82,0.15)
    {identity band\\[-1pt]
     $p_\beta(x)\leftrightarrow x$};

  \node[
    black!45
  ]
    at (-0.55,-1.63)
    {odd copy};

  \node[
    anchor=east,
    align=right
  ]
    at (-2.82,1.08)
    {$\widehat{\zeta}_1=\tfrac{\pi}{3}$};

  \draw[black!35]
    (-2.72,1.00)--(-2.30,0.83);

  \node[
    anchor=east
  ]
    at (-2.78,-0.73)
    {equator};

  \draw[black!35]
    (-2.68,-0.67)--(-2.30,-0.54);

\begin{scope}[xshift=14.7cm,xscale=-1]

  \coordinate (TO) at (7.35,0);

  \coordinate (TvOne)
    at ($(TO)+(150:\Rt)$);

  \coordinate (TvTwo)
    at ($(TO)+(-90:\Rt)$);

  \coordinate (TvThree)
    at ($(TO)+(30:\Rt)$);

  \coordinate (Tx)
    at ($(TO)+(\xang:\Rt)$);

  \coordinate (TH)
    at ($(TO)+(\Hang:\Rt)$);

  \coordinate (Tc)
    at ($(TvTwo)!.6666667!(Tx)$);

  \coordinate (Tchordmid)
    at ($(TvTwo)!.48!(Tx)$);

  \fill[gray!3]
    (TO) circle[radius=\Rt];

  \draw[
    blue!65!black,
    line width=2.2pt
  ]
    ($(TO)+(90:\Rt)$)
    arc[
      start angle=90,
      end angle=210,
      radius=\Rt
    ];

  \draw[
    orange!80!black,
    line width=2.2pt
  ]
    ($(TO)+(-150:\Rt)$)
    arc[
      start angle=-150,
      end angle=-30,
      radius=\Rt
    ];

  \draw[
    green!52!black,
    line width=2.2pt
  ]
    ($(TO)+(-30:\Rt)$)
    arc[
      start angle=-30,
      end angle=90,
      radius=\Rt
    ];

  \draw[
    black!70,
    semithick
  ]
    (TO) circle[radius=\Rt];

  \draw[
    black!58,
    densely dotted,
    line width=1pt
  ]
    (TvOne)--(TvTwo)--(TvThree)--cycle;

  \foreach \p in {-150,-30,90}{
    \draw[black!48]
      ($(TO)+(\p:{\Rt-0.09})$)
      --
      ($(TO)+(\p:{\Rt+0.09})$);

    \draw[
      black!42,
      fill=white
    ]
      ($(TO)+(\p:\Rt)$)
      circle[radius=1.45pt];
  }

\draw[
    orange!78!black,
    thick
  ]
    (TvTwo)--(Tx);

  \draw[chordpath]
    ($(TO)+(-90:\Rt)$)
    arc[
      start angle=-90,
      end angle=\xang,
      radius=\Rt
    ];

  \draw[black!32]
    (TO)--(Tc);

\draw[
      teal!70!black,
      semithick
    ]
      (Tc)--(TH);

  \fill
    (TO)
    circle[radius=1.3pt]
    node[above left=-1pt] {$0$};

  \fill[blue!65!black]
    (TvOne)
    circle[radius=2.2pt]
    node[above right=-1pt] {$v_1$};

  \fill[orange!80!black]
    (TvTwo)
    circle[radius=2.2pt]
    node[below=2pt] {$v_2$};

  \fill[green!52!black]
    (TvThree)
    circle[radius=2.2pt]
    node[above left=-1pt] {$v_3$};

  \fill[black!58]
    (Tc)
    circle[radius=1.7pt];

  \fill[teal!70!black]
    (TH)
    circle[radius=2.3pt];

  \fill[orange!80!black]
    (Tx)
    circle[radius=2.3pt];

  \node[
    draw=teal!35,
    fill=white,
    rounded corners=1pt,
    inner sep=1.5pt
  ]
    (Hlabel)
    at (8.72,-2.19)
    {$H_{\alpha,2}(x)$};

  \draw[teal!45]
    (Hlabel)
    to[out=90,in=-70]
    (TH);

  \node[
    anchor=west
  ]
    at ($(Tx)+(0.10,0.10)$)
    {$x$};

  \node[
    orange!72!black,
    align=center,
    fill=white,
    font=\scriptsize,
    rounded corners=1pt,
    inner sep=1.5pt
  ]
    (chordlabel)
    at (4.85,-1.86)
    {Euclidean chord};

  \draw[orange!52]
    (chordlabel)
    to[out=-35,in=180]
    (Tchordmid);

  \node[
    font=\small\bfseries
  ]
    at (7.35,3.3)
    {target $\Sp^1$};

  \node[blue!65!black]
    at ($(TO)+(150:1.43)$)
    {$M_1$};

  \node[orange!80!black]
    at ($(TO)+(-90:1.43)$)
    {$M_2$};

  \node[green!52!black]
    at ($(TO)+(30:1.43)$)
    {$M_3$};
\end{scope}

  \node[draw=teal!35,fill=white,rounded corners=1pt,
    align=center,inner sep=1.5pt] (qlabel) at (2.08,1.88)
      {$q:=p_\alpha(x)$\\[-1pt]$r_1(\alpha)=2$};
  \draw[teal!58] (qlabel) to[out=-55,in=115] (Sq);

  \draw[
    coupling,
    teal!62!black
  ]
    (Sq)
    to[out=-4,in=175]
      node[
        pos=.53,
        above,
        sloped,
        fill=white,
        inner sep=1pt
      ]
        {paired in $\mathcal R_1$}
    (TH);

  \draw[
    coupling,
    purple!58!black
  ]
    (Sqbeta)
    to[out=-8,in=245]
      node[
        pos=.52,
        below,
        sloped,
        fill=white,
        inner sep=1pt
      ]
        {identity pairing}
    (Tx);

  \node[
    draw=orange!45,
    fill=orange!6,
    rounded corners=2pt,
    align=center,
    inner sep=3pt
  ]
    (Nfiber)
    at (3.92,2.88)
    {north-pole fiber\\
     $\mathcal R_1[N]=\{v_1,v_2,v_3\}$};

  \draw[
    orange!45,
    semithick
  ]
    (spNorth)
    to[out=8,in=180]
    (Nfiber);

  \node[
    black!58,
    align=center
  ]
    at (3.60,-2.82)
    {odd extension:
     $(q,z)\longmapsto(-q,-z)$};

\end{tikzpicture}
 }
\caption{The optimal correspondence
\(\mathcal R_1\subseteq\Sp^2\times\Sp^1\).  The latitude circles at
colatitudes \(\widehat{\zeta}_1=\tfrac{\pi}{3}\) and
\(\pi-\widehat{\zeta}_1=\tfrac{2\pi}{3}\) bound the identity band, where
\(p_\beta(x)\) is paired with the equatorial point \(x\).  Above the
band, the colored regions correspond to the closed Voronoi cells
\(M_1,M_2,M_3\), and the lower regions are obtained antipodally.  The
highlighted pair has \(x\in M_2\) and \(r_1(\alpha)=2\), so
\(H_{\alpha,2}(x)\) is the radial normalization of the point two-thirds
of the way along \([v_2,x]\).  At the north pole,
\(\mathcal R_1[N]=\{v_1,v_2,v_3\}\).}

\label{fig:anchored-chord-n-one}
\end{figure}
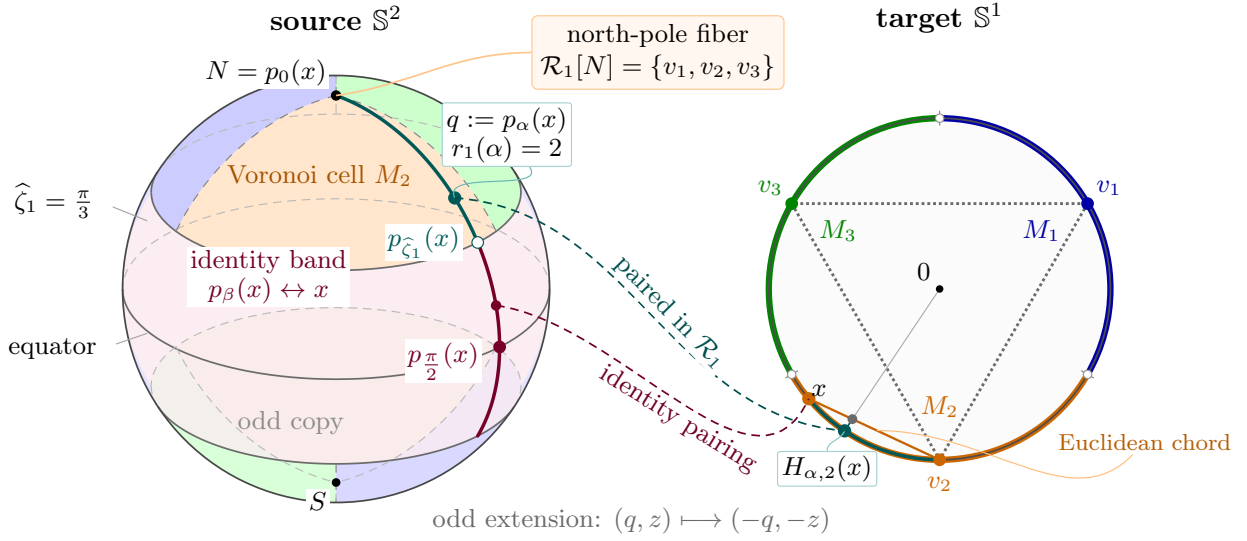
 
\subsection{Sharpness mechanism and proof strategy}
\label{sec:proof-strategy}

The north pole already explains the value of the final bound.  Since
\(p_0(x)\) is independent of \(x\) and \(r_n(0)=0\), its fiber is
\[
 \mathcal R_n[N]:=\{z:(N,z)\in\mathcal R_n\}
 =\{v_1,\ldots,v_{n+2}\}.
\]

The diameter of this regular simplex is \(\zeta_n\).  Thus the
construction cannot have distortion smaller than \(\zeta_n\), and the
upper-bound proof must show that no other pair has a larger defect.

As explained in \Cref{sec:anchored-chord-definition}, sharpness on the
equal-level swapped-anchor family and reciprocal compatibility under
\(F_n\) select the gain \(r_n\) within the normalized-chord ansatz;
\Cref{prop:canonical-gain-characterization} proves the corresponding
uniqueness statement.  The same family realizes the lower boundary of
the admissible \((Q,R)\)-region below.  To prove the global upper bound,
it remains to show that every other realized comparison also belongs to
that region.
Consider two anchored--chord pairs
\[
 (q,z)=(p_\alpha(x),H_{\alpha,i}(x)),\qquad
 (q',z')=(p_\beta(y),H_{\beta,j}(y)).
\]
Define their \emph{source inner product} and \emph{target inner product} by
\phantomsection\label{def:source-target-inner-products}
\[
 Q:=q\cdot q',\qquad R:=z\cdot z'.
\]
The corresponding distances are \(\arccos Q\) and \(\arccos R\).
We call \((Q,R)\) \emph{admissible} if the associated source and target
distances have defect at most \(\zeta_n\), and write
\phantomsection\label{def:admissible-region}
\[
 \mathcal A_n
 :=\bigl\{(Q,R)\in[-1,1]\times[-1,1]:
       \big|\arccos Q-\arccos R\big|\leq\zeta_n\bigr\}
\]
for the admissible \((Q,R)\)-region.  Thus, for the two anchored--chord
pairs under consideration, the upper-bound problem is to show that their
realized \((Q,R)\)-pair belongs to \(\mathcal A_n\).

The geometry of \(\mathcal A_n\) divides the proof into three
\(Q\)-ranges.
Using \(\cos\zeta_n=-\rho_n\), define its upper and lower boundary
functions by
\begin{equation}\label{eq:admissible-boundary-functions}
 U_n(Q):=-\rho_nQ+\sin\zeta_n\sqrt{1-Q^2},\qquad
 L_n(Q):=-\rho_nQ-\sin\zeta_n\sqrt{1-Q^2}.
\end{equation}
Then
\[
 \boxed{
 \begin{gathered}
 (Q,R)\in\mathcal A_n
 \quad\Longleftrightarrow\quad
 \left\{
 \begin{array}{r@{\;}c@{\;}c@{\;}c@{\;}l@{\qquad}r@{\;}c@{\;}c@{\;}c@{\;}l}
 {}     &{}   &R&\leq&U_n(Q),&-1     &\leq&Q&\leq&-\rho_n,\\[2pt]
 -1     &\leq&R&\leq&1,     &-\rho_n&\leq&Q&\leq& \rho_n,\\[2pt]
 L_n(Q) &\leq&R,&{}  &{}    & \rho_n&\leq&Q&\leq&1.
 \end{array}
 \right.
 \end{gathered}}
\]
Indeed, in the middle strip the source distance lies between
\(\widehat{\zeta}_n\) and \(\zeta_n\), so every target distance in
\([0,\pi]\) gives an admissible pair.  The other two ranges have one
nontrivial boundary each.  These boundaries, together with the smaller
region defined by the bounds proved below for all realized
distinct-anchor \((Q,R)\)-pairs, are shown in
\Cref{fig:qr-distortion-region}.

\begin{figure}[!ht]
\centering
\begingroup
\def\qrlensrho{0.3333333333}\resizebox{0.88\textwidth}{!}{\providecommand{\qrlensrho}{0.3333333333}
\pgfmathsetmacro{\qrlenss}{sqrt(1-\qrlensrho*\qrlensrho)}

\begin{tikzpicture}[
  x=2.45cm,
  y=2.45cm,
  line cap=round,
  line join=round,
  every node/.style={font=\small},
  axis/.style={black!70,thin},
  threshold/.style={black!45,densely dotted},
  exact/.style={blue!65!black,very thick},
  stronger/.style={orange!85!black,very thick},
  exact ghost/.style={blue!65!black,thick,dashed},
  region/.style={fill=blue!10},
  proved region/.style={fill=red!9!white},
  middle hatch/.style={draw=black!55,opacity=.16,line width=.3pt}
]

\newcommand{\qrlensaxes}{\draw[axis,-{Latex[length=1.7mm]}] (-1.08,0)--(1.12,0)
    node[below right=-1pt] {$Q$};
  \draw[axis,-{Latex[length=1.7mm]}] (0,-1.08)--(0,1.12)
    node[above left=-1pt] {$R$};
  \foreach \t/\lab in {-1/$-1$,1/$1$}{
    \draw[axis] (\t,0.018)--(\t,-0.018) node[below=2pt] {\lab};
    \draw[axis] (0.018,\t)--(-0.018,\t) node[left=2pt] {\lab};
  }
  \draw[threshold] (-\qrlensrho,-1)--(-\qrlensrho,1);
  \draw[threshold] ( \qrlensrho,-1)--( \qrlensrho,1);
  \node[below=3pt,fill=white,inner sep=1pt] at (-\qrlensrho,0)
    {$-\rho_n$};
  \node[below=3pt,fill=white,inner sep=1pt] at ( \qrlensrho,0)
    {$\rho_n$};
}

\newcommand{\qrlensmiddlehatch}{\begin{scope}
    \clip (-\qrlensrho,-1) rectangle (\qrlensrho,1);
    \foreach \h in {-2.4,-2.3,...,2.4}{
      \draw[middle hatch]
        (-1.2,{\h-1.2})--(1.2,{\h+1.2});
    }
  \end{scope}
}

\node[
  anchor=south west,
  align=left,
  font=\footnotesize,
  draw=black!25,
  fill=white,
  rounded corners=1pt,
  inner xsep=4pt,
  inner ysep=3pt
] at (-1.08,1.72)
  {$U_n(Q):=-\rho_nQ+\sin\zeta_n\sqrt{1-Q^2}$\qquad\qquad
   $L_n(Q):=-\rho_nQ-\sin\zeta_n\sqrt{1-Q^2}$};

\begin{scope}
  \path[region]
    (-1,-1)--(\qrlensrho,-1)
    --plot[domain=\qrlensrho:1,samples=100]
      (\x,{-\qrlensrho*\x-\qrlenss*sqrt(1-\x*\x)})
    --(1,1)--(-\qrlensrho,1)
    --plot[domain=-\qrlensrho:-1,samples=100]
      (\x,{-\qrlensrho*\x+\qrlenss*sqrt(1-\x*\x)})
    --cycle;
  \qrlensmiddlehatch
  \draw[exact,domain=\qrlensrho:1,samples=100]
    plot (\x,{-\qrlensrho*\x-\qrlenss*sqrt(1-\x*\x)});
  \draw[exact,domain=-1:-\qrlensrho,samples=100]
    plot (\x,{-\qrlensrho*\x+\qrlenss*sqrt(1-\x*\x)});
  \qrlensaxes
  \node[align=center] at (0,1.34) {admissible region $\mathcal A_n$};
  \node[blue!65!black,align=center,font=\scriptsize]
    at (0.55,-0.47) {$R\geq L_n(Q)$};
  \node[blue!65!black,align=center,font=\scriptsize]
    at (-0.55,0.47) {$R\leq U_n(Q)$};
  \node at (0,-1.29) {\textup{(a)}};
\end{scope}

\begin{scope}[xshift=6cm]
  \path[proved region]
    (-1,-1)--(\qrlensrho,-1)
    --plot[domain=\qrlensrho:1,samples=100]
      (\x,{-\qrlensrho*\x-\qrlenss*sqrt(1-\x*\x)})
    --(1,1)--(-\qrlensrho,1)--(-1,\qrlensrho)--cycle;
  \qrlensmiddlehatch
  \draw[exact,domain=\qrlensrho:1,samples=100]
    plot (\x,{-\qrlensrho*\x-\qrlenss*sqrt(1-\x*\x)});
  \draw[exact ghost,domain=-1:-\qrlensrho,samples=100]
    plot (\x,{-\qrlensrho*\x+\qrlenss*sqrt(1-\x*\x)});
  \draw[stronger] (-1,\qrlensrho)--(-\qrlensrho,1);
  \qrlensaxes
  \node[align=center] at (0,1.50)
    {enclosure of realized $(Q,R)$-pairs\\[-1pt]
     with distinct anchors};
  \node[blue!65!black,font=\scriptsize,align=center,anchor=south]
    at (-0.80,1.00)
    {\Cref{lem:ac-far}\\[-1pt]$R\leq U_n(Q)$};
  \node[font=\scriptsize,align=center,anchor=north]
    at (-0.45,0.48)
    {\Cref{prop:ac-far-affine}\\[-1pt]$R\leq 1+\rho_n+Q$};
  \node[font=\scriptsize,align=center] at (0,-0.72)
    {no additional\\[-1pt]constraint on \(R\)};
  \node[font=\scriptsize,align=center]
    at (0.57,-0.43)
    {\Cref{lem:ac-near}\\[-1pt]$R\geq L_n(Q)$};
  \node at (0,-1.29) {\textup{(b)}};
\end{scope}

\end{tikzpicture}
 }
\endgroup
\caption{The admissible \((Q,R)\)-region and the bounds satisfied by
realized distinct-anchor pairs, illustrated for \(n=2\), so
\(\rho_n=1/3\).  Here \(Q\) and
\(R\) are the source and target inner products.  Panel~\textup{(a)}
shows the complete admissible region \(\mathcal A_n\), with its middle strip lightly
hatched; the legend defines its boundary functions \(U_n\) and \(L_n\).
In panel~\textup{(b)}, \Cref{lem:ac-near} proves the lower-boundary
inequality \(R\geq L_n(Q)\); equality is attained for every
\(Q\in[\rho_n,1]\) by the equal-level swapped-anchor family from
\Cref{sec:swapped-anchor-configuration}.  On the other side,
\Cref{prop:ac-far-affine} proves the stronger straight-line bound
\(R\leq1+\rho_n+Q\), which implies the upper-boundary inequality
\(R\leq U_n(Q)\) in \Cref{lem:ac-far}; the dashed curve is this upper
boundary of the full admissible region.  The spherical-cap optimization
underlying this upper estimate is reduced exactly to a one-variable
concave maximization in \Cref{prop:ac-cap-dimension-reduction}.
\label{fig:qr-distortion-region}}
\end{figure}
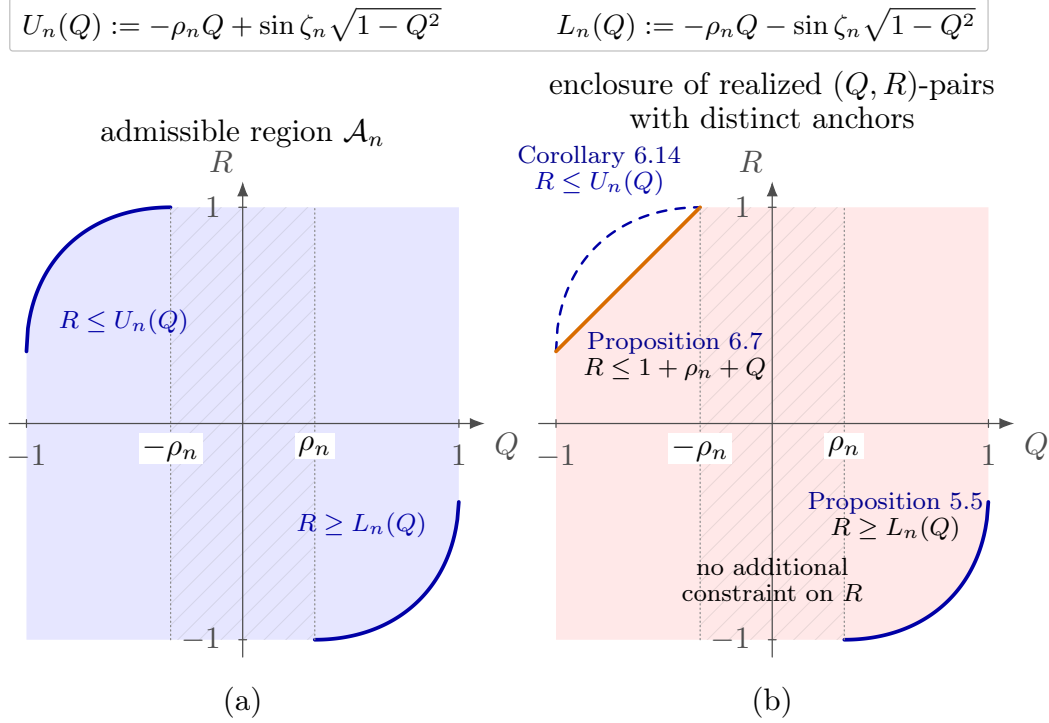
\FloatBarrier

For the remainder of this overview, assume \(n\geq2\); the circle case is
treated separately in \Cref{prop:ac-n-one}.

\begin{itemize}
 \item If \(i=j\), the vertex-dependent isometric embedding of
 \Cref{lem:ac-common-anchor-pointwise} places each embedded target within
 \(\widehat{\zeta}_n/2\) of its source.  The triangle inequality then
 gives the same-anchor estimate in \Cref{lem:ac-same-anchor}.

 \item Suppose \(i\neq j\).  The three ranges are shown in
 panel~\textup{(b)} of \Cref{fig:qr-distortion-region}.  In the middle
 range \(-\rho_n<Q<\rho_n\), admissibility is automatic.  For
 \(Q\geq\rho_n\), Ptolemy's inequality and a direct comparison of
 spherical distances give the lower-boundary result
 \Cref{lem:ac-near}.  For \(Q\leq-\rho_n\), the exact reduction of an
 optimization over two spherical caps to a one-variable concave
 maximization gives the stronger affine estimate in
 \Cref{prop:ac-far-affine}, and hence the upper-boundary result
 \Cref{lem:ac-far}.
\end{itemize}
Thus every realized \((Q,R)\)-pair belongs to \(\mathcal A_n\).  The
critical involution of \Cref{prop:ac-critical-involution} and the
reciprocal identity for \(r_n\) connect the low- and high-radial ranges
used in the distinct-anchor estimates.

Finally, \Cref{lem:ac-identity-band} handles comparisons involving the
identity band, while \Cref{lem:relation-helmet-trick} shows that
antipodal extension preserves distortion.  The case \(n=1\) is completed separately in
\Cref{prop:ac-n-one}, using the lower-boundary estimate of
\Cref{lem:ac-near}; for \(n\geq2\), all estimates are assembled in
\Cref{sec:anchored-chord-assembly} to prove \Cref{thm:main}.

\subsection{A self-contained proof of the lower bound}
\label{sec:self-contained-lower-bound}

The lower bound is an instance of what is now called a
\emph{quantitative Borsuk--Ulam theorem}: rather than merely ruling out a
continuous odd map, it quantifies how discontinuous an arbitrary odd map
must be.  The adjacent-dimensional result at the core of this viewpoint
was first established by Dubins and Schwarz
\cite{dubins1981equidiscontinuity}.  Lim, M\'emoli and Smith adapted this
quantitative obstruction to the Gromov--Hausdorff distance; see \cite[Section~5 and Appendix~A]{lim2021gromov} for
the broader context and  \cite{adams2022gromov} for an extension.

To keep the paper self-contained, we give a direct proof of the known lower
bound, different in organization from the argument in
\cite{lim2021gromov}.

\begin{proposition}[Adjacent-sphere lower bound]
\label{prop:adjacent-lower-bound}
For every integer \(n\geq1\) and every function
\(f\colon\Sp^{n+1}\to\Sp^n\),
\begin{equation}\label{eq:function-lower-bound}
 \dis(f)\geq\zeta_n.
\end{equation}
Consequently, every correspondence \(\mathcal R\) between \(\Sp^n\) and
\(\Sp^{n+1}\) satisfies \(\dis(\mathcal R)\geq\zeta_n\), and therefore
\begin{equation}\label{eq:gh-lower-bound}
 2\dgh(\Sp^n,\Sp^{n+1})\geq\zeta_n.
\end{equation}
More generally, for every \(d\geq1\), every correspondence
\(\mathcal R\subseteq\Sp^{n+d}\times\Sp^n\) satisfies
\(\dis(\mathcal R)\geq\zeta_n\).  Consequently,
\begin{equation}\label{eq:higher-codimension-lower-bound}
 \dgh(\Sp^n,\Sp^{n+d})\geq\frac{\zeta_n}{2}.
\end{equation}
\end{proposition}

Its directness comes from isolating two ingredients---the helmet trick and a
sharp finite open-hemisphere Jung-type lemma (see
\cite{dekster1995jung} and \cite[Lemma~A.1]{lim2021gromov})---and then
applying the classical Borsuk--Ulam theorem, without passing through the more
general Dubins--Schwarz formalism used in \cite{lim2021gromov}.
Although the principal contribution of the paper is the matching
upper-bound construction, this argument provides a streamlined and
transparent proof of the lower-bound mechanism.

The function-valued helmet trick appears in
\cite[Lemma~5.7 and Corollary~5.8]{lim2021gromov}; its relation-valued
extension is \cite[Proposition~2.2]{rodriguezmartin2024novel}.  We
reproduce both short arguments for self-containedness.
Recall that an antipodal correspondence is invariant under simultaneous
negation of its two coordinates.

\begin{lemma}[Helmet trick for relations]
\label{lem:relation-helmet-trick}
Let \(m,n\geq1\).  For a nonempty relation
\(\mathcal R\subseteq\Sp^m\times\Sp^n\), put
\[
 -\mathcal R:=\{(-x,-y):(x,y)\in\mathcal R\}.
\]
Then \(\mathcal R\cup(-\mathcal R)\) is antipodal; if
\(\mathcal R\) is a correspondence, so is
\(\mathcal R\cup(-\mathcal R)\).  Moreover,
\[
 \dis(\mathcal R\cup-\mathcal R)=\dis(\mathcal R).
\]
\end{lemma}

\begin{proof}
The first two assertions follow directly from the definitions.
Compare \((\varepsilon x,\varepsilon y)\) and
\((\delta x',\delta y')\), where
\((x,y),(x',y')\in\mathcal R\) and
\(\varepsilon,\delta\in\{\pm1\}\).  Equal signs leave both distances
unchanged, while opposite signs replace both by their complements to
\(\pi\).  Thus every defect in \(\mathcal R\cup-\mathcal R\) is a
defect in \(\mathcal R\); the reverse inequality follows from
\(\mathcal R\subseteq\mathcal R\cup-\mathcal R\).
\end{proof}

\begin{corollary}[Helmet trick for functions]
\label{lem:helmet-trick}
For every function \(f\colon\Sp^{n+1}\to\Sp^n\), there is an odd function
\(\widehat f\colon\Sp^{n+1}\to\Sp^n\) such that
\(\dis(\widehat f)\leq\dis(f)\).
\end{corollary}

\begin{proof}
For \(x\in\Sp^k\subset\R^{k+1}\), let
\(\ell(x):=\max\{j:x_j\neq0\}\), and define the \emph{helmet}
\[
 \mathcal H_k:=\{x\in\Sp^k:x_{\ell(x)}>0\}.
\]
Thus \(\mathcal H_k\cap(-\mathcal H_k)=\varnothing\) and
\(\mathcal H_k\cup(-\mathcal H_k)=\Sp^k\).  This is a concrete form of
the recursively defined helmet of Lim--M\'emoli--Smith.  Apply
\Cref{lem:relation-helmet-trick} to
\(\mathcal R:=\operatorname{graph}(f|_{\mathcal H_{n+1}})\).
Its antipodal extension is the graph of the odd function determined by
\(\widehat f(u):=f(u)\) and
\(\widehat f(-u):=-f(u)\) for \(u\in\mathcal H_{n+1}\).  Hence
\(\dis(\widehat f)=\dis(f|_{\mathcal H_{n+1}})\leq\dis(f)\).
\end{proof}

\begin{lemma}[Spherical Jung theorem]
\label{lem:spherical-jung}
Let \(q_1,\ldots,q_r\in\Sp^n\), where \(1\leq r\leq n+2\).  If
\[
 d_n(q_i,q_j)<\zeta_n
 \qquad(i\neq j),
\]
then \(q_1,\ldots,q_r\) lie in a common open hemisphere of \(\Sp^n\).
Equivalently,
\[
 0\notin\operatorname{conv}\{q_1,\ldots,q_r\}.
\]
\end{lemma}

\begin{proof}
For a finite subset of a sphere, containment in an open hemisphere is
equivalent, by strict hyperplane separation, to exclusion of the origin
from its Euclidean convex hull.  Suppose instead that
\[
 0=\sum_{i=1}^r\lambda_iq_i,
 \qquad
 \lambda_i\geq0,
 \qquad
 \sum_{i=1}^r\lambda_i=1.
\]
Discard zero coefficients and relabel.  The case \(r=1\) is impossible,
so \(r\geq2\).  Put \(S:=\sum_i\lambda_i^2\).  Since
\(\sum_i\lambda_i=1\), one has
\(2\sum_{i<j}\lambda_i\lambda_j=1-S\).  Therefore
\(q_i\cdot q_j>\cos\zeta_n=-1/(n+1)\) for \(i\neq j\) gives
\[
 0=\left\|\sum_i\lambda_iq_i\right\|^2
 >S-\frac{1-S}{n+1}
 =\frac{(n+2)S-1}{n+1}\geq0.
\]
The last inequality follows from
\(S\geq1/r\geq1/(n+2)\), a contradiction.
\end{proof}

\begin{proof}[Proof of Proposition \ref{prop:adjacent-lower-bound}]
Suppose that \(\dis(f)<\zeta_n\).  By \Cref{lem:helmet-trick}, there is
an odd function \(\widehat f\colon\Sp^{n+1}\to\Sp^n\) with
\(\dis(\widehat f)\leq\dis(f)\).  Choose \(\varepsilon>0\) such that
\[
 \dis(f)+\varepsilon<\zeta_n.
\]
Choose an antipodally symmetric triangulation \(K_\varepsilon\) of
\(\Sp^{n+1}\)
for which the antipodal map is simplicial and
\[
 \operatorname{mesh}(K_\varepsilon)
 :=\max_{\sigma\in K_\varepsilon}
   \operatorname{diam}_{d_{n+1}}(\sigma)
 <\varepsilon.
\]
Such triangulations are obtained, for example, from sufficiently many
antipodal barycentric subdivisions of the boundary of the cross-polytope,
realized equivariantly on the sphere.

Let \(\sigma\in K_\varepsilon\).  For any two vertices \(v,w\) of
\(\sigma\),
\[
 d_n(\widehat f(v),\widehat f(w))
 \leq d_{n+1}(v,w)+\dis(\widehat f)
 <\varepsilon+\dis(f)<\zeta_n.
\]
Because \(K_\varepsilon\) has dimension \(n+1\), the simplex
\(\sigma\) has at most
\(n+2\) vertices.  Hence \Cref{lem:spherical-jung} gives
\begin{equation}\label{eq:simplex-convex-hull-avoids-zero}
 0\notin\operatorname{conv}
 \{\widehat f(v):v\in\operatorname{Vert}(\sigma)\}.
\end{equation}

Extend the vertex map piecewise linearly.  Namely, if \(x\in\sigma\) has
barycentric coordinates \((\lambda_v)_{v\in\operatorname{Vert}(\sigma)}\),
put
\[
 G(x):=\sum_{v\in\operatorname{Vert}(\sigma)}
       \lambda_v\widehat f(v)\in\R^{n+1}.
\]
The formulas agree on common faces, so
\(G\colon\Sp^{n+1}\to\R^{n+1}\) is continuous.
Equation~\eqref{eq:simplex-convex-hull-avoids-zero} shows that
\(G(x)\neq0\) everywhere.  Since \(K_\varepsilon\) is antipodally
simplicial and
the vertex map is odd, \(G(-x)=-G(x)\).  Therefore
\[
 F(x):=\frac{G(x)}{\lVert G(x)\rVert}
\]
is a continuous odd map \(F\colon\Sp^{n+1}\to\Sp^n\), contradicting the
Borsuk--Ulam theorem.  This proves \eqref{eq:function-lower-bound}.

Finally, let
\(\mathcal R\subseteq\Sp^n\times\Sp^{n+1}\) be a correspondence.  For
each \(x\in\Sp^{n+1}\), choose \(f(x)\in\Sp^n\) with
\((f(x),x)\in\mathcal R\).  Then
\(\dis(f)\leq\dis(\mathcal R)\), so
\eqref{eq:function-lower-bound} yields
\(\dis(\mathcal R)\geq\zeta_n\).  Taking the infimum over
correspondences in \eqref{eq:gh-correspondence-formula} proves
\eqref{eq:gh-lower-bound}.

For the final assertion, let \(d\geq1\), restrict the source projection
of \(\mathcal R\subseteq\Sp^{n+d}\times\Sp^n\) to an equatorial
\(\Sp^{n+1}\subseteq\Sp^{n+d}\), and choose \(f(x)\in\Sp^n\) with
\((x,f(x))\in\mathcal R\).  Then
\(\operatorname{graph}(f)\subseteq\mathcal R\), so
\(\dis(\mathcal R)\geq\dis(f)\geq\zeta_n\) by
\eqref{eq:function-lower-bound}.  Taking the infimum over
correspondences gives \eqref{eq:higher-codimension-lower-bound}.
\end{proof}

\section{\texorpdfstring{Spherical joins of correspondences and bounds for nonconsecutive spheres}{Spherical joins of correspondences and bounds for nonconsecutive spheres}}
\label{sec:spherical-join-consequences}

The family \((\mathcal R_n)_{n\geq1}\) of sharp relations from
\Cref{thm:main} supplies the inputs for the applications of
\Cref{thm:exact-join-distortion}.  We first prove that theorem and then use
it in three ways: suspension stabilizes a correspondence without changing
its distortion, balanced joins combine {sharp
consecutive-sphere correspondences} \(\mathcal R_n\),
and joins of exact circle correspondences give bounds in every
codimension.  Throughout this section, subscripts on \(O\) and \(\Omega\)
record the parameters on which the implicit constants may depend; the
remaining variable tends to infinity.  Thus \(O_d(m^{-2})\) is taken as
\(m\to\infty\) with \(d\) fixed, while \(O_m(n^{-2})\) and
\(\Omega_m(n^{-1/2})\) are taken as \(n\to\infty\) with \(m\) fixed.  As
usual, \(d=o(m)\) means \(d/m\to0\).

\subsection{Synchronized spherical joins of correspondences}

We begin the proof of \Cref{thm:exact-join-distortion} with the following
distance estimate.

\begin{lemma}[A join-distance estimate]
\label{lem:join-distance-estimate}
Let \(a,b\geq0\) and \(a+b\leq1\).  Then
\(\psi(u_1,u_2):=\arccos(a\cos u_1+b\cos u_2)\) is \(1\)-Lipschitz
on \([0,\pi]^2\) for the \(\ell_\infty\)-metric.
\end{lemma}

\begin{proof}
At every point where the derivatives are finite, the $\ell_1$ norm of the gradient of $\psi$ is
\[
 \frac{a\sin u_1+b\sin u_2}
 {\sqrt{1-(a\cos u_1+b\cos u_2)^2}} \leq 1.
\]
If the denominator
vanishes, replace \((a,b)\) by
\((1-\varepsilon)(a,b)\), apply the preceding estimate, and pass to the
uniform limit as \(\varepsilon\downarrow0\). This proves the global
Lipschitz estimate.  The fundamental theorem of calculus along
the segment joining \(u,v\in[0,\pi]^2\), together with
\(\ell_1\)--\(\ell_\infty\) duality, gives the required estimate.
\end{proof}

\begin{proof}[Proof of \Cref{thm:exact-join-distortion}]
We first prove the result for \(k=2\).
Every point of \(\Sp^{m_1+m_2+1}\) has a block decomposition
\((x_1\cos t,x_2\sin t)\), with arbitrary choices for a component
whose coefficient vanishes.  Choose \(y_i\) with
\((x_i,y_i)\in\mathcal R^{(i)}\); this gives a joined pair above the chosen
source point.  Applying the same argument to the target blocks proves
surjectivity of the other coordinate projection.

For \(i=1,2\), choose
\((x_i,y_i),(x_i',y_i')\in\mathcal R^{(i)}\), and form the two
elements of the joined relation given by
\eqref{eq:synchronized-join-definition} with join parameters \(s,t\).

Write \(D_i^{\mathrm{src}}:=d_{m_i}(x_i,x_i')\) and
\(D_i^{\mathrm{tar}}:=d_{n_i}(y_i,y_i')\), \(i=1,2\), for their
component distances, and put \(a:=\cos s\cos t\) and
\(b:=\sin s\sin t\).  Then
\(a,b\geq0\) and \(a+b=\cos(s-t)\leq1\).

Indeed, the source inner product is
\(a(x_1\cdot x_1')+b(x_2\cdot x_2')
=a\cos D_1^{\mathrm{src}}+b\cos D_2^{\mathrm{src}}\), and
the target inner product is
\(a(y_1\cdot y_1')+b(y_2\cdot y_2')
=a\cos D_1^{\mathrm{tar}}+b\cos D_2^{\mathrm{tar}}\).

The source and target distances between the two joined pairs are
\(\psi(D_1^{\mathrm{src}},D_2^{\mathrm{src}})\) and
\(\psi(D_1^{\mathrm{tar}},D_2^{\mathrm{tar}})\), respectively, where
\(\psi\) is the function defined in
\Cref{lem:join-distance-estimate}.
By \Cref{lem:join-distance-estimate},
\[
 \left|
 \psi(D_1^{\mathrm{src}},D_2^{\mathrm{src}})
 -
 \psi(D_1^{\mathrm{tar}},D_2^{\mathrm{tar}})
 \right|
 \leq
 \max_{i=1,2}|D_i^{\mathrm{src}}-D_i^{\mathrm{tar}}|.
\]
Taking the supremum gives
\(\dis(\mathcal R^{(1)}*\mathcal R^{(2)})
\leq\max\{\dis(\mathcal R^{(1)}),\dis(\mathcal R^{(2)})\}\).

At the endpoint \(t=0\), the joined relation contains an isometric copy
of \(\mathcal R^{(1)}\); at \(t=\tfrac{\pi}{2}\), it contains an isometric copy of
\(\mathcal R^{(2)}\).  These two slices give the reverse inequality and prove
the distortion formula for \(k=2\).  If both factors are antipodal, then
their join is antipodal: replace each component pair \((x_i,y_i)\) by
\((-x_i,-y_i)\).  Applying the two-factor result successively proves all
assertions by induction on \(k\).
\end{proof}

\Cref{thm:exact-join-distortion} turns the construction of joined correspondences with prescribed
dimensions into a finite minimax allocation problem: within a chosen family
of factors, one minimizes the largest factor distortion.  For joins of
{sharp consecutive-sphere correspondences}
\(\mathcal R_n\), a factor
with target \(\Sp^{a_i}\) has
distortion \(\zeta_{a_i}\); since \(\zeta_j\) decreases with \(j\), the
integers \(a_i+1\) should be distributed as evenly as possible.  For joins of
exact circle correspondences, a factor receiving codimension \(d_i\) has
distortion \(\delta(\lceil \tfrac{d_i}{2}\rceil)\), with \(\delta\) defined in
\eqref{eq:circle-delta}; since \(\delta\) increases, the
codimensions \(d_i\) should likewise be distributed as evenly as the integral
constraints permit.  {These ideas lead,
respectively, to a construction from balanced joins of sharp
consecutive-sphere correspondences and to a construction from joins of
exact circle correspondences.}

\subsection{Suspension of correspondences and stabilization}

We regard \(\Sp^0=\{-1,1\}\) as the unit zero-sphere, whose two points
are at geodesic distance \(\pi\).

\begin{definition}[Synchronized spherical suspension]
\label{def:synchronized-spherical-suspension}
Let \(\mathcal R\subseteq\Sp^m\times\Sp^n\) be a correspondence, and let
\(\Delta_{\Sp^0}\subseteq\Sp^0\times\Sp^0\) be the identity
correspondence.  The \emph{synchronized spherical suspension} of
\(\mathcal R\) is the relation
\(\Sigma\mathcal R\subseteq\Sp^{m+1}\times\Sp^{n+1}\) defined by
\[
 \Sigma\mathcal R:=\mathcal R*\Delta_{\Sp^0}
 =\left\{\bigl((x\sin t,\cos t),(y\sin t,\cos t)\bigr):
 (x,y)\in\mathcal R,\ 0\leq t\leq\pi\right\}.
\]
\end{definition}

\begin{corollary}[Suspension preserves distortion]
\label{cor:suspension-isometric}
For every correspondence \(\mathcal R\),
\begin{equation}\label{eq:suspension-isometric}
 \dis(\Sigma\mathcal R)=\dis(\mathcal R).
\end{equation}
Consequently,
\begin{equation}\label{eq:diagonal-gh-monotonicity}
 \dgh(\Sp^{m+1},\Sp^{n+1})
 \leq\dgh(\Sp^m,\Sp^n).
\end{equation}
In particular, for every fixed \(d\geq1\), the sequence
\[
 m\longmapsto\dgh(\Sp^m,\Sp^{m+d})
\]
is nonincreasing.
\end{corollary}

\begin{proof}
The identity relation of \(\Sp^0\) has distortion zero, so
\eqref{eq:suspension-isometric} is
\Cref{thm:exact-join-distortion}.  Applying suspension to an arbitrary
correspondence and then taking the infimum in
\eqref{eq:gh-correspondence-formula} gives
\eqref{eq:diagonal-gh-monotonicity}.
\end{proof}

\begin{remark}[Earlier use of spherical suspension]
\label{rem:earlier-map-suspension}
Spherical suspension was used for maps in
\cite[Lemma~4.3 and Corollary~4.4]{lim2021gromov} to construct
continuous, surjective, antipode-preserving maps between spheres of
arbitrary dimensions; these maps have distortion strictly less than
\(\pi\), as observed in \cite[the proof of Theorem~A]{lim2021gromov}.
\Cref{def:synchronized-spherical-suspension}
extends the same operation from graphs of maps to arbitrary
correspondences.
\end{remark}

The suspension operation gives a direct all-dimensional stabilization of
the first codimension-two example.  Following the
orientation used for \(\mathcal R_n\), we transpose correspondences from
the literature when necessary so that the higher-dimensional sphere
appears in the first factor.  Transposition preserves both the
correspondence property and distortion.  Let
 $\mathcal L_{3,1}\subseteq\Sp^3\times\Sp^1$
be the Lim--M\'emoli--Smith correspondence constructed in
\cite[Proposition~1.18 and Section~7]{lim2021gromov}.  It is based on the
rotation procedure they describe as reminiscent of the Hopf fibration.
Their construction and lower bound give $\dis(\mathcal L_{3,1})=\tfrac{2\pi}{3}.$

\begin{corollary}[Suspended Lim--M\'emoli--Smith relations]
\label{cor:suspended-lms}
For every \(m\geq1\), the correspondence
\[
 \mathcal L_{m+2,m}
 :=
 \Sigma^{m-1}\mathcal L_{3,1}
 \subseteq\Sp^{m+2}\times\Sp^m
\]
has distortion \(\tfrac{2\pi}{3}\).  Consequently,
\begin{equation}\label{eq:uniform-gap-two-bound}
 \frac{\zeta_m}{2}
 \leq\dgh(\Sp^m,\Sp^{m+2})
 \leq\frac\pi3.
\end{equation}
The two bounds agree when \(m=1\).
\end{corollary}

\begin{proof}
The distortion statement follows by iterating
\Cref{cor:suspension-isometric}.  The upper bound then follows from
\eqref{eq:gh-correspondence-formula}, and the lower bound is
\eqref{eq:higher-codimension-lower-bound}.  For \(m=1\), one has
\(\zeta_1=\tfrac{2\pi}{3}\), so equality follows.
\end{proof}

Suspension does not, however, propagate optimal consecutive-sphere
relations.  It preserves the old distortion while the sharp constant
strictly decreases.
We use below the elementary asymptotic formula
\(\zeta_j=\tfrac{\pi}{2}+1/(j+1)+O(j^{-3})\), which follows from the Taylor
expansion of \(\arccos(-x)\) at \(x=0\).

\begin{remark}[Suspension does not preserve optimality]
\label{cor:suspension-not-optimal}
Let \(n,k\geq1\).  The \(k\)-fold suspension of the sharp relation
 \(\mathcal R_n\) from \Cref{thm:main} satisfies
\[
 \Sigma^k\mathcal R_n
 \subseteq\Sp^{n+k+1}\times\Sp^{n+k},
 \qquad
 \dis(\Sigma^k\mathcal R_n)=\zeta_n.
\]
It is not optimal, because
\[
 \dgh(\Sp^{n+k},\Sp^{n+k+1})
 =\frac{\zeta_{n+k}}{2}<\frac{\zeta_n}{2}.
\]
For fixed \(k\), the excess of its half-distortion over the optimal
Gromov--Hausdorff value is
\begin{equation}\label{eq:suspension-excess}
 \frac12\dis(\Sigma^k\mathcal R_n)
 -\dgh(\Sp^{n+k},\Sp^{n+k+1})
 =
 \frac{k}{2n^2}+O_k(n^{-3}).
\end{equation}
The last identity follows immediately from the preceding asymptotic
formula.
\end{remark}

\subsection{\texorpdfstring{{Balanced joins of sharp
consecutive-sphere correspondences}}{Balanced joins of sharp
consecutive-sphere correspondences}}

Suspension raises both sphere dimensions by one while preserving the
distortion of the original correspondence. For nonconsecutive spheres, balanced joins instead
distribute the dimension gap among several sharp consecutive-sphere correspondences; by
Theorem~B, the resulting distortion is the maximum of their distortions.
For \(m,d\geq1\), put
\[
 b(m,d):=\left\lfloor\frac{m+1}{d}\right\rfloor.
\]

\begin{proposition}[{Correspondences from balanced joins}]
\label{thm:balanced-join-bound}
Let \(m,d\geq1\) satisfy \(2d\leq m+1\).
Write \(m+1=d\,b(m,d)+r\), where \(0\leq r<d\), and define
\begin{equation}\label{eq:balanced-join-relation}
 \mathcal J_{m,d}
 :=
 \Sigma^r\bigl(\mathcal R_{b(m,d)-1}^{*d}\bigr)
 \subseteq\Sp^{m+d}\times\Sp^m.
\end{equation}
Here \(\Sigma^0\) means that no suspension is applied, and a one-fold
join is interpreted as the relation itself.  Then
\(\mathcal J_{m,d}\) is a correspondence and
\begin{equation}\label{eq:balanced-join-distortion}
 \dis(\mathcal J_{m,d})
 =
 \zeta_{b(m,d)-1}
 =
 \arccos\left(-\frac1{b(m,d)}\right).
\end{equation}
Consequently,
\begin{equation}\label{eq:balanced-gh-bracket}
 \frac{\zeta_m}{2}
 \leq \dgh(\Sp^m,\Sp^{m+d})
 \leq \frac{\zeta_{b(m,d)-1}}{2}.
\end{equation}
Moreover, \(\mathcal J_{m,d}\) has the least distortion among all
relations of the form
\(
 \Sigma^h(\mathcal R_{a_1}*\cdots*\mathcal R_{a_d})
 \subseteq\Sp^{m+d}\times\Sp^m,
\)
where \(h\geq0\) and \(a_1,\ldots,a_d\geq1\).
\end{proposition}

\begin{proof}
The hypothesis \(2d\leq m+1\) gives \(b(m,d)\geq2\), so the sharp
relation \(\mathcal R_{b(m,d)-1}\) from \Cref{thm:main} is available.
Before suspension, the target and source dimensions are
\(d\,b(m,d)-1\) and \(d(b(m,d)+1)-1\), respectively.  After \(r\)
suspensions they become \(m\) and \(m+d\).  By
\eqref{eq:iterated-join-distortion} and
\eqref{eq:suspension-isometric},
\(\dis(\mathcal J_{m,d})=\zeta_{b(m,d)-1}\), which proves
\eqref{eq:balanced-join-distortion} and the upper bound in
\eqref{eq:balanced-gh-bracket}.

This value is best among constructions obtained by joining \(d\) sharp
adjacent relations and then applying suspensions.  Indeed, if
\(h\geq0\), \(a_i\geq1\), and
\(\Sigma^h(\mathcal R_{a_1}*\cdots*\mathcal R_{a_d})\) has target
\(\Sp^m\), then
\(\sum_i(a_i+1)+h=m+1\).  Hence
\(\min_i(a_i+1)\leq b(m,d)\), and therefore
\(\max_i\zeta_{a_i}\geq\zeta_{b(m,d)-1}\).

The lower bound is \eqref{eq:higher-codimension-lower-bound}.
\end{proof}

\begin{corollary}[Fixed and sublinear codimension]
\label{cor:balanced-join-asymptotics}
For each fixed \(d\geq1\), as \(m\to\infty\),
\begin{equation}\label{eq:fixed-gap-error}
 0\leq
 \dgh(\Sp^m,\Sp^{m+d})-\frac{\zeta_m}{2}
 \leq
 \frac{d-1}{2(m+1)}+O_d(m^{-2}).
\end{equation}
In particular,
\[
 \lim_{m\to\infty}\dgh(\Sp^m,\Sp^{m+d})=\frac\pi4.
\]
Together with \Cref{cor:suspension-isometric}, this shows that the
fixed-gap sequence decreases to \(\tfrac{\pi}{4}\).
The same limit holds whenever \(d=d(m)=o(m)\).
\end{corollary}

\begin{proof}
For fixed \(d\), one has
\(b(m,d)^{-1}=d/(m+1)+O_d(m^{-2})\); hence the preceding asymptotic
formula, applied to \eqref{eq:balanced-gh-bracket}, gives
\eqref{eq:fixed-gap-error}.  If \(d=o(m)\), then \(b(m,d)\to\infty\),
so both endpoints of \eqref{eq:balanced-gh-bracket} converge to
\(\tfrac{\pi}{4}\).
\end{proof}

The fixed-gap limit in \Cref{cor:balanced-join-asymptotics} answers the
question posed in \cite[Question~7.1]{harrison2023quantitative}.  It also
substantially improves the general estimates in
\cite[Theorems~1.2 and~1.3]{harrison2023quantitative} in this regime:
 {the upper bound obtained from balanced joins for
 $\dgh(\Sp^m,\Sp^{m+d})$ is}
\(\tfrac{\zeta_{b(m,d)-1}}{2}
=\tfrac{\pi}{4}+\tfrac{d}{2(m+1)}+O_d(m^{-2})\),
whereas the upper bounds in those two theorems
both tend to \(\tfrac{\pi}{2}\) as \(m\to\infty\) with \(d\) fixed.

\subsection{The gap-2 family}

For \(m=1,2\), \Cref{cor:suspended-lms} gives the upper bound
\(\tfrac{\pi}{3}\).  For \(m\geq3\),
\Cref{thm:balanced-join-bound} applies: \(b(m,2)=2\) when \(m=3,4\),
while \(b(m,2)\geq3\) when \(m\geq5\).  Together with
\eqref{eq:higher-codimension-lower-bound}, these observations give the
following bounds.

\begin{corollary}[Gap-2]
\label{cor:gap-two-family}
For every \(m\geq1\),
\begin{equation}\label{eq:gap-two-piecewise}
 \frac{\zeta_m}{2}
 \leq\dgh(\Sp^m,\Sp^{m+2})
 \leq
 \begin{cases}
  \displaystyle\frac{\pi}{3},
    &1\leq m\leq4,\\[1.2ex]
  \displaystyle\frac{\zeta_{b(m,2)-1}}2,
    &m\geq5.
 \end{cases}
\end{equation}

\end{corollary}

{The correspondences \(\mathcal J_{m,2}\) used above for
\(m\geq3\) are the following two specializations of
\eqref{eq:balanced-join-relation}.}  If
\(m=2\ell+1\geq3\), then \(b(m,2)=\ell+1\) and \(r=0\), so
\[
 \mathcal J_{m,2}
 =\mathcal R_\ell*\mathcal R_\ell
 \subseteq\Sp^{m+2}\times\Sp^m,
 \qquad
 \dis(\mathcal J_{m,2})=\zeta_\ell.
\]
If \(m=2\ell\geq4\), then \(b(m,2)=\ell\) and \(r=1\), so
\[
 \mathcal J_{m,2}
 =\Sigma(\mathcal R_{\ell-1}*\mathcal R_{\ell-1})
 \subseteq\Sp^{m+2}\times\Sp^m,
 \qquad
 \dis(\mathcal J_{m,2})=\zeta_{\ell-1}.
\]

\subsection{\texorpdfstring{{All-pairs bounds from
exact circle correspondences}}{All-pairs bounds from exact circle
correspondences}}

For \(1\leq m<n\), put \(d:=n-m\).  {The construction
from balanced joins is}
strongest near the diagonal but requires \(2d\leq m+1\).  The join
theorem also combines naturally with the
exact circle correspondences of Harrison and Jeffs and thereby gives a
bound for every \(1\leq m<n\). For every integer \(k\geq0\), put
\begin{equation}\label{eq:circle-delta}
 \delta(k):=\frac{2\pi k}{2k+1},
\end{equation}
For \(k\geq1\), this is the geodesic diameter of the vertices of a
regular \((2k+1)\)-gon in \(\Sp^1\).  The Polymath
collaboration proved that every correspondence
\(\mathcal R\subseteq\Sp^{d+1}\times\Sp^1\) satisfies
\(\dis(\mathcal R)\geq\delta(\lceil \tfrac{d}{2}\rceil)\)
\cite[Theorem~5.1]{adams2022gromov}; Harrison and Jeffs constructed, for
every \(d\geq1\), a correspondence \(\mathcal P_d\) attaining this lower
bound \cite[Theorem~1.1]{harrison2023quantitative}.  Thus
\[
 \mathcal P_d\subseteq\Sp^{d+1}\times\Sp^1,
 \qquad
 \dis(\mathcal P_d)=\delta(\lceil d/2\rceil).
\]
See also the alternative constructions in
\cite[Theorems~1.3 and~1.4, Sections~3 and~4]
{rodriguezmartin2024novel}.
We also put \(\mathcal P_0:=\Delta_{\Sp^1}\), so that the same formula
holds for \(d=0\).

For fixed \(d\geq1\), repeatedly suspending the 
correspondence \(\mathcal P_d\) gives, for every \(m\geq1\), a
correspondence between \(\Sp^{m+d}\) and \(\Sp^m\) with the same
distortion.  {The stronger all-pairs construction in
\Cref{thm:circle-join-all-pairs} joins several of these exact circle
correspondences and distributes the codimension among them.}

For \(m\geq1\), define the \emph{band width}
\[
 L_m:=2b(m,2)
 =2\left\lfloor\frac{m+1}{2}\right\rfloor
 =
 \begin{cases}
  m,&m\ \text{even},\\
  m+1,&m\ \text{odd}.
 \end{cases}
\]
{For \(1\leq m<n\), put
\[
 K(m,n):=\left\lceil\frac{n-m}{L_m}\right\rceil.
\]}

\begin{proposition}[{Correspondences from circle factors}]
\label{thm:circle-join-all-pairs}
Let \(1\leq m<n\), and write
\(m+1=2b(m,2)+s\), where \(s\in\{0,1\}\).
Choose integers \(d_1,\ldots,d_{b(m,2)}\geq0\) satisfying
\[
 \sum_{i=1}^{b(m,2)}d_i=n-m,
 \qquad
 d_i\leq2{K(m,n)},
\]
and define
\begin{equation}\label{eq:circle-join-relation}
 \mathcal C_{m,n}
 :=
 \mathcal P_{d_1}*\cdots*\mathcal P_{d_{b(m,2)}}
 *(\Delta_{\Sp^0})^{*s}
 \subseteq\Sp^n\times\Sp^m.
\end{equation}
Here the subscripts \(m<n\) record the dimensions
in increasing order, although the higher-dimensional sphere is the first
factor.
Then \(\mathcal C_{m,n}\) is a correspondence with
\begin{equation}\label{eq:circle-join-distortion}
 \dis(\mathcal C_{m,n})
 ={\delta\bigl(K(m,n)\bigr)}
 ={\frac{2\pi K(m,n)}{2K(m,n)+1}}.
\end{equation}
Consequently,
\begin{equation}\label{eq:circle-join-gh-bound}
 \dgh(\Sp^m,\Sp^n)
 \leq
 {\frac{\pi K(m,n)}{2K(m,n)+1}}.
\end{equation}
\end{proposition}

\begin{proof}
Such a choice exists because
{\(2K(m,n)b(m,2)\geq n-m\)}.

According to \Cref{def:synchronized-spherical-join}, the dimension of
the resulting target sphere is
\(b(m,2)+(b(m,2)+s-1)=m\), while the dimension of the resulting source
sphere is
\(\sum_{i=1}^{b(m,2)}(1+d_i)+(b(m,2)+s-1)=n\).
Hence \(\mathcal C_{m,n}\) is a correspondence between \(\Sp^n\) and
\(\Sp^m\).

By \eqref{eq:iterated-join-distortion},
\[
 \dis(\mathcal C_{m,n})
 =
 \max_i\delta(\lceil d_i/2\rceil)
 \leq{\delta\bigl(K(m,n)\bigr)}.
\]
{The definition of \(K(m,n)\) gives
\(n-m>2b(m,2)(K(m,n)-1)\), so at least one \(d_i\)
is greater than \(2(K(m,n)-1)\).  For that index,
\(\lceil d_i/2\rceil=K(m,n)\),} proving equality in
\eqref{eq:circle-join-distortion}.  The Gromov--Hausdorff bound follows
from \eqref{eq:gh-correspondence-formula}.
\end{proof}

\begin{corollary}[Staircase bound]
\label{cor:circle-join-two-pi-thirds}
Let \(1\leq m<n\) and \(k\geq1\).  If \(n-m\leq kL_m\), then
\[
 \dgh(\Sp^m,\Sp^n)\leq\frac{\pi k}{2k+1}.
\]
More precisely, if
\((k-1)L_m<n-m\leq kL_m\), then the correspondence
\(\mathcal C_{m,n}\) of \Cref{thm:circle-join-all-pairs} has distortion
exactly \(2\pi k/(2k+1)\).
\end{corollary}

\begin{proof}
The first hypothesis gives {\(K(m,n)\leq k\)}, and the function
\(j\mapsto\pi j/(2j+1)\) is increasing.  The two inequalities in the
second assertion are equivalent to {\(K(m,n)=k\)}, so both claims follow
from \Cref{thm:circle-join-all-pairs}.
\end{proof}

{
The upper bounds furnished by \(\mathcal J_{m,n-m}\) and
\(\mathcal C_{m,n}\) can be compared explicitly wherever the former is
defined.  Put \(d:=n-m\).  If \(2d\leq m+1\), then
\(K(m,m+d)=1\), so \(\mathcal C_{m,n}\) gives the upper bound
\(\tfrac{\pi}{3}\), whereas \(\mathcal J_{m,d}\) gives
\(\zeta_{b(m,d)-1}/2\).  The latter is strictly smaller when
\(b(m,d)\geq3\), or equivalently
\(m<n\leq\lfloor(4m+1)/3\rfloor\), and agrees with
\(\tfrac{\pi}{3}\) when \(b(m,d)=2\).  For
\(n>\lfloor(3m+1)/2\rfloor\), \(\mathcal J_{m,d}\) is unavailable,
whereas \(\mathcal C_{m,n}\) continues to apply.  Thus the upper bound
furnished by \(\mathcal J_{m,d}\) is never weaker where both
correspondences are available; the advantage of \(\mathcal C_{m,n}\)
is that it applies to every \(m<n\).
\Cref{fig:circle-join-staircase,fig:balanced-circle-comparison}
display these two complementary aspects.
}
\hypersetup{linkcolor=blue!55!black}

\begin{figure}[ht]
\centering

\begin{tikzpicture}[x=0.27cm,y=0.27cm]

\foreach \m in {1,...,35}{
    \pgfmathtruncatemacro{\Lm}{2*floor((\m+1)/2)}

    \foreach \n in {\m,...,35}{
        \ifnum\n=\m
            \filldraw[
                fill=gray!15,
                draw=gray!45,
                line width=0.25pt
            ]
            (\n-1,-\m+1) rectangle (\n,-\m);

            \node[font=\tiny] at (\n-.5,-\m+.5) {$0$};
        \else
            \pgfmathtruncatemacro{\gap}{\n-\m}
            \pgfmathtruncatemacro{\Kcell}{ceil(\gap/\Lm)}
            \pgfmathtruncatemacro{\cellshade}
                {round(10+60*(\Kcell-1)/16)}

            \filldraw[
                fill=red!\cellshade,
                draw=gray!45,
                line width=0.25pt
            ]
            (\n-1,-\m+1) rectangle (\n,-\m);

            \node[font=\tiny] at (\n-.5,-\m+.5) {$\Kcell$};
        \fi
    }
}

\foreach \m in {1,...,35}{
    \pgfmathtruncatemacro{\Lm}{2*floor((\m+1)/2)}

    \foreach \n in {\m,...,34}{

\ifnum\n=\m
            \pgfmathtruncatemacro{\Kleft}{0}
        \else
            \pgfmathtruncatemacro{\gapleft}{\n-\m}
            \pgfmathtruncatemacro{\Kleft}{ceil(\gapleft/\Lm)}
        \fi

\pgfmathtruncatemacro{\gapright}{\n+1-\m}
        \pgfmathtruncatemacro{\Kright}{ceil(\gapright/\Lm)}

\ifnum\Kleft=\Kright
        \else
            \draw[
                black,
                line width=0.6pt
            ]
            (\n,-\m+1) -- (\n,-\m);
        \fi
    }
}

\foreach \m in {1,...,34}{
    \pgfmathtruncatemacro{\Lm}{2*floor((\m+1)/2)}
    \pgfmathtruncatemacro{\Lnext}{2*floor((\m+2)/2)}
    \pgfmathtruncatemacro{\mstart}{\m+1}

    \foreach \n in {\mstart,...,35}{

\pgfmathtruncatemacro{\gapupper}{\n-\m}
        \pgfmathtruncatemacro{\Kupper}{ceil(\gapupper/\Lm)}

\ifnum\n=\mstart
            \pgfmathtruncatemacro{\Klower}{0}
        \else
            \pgfmathtruncatemacro{\gaplower}{\n-\m-1}
            \pgfmathtruncatemacro{\Klower}{ceil(\gaplower/\Lnext)}
        \fi

\ifnum\Kupper=\Klower
        \else
            \draw[
                black,
                line width=0.6pt
            ]
            (\n-1,-\m) -- (\n,-\m);
        \fi
    }
}

\foreach \n in {1,...,35}
    \node[font=\tiny] at (\n-.5,0.45) {$\n$};

\foreach \m in {1,...,35}
    \node[font=\tiny] at (-.5,-\m+.5) {$\m$};

\node[font=\scriptsize] at (17.5,1.85) {$n$};
\node[font=\scriptsize] at (-2.1,-17.5) {$m$};

\node at (36.02,0.45) {$\cdots$};
\node at (-.5,-35.7) {$\vdots$};
\node at (35.75,-35.75) {$\ddots$};

\begin{scope}[xshift=3mm]

\node[anchor=west,font=\small] at (37,.45)
    {$k:=K(m,n)$};

\node[anchor=west,font=\scriptsize] at (37,-1.2)
    {$\delta(k):=\tfrac{2\pi k}{2k+1}$};

\node[anchor=west,font=\scriptsize] at (37,-3)
    {$\dgh(\Sp^m,\Sp^n)\leq\tfrac12\delta(k)=\tfrac{\pi k}{2k+1}$};

\foreach \k in {1,...,17}{
    \pgfmathtruncatemacro{\shade}
        {round(10+60*(\k-1)/16)}
    \pgfmathsetmacro{\ytop}{-4.6-0.8*(\k-1)}
    \pgfmathsetmacro{\ybot}{\ytop-0.8}

    \filldraw[
        fill=red!\shade,
        draw=gray!45,
        line width=0.25pt
    ]
    (37,\ytop) rectangle (38.8,\ybot);
}

\node[anchor=west,font=\scriptsize] at (39.25,-5) {$k=1$};
\node[anchor=west,font=\scriptsize] at (39.25,-8.2) {$k=5$};
\node[anchor=west,font=\scriptsize] at (39.25,-11.4) {$k=9$};
\node[anchor=west,font=\scriptsize] at (39.25,-14.6) {$k=13$};
\node[anchor=west,font=\scriptsize] at (39.25,-17.8) {$k=17$};

\end{scope}

\end{tikzpicture}
 \caption{The staircase regions determined by {the
correspondences \(\mathcal C_{m,n}\subseteq\Sp^n\times\Sp^m\) obtained
from circle factors}, whose
distortion is computed in \Cref{thm:circle-join-all-pairs}; see also
\Cref{cor:circle-join-two-pi-thirds}.  In the displayed finite window, the
cell \((m,n)\), \(m<n\), is labeled by {\(k=K(m,n)\)}; its color records
the bound \(\dgh(\Sp^m,\Sp^n)\leq\pi k/(2k+1)\).  The diagonal entries
are zero.  Since \(L_{2r-1}=L_{2r}=2r\), consecutive pairs of rows have
the same band width.  {The sharper bounds obtained
from balanced joins nearer the diagonal are not shown; see \Cref{fig:balanced-circle-comparison}.}}
\label{fig:circle-join-staircase}
\end{figure}
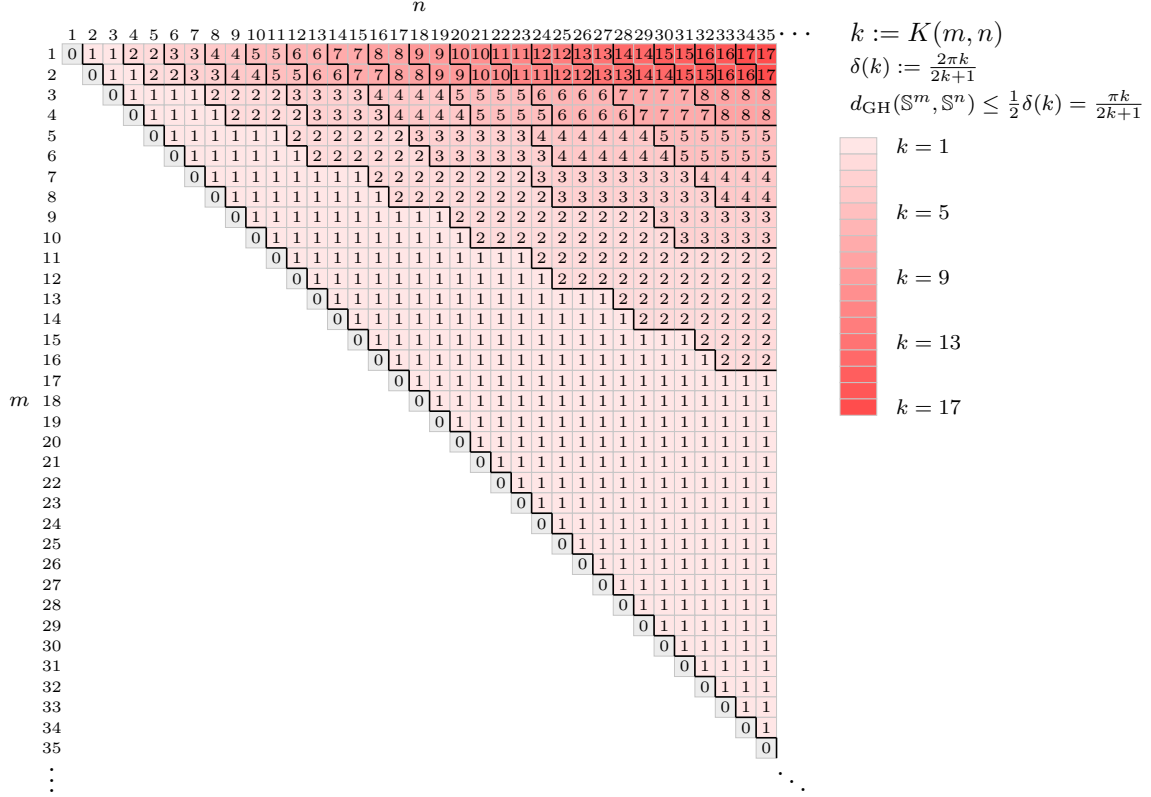

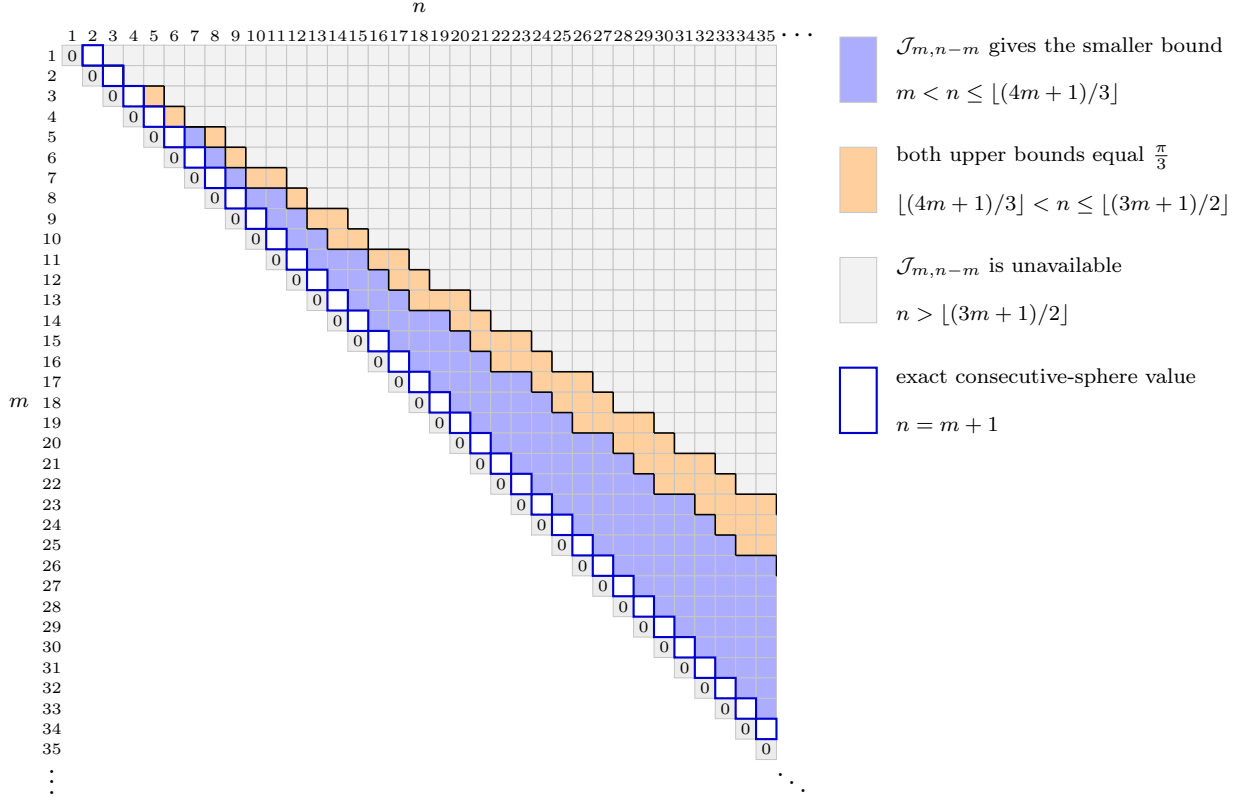
\begin{figure}[ht]
\centering
\begin{tikzpicture}[x=0.27cm,y=0.27cm]

\foreach \m in {1,...,35}{
  \pgfmathtruncatemacro{\strictend}{floor((4*\m+1)/3)}
  \pgfmathtruncatemacro{\balancedend}{floor((3*\m+1)/2)}
  \pgfmathtruncatemacro{\nexact}{\m+1}
  \foreach \n in {\m,...,35}{
    \ifnum\n=\m
      \def\cellcolor{gray!15}
    \else
      \ifnum\n=\nexact
        \def\cellcolor{white}
      \else
        \ifnum\n>\balancedend
          \def\cellcolor{gray!10}
        \else
          \ifnum\n>\strictend
            \def\cellcolor{orange!38}
          \else
            \def\cellcolor{blue!32}
          \fi
        \fi
      \fi
    \fi
    \filldraw[
      fill=\cellcolor,
      draw=gray!45,
      line width=0.25pt
    ]
      (\n-1,-\m+1) rectangle (\n,-\m);
    \ifnum\n=\m
      \node[font=\tiny] at (\n-.5,-\m+.5) {\(0\)};
    \fi
  }
}

\begin{scope}
\clip (0,-35) rectangle (35,0);
\foreach \m in {1,...,35}{
  \pgfmathtruncatemacro{\strictend}{floor((4*\m+1)/3)}
  \draw[black,line width=0.6pt]
    (\strictend,-\m+1)--(\strictend,-\m);
}
\foreach \m in {1,...,34}{
  \pgfmathtruncatemacro{\strictend}{floor((4*\m+1)/3)}
  \pgfmathtruncatemacro{\strictnext}{floor((4*(\m+1)+1)/3)}
  \draw[black,line width=0.6pt]
    (\strictend,-\m)--(\strictnext,-\m);
}

\foreach \m in {1,...,35}{
  \pgfmathtruncatemacro{\balancedend}{floor((3*\m+1)/2)}
  \draw[black,line width=0.6pt]
    (\balancedend,-\m+1)--(\balancedend,-\m);
}
\foreach \m in {1,...,34}{
  \pgfmathtruncatemacro{\balancedend}{floor((3*\m+1)/2)}
  \pgfmathtruncatemacro{\balancednext}{floor((3*(\m+1)+1)/2)}
  \draw[black,line width=0.6pt]
    (\balancedend,-\m)--(\balancednext,-\m);
}
\end{scope}

\foreach \m in {1,...,34}{
  \pgfmathtruncatemacro{\nexact}{\m+1}
  \draw[blue!80!black,line width=0.8pt]
    (\nexact-1,-\m+1) rectangle (\nexact,-\m);
}

\foreach \n in {1,...,35}
  \node[font=\tiny] at (\n-.5,0.45) {\(\n\)};
\foreach \m in {1,...,35}
  \node[font=\tiny] at (-.5,-\m+.5) {\(\m\)};
\node[font=\scriptsize] at (17.5,1.85) {\(n\)};
\node[font=\scriptsize] at (-2.1,-17.5) {\(m\)};
\node at (36.02,0.45) {\(\cdots\)};
\node at (-.5,-35.7) {\(\vdots\)};
\node at (35.75,-35.75) {\(\ddots\)};

\begin{scope}[xshift=3mm]
  \filldraw[fill=blue!32,draw=gray!45,line width=0.25pt]
    (37,0.4) rectangle (38.8,-2.8);
  \node[anchor=west,font=\scriptsize] at (39.25,-0.1)
    {\(\mathcal J_{m,n-m}\) gives the smaller bound};
  \node[anchor=west,font=\scriptsize] at (39.25,-2.4)
    {\(m<n\leq\lfloor(4m+1)/3\rfloor\)};

  \filldraw[fill=orange!38,draw=gray!45,line width=0.25pt]
    (37,-5.0) rectangle (38.8,-8.2);
  \node[anchor=west,font=\scriptsize] at (39.25,-5.5)
    {both upper bounds equal \(\tfrac{\pi}{3}\)};
  \node[anchor=west,font=\scriptsize] at (39.25,-7.8)
    {\(\lfloor(4m+1)/3\rfloor<n\leq\lfloor(3m+1)/2\rfloor\)};

  \filldraw[fill=gray!10,draw=gray!45,line width=0.25pt]
    (37,-10.4) rectangle (38.8,-13.6);
  \node[anchor=west,font=\scriptsize] at (39.25,-10.9)
    {\(\mathcal J_{m,n-m}\) is unavailable};
  \node[anchor=west,font=\scriptsize] at (39.25,-13.2)
    {\(n>\lfloor(3m+1)/2\rfloor\)};

  \draw[blue!80!black,line width=0.8pt]
    (37,-15.8) rectangle (38.8,-19.0);
  \node[anchor=west,font=\scriptsize] at (39.25,-16.3)
    {exact consecutive-sphere value};
  \node[anchor=west,font=\scriptsize] at (39.25,-18.6)
    {\(n=m+1\)};
\end{scope}

\end{tikzpicture}
 {
\caption{Comparison of the upper bounds furnished by the
correspondences
\(\mathcal J_{m,n-m}\subseteq\Sp^n\times\Sp^m\), obtained from
balanced joins and studied in \Cref{thm:balanced-join-bound}, and
\(\mathcal C_{m,n}\subseteq\Sp^n\times\Sp^m\), obtained from circle
factors and studied in \Cref{thm:circle-join-all-pairs}.  Those
propositions compute their distortions.  The axes follow the convention
of \Cref{fig:circle-join-staircase}.  In the blue region,
\(b(m,n-m)\geq3\), and the upper bound furnished by
\(\mathcal J_{m,n-m}\), namely \(\zeta_{b(m,n-m)-1}/2\), is strictly
smaller than the upper bound \(\tfrac{\pi}{3}\) furnished by
\(\mathcal C_{m,n}\).  In the orange region, \(b(m,n-m)=2\), and the
two upper bounds agree at \(\tfrac{\pi}{3}\).  The correspondence
\(\mathcal J_{m,n-m}\) is unavailable in the gray region, while
\(\mathcal C_{m,n}\) continues to apply.  The outlined first
superdiagonal records the exact value realized by the anchored--chord
correspondences \(\mathcal R_m\subseteq\Sp^{m+1}\times\Sp^m\) of
\Cref{thm:main}.}
\label{fig:balanced-circle-comparison}
}
\hypersetup{linkcolor=blue!55!black}
\end{figure}
\FloatBarrier

The case \(k=1\) gives the broad strip
\(n-m\leq L_m\), on which \(\dgh(\Sp^m,\Sp^n)\leq\tfrac{\pi}{3}\).
For \(n=m+1\), this recovers the dimension-independent bound of the
Polymath collaboration \cite[Theorem~1.2]{adams2022gromov}.
Thus the same bound holds throughout the larger region
\(1\leq n-m\leq L_m=2\lfloor(m+1)/2\rfloor\).

The general bound in \cite[Theorem~1.2]{harrison2023quantitative} is
\(\dgh(\Sp^m,\Sp^n)\leq\pi n/[2(n+1)]\).  {The
correspondence \(\mathcal C_{m,n}\) obtained from circle factors in
\Cref{thm:circle-join-all-pairs} gives}
\[
 \dgh(\Sp^m,\Sp^n)
 \leq{\frac{\pi K(m,n)}{2K(m,n)+1}}.
\]
Since {\(2K(m,n)\leq n\)}, this upper bound is at most the preceding
Harrison--Jeffs bound:
{\(\tfrac{\pi K(m,n)}{2K(m,n)+1}
\leq\tfrac{\pi n}{2(n+1)}\)}.
The inequality is strict for \(m\geq2\), while for \(m=1\) our
construction gives the exact circle value of
\cite[Theorem~1.1]{harrison2023quantitative}.
There is no global ordering with the refined packing bound in
\cite[Theorem~1.3]{harrison2023quantitative}.  For example, when
\((m,n)=(2,4)\), our upper bound is \(\tfrac{\pi}{3}\), whereas their
upper bound is \(\tfrac12\arccos(-3/5)>\tfrac{\pi}{3}\).  In the opposite
direction, for fixed \(m\geq2\) and \(n\to\infty\), the difference between
\(\tfrac{\pi}{2}\) and our upper bound is
{\(\tfrac{\pi}{2}-\tfrac{\pi K(m,n)}{2K(m,n)+1}
=\tfrac{\pi L_m}{4n}+O_m(n^{-2})\)}, whereas the upper bound obtained in the
proof of \cite[Corollary~1.4]{harrison2023quantitative} lies below
\(\tfrac{\pi}{2}\) by \(\Omega_m(n^{-1/2})\).  Their refined upper bound is
therefore eventually stronger.
 
\section{Comparison with previous correspondence constructions}
\label{sec:previous-work}

The anchored--chord correspondence \(\mathcal R_n\), defined in
\Cref{def:anchored-chord-relation}, belongs to the regular-simplex and
hemisphere framework introduced by Lim, M\'emoli and Smith.  This
section compares the successive rules used to pair
points in the upper hemisphere of \(\Sp^{n+1}\) with points of
\(\Sp^n\): the fixed-vertex construction, its discrete refinement, the
introduction of an identity band, continuous geodesic interpolation,
and the normalized-chord interpolation used here.  We use the notation
\(v_i\), \(M_i\), and \(p_\alpha(x)\) from
\Cref{sec:anchored-chord-definition}.

\smallskip
\noindent\emph{The fixed-vertex construction.}
For the \(\Sp^1\)-versus-\(\Sp^2\) case,
\cite[Proposition~1.16]{lim2021gromov} inscribes a regular triangle in
the equator of \(\Sp^2\), maps each open upper-hemisphere cone to its
simplex vertex, uses the identity on the equator, and extends oddly
over the lower hemisphere.  The all-dimensional version in
\cite[Proposition~1.20 and Section~6.2]{lim2021gromov} uses a regular
\((n+1)\)-simplex \(v_1,\ldots,v_{n+2}\subset\Sp^n\) and the same rule.
Its distortion \(\eta_n\) is the diameter of a closed spherical
Voronoi cell of the simplex
\cite[Remarks~6.4--6.6]{lim2021gromov}; the diameter computation follows
from \cite[Lemma~1]{santalo1946convex}.  Thus
\(\eta_1=\zeta_1\), whereas \(\eta_n>\zeta_n\) for \(n\geq2\);
the direct all-dimensional extension is not sharp.

\medskip
\noindent\emph{The discrete refinement for
\(\Sp^2\) versus \(\Sp^3\).}
The sharp construction of \cite[Proposition~1.19 and
Section~8]{lim2021gromov} retains the tetrahedral vertices \(v_i\)
near the north pole but subdivides each Voronoi region closer to the
equator.  A point in the subcell indexed by \(j\neq i\) is assigned to
an auxiliary point \(v_{ij}\) on the minimizing geodesic from \(v_i\)
to \(-v_j\).  This discrete change of target with the radial level,
followed by the identity on the equator and odd extension, gives the
sharp \(n=2\) correspondence.

\medskip
\noindent\emph{A fixed-anchor correspondence with an identity band.}
The proof of \cite[Theorem~1.2 and Section~6]{adams2022gromov}
constructs, for every \(n\geq1\), a correspondence between the closed
upper hemisphere of \(\Sp^{n+1}\) and its equatorial \(\Sp^n\), with
distortion at most \(\tfrac{2\pi}{3}\).  After relabeling the antipodes of the
regular-simplex vertices used there as \(v_i\), the corresponding
spherical facet---the geodesic convex hull of
\(\{-v_j:j\neq i\}\)---is precisely the closed Voronoi cell \(M_i\)
associated with \(v_i\).  In our notation, their relation consists of
\[
 \begin{aligned}
 \bigl\{(p_\alpha(x),v_i):
       0\leq\alpha\leq\tfrac{\pi}{3},\ x\in M_i,\
       1\leq i\leq n+2\bigr\}\cup
 \bigl\{(p_\alpha(x),x):
       \tfrac{\pi}{3}<\alpha\leq\tfrac{\pi}{2},\ x\in\Sp^n\bigr\}.
 \end{aligned}
\]
Thus this construction introduces an identity band of positive width
near the equator in every dimension.  Unlike the constructions below,
however, it does not interpolate between \(v_i\) and \(x\): the target
changes abruptly at \(\alpha=\tfrac{\pi}{3}\).  The resulting upper bound
\(\tfrac{2\pi}{3}\) is sharp for \(n=1\), but is strictly larger than
\(\zeta_n\) for every \(n\geq2\).

\medskip
\noindent\emph{Continuous geodesic interpolation.}
Rodr\'iguez Mart\'in's family was inspired by the proof of
\cite[Proposition~1.20]{lim2021gromov}.  It retains the regular-simplex
cells, the upper- and lower-hemisphere decomposition, and the odd
extension, but replaces the fixed target within a cell by continuous
motion from \(v_i\) toward \(x\).  Denoting his map by
\(\Phi_n\), one has, for \(x\) in the open cell associated with \(v_i\),
that \(\Phi_n(p_\alpha(x))\) lies a fraction
\(\max\{0,\alpha+1-\tfrac{\pi}{2}\}\) of the way along the minimizing geodesic
from \(v_i\) to \(x\).
This proves the \(n=3\) case and also recovers \(n=1,2\)
\cite[Theorem~1.5 and Section~5]{rodriguezmartin2024novel}.  The same
family does not settle \(n=4,5,6\), and its distortion is strictly
larger than \(\zeta_n\) for \(n\geq7\)
\cite[Section~5.4]{rodriguezmartin2024novel}.

\medskip
\noindent\emph{The anchored--chord relation.}
Our construction retains the regular-simplex and hemisphere framework
and includes an identity band near the equator.  On the remaining part
of the upper hemisphere, it replaces the preceding interpolation rules
by normalized Euclidean-chord motion.  Its dimension-dependent gain is
determined by sharpness on the equal-level swapped-anchor family and by
reciprocal compatibility under the involution \(F_n\); see
\Cref{sec:swapped-anchor-configuration,sec:critical-involution}.
Within the normalized-chord ansatz, these requirements determine the
gain uniquely.  The subsequent estimates prove that the resulting
relation has distortion exactly \(\zeta_n\) in every dimension.

\providecommand{\acConstructionHeading}{\paragraph{A sharp anchored--chord correspondence.}\label{sec:canonical-anchored-chord-all-n}}
\providecommand{\acEstimatesHeading}{}
\providecommand{\acUpperBoundaryHeading}{}
\providecommand{\acAssemblyHeading}{}
\providecommand{\acSharpDistortionEquation}{\eqref{eq:ac-109}}
\providecommand{\acMainResultEquation}{\eqref{eq:ac-110}}
\newcommand{\acEquationTag}[1]{\refstepcounter{equation}\tag{\theequation}\label{#1}}
\providecommand{\acLowerBoundArgument}{The reverse inequality is \cite[Theorem~B]{lim2021gromov}, which states
  that
  \[
   2\dgh(\Sp^n,\Sp^{n+1})
   \geq\arccos\left(-\frac1{n+1}\right).
  \]}

\acConstructionHeading

Within the class of continuous normalized-chord gains starting at the
anchor, we now derive the two sharpness requirements that
\emph{canonically} determine \(r_n\), prove its characterization, and
establish the basic properties of the resulting correspondence.

\subsection{Chord and Voronoi geometry}
\label{sec:chord-voronoi-preliminaries}

We first isolate two elementary facts used throughout the construction
and its distortion estimates: the normalized-chord computation and the
geometry of the simplex Voronoi cells.

\begin{lemma}[Chord geometry]
\label{lem:ac-chord-geometry}
Let \(u,v\) be unit vectors in a real inner-product space, and let
\(a\geq0\).
\begin{enumerate}[label=\textup{(\roman*)},nosep]
\item If \(u\cdot v\geq\rho_n\), then
\begin{equation}
 \lVert u+av\rVert
 \geq\sqrt{1+2\rho_n a+a^2}.
 \label{eq:chord-norm-lower-bound}
\end{equation}
For \(a>0\), equality holds if and only if \(u\cdot v=\rho_n\).

\item Assume that \(u\cdot v=\rho_n\), and put
\[
 z:=\frac{u+av}{\lVert u+av\rVert},
 \qquad
 \theta:=\arccos(u\cdot z).
\]
Then
\begin{equation}
 \lVert u+av\rVert
 =\sqrt{1+2\rho_n a+a^2}
 =\frac{\sin\widehat{\zeta}_n}
 {\sin(\widehat{\zeta}_n-\theta)},
 \qquad
 a=\frac{\sin\theta}
 {\sin(\widehat{\zeta}_n-\theta)}.
 \label{eq:chord-geometry}
\end{equation}

The corresponding component identities are
\[
 \cos\theta
 =\frac{1+\rho_na}{\sqrt{1+2\rho_n a+a^2}},
 \qquad
 \sin\theta
 =\frac{a\sin\widehat{\zeta}_n}
 {\sqrt{1+2\rho_n a+a^2}}.
\]

\item Under the assumptions and notation of \textup{(ii)}, if
\[
 v=\rho_nu+\xi\sin\widehat{\zeta}_n,
 \qquad \xi\perp u,\quad \lVert\xi\rVert=1,
\]
then
\begin{equation}
 z=u\cos\theta+\xi\sin\theta.
 \label{eq:chord-angular-form}
\end{equation}
\end{enumerate}
\end{lemma}

\begin{proof}
Expanding the squared norm gives
\[
 \lVert u+av\rVert^2
 =1+2a(u\cdot v)+a^2
 \geq1+2\rho_n a+a^2.
\]
This proves \eqref{eq:chord-norm-lower-bound}; when \(a>0\), equality
holds precisely when \(u\cdot v=\rho_n\).

Assume now that \(u\cdot v=\rho_n\).  The norm identity in
\eqref{eq:chord-geometry} follows from the displayed expansion.  If
\(a=0\), the remaining assertions are immediate.  Suppose \(a>0\),
and consider the Euclidean triangle with vertices \(0,u,u+av\).  Its
sides of lengths \(a\), \(1\), and \(\lVert u+av\rVert\) are opposite
the angles \(\theta\), \(\widehat{\zeta}_n-\theta\), and \(\zeta_n\),
respectively.  Since
\(\sin\zeta_n=\sin\widehat{\zeta}_n\), the sine rule gives
\[
 \frac{a}{\sin\theta}
 =\frac{1}{\sin(\widehat{\zeta}_n-\theta)}
 =\frac{\lVert u+av\rVert}{\sin\widehat{\zeta}_n}.
\]
Furthermore,
\[
 u\cdot z
 =\frac{1+a(u\cdot v)}{\lVert u+av\rVert}
 =\frac{1+\rho_na}{\sqrt{1+2\rho_n a+a^2}},
\]
which proves the formula for \(\cos\theta\).  The component of \(z\)
orthogonal to \(u\) is
\[
 \frac{a(v-\rho_nu)}{\lVert u+av\rVert}.
\]
Its norm is \(\sin\theta\), while
\(\lVert v-\rho_nu\rVert=\sin\widehat{\zeta}_n\); this proves the
formula for \(\sin\theta\).
Finally, \(z\) lies in the plane spanned by \(u\) and \(\xi\), on the
side of \(u\) toward \(v\), and makes angle \(\theta\) with \(u\).
This proves \eqref{eq:chord-angular-form}.
\end{proof}

\begin{lemma}[The Voronoi-cell bound and its equality case]
\label{lem:ac-simplex-cell}
For every \(1\leq i\leq n+2\),
\[
 x\in M_i\quad\Longrightarrow\quad
 v_i\cdot x\geq\rho_n,                                      \acEquationTag{eq:ac-10}
\]
and equality holds in \eqref{eq:ac-10} precisely for
\[
 \{x\in M_i:v_i\cdot x=\rho_n\}
 =\{-v_j:j\neq i\}.
\]
\end{lemma}

\begin{proof}
Fix \(x\in M_i\), and put \(a_k:=v_k\cdot x\).  Then
\[
 a_i=\max_k a_k,\qquad
 \sum_ka_k=0,\qquad
 \sum_ka_k^2=\frac{n+2}{n+1},\qquad
 -1\leq a_k\leq1.
\]
Set
\[
 a_{\max}:=\max_k a_k,
 \qquad
 a_+:=\sum_{a_k>0}a_k.
\]
For the positive terms, \(a_k^2\leq a_{\max}a_k\), whereas for the
nonpositive terms, \(a_k^2\leq-a_k\).  Since
\(\sum_{a_k\leq0}(-a_k)=a_+\), and since at most \(n+1\) of the
\(a_k\) are positive, it follows that
\[
 \sum_ka_k^2
 \leq a_{\max}a_++a_+
 \leq(n+1)a_{\max}(1+a_{\max}).
\]
The rightmost expression is increasing for \(a_{\max}\geq0\).
Comparison with
\[
 \frac{n+2}{n+1}
 =(n+1)\rho_n(1+\rho_n)
\]
therefore gives \(a_i=a_{\max}\geq\rho_n\), proving
\eqref{eq:ac-10}.

Suppose now that \(a_i=\rho_n\).  Then \(a_k\leq\rho_n\) for every
\(k\), and the identities displayed at the beginning of the proof give
\[
 \sum_{k=1}^{n+2}(\rho_n-a_k)(1+a_k)
 =(n+2)\rho_n-\sum_{k=1}^{n+2}a_k^2=0.
\]
Every summand is nonnegative.  Hence each \(a_k\) equals either
\(\rho_n\) or \(-1\).  The identity \(\sum_ka_k=0\) then shows that
exactly one index \(j\) satisfies \(a_j=-1\).  Since \(x\) and \(v_j\)
are unit vectors, \(x\cdot v_j=-1\) implies \(x=-v_j\); moreover,
\(j\neq i\).

Conversely, if \(j\neq i\), then
\[
 v_j\cdot(-v_j)=-1,
 \qquad
 v_k\cdot(-v_j)=\rho_n\quad(k\neq j).
\]
Thus \(-v_j\in M_i\) and \(v_i\cdot(-v_j)=\rho_n\), proving the
asserted equality locus.
\end{proof}

\subsection{The swapped-anchor configuration}
\label{sec:swapped-anchor-configuration}

The formula for the gain is dictated by an explicit family of pairs.
Fix distinct indices \(i,j\).  For every \(k\neq j\),
\(v_k\cdot(-v_j)=\rho_n\), whereas \(v_j\cdot(-v_j)=-1\).  Therefore
\[
 d_n(-v_j,v_k)=\widehat{\zeta}_n\quad(k\neq j),
 \qquad
 d_n(-v_j,v_j)=\pi.
\]
Thus every \(v_k\) with \(k\neq j\) is a simplex vertex nearest to
\(-v_j\).  Equivalently, the equality case in
\Cref{lem:ac-simplex-cell} gives \(-v_j\in M_i\).  Interchanging
\(i\) and \(j\) gives \(-v_i\in M_j\).  For this fixed pair of indices
and for \(r,t\geq0\),
define the \emph{normalized target points}
\phantomsection\label{def:swapped-anchor-targets}
\[
 z_i(r):=\frac{v_i-rv_j}{\lVert v_i-rv_j\rVert},
 \qquad
 z_j(t):=\frac{v_j-tv_i}{\lVert v_j-tv_i\rVert},
\]
as shown in \Cref{fig:swapped-anchor-pairings-n1}.

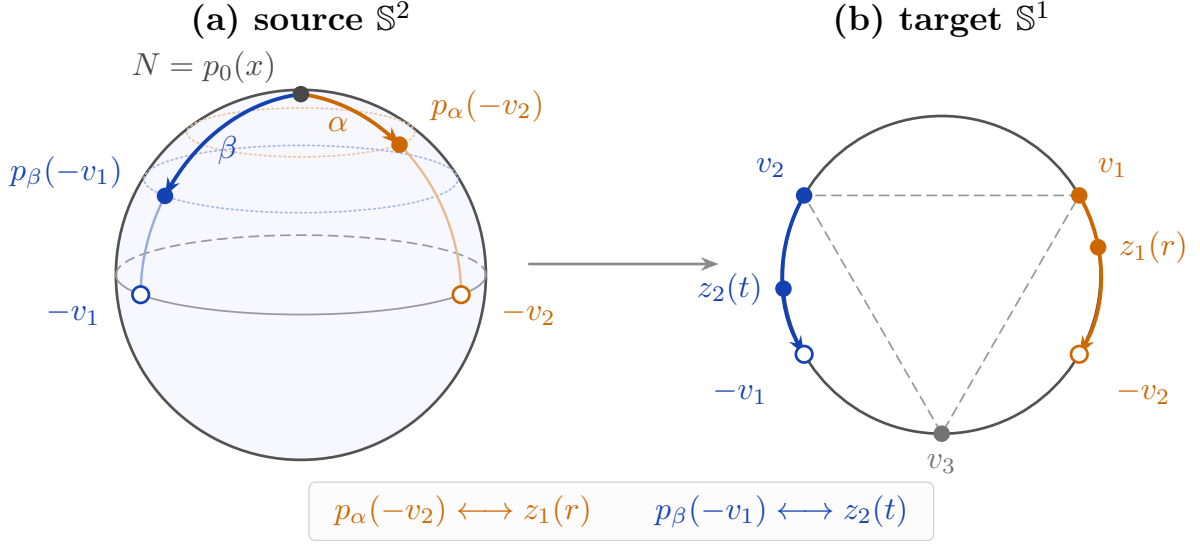
\begin{figure}[!ht]
    \centering
    \resizebox{1.0\textwidth}{!}{\definecolor{anchori}{RGB}{205,103,0}
\definecolor{anchorj}{RGB}{20,66,180}
\definecolor{guidecol}{RGB}{105,110,118}

\tikzset{
  sphere line/.style={line width=.85pt,draw=black!68},
  guide/.style={line width=.55pt,draw=guidecol!70},
  point/.style={circle,inner sep=0pt,minimum size=5.2pt},
  open point/.style={point,fill=white,line width=.85pt},
  pairing/.style={line width=1.2pt,-{Stealth[length=2.0mm]}}
}

\begin{tikzpicture}[line cap=round,line join=round,font=\small]

\begin{scope}[xshift=-3.55cm]
  \def\Rs{2.05}
  \def\flat{0.44}   \def\alph{38}     \def\bet{58}      \def\pLabelScale{1.30}
  \def\minusLabelScale{1.24}
  \pgfmathsetmacro{\up}{sqrt(\Rs*\Rs-\flat*\flat)}

  \coordinate (O) at (0,0);
  \coordinate (N) at (0,\up);

\coordinate (minusvTwo) at
    ({\Rs*cos(330)},{\flat*sin(330)});
  \coordinate (minusvOne) at
    ({\Rs*cos(210)},{\flat*sin(210)});
  \coordinate (qOne) at
    ({\Rs*sin(\alph)*cos(330)},
     {\up*cos(\alph)+\flat*sin(\alph)*sin(330)});
  \coordinate (qTwo) at
    ({\Rs*sin(\bet)*cos(210)},
     {\up*cos(\bet)+\flat*sin(\bet)*sin(210)});

  \pgfmathsetmacro{\latyA}{\up*cos(\alph)}
  \pgfmathsetmacro{\latxA}{\Rs*sin(\alph)}
  \pgfmathsetmacro{\latflatA}{\flat*sin(\alph)}
  \pgfmathsetmacro{\latyB}{\up*cos(\bet)}
  \pgfmathsetmacro{\latxB}{\Rs*sin(\bet)}
  \pgfmathsetmacro{\latflatB}{\flat*sin(\bet)}

  \fill[blue!3] (O) circle[radius=\Rs];
  \draw[sphere line] (O) circle[radius=\Rs];

\draw[guide,densely dashed]
    (\Rs,0) arc[start angle=0,end angle=180,
                 x radius=\Rs,y radius=\flat];
  \draw[guide]
    (-\Rs,0) arc[start angle=180,end angle=360,
                  x radius=\Rs,y radius=\flat];

\draw[draw=anchori!38,line width=.55pt,densely dotted]
    (0,\latyA) ellipse[x radius=\latxA,y radius=\latflatA];
  \draw[draw=anchorj!38,line width=.55pt,densely dotted]
    (0,\latyB) ellipse[x radius=\latxB,y radius=\latflatB];

\draw[anchori!45,line width=.75pt]
    plot[domain=0:90,samples=45,variable=\t]
      ({\Rs*sin(\t)*cos(330)},
       {\up*cos(\t)+\flat*sin(\t)*sin(330)});
  \draw[anchorj!45,line width=.75pt]
    plot[domain=0:90,samples=45,variable=\t]
      ({\Rs*sin(\t)*cos(210)},
       {\up*cos(\t)+\flat*sin(\t)*sin(210)});
  \draw[anchori,pairing]
    plot[domain=3:\alph,samples=32,variable=\t]
      ({\Rs*sin(\t)*cos(330)},
       {\up*cos(\t)+\flat*sin(\t)*sin(330)});
  \draw[anchorj,pairing]
    plot[domain=3:\bet,samples=32,variable=\t]
      ({\Rs*sin(\t)*cos(210)},
       {\up*cos(\t)+\flat*sin(\t)*sin(210)});

  \node[point,fill=anchori] at (qOne) {};
  \node[
    fill=white,
    draw=none,
    text=anchori,
    inner sep=0.4pt,
    anchor=west
  ] at ($(O)!\pLabelScale!(qOne)$) {$p_\alpha(-v_2)$};

  \node[point,fill=anchorj] at (qTwo) {};
  \node[
    fill=white,
    draw=none,
    text=anchorj,
    inner sep=0.4pt,
    anchor=east
  ] at ($(O)!\pLabelScale!(qTwo)$) {$p_\beta(-v_1)$};

  \node[fill=white,inner sep=1.2pt,open point,draw=anchori]
    at (minusvTwo) {};
  \node[
    fill=white,
    draw=none,
    text=anchori,
    inner sep=0.4pt,
    anchor=north west
  ] at ($(O)!\minusLabelScale!(minusvTwo)$) {$-v_2$};

  \node[fill=white,inner sep=1.2pt,open point,draw=anchorj]
    at (minusvOne) {};
  \node[
    fill=white,
    draw=none,
    text=anchorj,
    inner sep=0.4pt,
    anchor=north east
  ] at ($(O)!\minusLabelScale!(minusvOne)$) {$-v_1$};

  \node[point,fill=black!72] at (N) {};
  \node[black!72,anchor=east] at (-.14,2.30) {\(N=p_0(x)\)};
  \node[font=\bfseries] at (0,2.82) {(a) source \(\mathbb S^2\)};
\node[anchori] at (.41,1.67) {\(\alpha\)};
  \node[anchorj] at (-.81,1.38) {\(\beta\)};
\end{scope}

\draw[black!45,line width=.85pt,-{Stealth[length=2.0mm]}]
  (-1.02,.12) -- (1.08,.12);

\begin{scope}[xshift=3.55cm]
  \def\Rt{1.76}
  \def\minusTargetLabelScale{1.18}
  \coordinate (vOne) at (30:\Rt);
  \coordinate (vTwo) at (150:\Rt);
  \coordinate (vThree) at (270:\Rt);
  \coordinate (minusvTwoT) at (-30:\Rt);
  \coordinate (minusvOneT) at (210:\Rt);
  \coordinate (zOne) at (10:\Rt);
  \coordinate (zTwo) at (185:\Rt);

  \draw[sphere line] (0,0) circle[radius=\Rt];
  \draw[guide,densely dashed] (vOne)--(vTwo)--(vThree)--cycle;

\draw[anchori,pairing]
    (vOne) arc[start angle=30,end angle=-30,radius=\Rt];
  \draw[anchorj,pairing]
    (vTwo) arc[start angle=150,end angle=210,radius=\Rt];

  \node[point,fill=anchori,label={[anchori]above right=1pt:\(v_1\)}]
    at (vOne) {};
  \node[point,fill=anchorj,label={[anchorj]above left=1pt:\(v_2\)}]
    at (vTwo) {};
  \node[point,fill=black!55,label={[black!55]below:\(v_3\)}]
    at (vThree) {};

  \node[open point,draw=anchori] at (minusvTwoT) {};
  \node[text=anchori,anchor=north west]
    at ($(0,0)!\minusTargetLabelScale!(minusvTwoT)$) {$-v_2$};
  \node[open point,draw=anchorj] at (minusvOneT) {};
  \node[text=anchorj,anchor=north east]
    at ($(0,0)!\minusTargetLabelScale!(minusvOneT)$) {$-v_1$};

  \node[point,fill=anchori,label={[anchori]right=3pt:\(z_1(r)\)}]
    at (zOne) {};
  \node[point,fill=anchorj,label={[anchorj]left=3pt:\(z_2(t)\)}]
    at (zTwo) {};

  \node[font=\bfseries] at (0,2.82) {(b) target \(\mathbb S^1\)};
\end{scope}

\node[
  draw=black!18,
  fill=black!1,
  rounded corners=2pt,
  inner xsep=8pt,
  inner ysep=4pt
] at (0,-2.62) {\(
  \color{anchori}{p_\alpha(-v_2)\longleftrightarrow z_1(r)}
  \qquad
  \color{anchorj}{p_\beta(-v_1)\longleftrightarrow z_2(t)}
\)};

\end{tikzpicture} }
  \caption{The swapped-anchor pairings for \(n=1\).  The source points
\(p_\alpha(-v_2)\) and \(p_\beta(-v_1)\) are paired with the target
points \(z_1(r)\) and \(z_2(t)\), which move from \(v_1\) toward
\(-v_2\) and from \(v_2\) toward \(-v_1\), respectively.}
    \label{fig:swapped-anchor-pairings-n1}
\end{figure}

Thus the equatorial point paired with the anchor \(v_i\) is \(-v_j\),
while the point paired with \(v_j\) is \(-v_i\).  We call
the two pairings
\[
 \bigl(p_\alpha(-v_j),z_i(r)\bigr),\qquad
 \bigl(p_\beta(-v_i),z_j(t)\bigr)
\]
the \emph{swapped-anchor configuration}.  This configuration is
extremal in the sense relevant to the construction: within the
normalized-chord ansatz, making the desired distortion bound sharp on
its equal-level subfamily determines the first branch of \(r_n\), while
its antipodal threshold determines the radial pairing \(F_n\) and the
reciprocal second branch.  This forcing is made precise in
\Cref{prop:canonical-gain-characterization}; the subsequent distortion
estimates show that no other realized configuration produces a larger
defect.

First give the two source points the same radial parameter \(\alpha\).
By direct calculation, the source-distance function
introduced in \Cref{sec:anchored-chord-definition} satisfies
\[
 \begin{aligned}
 D_n(\alpha)
 &=d_{n+1}\bigl(p_\alpha(-v_j),p_\alpha(-v_i)\bigr)\\
 &=\arccos\bigl(\cos^2\alpha-\rho_n\sin^2\alpha\bigr).
 \end{aligned}
\]
When \(0\leq r\leq1\), the point \(z_i(r)\) moves along the shorter
great-circle arc from \(v_i\) toward \(-v_j\).  Put
\(\theta:=d_n(v_i,z_i(r))\).  The point \(z_j(r)\) makes the symmetric
motion from \(v_j\) toward \(-v_i\).  The four endpoints
\(v_i,-v_j,v_j,-v_i\), together with the two moving points
\(z_i(r),z_j(r)\), lie on the great circle
\(\Sp(\operatorname{span}\{v_i,v_j\})\), as shown in
\Cref{fig:swapped-anchor-great-circle}.  Consequently,
\[
 d_n\bigl(z_i(r),z_j(r)\bigr)=\zeta_n+2\theta.
\]

\begin{figure}[htbp]
\centering
\begin{tikzpicture}[line cap=round,line join=round]
\begin{scope}[xshift=-3.45cm]
 \def\R{2.02}
 \coordinate (avi)  at (20:\R);
 \coordinate (avj)  at (135:\R);
 \coordinate (anvi) at (200:\R);
 \coordinate (anvj) at (315:\R);
 \coordinate (azi)  at (-2:\R);
 \coordinate (azj)  at (157:\R);

 \node[font=\bfseries] at (0,3.02) {(a) \(0<r<1\)};
 \draw[semithick,black!70] (0,0) circle[radius=\R];

\draw[line width=1.35pt,orange!85!black,-{Stealth[length=2.1mm]}]
   (avi) arc[start angle=20,end angle=-45,radius=\R];
 \draw[line width=1.35pt,blue!70!black,-{Stealth[length=2.1mm]}]
   (avj) arc[start angle=135,end angle=200,radius=\R];

\draw[semithick,black!65]
   (20:2.31) arc[start angle=20,end angle=135,radius=2.31];
 \node[fill=white,inner sep=1pt] at (77.5:2.48) {\(\zeta_n\)};

 \draw[semithick,orange!85!black]
   (20:1.68) arc[start angle=20,end angle=-2,radius=1.68];
 \node[fill=white,inner sep=0.6pt] at (9:1.50) {\(\theta\)};
 \draw[semithick,blue!70!black]
   (135:1.68) arc[start angle=135,end angle=157,radius=1.68];
 \node[fill=white,inner sep=0.6pt] at (146:1.50) {\(\theta\)};

 \draw[line width=1pt,purple!70!black,{Stealth[length=1.8mm]}-{Stealth[length=1.8mm]}]
   (-2:1.15) arc[start angle=-2,end angle=157,radius=1.15];
 \node[fill=white,inner sep=1.2pt,align=center] at (77.5:0.89)
   {\(\zeta_n+2\theta\)};

 \node[fill=white,inner sep=0.8pt] at (-20:2.55) {\(\widehat{\zeta}_n\)};
 \node[fill=white,inner sep=0.8pt] at (182:2.55) {\(\widehat{\zeta}_n\)};

\fill[orange!85!black] (avi) circle[radius=2.2pt]
   node[fill=white,inner sep=1.2pt,above right=2pt] {\(v_i\)};
 \fill[blue!70!black] (avj) circle[radius=2.2pt]
   node[fill=white,inner sep=1.2pt,above left=2pt] {\(v_j\)};
 \filldraw[fill=white,draw=blue!70!black,line width=.8pt]
   (anvi) circle[radius=2.2pt] node[below left=1pt] {\(-v_i\)};
 \filldraw[fill=white,draw=orange!85!black,line width=.8pt]
   (anvj) circle[radius=2.2pt] node[below right=1pt] {\(-v_j\)};
 \fill[purple!75!black] (azi) circle[radius=2.5pt]
   node[right=1pt] {\(z_i(r)\)};
 \fill[purple!75!black] (azj) circle[radius=2.5pt]
   node[left=1pt] {\(z_j(r)\)};
\end{scope}

\begin{scope}[xshift=3.35cm]
 \def\R{2.02}
 \coordinate (bvi)  at (20:\R);
 \coordinate (bvj)  at (135:\R);
 \coordinate (bnvi) at (200:\R);
 \coordinate (bnvj) at (315:\R);
 \coordinate (bzi)  at (-12.5:\R);
 \coordinate (bzj)  at (167.5:\R);

 \node[font=\bfseries] at (0,3.02) {(b) \(r=1\)};
 \draw[semithick,black!70] (0,0) circle[radius=\R];
 \draw[line width=1.35pt,orange!85!black,-{Stealth[length=2.1mm]}]
   (bvi) arc[start angle=20,end angle=-45,radius=\R];
 \draw[line width=1.35pt,blue!70!black,-{Stealth[length=2.1mm]}]
   (bvj) arc[start angle=135,end angle=200,radius=\R];
 \draw[densely dashed,black!55] (bzi)--(bzj);

 \draw[semithick,orange!85!black]
   (20:1.68) arc[start angle=20,end angle=-12.5,radius=1.68];
 \draw[semithick,blue!70!black]
   (135:1.68) arc[start angle=135,end angle=167.5,radius=1.68];
 \node[fill=white,inner sep=.7pt] at (3.75:1.39) {\(\widehat{\zeta}_n/2\)};
 \node[fill=white,inner sep=.7pt] at (151.25:1.39) {\(\widehat{\zeta}_n/2\)};

 \fill[orange!85!black] (bvi) circle[radius=2.2pt]
   node[above right=1pt] {\(v_i\)};
 \fill[blue!70!black] (bvj) circle[radius=2.2pt]
   node[above left=1pt] {\(v_j\)};
 \filldraw[fill=white,draw=blue!70!black,line width=.8pt]
   (bnvi) circle[radius=2.2pt] node[below left=1pt] {\(-v_i\)};
 \filldraw[fill=white,draw=orange!85!black,line width=.8pt]
   (bnvj) circle[radius=2.2pt] node[below right=1pt] {\(-v_j\)};
 \fill[purple!75!black] (bzi) circle[radius=2.5pt]
   node[below right=1pt] {\(z_i(1)\)};
 \fill[purple!75!black] (bzj) circle[radius=2.5pt]
   node[above left=1pt] {\(z_j(1)\)};

 \node[align=center] at (0,-2.72)
   {\(z_i(1)=-z_j(1)\)\\[+2pt]
    \(\theta=\widehat{\zeta}_n/2\)};
\end{scope}
\end{tikzpicture}
\caption{The swapped-anchor configuration in the great circle
\(\Sp(\operatorname{span}\{v_i,v_j\})\).
In panel~(a), the two targets move symmetrically through angle
\(\theta\), so their distance is \(\zeta_n+2\theta\).
In panel~(b), \(r=1\): \(z_i(1)\) is the midpoint of the arc from
\(v_i\) to \(-v_j\), \(z_j(1)\) is the midpoint of the arc from
\(v_j\) to \(-v_i\), and the two targets are antipodal.  This first
occurs at \(\alpha=\alpha_n^\ast\), where
\(D_n(\alpha_n^\ast)=\widehat{\zeta}_n\).}
\label{fig:swapped-anchor-great-circle}
\end{figure}
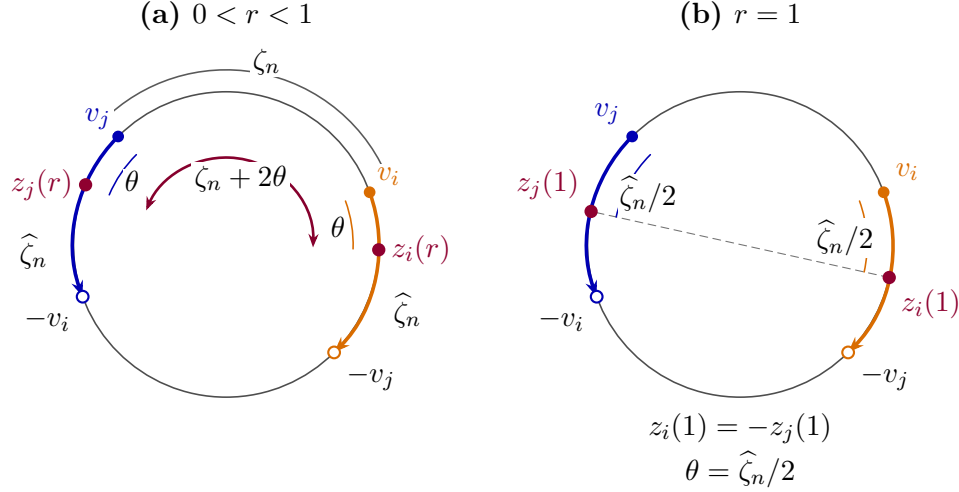

For \(0\leq\alpha\leq\alpha_n^\ast\), make the desired distortion
bound sharp on this family by requiring
\[
 d_n\bigl(z_i(r_n(\alpha)),z_j(r_n(\alpha))\bigr)
 =\zeta_n+D_n(\alpha).
\]
This forces \(\theta=D_n(\alpha)/2\).  The identity
\[
 \Theta_n(\alpha)=\frac{D_n(\alpha)}2
 =\arcsin(c_n\sin\alpha)
\]
shows that \(\Theta_n(\alpha)\) has a direct meaning: it is half the source
distance in the equal-level swapped-anchor configuration.  On the lower
radial range \(0\leq\alpha\leq\alpha_n^\ast\), it also equals the
geodesic displacement of the target from its anchor:
\[
 d_n\bigl(v_i,z_i(r_n(\alpha))\bigr)=\Theta_n(\alpha).
\]

\paragraph{Regularity of \(\Theta_n\).}
For \(0\leq\alpha\leq\widehat{\zeta}_n\), one has
\(\cos\Theta_n(\alpha)>0\).  Direct differentiation gives
\begin{equation}
\begin{gathered}
 \Theta_n'(\alpha)
 =\frac{c_n\cos\alpha}{\cos\Theta_n(\alpha)},\qquad
 1-\Theta_n'(\alpha)^2
 =\frac{s_n^2}{\cos^2\Theta_n(\alpha)}\geq0,\\
 \Theta_n''(\alpha)
 =-\frac{s_n^2\sin\Theta_n(\alpha)}
 {\cos^3\Theta_n(\alpha)}\leq0.
\end{gathered}
\label{eq:theta-derivatives}
\end{equation}
Thus \(\Theta_n\) is increasing, concave, and \(1\)-Lipschitz on
\([0,\widehat{\zeta}_n]\).

Apply \Cref{lem:ac-chord-geometry} with
\(u:=v_i\), \(v:=-v_j\), and \(a:=r\).  Since
\(v_i\cdot(-v_j)=\rho_n\), it gives
\[
 r=\frac{\sin\theta}{\sin(\widehat{\zeta}_n-\theta)}.
\]
The underlying Euclidean triangle is shown in
\Cref{fig:chord-triangle-geometry}.
Substituting \(\theta=\Theta_n(\alpha)\) gives exactly the first branch
of \eqref{eq:ac-1}.  Thus the gain is not an arbitrary interpolation:
it makes the distortion exactly \(\zeta_n\) on the equal-level
swapped-anchor family.

As \(r\) increases from \(0\) to \(1\), the two targets
\(z_i(r),z_j(r)\) move from distance \(\zeta_n\) to distance \(\pi\).
Indeed,
\[
 z_i(1)=-z_j(1).
\]
Distortion at most \(\zeta_n\) therefore requires the source distance
\(D_n(\alpha)\) to satisfy
\(\pi-D_n(\alpha)\leq\zeta_n\), equivalently
\(D_n(\alpha)\geq\widehat{\zeta}_n\).  Hence the first radial level at
which the swapped-anchor targets may become antipodal is determined by
\[
 D_n(\alpha)=\widehat{\zeta}_n.
\]

Its unique solution is \(\alpha=\alpha_n^\ast\), because
\[
 \sin^2\alpha_n^\ast
 =\frac{1-\rho_n}{1+\rho_n}=\frac{s_n^2}{c_n^2}.
\]
Thus
\[
 D_n(\alpha_n^\ast)=\widehat{\zeta}_n,\qquad
 \Theta_n(\alpha_n^\ast)=\frac{\widehat{\zeta}_n}{2},\qquad
 r_n(\alpha_n^\ast)=1.
\]
In particular, \(\alpha_n^\ast\) is the first radial level at which the
critical swapped-anchor targets may be antipodal; at this level each
target is also the spherical midpoint of the arc from \(v_i\) to
\(-v_j\), or from \(v_j\) to \(-v_i\), respectively.

\subsection{The critical involution and characterization of the gain}
\label{sec:critical-involution}

Unequal radial levels force the rest of the gain. A simple calculation shows that the two targets
\(z_i(r)\) and \(z_j(t)\) are antipodal exactly when \(rt=1\). Similarly, a direct calculation shows that, on the
source side, the points \(p_\alpha(-v_j)\) and \(p_\beta(-v_i)\) are
exactly \(\widehat{\zeta}_n\) apart precisely when
\[
 \cos\alpha\cos\beta-\rho_n\sin\alpha\sin\beta=\rho_n.
\]
For each \(\alpha\), let \(F_n(\alpha)\) be the unique value of \(\beta\)
on this equality locus.  Requiring the targets to become antipodal
exactly on that locus gives
\[
 r_n(\alpha)\,r_n(F_n(\alpha))=1
 \qquad(0<\alpha<\widehat{\zeta}_n).
\]
This is the origin of both the radial involution and the reciprocal
definition of the second part of the gain.

We call
\phantomsection\label{def:critical-source-distance}
\[
\begin{aligned}
 D_n^\times\colon
 [0,\widehat{\zeta}_n]\times[0,\widehat{\zeta}_n]
 &\longrightarrow[0,\pi],\\
 D_n^\times(\alpha,\beta)
 &:=
 \arccos\bigl(
 \cos\alpha\cos\beta-\rho_n\sin\alpha\sin\beta
 \bigr)
\end{aligned}
\]
the \emph{critical source-distance function}.

\begin{proposition}[The critical involution]
\label{prop:ac-critical-involution}
The function \(F_n\) defined in \eqref{eq:ac-F-explicit} is the unique
function from \([0,\widehat{\zeta}_n]\) to itself satisfying
\begin{equation}
 \cos\alpha\cos F_n(\alpha)
 -\rho_n\sin\alpha\sin F_n(\alpha)=\rho_n.
 \label{eq:critical-involution-level}
\end{equation}
It is continuous, strictly decreasing, and involutive.  Its endpoint
values and unique fixed point are
\[
 F_n(0)=\widehat{\zeta}_n,\qquad
 F_n(\widehat{\zeta}_n)=0,\qquad
 F_n(\alpha_n^\ast)=\alpha_n^\ast.
\]
Moreover, for \(0\leq\alpha,\beta\leq\widehat{\zeta}_n\),
\begin{equation}
 D_n^\times(\alpha,\beta)\leq\widehat{\zeta}_n
 \quad\Longleftrightarrow\quad
 \beta\leq F_n(\alpha),
 \label{eq:critical-involution-sublevel}
\end{equation}
and
\begin{equation}
 D_n^\times(\alpha,\beta)\geq\widehat{\zeta}_n
 \quad\Longleftrightarrow\quad
 \beta\geq F_n(\alpha).
 \label{eq:critical-involution-superlevel}
\end{equation}
\end{proposition}

\begin{proof}
Put
\[
 \begin{aligned}
 G_n\colon[0,\widehat{\zeta}_n]\times[0,\widehat{\zeta}_n]&\longrightarrow\R,\\
 G_n(\alpha,\beta)&:=
 \cos\alpha\cos\beta-{\rho_n}\sin\alpha\sin\beta.
 \end{aligned}
\]
For fixed \(\alpha\), the function \(\beta\mapsto G_n(\alpha,\beta)\)
is strictly decreasing on \([0,\widehat{\zeta}_n]\).  Indeed,
\[
 \partial_\beta G_n(\alpha,\beta)
 =-\cos\alpha\sin\beta-{\rho_n}\sin\alpha\cos\beta\leq0,
\]
Because \(\rho_n\), \(\cos\alpha\), and \(\cos\beta\) are positive on
the parameter square, equality forces
\(\sin\alpha=\sin\beta=0\).  Thus the derivative vanishes only at the
corner \((\alpha,\beta)=(0,0)\).  Moreover,
\[
 G_n(\alpha,0)=\cos\alpha\geq{\rho_n},\qquad
 G_n(\alpha,\widehat{\zeta}_n)
 ={\rho_n}(\cos\alpha-\sin \widehat{\zeta}_n\sin\alpha)\leq{\rho_n}.
\]

Consequently, for every \(\alpha\in[0,\widehat{\zeta}_n]\), there is a unique
\(F_n(\alpha)\in[0,\widehat{\zeta}_n]\) satisfying
\[
 \cos\alpha\cos F_n(\alpha)
 -\rho_n\sin\alpha\sin F_n(\alpha)=\rho_n.
 \tag{\ref{eq:critical-involution-level}}
\]

To recover the formula \eqref{eq:ac-F-explicit}, put
\(u_\alpha:=\tan F_n(\alpha)\).  Since
\(0\leq F_n(\alpha)\leq\widehat{\zeta}_n<\tfrac{\pi}{2}\), its cosine is
positive and
\[
 \sec F_n(\alpha)
 =\sqrt{1+\tan^2F_n(\alpha)}
 =\sqrt{1+u_\alpha^2}.
\]

Dividing both sides of \eqref{eq:critical-involution-level} by
\(\cos F_n(\alpha)\) therefore gives
\begin{equation}\label{eq:ac-2-1}
     \cos\alpha-\rho_n\sin\alpha\,u_\alpha
 =\rho_n\sqrt{1+u_\alpha^2}.
\end{equation}
Squaring, moving all terms to one side, and dividing by
\(-\cos^2\alpha\) gives
\[
 \rho_n^2u_\alpha^2+2\rho_n\tan\alpha\,u_\alpha
 +\rho_n^2\sec^2\alpha-1=0.
\]
The nonnegative root is
\[
 \tan F_n(\alpha)
 =\frac{\sqrt{1-\rho_n^2}-\sin\alpha}
        {\rho_n\cos\alpha}
 =\frac{\sin\widehat{\zeta}_n-\sin\alpha}
        {\rho_n\cos\alpha}.
\]

Squaring could have introduced a solution on the negative square-root
branch.  To exclude this possibility, insert the candidate root into
the left-hand side of the unsquared equation:
\[
 \cos\alpha-\rho_n\sin\alpha\,u_\alpha
 =\frac{1-\sin\alpha\sin\widehat{\zeta}_n}{\cos\alpha}.
\]
Since \(0\leq\alpha\leq\widehat{\zeta}_n<\tfrac{\pi}{2}\), its denominator is
positive and its numerator satisfies
\[
 1-\sin\alpha\sin\widehat{\zeta}_n
 \geq 1-\sin^2\widehat{\zeta}_n
 =\rho_n^2>0.
\]
Thus both sides of \eqref{eq:ac-2-1} are positive, and hence the
candidate satisfies the original unsquared equation.  This proves the
formula \eqref{eq:ac-F-explicit}; in particular, \(F_n\) is continuous.

Since \(G_n\) is symmetric and strictly decreasing in either variable,
\(F_n\) is strictly decreasing and uniqueness gives
\(F_n(F_n(\alpha))=\alpha\).  Thus \(F_n\) is a decreasing involution.
In particular,
\[
 F_n(0)=\widehat{\zeta}_n,\qquad F_n(\widehat{\zeta}_n)=0.
\]
The geometric derivation above predicts that \(\alpha_n^\ast\) is the
fixed point of \(F_n\).  Algebraically, \(F_n(\alpha)=\alpha\) is
equivalent to
\[
 1-(1+{\rho_n})\sin^2\alpha={\rho_n},
 \qquad
 \sin^2\alpha=\frac{1-{\rho_n}}{1+{\rho_n}}
 =\frac{s_n^2}{c_n^2}.
\]
Since
\((1-{\rho_n})/(1+{\rho_n})<1-\rho_n^2=\sin^2\widehat{\zeta}_n\), this solution lies
in \((0,\widehat{\zeta}_n)\).  Hence the unique fixed point is
\[
 F_n(\alpha_n^\ast)=\alpha_n^\ast,
 \qquad
 \alpha_n^\ast=\arcsin\!\left(\frac{s_n}{c_n}\right).
\]

Finally,
\[
 D_n^\times(\alpha,\beta)\leq\widehat{\zeta}_n
 \quad\Longleftrightarrow\quad
 G_n(\alpha,\beta)\geq\rho_n.
\]
Since \(\beta\mapsto G_n(\alpha,\beta)\) is strictly decreasing and
equals \(\rho_n\) precisely at \(\beta=F_n(\alpha)\), this is equivalent
to \(\beta\leq F_n(\alpha)\), proving
\eqref{eq:critical-involution-sublevel}.  The same strict monotonicity
gives
\[
 G_n(\alpha,\beta)\leq\rho_n
 \quad\Longleftrightarrow\quad
 \beta\geq F_n(\alpha),
\]
which is equivalent to
\eqref{eq:critical-involution-superlevel}.
\end{proof}

The reciprocal condition now gives
\[
 r_n(\alpha)=\frac1{r_n(F_n(\alpha))}
 \qquad(\alpha_n^\ast\leq\alpha<\widehat{\zeta}_n),
\]
which is precisely the second branch of \eqref{eq:ac-1}.  It also gives
\(r_n(\alpha_n^\ast)=1\) and
\(r_n(\alpha)\to+\infty\) as
\(\alpha\uparrow\widehat{\zeta}_n\).

The preceding two requirements characterize the gain among continuous
normalized-chord gains starting at the anchor that satisfy the stated
equal-level sharpness and reciprocal compatibility conditions.

\begin{proposition}[Characterization of the canonical gain]
\label{prop:canonical-gain-characterization}
Let
\[
 g\colon[0,\widehat{\zeta}_n)\longrightarrow[0,\infty)
\]
be continuous, with \(g(0)=0\).  Fix distinct indices
\(i,j\in\{1,\ldots,n+2\}\) and assume:
\begin{enumerate}[label=\textup{(\roman*)},nosep]
 \item For every \(0\leq\alpha\leq\alpha_n^\ast\), the equal-level
 swapped-anchor comparison is sharp:
 \[
  d_n\bigl(z_i(g(\alpha)),z_j(g(\alpha))\bigr)
  =\zeta_n+D_n(\alpha).
 \]
 \item For every \(0<\alpha<\widehat{\zeta}_n\), the reciprocal
 compatibility condition holds:
 \[
  g(\alpha)\,g(F_n(\alpha))=1.
 \]
\end{enumerate}
Then
\[
 g(\alpha)=r_n(\alpha)
 \qquad(0\leq\alpha<\widehat{\zeta}_n).
\]
In particular, \(g(\alpha)\to+\infty\) as
\(\alpha\uparrow\widehat{\zeta}_n\).
\end{proposition}

\begin{proof}
For \(r\geq0\), direct calculation gives
\begin{equation}\label{eq:ac-3-2-1}
     z_i(r)\cdot z_j(r)
 =\frac{-\rho_n(1+r^2)-2r}
        {1+2\rho_n r+r^2}.
\end{equation}

On
\(0\leq r\leq1\), put
\(\theta(r):=d_n(v_i,z_i(r))\).

The great-circle calculation in
\Cref{sec:swapped-anchor-configuration} and
\Cref{lem:ac-chord-geometry} give
\[
 d_n(z_i(r),z_j(r))=\zeta_n+2\theta(r),
 \qquad
 r=\frac{\sin\theta(r)}
         {\sin(\widehat{\zeta}_n-\theta(r))}.
\]
The right-hand side is strictly increasing for
\(0\leq\theta\leq\widehat{\zeta}_n/2\).  We now determine all positive
solutions of condition~\textup{(i)} for
\(0<\alpha<\alpha_n^\ast\).

Define
\[
 T(r):=d_n(z_i(r),z_j(r))\qquad(r>0).
\]
\Cref{eq:ac-3-2-1} shows that \(T(r)=T(1/r)\).  On
\(0<r\leq1\), \Cref{lem:ac-chord-geometry} gives a bijection between
\(r\) and
\(0<\theta\leq\widehat{\zeta}_n/2\), because
\[
 r=\frac{\sin\theta}{\sin(\widehat{\zeta}_n-\theta)}
\]
is strictly increasing in \(\theta\).  Condition~\textup{(i)} requires
\[
 T(g(\alpha))=\zeta_n+D_n(\alpha)
 =\zeta_n+2\Theta_n(\alpha).
\]

The equation
\(T(r)=\zeta_n+2\Theta_n(\alpha)\) therefore has a unique solution in
\(0<r\leq1\), namely
\[
 a(\alpha)
 :=\frac{\sin\Theta_n(\alpha)}
         {\sin(\widehat{\zeta}_n-\Theta_n(\alpha))}.
\]
Reciprocal symmetry supplies the second solution \(a(\alpha)^{-1}\).
Conversely, if \(r>0\) is any solution, then
\(\min\{r,r^{-1}\}\in(0,1]\) is also a solution, so uniqueness on that
interval forces \(\min\{r,r^{-1}\}=a(\alpha)\).  These are therefore
the only two positive solutions.
Here \(0<a(\alpha)<1<a(\alpha)^{-1}\).  Since \(g\) is continuous and
\(g(0)=0\), it must select the first solution throughout
\((0,\alpha_n^\ast)\).  Indeed, the two continuous solution branches
\(a(\alpha)<1<a(\alpha)^{-1}\) are disjoint there, so continuity and
\(g(0)=0\) prevent a switch between them.  At
\(\alpha=\alpha_n^\ast\), condition~\textup{(i)} requires the target
distance to be \(\pi\), which occurs exactly when
\(g(\alpha_n^\ast)=1\).  Together with the endpoint \(g(0)=0\), this gives
\[
 g(\alpha)=a(\alpha)=r_n(\alpha)
 \qquad(0\leq\alpha\leq\alpha_n^\ast).
\]

Now let \(\alpha_n^\ast\leq\alpha<\widehat{\zeta}_n\) and put
\(\beta:=F_n(\alpha)\).  By
\Cref{prop:ac-critical-involution}, one has
\(0<\beta\leq\alpha_n^\ast\).  Condition
\textup{(ii)} and the first part give
\[
 g(\alpha)
 =\frac1{g(\beta)}
 =\frac1{r_n(\beta)}
 =r_n(\alpha),
\]
where the last equality is the second branch of \eqref{eq:ac-1}.
Finally, \(F_n(\alpha)\to0\) as
\(\alpha\uparrow\widehat{\zeta}_n\), so the same identity gives
\(g(\alpha)\to+\infty\).
\end{proof}

\subsection{Chord displacement and endpoint continuity}
\label{sec:chord-displacement-endpoints}

We now express the normalized-chord motion in terms of its geodesic
displacement and introduce the bounded chord fraction used to treat
the endpoint \(\alpha=\widehat{\zeta}_n\).

\paragraph{The \emph{maximal-displacement function}
\(\vartheta_n\).}
We encode the angular displacement of the normalized chord by \phantomsection\label{def:maximal-displacement-function}
\[
 \vartheta_n\colon[0,\widehat{\zeta}_n]
 \longrightarrow[0,\widehat{\zeta}_n],\qquad
 \vartheta_n(\alpha):=
 \begin{cases}
 \Theta_n(\alpha),&0\leq\alpha\leq\alpha_n^\ast,\\
 \widehat{\zeta}_n-\Theta_n(F_n(\alpha)),
 &\alpha_n^\ast\leq\alpha\leq\widehat{\zeta}_n.
 \end{cases}                                               \acEquationTag{eq:ac-5}
\]

The definitions of \(r_n\) and \(\vartheta_n\) give
\[
 \begin{gathered}
 r_n(\alpha)
 =\frac{\sin\vartheta_n(\alpha)}
        {\sin(\widehat{\zeta}_n-\vartheta_n(\alpha))}
 \qquad(0\leq\alpha<\widehat{\zeta}_n),\\
 \vartheta_n(F_n(\alpha))
 =\widehat{\zeta}_n-\vartheta_n(\alpha)
 \qquad(0\leq\alpha\leq\widehat{\zeta}_n).
 \end{gathered}                                             \acEquationTag{eq:ac-6}
\]
The endpoint in the first identity is the following precise
extended-real statement:
\[
 r_n(\widehat{\zeta}_n)=+\infty
 =\lim_{\alpha\uparrow\widehat{\zeta}_n}
 \frac{\sin\vartheta_n(\alpha)}
      {\sin(\widehat{\zeta}_n-\vartheta_n(\alpha))}.
\]
Indeed, \(\vartheta_n(\alpha)\to\widehat{\zeta}_n\), so the numerator tends
to \(\sin\widehat{\zeta}_n>0\), whereas the denominator tends to
\(0\) through positive values.

For the second identity, use the decreasing-involution and fixed-point
properties in \Cref{prop:ac-critical-involution}.  If
\(0\leq\alpha\leq\alpha_n^\ast\), then
\(F_n(\alpha)\geq\alpha_n^\ast\), and the two branches in
\eqref{eq:ac-5} give
\[
 \vartheta_n(F_n(\alpha))
 =\widehat{\zeta}_n-\Theta_n(F_n(F_n(\alpha)))
 =\widehat{\zeta}_n-\Theta_n(\alpha)
 =\widehat{\zeta}_n-\vartheta_n(\alpha).
\]
If \(\alpha_n^\ast\leq\alpha\leq\widehat{\zeta}_n\), then
\(F_n(\alpha)\leq\alpha_n^\ast\), and instead
\[
 \vartheta_n(F_n(\alpha))
 =\Theta_n(F_n(\alpha))
 =\widehat{\zeta}_n-\vartheta_n(\alpha),
\]
again by \eqref{eq:ac-5}.

The functions \(\Theta_n\) and \(\vartheta_n\) have different geometric
roles.  The value \(\Theta_n(\alpha)\) is half the source distance in
the equal-level swapped-anchor configuration; it prescribes the target
motion only on \(0\leq\alpha\leq\alpha_n^\ast\).

The relationship between the Euclidean chord ratio \(r_n(\alpha)\)
and the corresponding spherical displacement \(\vartheta_n(\alpha)\)
is depicted in \Cref{fig:chord-triangle-geometry}.

\begin{figure}[!t]
\centering
\begin{tikzpicture}[
  scale=1.18,
  line cap=round,
  line join=round,
  every node/.style={font=\small}
]
\coordinate (O) at (0,0);
\coordinate (V) at (55:2.35);
\coordinate (X) at (-25:2.35);
\coordinate (C) at ($(V)!.5454545!(X)$);
\coordinate (H) at ($(O)!2.35cm!(C)$);
\coordinate (W) at ($(O)!2.2!(C)$);

\draw[semithick] (O) circle[radius=2.35];
\draw[teal!70!black,very thick]
  (55:2.35) arc[start angle=55,end angle=10.62,radius=2.35];
\draw[blue!65!black,very thick]
  (10.62:2.35) arc[start angle=10.62,end angle=-25,radius=2.35];

\draw[black!32] (O)--(V);
\draw[black!32,dashed] (O)--(X);
\draw[orange!80!black,very thick] (V)--(X);
\draw[black!35] (O)--(C);
\draw[black!35,dashed] (H)--(W);
\draw[-{Stealth[length=2.0mm]},teal!70!black,thick] (C)--(H);

\draw[-{Stealth[length=2.0mm]},blue!65!black,semithick]
  (V)--(W)
  node[pos=.58,above right=-1pt]
    {\(r_n(\alpha)x\)};
\draw[-{Stealth[length=2.0mm]},red!70!black,semithick]
  (O)--(W);

\draw[teal!70!black,semithick]
  (55:.72) arc[start angle=55,end angle=10.62,radius=.72];
\node[teal!70!black] at (33:1.08)
  {\(\vartheta_n(\alpha)\)};
\draw[blue!65!black,semithick]
  (10.62:.48) arc[start angle=10.62,end angle=-25,radius=.48];
\node[blue!65!black,align=center] at (-16:1.28)
  {\(\widehat{\zeta}_n-\vartheta_n(\alpha)\)};

\fill (O) circle[radius=1.35pt]
  node[below left=2pt] {\(O\)};
\fill[orange!80!black] (V) circle[radius=2.0pt]
  node[above=2pt] {\(v_i\)};
\fill[orange!80!black] (X) circle[radius=2.0pt]
  node[below right=2pt] {\(x\)};
\fill[black!65] (C) circle[radius=1.8pt]
  node[above right=1pt] {\(c\)};
\fill[teal!70!black] (H) circle[radius=2.1pt]
  node[below right=2pt] {\(H_{\alpha,i}(x)\)};
\fill[red!70!black] (W) circle[radius=1.8pt]
  node[above right=1pt] {\(v_i+r_n(\alpha)x\)};

\node[align=center,anchor=north] at (1.70,-2.72)
  {Euclidean chord ratio\\[2pt]
   \(\displaystyle
     \frac{\lVert c-v_i\rVert}{\lVert x-c\rVert}=r_n(\alpha)\)};
\node[align=center,anchor=south] at (-1.55,2.73)
  {geodesic displacement\\[2pt]
   \(d_n(v_i,H_{\alpha,i}(x))=\vartheta_n(\alpha)\)};
\end{tikzpicture}
 \caption{The Euclidean gain and spherical displacement for
\(0\leq\alpha<\widehat{\zeta}_n\) and a point \(x\in M_i\) satisfying
\(v_i\cdot x=\rho_n\), equivalently
\(d_n(v_i,x)=\widehat{\zeta}_n\).  Here \(c\) is the unique point of
\([v_i,x]\) satisfying
\(\lVert c-v_i\rVert/\lVert x-c\rVert=r_n(\alpha)\).
Radial normalization sends \(c\) to
\(H_{\alpha,i}(x)=c/\lVert c\rVert\), whose geodesic displacement from
\(v_i\) is \(\vartheta_n(\alpha)\).  The vector triangle ending at
\(v_i+r_n(\alpha)x\) gives the sine-rule relation between these two
parameters. }
\label{fig:chord-triangle-geometry}
\end{figure}

\FloatBarrier

\begin{proposition}[Maximal and complementary displacement]
\label{prop:ac-maximal-displacement}
For every \(1\leq i\leq n+2\) and every
\(0\leq\alpha\leq\widehat{\zeta}_n\),

\[
 \begin{aligned}
 \vartheta_n(\alpha)
 &=\max_{x\in M_i}d_n\bigl(v_i,H_{\alpha,i}(x)\bigr),\\
 \widehat{\zeta}_n-\vartheta_n(\alpha)
 &=\max_{x\in M_i}d_n\bigl(H_{\alpha,i}(x),x\bigr).
 \end{aligned}
\]
In addition,
\[
 0\leq\vartheta_n(\alpha)\leq\alpha
 \qquad(0\leq\alpha\leq\widehat{\zeta}_n).
\]
The maximizers in the first formula are all of \(M_i\) when
\(\alpha=0\), and exactly the points
\(\{-v_j:j\neq i\}\) when
\(0<\alpha\leq\widehat{\zeta}_n\).  The maximizers in the second
formula are exactly the points \(\{-v_j:j\neq i\}\) when
\(0\leq\alpha<\widehat{\zeta}_n\), and all of \(M_i\) when
\(\alpha=\widehat{\zeta}_n\).
\end{proposition}
\begin{proof}
First suppose \(0\leq\alpha<\widehat{\zeta}_n\), put
\(r:=r_n(\alpha)\), and, for \(x\in M_i\), put
\(a:=v_i\cdot x\).  Then
\[
 \begin{aligned}
 \cos d_n\bigl(v_i,H_{\alpha,i}(x)\bigr)
 &=v_i\cdot H_{\alpha,i}(x)
 =\frac{1+ra}{\sqrt{1+2ra+r^2}},\\
 \cos d_n\bigl(H_{\alpha,i}(x),x\bigr)
 &=x\cdot H_{\alpha,i}(x)
 =\frac{a+r}{\sqrt{1+2ra+r^2}}.
 \end{aligned}
\]
The Voronoi-cell bound \eqref{eq:ac-10} gives
\(\rho_n\leq a\leq1\), and
\[
 \begin{aligned}
 \frac{d}{da}\left(\frac{1+ra}{\sqrt{1+2ra+r^2}}\right)
 &=\frac{r^2(a+r)}{(1+2ra+r^2)^{3/2}}\geq0,\\
 \frac{d}{da}\left(\frac{a+r}{\sqrt{1+2ra+r^2}}\right)
 &=\frac{1+ra}{(1+2ra+r^2)^{3/2}}>0.
 \end{aligned}
\]

Because \(\arccos\) is decreasing, both displayed distances are largest
when \(a=\rho_n\), except that the first distance is identically zero
when \(r=0\).  By \Cref{lem:ac-simplex-cell}, the equality
\(a=\rho_n\) holds precisely at the points
\(\{-v_k:k\neq i\}\).

At any such boundary point,
\(d_n(v_i,x)=\widehat{\zeta}_n\), and
\(H_{\alpha,i}(x)\) lies on the shorter great-circle arc from \(v_i\)
to \(x\).  Moreover,
\Cref{lem:ac-chord-geometry}, together with \eqref{eq:ac-6}, gives
\[
 d_n\bigl(v_i,H_{\alpha,i}(x)\bigr)=\vartheta_n(\alpha).
\]
Therefore
\[
 d_n\bigl(H_{\alpha,i}(x),x\bigr)
 =\widehat{\zeta}_n-\vartheta_n(\alpha).
\]
The derivative calculations and the boundary description prove both
maximum formulas and their equality statements for
\(\alpha<\widehat{\zeta}_n\).

At the endpoint \(\alpha=\widehat{\zeta}_n\), one has
\(H_{\widehat{\zeta}_n,i}(x)=x\) and
\(\vartheta_n(\widehat{\zeta}_n)=\widehat{\zeta}_n\).  The first maximum
formula follows from \(v_i\cdot x\geq\rho_n\), with equality exactly
at the points specified in \Cref{lem:ac-simplex-cell}.  The second distance is
identically zero, so every point of \(M_i\) is a maximizer.  This
completes the endpoint equality statements.

It remains to prove the comparison between the displacement and the
radial level.  Nonnegativity follows directly from the definition of
\(\vartheta_n\).  If
\(0\leq\alpha\leq\alpha_n^\ast\), then
\[
 \vartheta_n(\alpha)=\Theta_n(\alpha)\leq\alpha,
\]
because \(\Theta_n(0)=0\) and \(\Theta_n\) is \(1\)-Lipschitz by
\eqref{eq:theta-derivatives}.  Now suppose
\(\alpha_n^\ast\leq\alpha\leq\widehat{\zeta}_n\), and put
\(\widetilde{\alpha}:=F_n(\alpha)\).  Then
\(0\leq\widetilde{\alpha}\leq\alpha_n^\ast\), and
\Cref{prop:ac-critical-involution} gives
\[
 D_n^\times(\widetilde{\alpha},\alpha)=\widehat{\zeta}_n.
\]
Choose \(j\neq i\).  Consider the three source points
\[
 p_{\widetilde{\alpha}}(-v_j),\qquad
 p_{\widetilde{\alpha}}(-v_i),\qquad
 p_\alpha(-v_i).
\]
The first two points are
\[
D_n^\times(\widetilde{\alpha},\widetilde{\alpha})
=2\Theta_n(\widetilde{\alpha})
\]
apart.  The last two points lie on the same meridian and are
\(\alpha-\widetilde{\alpha}\) apart.  The distance between the first and third
points is
\(\widehat{\zeta}_n\).  The triangle inequality therefore gives
\[
 2\Theta_n(\widetilde{\alpha})
 \geq\widehat{\zeta}_n-\alpha+\widetilde{\alpha}.
\]
The path between the first and third points through the north pole has
length \(\alpha+\widetilde{\alpha}\), so
\(\widetilde{\alpha}\geq\widehat{\zeta}_n-\alpha\).  Consequently,
\[
 \Theta_n(\widetilde{\alpha})\geq\widehat{\zeta}_n-\alpha.
\]
The second branch of \eqref{eq:ac-5} now yields
\[
 \vartheta_n(\alpha)
 =\widehat{\zeta}_n-\Theta_n(\widetilde{\alpha})\leq\alpha,
\]
as required.
\end{proof}

The function \(\vartheta_n\) starts at \(0\), reaches
\(\widehat{\zeta}_n/2\) at \(\alpha_n^\ast\), and reaches
\(\widehat{\zeta}_n\) when \(\alpha=\widehat{\zeta}_n\), the identity-band threshold.

More precisely, fix a point \(x\in M_i\) satisfying
\(v_i\cdot x=\rho_n\).  For
\(0\leq\alpha\leq\widehat{\zeta}_n\), the points
\(H_{\alpha,i}(x)\) and
\(H_{F_n(\alpha),i}(x)\) lie on the same minimizing geodesic segment
from \(v_i\) to \(x\), whose length is \(\widehat{\zeta}_n\), and

\[
\begin{aligned}
 d_n\bigl(v_i,H_{\alpha,i}(x)\bigr)
 &=\vartheta_n(\alpha),\\
 d_n\bigl(v_i,H_{F_n(\alpha),i}(x)\bigr)
 &=\widehat{\zeta}_n-\vartheta_n(\alpha).
\end{aligned}
\]
Thus reflection across the midpoint of this geodesic segment exchanges
these two target points.

\paragraph{Application to the canonical gain.}

For \(0\leq\alpha<\widehat{\zeta}_n\) and
\(x\in M_i\) satisfying \(v_i\cdot x=\rho_n\),
\Cref{lem:ac-chord-geometry} and \eqref{eq:ac-6} relate the Euclidean
chord ratio \(r_n(\alpha)\) to the geodesic displacement
\(\vartheta_n(\alpha)\) produced by radial normalization through the
sine-rule identity.  Both quantities are depicted in
\Cref{fig:chord-triangle-geometry}.

\paragraph{The bounded chord fraction \(\lambda_n\).}
The geometrically relevant bounded form of the gain is
\phantomsection\label{def:bounded-chord-fraction}
\[
 \lambda_n\colon[0,\frac{\pi}{2}]\longrightarrow[0,1],\qquad
 \lambda_n(\alpha):=
 \begin{cases}
 \displaystyle\frac{r_n(\alpha)}{1+r_n(\alpha)},
   &0\leq\alpha<\widehat{\zeta}_n,\\[6pt]
 1,&\widehat{\zeta}_n\leq\alpha\leq\frac{\pi}{2}.
 \end{cases}
\]
The role of \(\lambda_n\) is twofold.  First, it replaces the extended gain
\(r_n\in[0,+\infty]\) by a bounded chord parameter in \([0,1]\);
the value \(r_n=+\infty\) corresponds to \(\lambda_n=1\).
Second, it removes the unbounded quantity \(r_n\) from the
formula for the target.  For
\(0\leq\alpha<\widehat{\zeta}_n\), multiplication of the vector being
normalized by the positive scalar \(1/(1+r_n(\alpha))\) gives

\[
\begin{aligned}
 H_{\alpha,i}(x)
 &=\frac{v_i+r_n(\alpha)x}
 {\lVert v_i+r_n(\alpha)x\rVert}\\
 &=\frac{(1-\lambda_n(\alpha))v_i+\lambda_n(\alpha)x}
 {\lVert(1-\lambda_n(\alpha))v_i+\lambda_n(\alpha)x\rVert}.
\end{aligned}
\]
The rightmost expression involves only the bounded parameter
\(\lambda_n(\alpha)\), and it remains meaningful when
\(\lambda_n(\alpha)=1\), where its value is \(x\).  This observation is
the basis of the endpoint-continuity statement below.

The following lemma shows that the endpoint convention in
\Cref{def:anchored-chord-relation} is the continuous extension of the
finite-gain formula.  It will allow the endpoint cases in
\Cref{lem:ac-near,prop:ac-far-affine} to be recovered from the finite-gain
cases.

\begin{lemma}[Endpoint continuity of the chord targets]
\label{lem:ac-endpoint-continuity}
The function \(\lambda_n\) is continuous on \([0,\tfrac{\pi}{2}]\).  For every
\(1\leq i\leq n+2\), the map
\[
 [0,\widehat{\zeta}_n]\times M_i\longrightarrow\Sp(W_{n+2}),
 \qquad
 (\alpha,x)\longmapsto H_{\alpha,i}(x),
\]
is continuous when \(H_{\widehat{\zeta}_n,i}(x):=x\).  In particular,
\[
 H_{\alpha,i}(x)\longrightarrow x
 \qquad\text{uniformly for \(x\in M_i\) as }
 \alpha\uparrow\widehat{\zeta}_n.
\]
\end{lemma}

\begin{proof}
Since \(r_n(\alpha)\to+\infty\) as
\(\alpha\uparrow\widehat{\zeta}_n\), one has
\(\lambda_n(\alpha)\to1=\lambda_n(\widehat{\zeta}_n)\).
Thus \(\lambda_n\) is continuous at the only point requiring
verification, and it is constant on
\([\widehat{\zeta}_n,\tfrac{\pi}{2}]\).

If \(x\in M_i\), then \eqref{eq:ac-10} gives
\(v_i\cdot x\geq\rho_n>0\).  Therefore, for
\(0\leq\lambda\leq1\),
\[
\begin{aligned}
 \lVert(1-\lambda)v_i+\lambda x\rVert^2
 &=(1-\lambda)^2+\lambda^2
   +2\lambda(1-\lambda)(v_i\cdot x)\\
 &\geq(1-\lambda)^2+\lambda^2\geq\frac12.
\end{aligned}
\]
The vector being normalized is consequently bounded uniformly away
from the origin.  Hence the bounded target formula above defines a
jointly continuous map through
\(\alpha=\widehat{\zeta}_n\), where it equals \(x\).
Moreover,
\[
 \lVert(1-\lambda_n(\alpha))v_i+\lambda_n(\alpha)x-x\rVert
 \leq2(1-\lambda_n(\alpha))
\]
uniformly in \(x\).  Radial normalization is uniformly continuous on
the region where the norm is at least \(1/\sqrt2\), proving the final
assertion.
\end{proof}

\begin{corollary}[Finite-gain closure]
\label{cor:ac-finite-gain-closure}
Let \(g\colon[-1,1]\to\R\) be continuous and let \(c\in(-1,1)\).

Fix \(1\leq i,j\leq n+2\), \(x\in M_i\), and \(y\in M_j\).  For
\(0\leq\alpha,\beta\leq\widehat{\zeta}_n\), put
\[
 Q:=p_\alpha(x)\cdot p_\beta(y),
 \qquad
 R:=H_{\alpha,i}(x)\cdot H_{\beta,j}(y).
\]
The following two closure statements hold.
\begin{enumerate}[label=\textup{(\roman*)},nosep]
 \item If \(R\geq g(Q)\) for every
 \(0\leq\alpha,\beta<\widehat{\zeta}_n\) for which \(Q>c\), then the same
 estimate holds when endpoint radial levels are allowed, provided
 \(Q>c\).
 \item If \(R\leq g(Q)\) for every
 \(0\leq\alpha,\beta<\widehat{\zeta}_n\) for which \(Q<c\), then the same
 estimate holds when endpoint radial levels are allowed, provided
 \(Q<c\).
\end{enumerate}
\end{corollary}

\begin{proof}
Only configurations for which at least one radial level equals
\(\widehat{\zeta}_n\) require consideration.  Hold \(x,y\) and each
non-endpoint radial level fixed, and replace every endpoint radial level
by a sequence increasing to \(\widehat{\zeta}_n\).  The corresponding
source inner products converge to \(Q\), while
\Cref{lem:ac-endpoint-continuity} gives convergence of the target inner
products to \(R\).  The strict inequality \(Q>c\), respectively
\(Q<c\), therefore persists along the approximating sequence for all
sufficiently large indices.  Applying the finite-gain estimate and
passing to the limit proves the assertion, because \(g\) is continuous.
\end{proof}

\subsection{Basic properties of the correspondence}
\label{sec:basic-correspondence-properties}
\label{sec:anchored-chord-relation}

The Voronoi-cell bound \eqref{eq:ac-10} was established in
\Cref{lem:ac-simplex-cell}.  We now check that the maps in
\eqref{eq:ac-7} are well-defined and then prove that the formulas in
\Cref{def:anchored-chord-relation} define a correspondence.
If \(x\in M_i\) and \(0\leq\alpha<\widehat{\zeta}_n\), then
\(v_i\cdot x\geq\rho_n>0\), and hence
\(v_i+r_n(\alpha)\,x\neq0\).
For \(\alpha<\widehat{\zeta}_n\), \(H_{\alpha,i}(x)\) is the radial normalization of
the point \((v_i+r_n(\alpha)\,x)/(1+r_n(\alpha))\) on the Euclidean chord
\([v_i,x]\).  Since \(r_n(0)=0\),
\[
 H_{0,i}(x)=v_i\qquad(x\in M_i).
\]
This is the precise meaning of \emph{anchor}: for every \(x\in M_i\),
the path \(\alpha\mapsto H_{\alpha,i}(x)\) begins at \(v_i\).

\begin{proposition}[Basic properties of the anchored--chord correspondence]
\label{prop:ac-basic}
The relation \(\mathcal R_n\) is a correspondence between
\(\Sp^{n+1}\) and \(\Sp^n\).

For \(x\in M_i\),
\[
 x\cdot H_{\alpha,i}(x)>0
 \qquad(0\leq\alpha\leq\widehat{\zeta}_n).                            \acEquationTag{eq:ac-11}
\]
If \(i,j\in I(x)\), then
\[
 d_n(H_{\alpha,i}(x),H_{\alpha,j}(x))\leq\zeta_n.        \acEquationTag{eq:ac-12}
\]
for every \(0\leq\alpha\leq\widehat{\zeta}_n\).

At the north pole,
\[
 \mathcal R_n[N]=\{v_1,\ldots,v_{n+2}\},
\]
and therefore this fiber has diameter exactly \(\zeta_n\).

\end{proposition}

\begin{proof}
\Cref{lem:ac-simplex-cell} gives
\(v_i\cdot x\geq\rho_n\).  For
\(0\leq\alpha<\widehat{\zeta}_n\), put
\(r:=r_n(\alpha)\in[0,\infty)\).  Then
\[
 \|v_i+r x\|^2=1+2r(v_i\cdot x)+r^2>0,
\qquad
 x\cdot H_{\alpha,i}(x)
 =\frac{v_i\cdot x+r}{\sqrt{1+2r(v_i\cdot x)+r^2}}>0.
\]

This proves \eqref{eq:ac-11} for
\(0\leq\alpha<\widehat{\zeta}_n\).  At
\(\alpha=\widehat{\zeta}_n\), the endpoint definition gives
\[
 x\cdot H_{\widehat{\zeta}_n,i}(x)=x\cdot x=1>0,
\]
which completes the proof of \eqref{eq:ac-11}.  If \(i=j\),
\eqref{eq:ac-12} is trivial.  Suppose that \(i\neq j\) and
\(i,j\in I(x)\).  First let
\(0\leq\alpha<\widehat{\zeta}_n\), put
\(r:=r_n(\alpha)\), and write
\(a:=v_i\cdot x=v_j\cdot x\geq {\rho_n}\).  The inner product of the
two targets is
\[
 \frac{-{\rho_n}+2ra+r^2}{1+2ra+r^2}\geq-{\rho_n},
\]
because adding \({\rho_n}\) to the fraction produces the nonnegative numerator
\((1+{\rho_n})r(2a+r)\).  Thus the target inner product is at least
\(-\rho_n=\cos\zeta_n\), which proves \eqref{eq:ac-12} for
\(0\leq\alpha<\widehat{\zeta}_n\).

At \(\alpha=\widehat{\zeta}_n\), one has
\[
 H_{\widehat{\zeta}_n,i}(x)
 =x
 =H_{\widehat{\zeta}_n,j}(x),
\]
so the distance in \eqref{eq:ac-12} is zero.  Thus
\eqref{eq:ac-12} holds throughout its stated radial range.
\Cref{lem:ac-endpoint-continuity} also gives
\(H_{\alpha,i}(x)\to x\) uniformly for \(x\in M_i\) as
\(\alpha\uparrow\widehat{\zeta}_n\).  Hence the two pieces of
\eqref{eq:ac-8} agree at
\(\alpha=\widehat{\zeta}_n\).

To verify surjectivity of the first-coordinate projection, fix an
upper-hemisphere point \(p_\alpha(x)\).  If
\(0\leq\alpha\leq\widehat{\zeta}_n\), choose \(i\in I(x)\); then
\(x\in M_i\) and
\[
 \bigl(p_\alpha(x),H_{\alpha,i}(x)\bigr)\in\mathcal R_n^+.
\]
If \(\widehat{\zeta}_n\leq\alpha\leq\tfrac{\pi}{2}\), then
\[
 \bigl(p_\alpha(x),x\bigr)\in\mathcal R_n^+
\]
by the identity-band definition.  Thus the first-coordinate projection
of \(\mathcal R_n^+\) is the entire closed upper hemisphere of
\(\Sp^{n+1}\).  Antipodal extension therefore makes the first-coordinate projection of
\(\mathcal R_n\) all of \(\Sp^{n+1}\).  The identity band makes its
second-coordinate projection all of \(\Sp^n\).

Finally, \(p_0(x)=N\) is independent of \(x\), and
\(H_{0,i}(x)=v_i\).  Thus the anchored--chord part contributes exactly
the targets \(v_1,\ldots,v_{n+2}\) over \(N\).  The identity band does
not meet \(N\), and the antipodal extension contributes no additional
pair over \(N\).  Hence
\[
 \mathcal R_n[N]=\{v_1,\ldots,v_{n+2}\},
\]
as asserted.
\end{proof}

\subsection{Asymptotic profile of the gain}
\label{sec:asymptotic-gain-profile}

For each fixed \(0\leq\alpha<\tfrac{\pi}{2}\), the value
\(r_n(\alpha)\) is defined for all sufficiently large \(n\), and along
this tail one has
\[
 r_n(\alpha)\longrightarrow
 r_\infty(\alpha):=\frac{\sin\alpha}
 {\sqrt{2-\sin^2\alpha}},
 \qquad
 \lambda_n(\alpha)\longrightarrow
 \lambda_\infty(\alpha):=
 \frac{\sin\alpha}{\sin\alpha+\sqrt{2-\sin^2\alpha}}.
\]
Indeed, \(\widehat{\zeta}_n\uparrow\tfrac{\pi}{2}\) and
\(\alpha_n^\ast\uparrow\tfrac{\pi}{2}\).  Hence, for all sufficiently large
\(n\), one has
\(\alpha\leq\alpha_n^\ast<\widehat{\zeta}_n\), so the first formula
defining \(r_n\) applies.  Since \({\rho_n}\to0\) and
\(c_n\to1/\sqrt2\), that formula gives the stated limit for
\(r_n(\alpha)\); the limit for
\(\lambda_n(\alpha)=r_n(\alpha)/(1+r_n(\alpha))\) follows.  The
convergence is not uniform near
\(\tfrac{\pi}{2}\): every finite-dimensional profile reaches \(1\) at \(\widehat{\zeta}_n\)
and remains there, whereas
\(\lambda_\infty(\alpha)\to1/2\) as
\(\alpha\uparrow\tfrac{\pi}{2}\).  This nonuniform behavior near \(\tfrac{\pi}{2}\) is visible in
\Cref{fig:canonical-gain-profiles}.

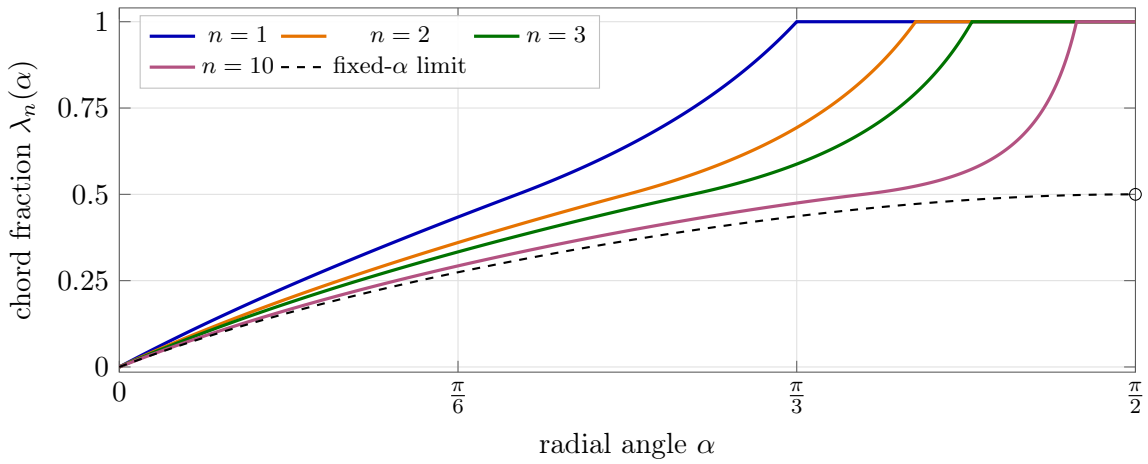
\begin{figure}[htbp]
\centering
\begin{tikzpicture}
\begin{axis}[
  width=0.91\textwidth,
  height=6.4cm,
  xmin=0,
  xmax=1.5707963268,
  ymin=-0.015,
  ymax=1.04,
  xlabel={radial angle \(\alpha\)},
  ylabel={chord fraction \(\lambda_n(\alpha)\)},
  xtick={0,0.5235987756,1.0471975512,1.5707963268},
  xticklabels={\(0\),\(\tfrac{\pi}{6}\),\(\tfrac{\pi}{3}\),\(\tfrac{\pi}{2}\)},
  ytick={0,0.25,0.5,0.75,1},
  grid=major,
  grid style={black!12},
  axis line style={black!65},
  tick style={black!65},
  legend style={
    at={(0.02,0.98)},
    anchor=north west,
    draw=black!25,
    fill=white,
    fill opacity=0.94,
    text opacity=1,
    legend columns=3,
    font=\footnotesize
  }
]
\addplot[blue!70!black,very thick]
 table[x=alpha,y=lambda]{figures/canonical-gain-n1.dat};
\addlegendentry{\(n=1\)}
\addplot[orange!90!black,very thick]
 table[x=alpha,y=lambda]{figures/canonical-gain-n2.dat};
\addlegendentry{\(n=2\)}
\addplot[green!45!black,very thick]
 table[x=alpha,y=lambda]{figures/canonical-gain-n3.dat};
\addlegendentry{\(n=3\)}
\addplot[magenta!70!black,very thick]
 table[x=alpha,y=lambda]{figures/canonical-gain-n10.dat};
\addlegendentry{\(n=10\)}
\addplot[black,thick,dashed]
 table[x=alpha,y=lambda]{figures/canonical-gain-limit.dat};
\addlegendentry{fixed-\(\alpha\) limit}
\addplot[only marks,mark=o,mark size=2.2pt,black,fill=white]
 coordinates {(1.5707963268,0.5)};
\end{axis}
\end{tikzpicture}
\caption{The canonical gain represented by its bounded Euclidean chord
fraction: \(\lambda_n(\alpha)=r_n(\alpha)/(1+r_n(\alpha))\) for
\(0\leq\alpha<\widehat{\zeta}_n\), and \(\lambda_n(\alpha)=1\) for
\(\widehat{\zeta}_n\leq\alpha\leq\tfrac{\pi}{2}\).  Each finite-dimensional
curve reaches \(1\) at \(\widehat{\zeta}_n\) and remains equal to \(1\)
throughout the identity band.  The dashed curve is the fixed-\(\alpha\) limit as
\(n\to\infty\); the open circle records its limiting value
\(1/2\) as \(\alpha\uparrow\tfrac{\pi}{2}\).}
\label{fig:canonical-gain-profiles}
\end{figure}
\FloatBarrier

\acEstimatesHeading

The organization of this section follows the admissible \((Q,R)\)-region
described in \Cref{sec:proof-strategy,fig:qr-distortion-region}.

\Cref{lem:ac-crossed-schedule} proves that
\[
 D_n^\times(\alpha,\beta)\leq\widehat{\zeta}_n
 \quad\Longrightarrow\quad
 \vartheta_n(\alpha)+\vartheta_n(\beta)
 \leq D_n^\times(\alpha,\beta).
\]
For pairs with a common anchor,
\Cref{lem:ac-common-anchor-pointwise} places each target in a single
isometric copy of \(\Sp^n\), with a pointwise distance bound proved by the
one-variable calculation in \Cref{lem:ac-common-anchor-scalar};
\Cref{lem:ac-same-anchor} then deduces the distortion estimate by the
triangle inequality.  For distinct anchors, the middle range
\(-\rho_n\leq Q\leq\rho_n\) requires no additional estimate: its source
distance lies in
\([\widehat{\zeta}_n,\zeta_n]\), so every target distance in \([0,\pi]\)
has defect at most \(\zeta_n\).
\Cref{lem:ac-near} handles \(Q\geq\rho_n\), while
\Cref{prop:ac-cap-dimension-reduction} reduces the cap optimization needed for
the remaining range to one variable and \Cref{prop:ac-far-affine} handles
\(Q\leq-\rho_n\).

\subsection{Chord geometry and radial compatibility}
\label{sec:ac-radial-compatibility}

For \(0\leq\alpha<\widehat{\zeta}_n\), the normalized-chord construction
is described by the Euclidean gain \(r_n(\alpha)\) and the maximal
geodesic displacement \(\vartheta_n(\alpha)\).  If \(x\in M_i\)
satisfies \(v_i\cdot x=\rho_n\), then
\(
 d_n\bigl(v_i,H_{\alpha,i}(x)\bigr)=\vartheta_n(\alpha),
\)
and \Cref{lem:ac-chord-geometry,fig:chord-triangle-geometry} give the
conversion between \(r_n(\alpha)\) and
\(\vartheta_n(\alpha)\).  At
\(\alpha=\widehat{\zeta}_n\), the endpoint convention gives
\(H_{\widehat{\zeta}_n,i}(x)=x\), and the corresponding displacement is
\(\vartheta_n(\widehat{\zeta}_n)=\widehat{\zeta}_n\).
The maximality assertion is \Cref{prop:ac-maximal-displacement}.  The
purpose of this subsection is to prove the two-level inequality in
\Cref{lem:ac-crossed-schedule}.

\begin{lemma}[Two-level radial inequality]
\label{lem:ac-crossed-schedule}
Let \(0\leq\alpha,\beta\leq\widehat{\zeta}_n\).  Recall from
\Cref{prop:ac-critical-involution} that
\[
 D_n^\times(\alpha,\beta)
 =
 \arccos\!\left(
 \cos\alpha\cos\beta-\rho_n\sin\alpha\sin\beta
 \right).
\]
In particular,
\(D_n(\alpha)=D_n^\times(\alpha,\alpha)\). Whenever \(D_n^\times(\alpha,\beta)\leq \widehat{\zeta}_n\), one has
\[
 \vartheta_n(\alpha)+\vartheta_n(\beta)
 \leq D_n^\times(\alpha,\beta).                         \acEquationTag{eq:ac-14}
\]
\end{lemma}

\begin{proof}

By symmetry assume \(\alpha\leq\beta\).  By
\eqref{eq:critical-involution-sublevel}, the hypothesis is equivalent
to \(\beta\leq F_n(\alpha)\).

\smallskip
\noindent\emph{Low--low levels.}
Suppose first that
\(\beta\leq\alpha_n^\ast\).  From
\(c_n^2+s_n^2=1\) and the identities
\(
 \sin\Theta_n(\alpha)=c_n\sin\alpha,\qquad
 \sin\Theta_n(\beta)=c_n\sin\beta,
\)
one obtains
\[
\begin{aligned}
 \bigl\|(\cos\alpha,s_n\sin\alpha)\bigr\|^2
 &=\cos^2\alpha+s_n^2\sin^2\alpha\\
 &=1-c_n^2\sin^2\alpha
 =\cos^2\Theta_n(\alpha),\\
 \bigl\|(\cos\beta,s_n\sin\beta)\bigr\|^2
 &=\cos^2\beta+s_n^2\sin^2\beta\\
 &=1-c_n^2\sin^2\beta
 =\cos^2\Theta_n(\beta).
\end{aligned}
\]

Since
\[
 0\leq\Theta_n(\alpha),\Theta_n(\beta)
 \leq\Theta_n(\alpha_n^\ast)
 =\frac{\widehat{\zeta}_n}{2}<\frac{\pi}{2},
\]
the two norms above are
\(\cos\Theta_n(\alpha)\) and \(\cos\Theta_n(\beta)\), respectively.
Cauchy--Schwarz
therefore gives
\[
 \cos\alpha\cos\beta+s_n^2\sin\alpha\sin\beta
 \leq
 \cos\Theta_n(\alpha)\cos\Theta_n(\beta).
\]
Using the two displayed sine identities and
\(s_n^2-c_n^2=-\rho_n\), we obtain
\[
 \begin{aligned}
 \cos\bigl(\Theta_n(\alpha)+\Theta_n(\beta)\bigr)
 &=\cos\Theta_n(\alpha)\cos\Theta_n(\beta)
   -c_n^2\sin\alpha\sin\beta\\
 &\geq\cos\alpha\cos\beta
   +(s_n^2-c_n^2)\sin\alpha\sin\beta\\
 &=\cos\alpha\cos\beta-\rho_n\sin\alpha\sin\beta\\
 &=\cos D_n^\times(\alpha,\beta).
 \end{aligned}
\]
Both \(\Theta_n(\alpha)+\Theta_n(\beta)\) and
\(D_n^\times(\alpha,\beta)\) belong to \([0,\pi]\).  Since cosine is
decreasing on this interval,
\[
 \vartheta_n(\alpha)+\vartheta_n(\beta)
 =\Theta_n(\alpha)+\Theta_n(\beta)
 \leq D_n^\times(\alpha,\beta),
\]
which proves \eqref{eq:ac-14} in the low--low case.

\smallskip
\noindent\emph{Mixed levels.}
Suppose next that \(\alpha<\alpha_n^\ast<\beta\), and put
\(\widetilde{\beta}:=F_n(\beta)\).  Since
\(\beta\leq F_n(\alpha)\), the properties in
\Cref{prop:ac-critical-involution} give
\(
 \widetilde{\beta}=F_n(\beta)
 \geq F_n(F_n(\alpha))=\alpha.
\)
Moreover, \(\beta>\alpha_n^\ast\) and
\Cref{prop:ac-critical-involution} gives
\(\widetilde{\beta}<\alpha_n^\ast\).  Thus
\(\alpha\leq\widetilde{\beta}<\alpha_n^\ast\).  The defining relation for
\(F_n\) and the second branch of \eqref{eq:ac-5} now give
\[
 D_n^\times(\widetilde{\beta},\beta)=\widehat{\zeta}_n,
 \qquad
 \vartheta_n(\beta)
 =\widehat{\zeta}_n-\Theta_n(\widetilde{\beta}).
\]

If \(\alpha=\widetilde{\beta}\), then
\[
 D_n^\times(\alpha,\beta)=\widehat{\zeta}_n,\qquad
 \vartheta_n(\alpha)+\vartheta_n(\beta)
 =\Theta_n(\alpha)+\widehat{\zeta}_n-\Theta_n(\widetilde{\beta})
 =\widehat{\zeta}_n,
\]
so \eqref{eq:ac-14} is an equality.  Suppose henceforth that
\(\alpha<\widetilde{\beta}\).  For
\(\tau\in[\alpha,\widetilde{\beta}]\), direct
differentiation gives
\[
 \partial_\tau D_n^\times(\tau,\beta)
 =\frac{\sin\tau\cos\beta+\rho_n\cos\tau\sin\beta}
 {\sin D_n^\times(\tau,\beta)}.
\]

Here \(0\leq\tau\leq\widetilde{\beta}<\beta<\tfrac{\pi}{2}\).  The quantity inside the
arccosine defining \(D_n^\times(\tau,\beta)\) satisfies
\[
 -1<-\rho_n
 \leq\cos\tau\cos\beta-\rho_n\sin\tau\sin\beta
 \leq\cos(\beta-\tau)<1.
\]
The last strict inequality follows from \(0<\beta-\tau<\tfrac{\pi}{2}\).
Thus \(0<D_n^\times(\tau,\beta)<\pi\), and consequently
\(
 \sin D_n^\times(\tau,\beta)>0,\qquad
 \cos\Theta_n(\tau)>0.
\)
The numerators of
\[
 \partial_\tau D_n^\times(\tau,\beta)
 \quad\text{and}\quad
 \Theta_n'(\tau)=\frac{c_n\cos\tau}{\cos\Theta_n(\tau)}
\]
are positive.  Cross multiplication and squaring therefore preserve
the desired inequality.  The resulting difference of squares is
\[
\begin{aligned}
&c_n^2\cos^2\tau\,
 \sin^2D_n^\times(\tau,\beta)-
 \bigl(\sin\tau\cos\beta+\rho_n\cos\tau\sin\beta\bigr)^2
 \cos^2\Theta_n(\tau)\\
&=
\frac{(1-\rho_n)(\tan\beta-\tan\tau)
\bigl(\tan\tau+(1+2\rho_n)\tan\beta\bigr)}
{2(1+\tan^2\tau)(1+\tan^2\beta)}
\geq0.
\end{aligned}
\]

The final inequality follows from \(0<\rho_n<1\) and
\(0\leq\tau<\beta<\tfrac{\pi}{2}\): both tangent factors in the numerator are
positive, as is the denominator.  Hence
 \(
 \partial_\tau D_n^\times(\tau,\beta)\leq\Theta_n'(\tau)
 \qquad(\alpha\leq\tau\leq\widetilde{\beta}).
 \)
After integration,
\[
 \widehat{\zeta}_n-D_n^\times(\alpha,\beta)
 =D_n^\times(\widetilde{\beta},\beta)-D_n^\times(\alpha,\beta)
 \leq\Theta_n(\widetilde{\beta})-\Theta_n(\alpha).
\]
Since
\[
 \vartheta_n(\alpha)=\Theta_n(\alpha),\qquad
 \vartheta_n(\beta)
 =\widehat{\zeta}_n-\Theta_n(\widetilde{\beta}),
\]
this inequality is equivalent to \eqref{eq:ac-14}.

\smallskip
\noindent\emph{The fixed-point case.}
If \(\alpha,\beta\geq\alpha_n^\ast\), the condition
\(\beta\leq F_n(\alpha)\) forces
\(\alpha=\beta=\alpha_n^\ast\); at this single point
\eqref{eq:ac-14} is an equality.

\end{proof}

\subsection{Pairs with a common anchor}
\label{sec:ac-common-anchor}

The common anchor permits a direct comparison between the source and
target geometries.  For each \(i\), define the linear map
\phantomsection\label{def:common-anchor-embedding}

\[
 \iota_{v_i}\colon W_{n+2}\longrightarrow W_{n+2}\oplus\R e_0,
 \qquad
 \iota_{v_i}(z)
 :=(z\cdot v_i)(c_ne_0+s_nv_i)+z-(z\cdot v_i)v_i.
\]
Equivalently, every \(z\in W_{n+2}\) has the unique orthogonal
decomposition
\[
 z=a v_i+w,\qquad
 a=z\cdot v_i,\qquad
 w=z-(z\cdot v_i)v_i\in W_{n+2}\cap v_i^\perp,
\]
and in these coordinates
\(
 \iota_{v_i}(a v_i+w)=a(c_ne_0+s_nv_i)+w.
\)

\smallskip
\noindent
\emph{Geometric meaning of the embedding.}
The formula shows that \(\iota_{v_i}\) fixes
\(W_{n+2}\cap v_i^\perp\) pointwise and sends \(v_i\) to
\[
 c_ne_0+s_nv_i
 =p_{\widehat{\zeta}_n/2}(v_i),
\]
the midpoint of the minimizing geodesic from
\(N=p_0(v_i)\) to \(p_{\widehat{\zeta}_n}(v_i)\).
Thus the half-angle constants in \(\iota_{v_i}\) center the comparison
on this source geodesic and make the scale
\(\widehat{\zeta}_n/2\) in \eqref{eq:ac-18} geometrically natural.
\par
\Cref{lem:ac-common-anchor-pointwise} proves that this map restricts to
an isometric embedding and satisfies the pointwise estimate used in
the same-anchor comparison.

\begin{proposition}[Pointwise comparison under the common-anchor embedding]
\label{lem:ac-common-anchor-pointwise}
Assume \(n\geq2\), and fix \(i\).  The linear map \(\iota_{v_i}\) defined
above restricts to an isometric embedding
\(\Sp^n\to\Sp^{n+1}\), and, for every \(x\in M_i\) and
\(0\leq\alpha\leq\widehat{\zeta}_n\),
\[
 d_{n+1}\bigl(p_\alpha(x),\iota_{v_i}H_{\alpha,i}(x)\bigr)
 \leq\frac{\widehat{\zeta}_n}{2}.                                      \acEquationTag{eq:ac-18}
\]
\end{proposition}

\begin{figure}[!ht]
    \centering
    \resizebox{1.00\textwidth}{!}{\colorlet{anchorone}{blue!70!black}
\colorlet{anchortwo}{orange!85!black}
\colorlet{construction}{teal!70!black}
\colorlet{resultcolor}{purple!70!black}
\colorlet{auxiliary}{black!35}
\colorlet{outline}{black!70}

\tikzset{
  >={Latex[length=2mm]},
  line cap=round,
  line join=round,
  every node/.style={font=\small},
  point/.style={circle,fill,inner sep=0pt,minimum size=4.6pt},
  comparison/.style={densely dashed,black!55,line width=.8pt},
  proofbox/.style={
    draw=black!28,
    rounded corners=2pt,
    fill=black!2,
    inner sep=6pt,
    align=center,
    text width=4.15cm
  },
  endpointbox/.style={
    draw=construction!55!black,
    rounded corners=2pt,
    fill=construction!6,
    inner sep=6pt,
    align=center,
    text width=4.25cm
  },
  implication/.style={-{Latex[length=2mm]},black!55,semithick}
}

\begin{tikzpicture}[every node/.style={font=\small}]

\begin{scope}[xshift=-4.65cm]
  \draw[outline,semithick] (0,0) circle[radius=2.35cm];
  \draw[construction,very thick,rotate=18]
    (0,0) ellipse[x radius=2.10cm,y radius=.66cm];

\draw[auxiliary,densely dashed]
  (2.25,0) arc[start angle=0,end angle=180,
               x radius=2.25,y radius=0.5];

\draw[auxiliary]
  (-2.25,0) arc[start angle=180,end angle=360,
                x radius=2.25,y radius=0.5];

  \coordinate (P) at (-.95,1.34);
  \coordinate (Hp) at (-.55,.50);
  \coordinate (Iv) at (1.80,.82);

  \fill[anchorone] (P) circle[radius=2.4pt]
    node[above left=2pt,anchor=south east,fill=white,inner sep=1.5pt]
      {\(p_\alpha(x)\)};
  \fill[construction] (Hp) circle[radius=2.4pt]
    node[below left=3pt,align=right,fill=white,inner sep=1.5pt]
      {\(\iota_{v_i}H_{\alpha,i}(x)\)};
  \fill[construction] (Iv) circle[radius=2.1pt]
    node[right=4pt,align=left,fill=white,inner sep=1.5pt]
      {\(\iota_{v_i}(v_i)\)\\[-2pt]
       \(=c_n e_0+s_n v_i\)};

  \draw[comparison]
    (P) to[bend right=12]
    node[midway,right=4pt,fill=white,inner sep=1.5pt]
      {\(\leq\widehat{\zeta}_n/2\)}
    (Hp);

  \node[construction,font=\footnotesize,fill=white,inner sep=1.5pt]
    at (.10,-.82)
    {\(\iota_{v_i}(\mathbb S^n)\subset\mathbb S^{n+1}\)};
  \node[outline,font=\footnotesize] at (0,2.66)
    {\(\mathbb S^{n+1}\)};
  \node[font=\scriptsize,black!58,align=center] at (0,-2.83)
    {schematic projection in
     \(\operatorname{span}\{e_0,v_i,w\}\)};
  \node[font=\footnotesize] at (0,-3.85)
    {\textup{(a)} one pointwise comparison};
\end{scope}

\begin{scope}[xshift=4.55cm]
  \coordinate (A)  at (-2.55,1.40);
  \coordinate (B)  at ( 2.55,1.40);
  \coordinate (Ap) at (-2.00,-1.15);
  \coordinate (Bp) at ( 2.00,-1.15);

\path[fill=construction!7,draw=construction!45,rounded corners=9pt]
    (-2.85,-1.70) rectangle (2.85,-.52);
  \node[construction,font=\footnotesize,anchor=west]
    at (-2.65,-1.52)
    {\(\iota_{v_i}(\mathbb S^n)\)};

  \draw[anchorone,very thick]
    (A) to[bend left=8]
    node[midway,above=4pt]
      {\(d_s=d_{n+1}(A,B)\)}
    (B);
  \draw[construction,very thick]
    (Ap) to[bend right=5]
    node[midway,below=11pt,align=center]
      {\(d_{n+1}(A',B')\)\\[0pt]
       \(=d_n(H_{\alpha,i}(x),H_{\beta,i}(y))=:d_t\)}
    (Bp);

  \draw[comparison]
    (A)--node[midway,left=5pt,align=right]
      {\(d_{n+1}(A,A')\)\\ \(\leq\widehat{\zeta}_n/2\)}
    (Ap);
  \draw[comparison]
    (B)--node[midway,right=5pt,align=left]
      {\(d_{n+1}(B,B')\)\\ \(\leq\widehat{\zeta}_n/2\)}
    (Bp);

  \node[point,fill=anchorone,label={[anchorone]above left:
    \(A=p_\alpha(x)\)}] at (A) {};
  \node[point,fill=anchorone,label={[anchorone]above right:
    \(B=p_\beta(y)\)}] at (B) {};
  \node[point,fill=construction,label={[construction]above left:
    \(A'=\iota_{v_i}H_{\alpha,i}(x)\)}] at (Ap) {};
  \node[point,fill=construction,label={[construction]above right:
    \(B'=\iota_{v_i}H_{\beta,i}(y)\)}] at (Bp) {};

  \node[
    draw=black!25,
    fill=white,
    rounded corners=2pt,
    inner sep=5pt,
    align=center,
    text width=7.1cm
  ] at (0,-3.10)
  {\(
    |d_s-d_t|
    \leq d_{n+1}(A,A')+d_{n+1}(B,B')
    \leq\widehat{\zeta}_n .
   \)};
  \node[font=\footnotesize] at (0,-3.85)
    {\textup{(b)} the metric comparison used in \Cref{lem:ac-same-anchor}};
\end{scope}

\end{tikzpicture}
 }
    \caption{The common-anchor comparison inside the source sphere. Panel
\textup{(a)} shows one source point and its corresponding target after
the isometric embedding \(\iota_{v_i}\); it is only a schematic
projection and does not assert that the displayed objects lie on one
great circle. Panel \textup{(b)} records the abstract four-point metric
comparison. The lower edge is the target distance because
\(\iota_{v_i}\) is an isometry, and the two dashed comparison distances
are at most \(\widehat{\zeta}_n/2\).}
    \label{fig:prop_4_2_fig}
\end{figure}
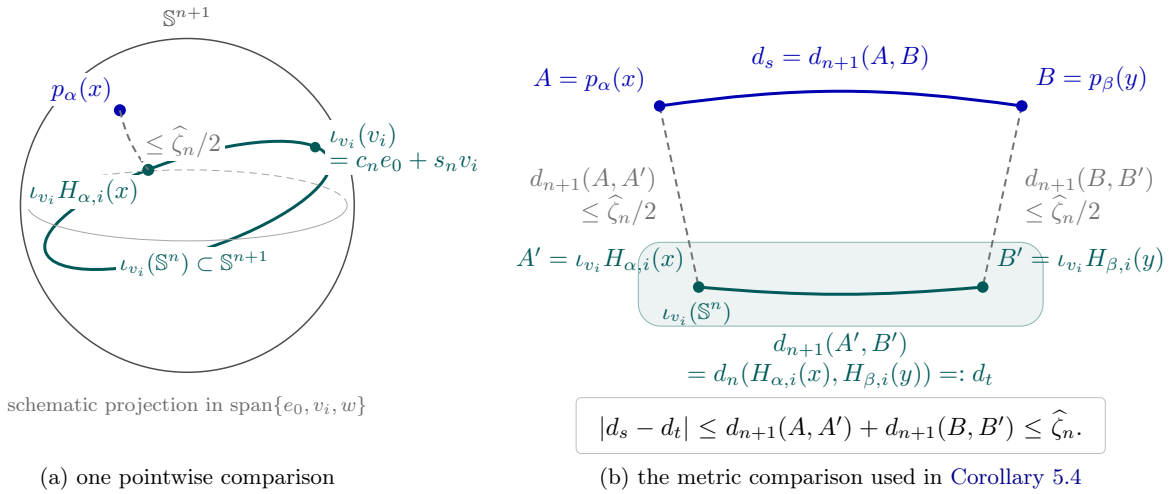

\paragraph{Analytic verification.}
Its proof
uses the auxiliary half-angle inequalities in
\Cref{lem:ac-half-angle-inequalities} and the one-variable endpoint
reduction in \Cref{lem:ac-common-anchor-scalar}.  Applying the pointwise
comparison to two pairs and using the triangle inequality then gives
the same-anchor distortion bound in \Cref{lem:ac-same-anchor}.

The auxiliary half-angle inequalities used in the analytic verification
are stated and proved in
Appendix~\ref{app:ac-half-angle-inequalities}.

The next lemma proves that the minimum of \(E_\alpha\) on
\([\rho_n,1]\) occurs at \(a=\rho_n\) or \(a=1\).  In the pointwise
common-anchor comparison, the scalar is \(a=v_i\cdot x\) and
\(
 E_\alpha(a)
 =p_\alpha(x)\cdot \iota_{v_i}H_{\alpha,i}(x).
\)
This identity is derived explicitly in the proof of
\Cref{lem:ac-common-anchor-pointwise}.

\begin{lemma}[Endpoint reduction for the comparison inner product]
\label{lem:ac-common-anchor-scalar}
Assume \(n\geq2\) and \(0<\alpha<\widehat{\zeta}_n\).  Put
\(r_\alpha:=r_n(\alpha)\), and, for \(a\in[\rho_n,1]\), define
\[
 E_\alpha(a)
 :=
 \frac{
 c_n\cos\alpha+r_\alpha\sin\alpha
 +(r_\alpha c_n\cos\alpha+s_n\sin\alpha)a
 -r_\alpha\sin\alpha(1-s_n)a^2}
 {\sqrt{1+2r_\alpha a+r_\alpha^2}}.
\]
Then
\[
 \min_{\rho_n\leq a\leq1}E_\alpha(a)
 =\min\{E_\alpha(\rho_n),E_\alpha(1)\}
 \geq c_n.
\]
\end{lemma}

The proof is given in
Appendix~\ref{app:ac-common-anchor-endpoint-reduction}.

\paragraph{Completion of the pointwise comparison.}
We now return to \(\iota_{v_i}\) and use the preceding analytic
estimates to prove \Cref{lem:ac-common-anchor-pointwise}.
\begin{proof}[Proof of \Cref{lem:ac-common-anchor-pointwise}]
The orthogonal decomposition
\(
 W_{n+2}=\R v_i\oplus\bigl(W_{n+2}\cap v_i^\perp\bigr)
\)
makes the definition of \(\iota_{v_i}\) unambiguous.  The vector
\(c_ne_0+s_nv_i\) has unit norm and is perpendicular to
\(W_{n+2}\cap v_i^\perp\).  Hence \(\iota_{v_i}\) preserves inner
products and restricts to an isometric embedding
\(
 \Sp(W_{n+2})\longrightarrow\Sp(W_{n+2}\oplus\R e_0).
\)
By the identification in \Cref{sec:anchored-chord-definition}, the
latter sphere is \(\Sp^{n+1}\).

Suppose first that \(0<\alpha<\widehat{\zeta}_n\), and put
\[
 a:=v_i\cdot x,\qquad
 x=av_i+w,\qquad
 w\in W_{n+2}\cap v_i^\perp.
\]
By the cell bound \eqref{eq:ac-10}, \(x\in M_i\) implies
\(a\in[\rho_n,1]\).  Since \(x\) is a unit vector,
\(
 \lVert w\rVert^2=1-a^2.
\)

Put \(r_\alpha:=r_n(\alpha)\).  The orthogonality of \(v_i\) and \(w\)
gives
\[
 \lVert v_i+r_\alpha x\rVert^2
 =(1+r_\alpha a)^2+r_\alpha^2\lVert w\rVert^2
 =1+2r_\alpha a+r_\alpha^2.
\]
It follows from the definitions of \(H_{\alpha,i}\) and \(\iota_{v_i}\)
that
\[
 \iota_{v_i}H_{\alpha,i}(x)
 =
 \frac{(1+r_\alpha a)(c_ne_0+s_nv_i)+r_\alpha w}
 {\sqrt{1+2r_\alpha a+r_\alpha^2}}.
\]
Taking the inner product with
\(
 p_\alpha(x)=e_0\cos\alpha+(av_i+w)\sin\alpha
\)
and using \(\lVert w\rVert^2=1-a^2\) gives
\[
 \begin{aligned}
 p_\alpha(x)\cdot \iota_{v_i}H_{\alpha,i}(x) = 
 \frac{
 c_n\cos\alpha+r_\alpha\sin\alpha
 +(r_\alpha c_n\cos\alpha+s_n\sin\alpha)a
 -r_\alpha\sin\alpha(1-s_n)a^2}
 {\sqrt{1+2r_\alpha a+r_\alpha^2}}
 =E_\alpha(a).
 \end{aligned}
\]
\Cref{lem:ac-common-anchor-scalar} now yields
\[
 p_\alpha(x)\cdot \iota_{v_i}H_{\alpha,i}(x)
 \geq c_n
 =\cos\!\left(\frac{\widehat{\zeta}_n}{2}\right).
\]

Since geodesic distance is the arccosine of the inner product,
\eqref{eq:ac-18} follows for
\(0<\alpha<\widehat{\zeta}_n\).

At \(\alpha=0\), one has \(p_0(x)=e_0\), \(H_{0,i}(x)=v_i\), and
\(
 p_0(x)\cdot \iota_{v_i}H_{0,i}(x)=e_0\cdot(c_ne_0+s_nv_i)=c_n,
\)
so \eqref{eq:ac-18} holds with equality.

At \(\alpha=\widehat{\zeta}_n\), let
\(\tau\uparrow\widehat{\zeta}_n\).  The source points
\(p_\tau(x)\) converge to \(p_{\widehat{\zeta}_n}(x)\), while
\Cref{lem:ac-endpoint-continuity} gives
\(
 H_{\tau,i}(x)\longrightarrow x
 =H_{\widehat{\zeta}_n,i}(x).
\)
Passing to the limit in the already proved pointwise inequality gives
\eqref{eq:ac-18} at the right endpoint and completes the proof on
\([0,\widehat{\zeta}_n]\).
\end{proof}

\begin{corollary}[Same-anchor distortion estimate]
\label{lem:ac-same-anchor}
Assume \(n\geq2\).  Fix \(i\), \(x,y\in M_i\), and
\(0\leq\alpha,\beta\leq \widehat{\zeta}_n\).  Then
\[
\left|
d_{n+1}(p_\alpha(x),p_\beta(y))
-d_n(H_{\alpha,i}(x),H_{\beta,i}(y))
\right|\leq \widehat{\zeta}_n<\zeta_n.                              \acEquationTag{eq:ac-16}
\]
\end{corollary}

\begin{proof}
Let \(\iota_{v_i}\) be the isometric embedding from
\Cref{lem:ac-common-anchor-pointwise}.  Applying the pointwise estimate
there to \((\alpha,x)\) and \((\beta,y)\), and then using the triangle
inequality twice, gives
\[
 \begin{aligned}
 &\left|
 d_{n+1}(p_\alpha(x),p_\beta(y))
 -d_n(H_{\alpha,i}(x),H_{\beta,i}(y))
 \right|\\
 &\quad=\left|
 d_{n+1}(p_\alpha(x),p_\beta(y))
 -d_{n+1}(\iota_{v_i}H_{\alpha,i}(x),\iota_{v_i}H_{\beta,i}(y))
 \right|\\
 &\quad\leq
 d_{n+1}(p_\alpha(x),\iota_{v_i}H_{\alpha,i}(x))
 +d_{n+1}(p_\beta(y),\iota_{v_i}H_{\beta,i}(y))
 \leq \widehat{\zeta}_n,
 \end{aligned}
\]
which is \eqref{eq:ac-16}.
\end{proof}

\subsection{The lower boundary for distinct anchors:
\texorpdfstring{\(Q\geq\rho_n\)}{Q >= rho\_n}}
\label{sec:ac-distinct-lower-boundary}

The lower boundary in the next proposition is attained by the equal-level
swapped-anchor family of \Cref{sec:swapped-anchor-configuration}; the
proof shows that all other distinct-anchor pairs lie above that same
boundary in the \((Q,R)\)-plane.  This boundary is the graph
\(R=L_n(Q)\), \(\rho_n\leq Q\leq1\), shown in
panel~\textup{(b)} of \Cref{fig:qr-distortion-region}.

\paragraph{Common notation for distinct-anchor comparisons.}
The lower- and upper-boundary arguments use the following
notation, but not the same dimension range.  For \(n\geq1\), fix
\(i\neq j\), \(x\in M_i\), \(y\in M_j\), and
\(0\leq\alpha,\beta\leq\widehat{\zeta}_n\), and put
\[
 Q:=\cos\alpha\cos\beta
 +\sin\alpha\sin\beta\,(x\cdot y),
 \qquad
 R:=H_{\alpha,i}(x)\cdot H_{\beta,j}(y).
\]

\begin{proposition}[Lower boundary for distinct anchors]
\label{lem:ac-near}
Assume \(n\geq1\).  In the common notation above,
\[
 Q\geq {\rho_n}\quad\Longrightarrow\quad
 R\geq L_n(Q)
 =-{\rho_n}Q-{\sin \zeta_n}\sqrt{1-Q^2}.                              \acEquationTag{eq:ac-33}
\]
\end{proposition}

Equality is realized explicitly as follows.  If
\(
 \beta=\alpha,\qquad x=-v_j,\qquad y=-v_i,
\)
then \Cref{sec:swapped-anchor-configuration} gives, for
\(0\leq\alpha\leq\alpha_n^\ast\),
\[
 Q=\cos D_n(\alpha),\qquad
 R=\cos\bigl(\zeta_n+D_n(\alpha)\bigr)=L_n(Q).
\]
As \(\alpha\) increases from \(0\) to \(\alpha_n^\ast\), \(Q\)
decreases from \(1\) to \(\rho_n\).  Thus this family parametrizes the
graph \(R=L_n(Q)\) over \(\rho_n\leq Q\leq1\), shown in
panel~\textup{(b)} of \Cref{fig:qr-distortion-region}.

All the subsidiary results below are valid for every
\(n\geq1\).  The only dimension-dependent verification occurs in
\Cref{lem:ac-near-mixed-slack}, where the cases \(n\geq2\) and \(n=1\)
are treated separately.

The supporting results and their roles are summarized in
\Cref{tab:ac-near-proof-structure}.
\begin{table}[h!]
\centering
\caption{Structure of the lower-boundary argument for distinct anchors.}
\label{tab:ac-near-proof-structure}
\renewcommand{\arraystretch}{1.15}
\begin{tabularx}{\textwidth}{@{}>{\raggedright\arraybackslash}p{0.30\textwidth}
                                  >{\raggedright\arraybackslash}X
                                  >{\raggedright\arraybackslash}p{0.23\textwidth}@{}}
\toprule
Result & Role or regime & Main technique \\
\midrule
\Cref{lem:ac-near-direct}
&
$u\leq-\rho_n$
&
Direct comparison
\\
\Cref{lem:ac-near-ptolemy}
&
Common estimates for $u\geq-\rho_n$
&
Ptolemy
\\
\Cref{lem:ac-near-low}
&
Low--low regime
&
Factorization
\\
\Cref{lem:ac-low-polynomial-sign}
&
Polynomial sign in the low--low case
&
Polynomial estimate
\\
\Cref{lem:ac-spherical-distance-comparison}
&
Distance comparison used in the high--high case
&
Lipschitz estimate
\\
\Cref{lem:ac-near-high}
&
High--high regime
&
Triangle comparison
\\
\Cref{lem:ac-near-mixed}
&
Mixed regime
&
Ptolemy and concavity
\\
\bottomrule
\end{tabularx}
\end{table}

\paragraph{Notation for this subsection.}
\phantomsection\label{def:lower-boundary-notation}
Whenever \(i\neq j\), \(x\in M_i\), \(y\in M_j\), and
\(0\leq\alpha,\beta\leq\widehat{\zeta}_n\), we write
\[
 u:=x\cdot y,\qquad
 r_\alpha:=r_n(\alpha),\qquad
 r_\beta:=r_n(\beta),\qquad
 \theta_\alpha:=\vartheta_n(\alpha),\qquad
 \theta_\beta:=\vartheta_n(\beta).
\]
We also use repeatedly the elementary identity
\(
 \sin\widehat{\zeta}_n=\sin\zeta_n.
\)
Thus
\[
 Q=\cos\alpha\cos\beta+\sin\alpha\sin\beta\,u.
\]
Whenever the two gains are finite, we also write
\[
\begin{gathered}
 \mathrm{Num}:=(v_i+r_\alpha x)\cdot(v_j+r_\beta y),\\
 X:=\lVert v_i+r_\alpha x\rVert,\qquad
 Y:=\lVert v_j+r_\beta y\rVert,
\end{gathered}
\qquad
\begin{gathered}
 X_\alpha:=\sqrt{1+2\rho_n r_\alpha+r_\alpha^2},\\
 X_\beta:=\sqrt{1+2\rho_n r_\beta+r_\beta^2}.
\end{gathered}
\]
In particular,
\[
 R=\frac{\mathrm{Num}}{XY}.
\]
The quantities \(X_\alpha\) and \(X_\beta\) are the norms of the corresponding
unnormalized chord vectors in the boundary cases
\(v_i\cdot x=\rho_n\) and \(v_j\cdot y=\rho_n\), respectively; see
\Cref{fig:chord-triangle-geometry}.
The subscripts on \(r_\alpha,r_\beta,\theta_\alpha,\theta_\beta\)
will always refer to the two radial levels fixed in the comparison.

\begin{lemma}[The range \texorpdfstring{\(u\leq-\rho_n\)}{u <= -rho\_n}]
\label{lem:ac-near-direct}
With the notation above, if
\(
 u\leq-\rho_n,\qquad Q\geq\rho_n,
\)
then \(R\geq L_n(Q)\).
\end{lemma}

\begin{proof}
Put
\[
 Q_0:=\cos D_n^\times(\alpha,\beta)
 =\cos\alpha\cos\beta-\rho_n\sin\alpha\sin\beta.
\]
Thus \(\arccos Q_0=D_n^\times(\alpha,\beta)\).
By \Cref{prop:ac-maximal-displacement}, the two targets lie in the
spherical caps of radii \(\theta_\alpha\) and \(\theta_\beta\) about
\(v_i\) and \(v_j\), respectively.  These cap inclusions use only
\(x\in M_i\) and \(y\in M_j\); the assumption
\(u\leq-\rho_n\) enters in the comparison with \(Q_0\).  Indeed, since
\(\rho_n\leq Q\leq Q_0\), \Cref{lem:ac-crossed-schedule} and the
triangle inequality give
\[
 d_n(H_{\alpha,i}(x),H_{\beta,j}(y))
 \leq\zeta_n+\theta_\alpha+\theta_\beta
 \leq\zeta_n+\arccos Q_0
 \leq\zeta_n+\arccos Q.
\]
Because \(Q\geq\rho_n\), the rightmost quantity is at most
\(\zeta_n+\widehat{\zeta}_n=\pi\). More explicitly, since cosine is decreasing on \([0,\pi]\),
\[
\begin{aligned}
 R
 &=\cos d_n\bigl(H_{\alpha,i}(x),H_{\beta,j}(y)\bigr)\\
 &\geq\cos\bigl(\zeta_n+\arccos Q\bigr)\\
 &=-\rho_nQ-\sin\zeta_n\sqrt{1-Q^2}
 =L_n(Q).
\end{aligned}
\]
\end{proof}

\begin{lemma}[Ptolemy estimates for two distinct cells]
\label{lem:ac-near-ptolemy}
With the notation above, assume that
\(\alpha,\beta<\widehat{\zeta}_n\).  Define
\[
 \ell_j(x):=\frac{\|x+v_j\|^2}{2},\qquad
 \ell_i(y):=\frac{\|y+v_i\|^2}{2},
\]
and
\[
 \Phi_{\rho_n}(u):=(1+{\rho_n})
 -\sqrt{(1+{\rho_n})(1-u)}.
\]
Define the corresponding Ptolemy lower bound for the numerator by
\[
 \mathrm{Num}_{\mathrm{Ptol}}
 :=-{\rho_n}-r_\alpha-r_\beta+r_\alpha r_\beta u
 +2\sqrt{r_\alpha r_\beta}\,\Phi_{\rho_n}(u).
\]
If \(u\geq-\rho_n\), then
\[
 \sqrt{\ell_j(x)\ell_i(y)}
 \geq\Phi_{\rho_n}(u)\geq0,                              \acEquationTag{eq:ac-35}
\]
\[
 r_\alpha\ell_j(x)+r_\beta\ell_i(y)
 \geq2\sqrt{r_\alpha r_\beta}\,\Phi_{\rho_n}(u),          \acEquationTag{eq:ac-36}
\]
and
\[
\begin{aligned}
 \mathrm{Num}
 &=-{\rho_n}-r_\alpha-r_\beta+r_\alpha r_\beta u
   +r_\alpha\ell_j(x)+r_\beta\ell_i(y)
   \geq\mathrm{Num}_{\mathrm{Ptol}},\\
 X^2&=X_\alpha^2+2r_\alpha(v_i\cdot x-{\rho_n})\geq X_\alpha^2,\qquad
 Y^2=X_\beta^2+2r_\beta(v_j\cdot y-{\rho_n})\geq X_\beta^2.
\end{aligned}                                             \acEquationTag{eq:ac-37}
\]
\end{lemma}

\begin{proof}

Apply Ptolemy's inequality in the ambient Euclidean space to the four
points
\(x,y,-v_i,-v_j\).

The cell bound \eqref{eq:ac-10} gives
\(
 v_i\cdot x\geq\rho_n,
 \qquad
 v_j\cdot y\geq\rho_n.
\)
Consequently, \eqref{eq:chord-norm-lower-bound} in
\Cref{lem:ac-chord-geometry}, applied to the pairs
\((v_i,x)\) and \((v_j,y)\) with \(a=1\), gives
\(
 \lVert x+v_i\rVert,\ \lVert y+v_j\rVert
 \geq\sqrt{2(1+\rho_n)}.
\)
This comparison value is the length of the sum of two unit vectors
whose inner product is \(\rho_n\).  The remaining two lengths in
Ptolemy's inequality are
\[
 \lVert x-y\rVert=\sqrt{2(1-u)},\qquad
 \lVert v_i-v_j\rVert=\sqrt{2(1+\rho_n)}.
\]

In the present labeling, Ptolemy's inequality reads
\[
 \lVert x-y\rVert\,\lVert v_i-v_j\rVert
 +\lVert x+v_j\rVert\,\lVert y+v_i\rVert
 \geq\lVert x+v_i\rVert\,\lVert y+v_j\rVert.
\]
Substituting the bounds and identities above, and then isolating the
second product, gives
\[
\begin{aligned}
 \lVert x+v_j\rVert\,\lVert y+v_i\rVert
 &\geq 2(1+\rho_n)
   -\sqrt{2(1-u)}\sqrt{2(1+\rho_n)}\\
 &=2(1+\rho_n)-2\sqrt{(1+\rho_n)(1-u)},
\end{aligned}
\]

After division by \(2\), this gives the first inequality in
\eqref{eq:ac-35}.  Moreover, \(u\geq-\rho_n\) gives
\(1-u\leq1+\rho_n\), and hence
\(
 \sqrt{(1+\rho_n)(1-u)}\leq1+\rho_n.
\)
Therefore \(\Phi_{\rho_n}(u)\geq0\), completing
\eqref{eq:ac-35}.

The quantities \(r_\alpha,r_\beta,\ell_j(x),\ell_i(y)\) are nonnegative.
Therefore the arithmetic--geometric mean inequality, applied to
\(r_\alpha\ell_j(x)\) and \(r_\beta\ell_i(y)\),
together with \eqref{eq:ac-35}, gives
\[
 r_\alpha\ell_j(x)+r_\beta\ell_i(y)
 \geq2\sqrt{\bigl(r_\alpha\ell_j(x)\bigr)
                  \bigl(r_\beta\ell_i(y)\bigr)}
 =2\sqrt{r_\alpha r_\beta}\,\sqrt{\ell_j(x)\ell_i(y)}
 \geq2\sqrt{r_\alpha r_\beta}\,\Phi_{\rho_n}(u).
\]
This proves \eqref{eq:ac-36}.  Combining it with the direct expansion
of \(\mathrm{Num}\) gives
\(\mathrm{Num}\geq\mathrm{Num}_{\mathrm{Ptol}}\); direct expansion also
gives the two norm identities in \eqref{eq:ac-37}.
\end{proof}

\begin{lemma}[Low radial levels: exact factorization]
\label{lem:ac-near-low}

With the notation above, assume
\[
 \alpha,\beta\leq\alpha_n^\ast<\widehat{\zeta}_n,\qquad
 u\geq-\rho_n,\qquad Q\geq\rho_n.
\]
Then \(R\geq L_n(Q)\).
\end{lemma}

We first state the polynomial-sign estimate used in the proof of
\Cref{lem:ac-near-low}.
The proof of the polynomial-sign estimate is given in
Appendix~\ref{app:ac-low-polynomial-sign}.

\begin{lemma}[Sign of the polynomial \texorpdfstring{\(P_{\rho_n}\)}{P}]
\label{lem:ac-low-polynomial-sign}
Let \(n\geq1\), let \(r_\alpha,r_\beta>0\), and let
\(0\leq z\leq1\).  Put
\[
 r_{\alpha\beta}^{\times}:=\sqrt{r_\alpha r_\beta},
 \qquad
 r_{\alpha\beta}^{+}:=r_\alpha+r_\beta,
 \qquad
 \chi_n:=1+2\rho_n-2\rho_n^2.
\]
Assume
\[
 (r_{\alpha\beta}^{\times})^2(1-\chi_nz^2)
 +2r_{\alpha\beta}^{\times}(1-z)
 -(1-\rho_n)r_{\alpha\beta}^{+}<0.                    \acEquationTag{eq:ac-low-poly-hyp}
\]
Define
\[
\begin{aligned}
 P_{\rho_n}
 (r_{\alpha\beta}^{\times},r_{\alpha\beta}^{+},z)
 :={}&-2(1-\rho_n)
 \bigl(r_{\alpha\beta}^{\times}(1+z)+2\bigr)r_{\alpha\beta}^{+}\\
 &+(r_{\alpha\beta}^{\times})^3(1+z)
 \bigl[1+(4\rho_n-5)z^2\bigr]\\
 &+4(r_{\alpha\beta}^{\times})^2(1-\chi_nz^2)\\
 &+4r_{\alpha\beta}^{\times}
 \bigl[2-\rho_n(2-\rho_n)(1+z)\bigr].
\end{aligned}
\]
Then
\[
 P_{\rho_n}
 (r_{\alpha\beta}^{\times},r_{\alpha\beta}^{+},z)\leq0.
\]
Consequently,
\[
 r_{\alpha\beta}^{\times}(1+\rho_n)^2(z-1)
 P_{\rho_n}
 (r_{\alpha\beta}^{\times},r_{\alpha\beta}^{+},z)\geq0.
\]
\end{lemma}

\begin{proof}[Proof of \Cref{lem:ac-near-low}]
We first encode \(u\) by a scalar \(z\) and introduce an auxiliary
angle \(\eta\), which we show is at most the source distance
\(\arccos Q\).  We then prove by exact factorization the comparison
\eqref{eq:ac-42} for the Ptolemy lower bound
\(\mathrm{Num}_{\mathrm{Ptol}}\) defined in
\Cref{lem:ac-near-ptolemy}.  The estimates in that lemma and the
comparison \(\eta\leq\arccos Q\) will then imply \(R\geq L_n(Q)\).

Because \(\alpha,\beta\leq\alpha_n^\ast\), the definition of
\(\vartheta_n\) gives
\(
 \theta_\alpha=\Theta_n(\alpha),
 \qquad
 \theta_\beta=\Theta_n(\beta).
\)
Put
\[
 z:=\sqrt{\frac{1-u}{1+{\rho_n}}}.
\]
Since \(-\rho_n\leq u\leq1\), one has \(0\leq z\leq1\), and
\(
 z=1\quad\Longleftrightarrow\quad u=-\rho_n.
\)
Thus \(z=1\) corresponds exactly to the boundary \(u=-\rho_n\) of the
range \(u\geq-\rho_n\) in which the Ptolemy estimates are applied.

The definitions of \(z\) and \(\Phi_{\rho_n}\) give
\(
 \Phi_{\rho_n}(u)=(1+\rho_n)(1-z).
\)
Define \(\eta\in[0,\widehat{\zeta}_n]\) by
\[
 \sin^2\frac\eta2=
 \sin^2\frac{\theta_\alpha-\theta_\beta}{2}
 +z^2\sin\theta_\alpha\sin\theta_\beta.                  \acEquationTag{eq:ac-40}
\]
This defines a unique \(\eta\in[0,\widehat{\zeta}_n]\).  Indeed,
\(u\geq-\rho_n\) gives \(0\leq z\leq1\), and therefore the right-hand
side is at most
\[
 \sin^2\frac{\theta_\alpha-\theta_\beta}{2}
 +\sin\theta_\alpha\sin\theta_\beta
 =\sin^2\frac{\theta_\alpha+\theta_\beta}{2}
 \leq\sin^2\frac{\widehat{\zeta}_n}{2}.
\]

If \(d_Q:=\arccos Q\), then
\[
\begin{aligned}
 \sin^2\frac{d_Q}{2}
 &=\frac{1-Q}{2}\\
 &=\frac{1-\cos\alpha\cos\beta-u\sin\alpha\sin\beta}{2}\\
 &=\frac{1-\cos(\alpha-\beta)}{2}
   +\frac{1-u}{2}\sin\alpha\sin\beta\\
 &=\sin^2\frac{\alpha-\beta}{2}
   +\frac{1-u}{2}\sin\alpha\sin\beta.
\end{aligned}
\]
Next, the identities
\[
 z^2=\frac{1-u}{1+\rho_n},\qquad
 \sin\theta_\alpha=c_n\sin\alpha,\qquad
 \sin\theta_\beta=c_n\sin\beta,\qquad
 \frac{c_n^2}{1+\rho_n}=\frac12
\]
give
\[
 z^2\sin\theta_\alpha\sin\theta_\beta
 =\frac{1-u}{1+\rho_n}c_n^2\sin\alpha\sin\beta
 =\frac{1-u}{2}\sin\alpha\sin\beta.
\]
Thus the defining relation \eqref{eq:ac-40} for \(\eta\) becomes
\[
 \sin^2\frac\eta2
 =\sin^2\frac{\theta_\alpha-\theta_\beta}{2}
  +\frac{1-u}{2}\sin\alpha\sin\beta.
\]

By \eqref{eq:theta-derivatives}, \(\Theta_n\) is \(1\)-Lipschitz, and hence
\[
 \lvert\theta_\alpha-\theta_\beta\rvert
 =\lvert\Theta_n(\alpha)-\Theta_n(\beta)\rvert
 \leq\lvert\alpha-\beta\rvert.
\]
Comparing the preceding two half-angle identities therefore gives
\(
 \sin^2\tfrac{\eta}{2}\leq\sin^2\tfrac{d_Q}{2}.
\)
Because \(s\mapsto\sin^2(s/2)\) is increasing on \([0,\pi]\), it follows
that
\(
 0\leq\eta\leq d_Q\leq \widehat{\zeta}_n.
\)
Here the final inequality follows from \(Q\geq\rho_n\).

To prove the lower-boundary estimate \eqref{eq:ac-33},
it suffices to establish the auxiliary comparison
\[
 \frac{\mathrm{Num}_{\mathrm{Ptol}}}{X_\alpha X_\beta}
 \geq\cos(\zeta_n+\eta),                                  \acEquationTag{eq:ac-42}
\]
Indeed, \eqref{eq:ac-35}--\eqref{eq:ac-37} give
\(\mathrm{Num}\geq\mathrm{Num}_{\mathrm{Ptol}}\) and
\(XY\geq X_\alpha X_\beta\).

If \(\mathrm{Num}\geq0\), then
\[
 R=\frac{\mathrm{Num}}{XY}\geq0\geq L_n(Q),
\]
so \eqref{eq:ac-33} holds.  If \(\mathrm{Num}<0\), then
\[
 R=\frac{\mathrm{Num}}{XY}
 \geq\frac{\mathrm{Num}}{X_\alpha X_\beta}
 \geq\frac{\mathrm{Num}_{\mathrm{Ptol}}}{X_\alpha X_\beta}.
\]
Moreover, \(\eta\leq d_Q\leq\widehat{\zeta}_n\) and monotonicity of
cosine on \([\zeta_n,\pi]\) give
\(\cos(\zeta_n+\eta)\geq\cos(\zeta_n+d_Q)\).  Thus
\eqref{eq:ac-42} implies the desired estimate \eqref{eq:ac-33}.
It therefore remains only to prove \eqref{eq:ac-42}.

\smallskip
\noindent\emph{Zero gains.}
The zero-gain cases can be handled directly.  If \(r_\alpha=0\), then
\(\theta_\alpha=0\), \(X_\alpha=1\), and \(\eta=\theta_\beta\) by
\eqref{eq:ac-40}.

By the sine-rule and norm identities in
\Cref{lem:ac-chord-geometry}, the formulas for \(r_\beta\) and
\(X_\beta\) give
\[
 \frac{\mathrm{Num}_{\mathrm{Ptol}}}{X_\alpha X_\beta}
 =\frac{-{\rho_n}-r_\beta}{X_\beta}
 =-\cos(\widehat{\zeta}_n-\theta_\beta)
 =\cos(\zeta_n+\eta).
\]
Thus \eqref{eq:ac-42} holds with equality.  The case \(r_\beta=0\) is
symmetric.

\smallskip
\noindent\emph{Positive gains: reduction to a polynomial inequality.}

Assume \(r_\alpha r_\beta>0\).  Using
\(
 \cos(\zeta_n+\eta)
 =-\rho_n\cos\eta-\sin\widehat{\zeta}_n\sin\eta,
\)
and multiplying \eqref{eq:ac-42} by \(X_\alpha X_\beta>0\), we see that the
required inequality is equivalent to

\[
 \mathcal G
 :=\mathrm{Num}_{\mathrm{Ptol}}
   +\rho_nX_\alpha X_\beta\cos\eta
 \geq-\sin\widehat{\zeta}_nX_\alpha X_\beta\sin\eta.       \acEquationTag{eq:ac-low-signed}
\]
The right-hand side is nonpositive.  Hence
\eqref{eq:ac-low-signed} holds immediately if \(\mathcal G\geq0\).
If \(\mathcal G<0\), both sides are nonpositive, and
\eqref{eq:ac-low-signed} is equivalent to
\[
 \sin^2\widehat{\zeta}_n(X_\alpha X_\beta)^2\sin^2\eta
 -\mathcal G^2\geq0.                                      \acEquationTag{eq:ac-low-squared}
\]

Assume henceforth that \(\mathcal G<0\).  It remains to prove
\eqref{eq:ac-low-squared}.

To evaluate the left side of \eqref{eq:ac-low-squared}, introduce the two symmetric
combinations of the gains
\[
 r_{\alpha\beta}^{\times}:=\sqrt{r_\alpha r_\beta},\qquad
 r_{\alpha\beta}^{+}:=r_\alpha+r_\beta,
\]

and define the dimension-dependent constant
\(
 \chi_n:=1+2\rho_n-2\rho_n^2.
\)

By \eqref{eq:chord-angular-form} and the norm identity in
\Cref{lem:ac-chord-geometry},
\[
\begin{aligned}
 X_\alpha\cos\theta_\alpha&=1+\rho_n r_\alpha,
 &\qquad
 X_\alpha\sin\theta_\alpha&=r_\alpha\sin\widehat{\zeta}_n,\\
 X_\beta\cos\theta_\beta&=1+\rho_n r_\beta,
 &
 X_\beta\sin\theta_\beta&=r_\beta\sin\widehat{\zeta}_n.
\end{aligned}
\]
On the other hand, \eqref{eq:ac-40} gives
\[
\begin{aligned}
 \cos\eta
 &=1-2\sin^2\frac{\eta}{2}\\
 &=\cos(\theta_\alpha-\theta_\beta)
   -2z^2\sin\theta_\alpha\sin\theta_\beta\\
 &=\cos\theta_\alpha\cos\theta_\beta
   +(1-2z^2)\sin\theta_\alpha\sin\theta_\beta.
\end{aligned}
\]
Multiplication by \(X_\alpha X_\beta\), followed by
\(\sin^2\widehat{\zeta}_n=1-\rho_n^2\), now gives
\[
\begin{aligned}
 X_\alpha X_\beta\cos\eta
 &=(1+\rho_n r_\alpha)(1+\rho_n r_\beta)
 +(1-2z^2)r_\alpha r_\beta
   \sin^2\widehat{\zeta}_n\\
 &=1+\rho_n r_{\alpha\beta}^{+}
 +(r_{\alpha\beta}^{\times})^2
 \bigl(1-2\sin^2\widehat{\zeta}_n z^2\bigr).
\end{aligned}
\]

The identities
\(
 u=1-(1+\rho_n)z^2,
 \qquad
 \Phi_{\rho_n}(u)=(1+\rho_n)(1-z)
\)
give
\[
 \mathrm{Num}_{\mathrm{Ptol}}
={}-\rho_n-r_{\alpha\beta}^{+}
 +(r_{\alpha\beta}^{\times})^2
   \bigl(1-(1+\rho_n)z^2\bigr)
 +2r_{\alpha\beta}^{\times}(1+\rho_n)(1-z).
\]
Adding \(\rho_nX_\alpha X_\beta\cos\eta\) and using the preceding
formula for \(X_\alpha X_\beta\cos\eta\), we obtain
\[
\begin{aligned}
 \mathcal G
 &=(1+\rho_n)\left[
 (r_{\alpha\beta}^{\times})^2(1-\chi_nz^2)
 +2r_{\alpha\beta}^{\times}(1-z)
 -(1-\rho_n)r_{\alpha\beta}^{+}\right].
\end{aligned}
 \acEquationTag{eq:ac-44}
\]

\medskip
\noindent\emph{Verification of the exact factorization.}

For the factorization, first note that direct multiplication gives
\[
 X_\alpha^2X_\beta^2
={}\bigl(1+\rho_n r_{\alpha\beta}^{+}
 +(r_{\alpha\beta}^{\times})^2\bigr)^2
 +\sin^2\widehat{\zeta}_n
 \bigl((r_{\alpha\beta}^{+})^2
 -4(r_{\alpha\beta}^{\times})^2\bigr).
\]
Moreover,
\[
 (X_\alpha X_\beta)^2\sin^2\eta
 =X_\alpha^2X_\beta^2
 -(X_\alpha X_\beta\cos\eta)^2.
\]
Substitute the preceding formulas for
\(X_\alpha^2X_\beta^2\),
\(X_\alpha X_\beta\cos\eta\), and \(\mathcal G\) into the left-hand
side of \eqref{eq:ac-low-squared}.

After these substitutions, the left-hand side of
\eqref{eq:ac-low-squared} is a polynomial in \(z\).  At \(z=1\), the
identities above give
\[
 \mathcal G
 =-\sin^2\widehat{\zeta}_n
 \bigl(r_{\alpha\beta}^{+}
 +2\rho_n(r_{\alpha\beta}^{\times})^2\bigr),
\qquad
 X_\alpha^2X_\beta^2
 -(X_\alpha X_\beta\cos\eta)^2
 =\sin^2\widehat{\zeta}_n
 \bigl(r_{\alpha\beta}^{+}
 +2\rho_n(r_{\alpha\beta}^{\times})^2\bigr)^2.
\]
Thus this polynomial vanishes at \(z=1\).

Since the left-hand side is a polynomial in \(z\), its vanishing at
\(z=1\) implies divisibility by \(z-1\).  The equivalence
\(
 z=1\quad\Longleftrightarrow\quad u=-\rho_n
\)
identifies this root with the boundary of the range
\(u\geq-\rho_n\) in \Cref{lem:ac-near-ptolemy}.  The resulting quotient
is identified coefficient by coefficient in
\Cref{app:ac-low-factorization}.

Restoring the original notation gives the exact factorization
\[
 \sin^2\widehat{\zeta}_n(X_\alpha X_\beta)^2\sin^2\eta-\mathcal G^2
 =r_{\alpha\beta}^{\times}(1+\rho_n)^2(z-1)
   P_{\rho_n}
   (r_{\alpha\beta}^{\times},r_{\alpha\beta}^{+},z),          \acEquationTag{eq:ac-45}
\]
where \(P_{\rho_n}\) is the polynomial defined in
\Cref{lem:ac-low-polynomial-sign}.

By \eqref{eq:ac-44} and the standing assumption \(\mathcal G<0\),
\[
 (r_{\alpha\beta}^{\times})^2(1-\chi_nz^2)
 +2r_{\alpha\beta}^{\times}(1-z)
 -(1-\rho_n)r_{\alpha\beta}^{+}<0.
\]
Thus \Cref{lem:ac-low-polynomial-sign} gives
\(
 P_{\rho_n}
 (r_{\alpha\beta}^{\times},r_{\alpha\beta}^{+},z)\leq0.
\)
Since \(r_{\alpha\beta}^{\times}>0\) and \(z-1\leq0\),
the right-hand side of \eqref{eq:ac-45} is nonnegative.  This proves
\eqref{eq:ac-low-squared}, hence \eqref{eq:ac-low-signed}, and
finally \eqref{eq:ac-42}.
\end{proof}

\begin{lemma}[Spherical-distance comparison]
\label{lem:ac-spherical-distance-comparison}
Let \(0\leq\alpha,\beta\leq\tfrac{\pi}{2}\), and define
\[
 \mathsf d_{\alpha,\beta}\colon[0,\pi]\longrightarrow[0,\pi],
\qquad
 \mathsf d_{\alpha,\beta}(\delta)
 :=
 \arccos\bigl(\cos\alpha\cos\beta
 +\sin\alpha\sin\beta\cos\delta\bigr).
\]
Then \(\mathsf d_{\alpha,\beta}\) is nondecreasing and
\(1\)-Lipschitz.  Equivalently, whenever
\(0\leq\delta_1\leq\delta_2\leq\pi\),
\[
 0\leq
 \mathsf d_{\alpha,\beta}(\delta_2)
 -\mathsf d_{\alpha,\beta}(\delta_1)
 \leq\delta_2-\delta_1.
\]
\end{lemma}

\begin{proof}

If \(\sin\alpha\sin\beta=0\), then
\(\mathsf d_{\alpha,\beta}\) is constant, so the conclusion is
immediate.  Assume henceforth that \(\sin\alpha\sin\beta>0\).
Put
\(
 a:=\cos\alpha\cos\beta,\qquad
 b:=\sin\alpha\sin\beta.
\)

For \(0<\delta<\pi\), the standing assumption \(b>0\) gives
\(
 \cos(\alpha+\beta)
 =a-b
 <a+b\cos\delta
 <a+b
 =\cos(\alpha-\beta).
\)
Since \(0\leq\alpha,\beta\leq\tfrac{\pi}{2}\), it follows that
\(
 -1<a+b\cos\delta<1.
\)
Consequently
\(\mathsf d_{\alpha,\beta}(\delta)\in(0,\pi)\), and hence
\(\sin\mathsf d_{\alpha,\beta}(\delta)>0\).
Direct differentiation now gives
\[
 \mathsf d_{\alpha,\beta}'(\delta)
 =\frac{\sin\alpha\sin\beta\sin\delta}
 {\sin\mathsf d_{\alpha,\beta}(\delta)}\geq0.
\]

Hence squaring preserves the inequality
\(\mathsf d_{\alpha,\beta}'(\delta)\leq1\), and rearrangement gives
\[
 \bigl(\mathsf d_{\alpha,\beta}'(\delta)\bigr)^2\leq1
 \iff
 b^2\sin^2\delta
 \leq1-(a+b\cos\delta)^2
 \iff
 a^2+b^2+2ab\cos\delta\leq1.
\]
The final inequality follows directly from
\[
 a^2+b^2+2ab\cos\delta
 =(a+b)^2-2ab(1-\cos\delta)
 \leq(a+b)^2
 =\cos^2(\alpha-\beta)\leq1.
\]

Consequently
\(
 0\leq\mathsf d_{\alpha,\beta}'(\delta)\leq1
 \qquad(0<\delta<\pi).
\)

Integrating this inequality over
\([\delta_1,\delta_2]\subset(0,\pi)\) proves the assertion whenever
\(0<\delta_1\leq\delta_2<\pi\).

Since \(\mathsf d_{\alpha,\beta}\) is continuous on \([0,\pi]\), the
endpoint cases follow by taking limits.
\end{proof}

\begin{lemma}[High radial levels: triangle comparison]
\label{lem:ac-near-high}
With the notation above, assume
\[
\alpha_n^\ast\leq\alpha,\beta<\widehat{\zeta}_n,\qquad
 u\geq-\rho_n,\qquad Q\geq\rho_n.
\]
Then \(R\geq L_n(Q)\).
\end{lemma}

\begin{proof}
Put
\(
 \delta:=\arccos u\leq\zeta_n.
\)

Applying the complementary-displacement formula in
\Cref{prop:ac-maximal-displacement} first in the \(i\)-cell and then in
the \(j\)-cell gives
\[
 d_n(H_{\alpha,i}(x),x)\leq\widehat{\zeta}_n-\theta_\alpha,
 \qquad
 d_n(H_{\beta,j}(y),y)\leq\widehat{\zeta}_n-\theta_\beta.  \acEquationTag{eq:ac-51}
\]

The equatorial points in the swapped-anchor configuration of
\Cref{sec:swapped-anchor-configuration}, depicted in
\Cref{fig:swapped-anchor-great-circle}, are separated by \(\zeta_n\).
For the function \(\mathsf d_{\alpha,\beta}\) defined in
\Cref{lem:ac-spherical-distance-comparison}, this gives
\(
 \mathsf d_{\alpha,\beta}(\zeta_n)
 =D_n^\times(\alpha,\beta),
\)

whereas
\[
 \mathsf d_{\alpha,\beta}(\delta)
 =d_{n+1}\bigl(p_\alpha(x),p_\beta(y)\bigr)
 =\arccos Q.
\]

Since \(\delta\leq\zeta_n\),
\Cref{lem:ac-spherical-distance-comparison} gives
\[
 \mathsf d_{\alpha,\beta}(\zeta_n)
 -\mathsf d_{\alpha,\beta}(\delta)
 \leq\zeta_n-\delta.                                      \acEquationTag{eq:ac-high-lipschitz}
\]

We next prove
\[
 \mathsf d_{\alpha,\beta}(\zeta_n)\geq\widehat{\zeta}_n.
 \acEquationTag{eq:ac-high-endpoint-distance}
\]

Because \(F_n\) is decreasing and fixes \(\alpha_n^\ast\),
\(
 F_n(\alpha)\leq\alpha_n^\ast\leq\beta.
\)
The superlevel characterization
\eqref{eq:critical-involution-superlevel} in
\Cref{prop:ac-critical-involution} therefore gives
\(
 \mathsf d_{\alpha,\beta}(\zeta_n)
 =D_n^\times(\alpha,\beta)
 \geq\widehat{\zeta}_n.
\)
Because
\(\theta_\alpha,\theta_\beta\geq\widehat{\zeta}_n/2\),
\(
 (\widehat{\zeta}_n-\theta_\alpha)
 +(\widehat{\zeta}_n-\theta_\beta)
 \leq \widehat{\zeta}_n.
\)
Combining this estimate with
\eqref{eq:ac-high-lipschitz}--\eqref{eq:ac-high-endpoint-distance}
gives
\[
\begin{aligned}
 \delta+(\widehat{\zeta}_n-\theta_\alpha)
 +(\widehat{\zeta}_n-\theta_\beta)
 &\leq\delta+\widehat{\zeta}_n
 &\leq\delta+\mathsf d_{\alpha,\beta}(\zeta_n)
 &\leq\zeta_n+\mathsf d_{\alpha,\beta}(\delta).
\end{aligned}
\]

Here \(d_n(x,y)=\delta\).  Therefore the triangle inequality in
\(\Sp^n\), followed by \eqref{eq:ac-51} and the preceding estimate,
gives
\[
\begin{aligned}
 \arccos R
 &\leq d_n\bigl(H_{\alpha,i}(x),x\bigr)
       +d_n(x,y)
       +d_n\bigl(y,H_{\beta,j}(y)\bigr)\\
 &\leq(\widehat{\zeta}_n-\theta_\alpha)
       +\delta
       +(\widehat{\zeta}_n-\theta_\beta)\\
 &\leq\zeta_n+\mathsf d_{\alpha,\beta}(\delta)
 =\zeta_n+\arccos Q.
\end{aligned}
\]
Since \(Q\geq\rho_n\), one has
\(\arccos Q\leq\widehat{\zeta}_n\), and hence
\(\zeta_n+\arccos Q\leq\pi\).  The monotonicity of cosine on
\([0,\pi]\) now yields
\[
 R
 \geq\cos(\zeta_n+\arccos Q)
 =-\rho_nQ-\sin\zeta_n\sqrt{1-Q^2}
 =L_n(Q).
\]
This proves the lemma.
\end{proof}

\begin{lemma}[Mixed radial levels: convexity and an endpoint estimate]
\label{lem:ac-near-mixed}
With the notation above, assume
\[
 0\leq\alpha\leq\alpha_n^\ast\leq\beta<\widehat{\zeta}_n,\qquad
 \alpha<\beta,\qquad
 u\geq-\rho_n,\qquad Q\geq\rho_n.
\]
Then
\(R\geq L_n(Q)\).
\end{lemma}

We prove the lemma after three subsidiary results.  The first reduces
the desired inequality, by the Ptolemy estimates, to the
nonnegativity of a scalar function.  The second controls the possible
interior minimum of the corresponding one-variable function.  The
third proves the endpoint estimate at \(z=0\) that makes the resulting
lower bound for this minimum nonnegative.

\begin{lemma}[Ptolemy reduction in the mixed radial range]
\label{lem:ac-near-mixed-reduction}
Under the hypotheses of \Cref{lem:ac-near-mixed}, for every
\(s\in[-\rho_n,1]\) satisfying
\(
 \cos\alpha\cos\beta+\sin\alpha\sin\beta\,s\geq\rho_n,
\)
define
\[
\begin{aligned}
 \widetilde{\mathscr E}_{\alpha,\beta}(s):={}&
 \sin^2\widehat{\zeta}_n
 \bigl(-L_n(\cos\alpha\cos\beta+\sin\alpha\sin\beta\,s)
 -\cos(\widehat{\zeta}_n-\theta_\alpha-\theta_\beta)\bigr)\\
 &+\sin\theta_\alpha\sin\theta_\beta\,(s+\rho_n)\\
 &+2\sqrt{\sin\theta_\alpha\sin\theta_\beta
 \sin(\widehat{\zeta}_n-\theta_\alpha)
 \sin(\widehat{\zeta}_n-\theta_\beta)}\,\Phi_{\rho_n}(s).
\end{aligned} \acEquationTag{eq:ac-39}
\]
Then
\[
 \mathrm{Num}-XYL_n(Q)
 \geq
 \frac{\widetilde{\mathscr E}_{\alpha,\beta}(u)}
 {\sin(\widehat{\zeta}_n-\theta_\alpha)
  \sin(\widehat{\zeta}_n-\theta_\beta)}.
\]
In particular,
\(\widetilde{\mathscr E}_{\alpha,\beta}(u)\geq0\) implies
\(R\geq L_n(Q)\).
\end{lemma}

\begin{proof}

Recall that
\(
 Q=\cos\alpha\cos\beta+\sin\alpha\sin\beta\,u.
\)
On \(Q\in[\rho_n,1]\), the definition of \(L_n\) in
\eqref{eq:admissible-boundary-functions} gives
\[
 -L_n(Q)
 =\rho_nQ+\sin\widehat{\zeta}_n\sqrt{1-Q^2},
\]
where \(\sin\widehat{\zeta}_n=\sin\zeta_n\).
The Ptolemy estimates \eqref{eq:ac-36}--\eqref{eq:ac-37} give
\(\mathrm{Num}\geq\mathrm{Num}_{\mathrm{Ptol}}\) and
\(XY\geq X_\alpha X_\beta\).  Since
\(L_n(Q)\leq0\),
\[
 \begin{aligned}
 \mathrm{Num}-XYL_n(Q)
 &={}\bigl(\mathrm{Num}_{\mathrm{Ptol}}
        -X_\alpha X_\beta L_n(Q)\bigr)
 +(\mathrm{Num}-\mathrm{Num}_{\mathrm{Ptol}})\\
 &\quad +(XY-X_\alpha X_\beta)\bigl(-L_n(Q)\bigr)\\
 &\geq \mathrm{Num}_{\mathrm{Ptol}}
        -X_\alpha X_\beta L_n(Q).
 \end{aligned}
\]

Both added terms are nonnegative because
\(\mathrm{Num}\geq\mathrm{Num}_{\mathrm{Ptol}}\),
\(XY\geq X_\alpha X_\beta\), and \(-L_n(Q)\geq0\).

By the definition of \(\mathrm{Num}_{\mathrm{Ptol}}\),
\[
\begin{aligned}
 \mathrm{Num}_{\mathrm{Ptol}}-X_\alpha X_\beta L_n(Q)
 &={}-\rho_n-r_\alpha-r_\beta+r_\alpha r_\beta u\\
 &\quad +2\sqrt{r_\alpha r_\beta}\,\Phi_{\rho_n}(u)
 -X_\alpha X_\beta L_n(Q).
\end{aligned} \acEquationTag{eq:ac-38}
\]

The quantities \(r_\alpha,\theta_\alpha\) correspond to the radial
level \(\alpha\), and \(r_\beta,\theta_\beta\) correspond to the radial
level \(\beta\).
Because the gains are finite,
\(\theta_\alpha,\theta_\beta\in[0,\widehat{\zeta}_n)\), and
\eqref{eq:ac-6} gives
\[
 r_\alpha=\frac{\sin\theta_\alpha}
 {\sin(\widehat{\zeta}_n-\theta_\alpha)},\qquad
 r_\beta=\frac{\sin\theta_\beta}
 {\sin(\widehat{\zeta}_n-\theta_\beta)}.
\]

The norm identity \eqref{eq:chord-geometry} in
\Cref{lem:ac-chord-geometry} gives
\[
 X_\alpha=\frac{\sin\widehat{\zeta}_n}
 {\sin(\widehat{\zeta}_n-\theta_\alpha)},\qquad
 X_\beta=\frac{\sin\widehat{\zeta}_n}
 {\sin(\widehat{\zeta}_n-\theta_\beta)}.
\]

Multiplying the left-hand side of \eqref{eq:ac-38} by the positive
factor
\(\sin(\widehat{\zeta}_n-\theta_\alpha)
\sin(\widehat{\zeta}_n-\theta_\beta)\), and substituting the formulas
above for \(r_\alpha,r_\beta,X_\alpha,X_\beta\), gives
\[
\begin{aligned}
&\sin(\widehat{\zeta}_n-\theta_\alpha)
 \sin(\widehat{\zeta}_n-\theta_\beta)
 \bigl(\mathrm{Num}_{\mathrm{Ptol}}
       -X_\alpha X_\beta L_n(Q)\bigr)\\
={}&-\rho_n
 \sin(\widehat{\zeta}_n-\theta_\alpha)
 \sin(\widehat{\zeta}_n-\theta_\beta)\\
&-\sin\theta_\alpha
 \sin(\widehat{\zeta}_n-\theta_\beta)
 -\sin\theta_\beta
 \sin(\widehat{\zeta}_n-\theta_\alpha)\\
&+\sin\theta_\alpha\sin\theta_\beta\,u\\
&+2\sqrt{\sin\theta_\alpha\sin\theta_\beta
 \sin(\widehat{\zeta}_n-\theta_\alpha)
 \sin(\widehat{\zeta}_n-\theta_\beta)}\,\Phi_{\rho_n}(u)\\
&-\sin^2\widehat{\zeta}_n L_n(Q).
\end{aligned}
\]
The first three terms on the right satisfy
\[
\begin{aligned}
&-\rho_n
 \sin(\widehat{\zeta}_n-\theta_\alpha)
 \sin(\widehat{\zeta}_n-\theta_\beta)\\
&\quad-\sin\theta_\alpha
 \sin(\widehat{\zeta}_n-\theta_\beta)
 -\sin\theta_\beta
 \sin(\widehat{\zeta}_n-\theta_\alpha)\\
={}&-\sin^2\widehat{\zeta}_n
 \cos(\widehat{\zeta}_n-\theta_\alpha-\theta_\beta)
 {}+\rho_n\sin\theta_\alpha\sin\theta_\beta.
\end{aligned}
\]
Consequently the preceding expression is exactly
\(\widetilde{\mathscr E}_{\alpha,\beta}(u)\) as defined in
\eqref{eq:ac-39}.

For the trigonometric simplification above, apply
\(
 \sin(\widehat{\zeta}_n-\theta)
 =\sin\widehat{\zeta}_n\cos\theta-\rho_n\sin\theta
\)
first with \(\theta=\theta_\alpha\) and then with
\(\theta=\theta_\beta\), and use
\(
 \cos(\widehat{\zeta}_n-\theta_\alpha-\theta_\beta)
 =\rho_n\cos(\theta_\alpha+\theta_\beta)
  +\sin\widehat{\zeta}_n\sin(\theta_\alpha+\theta_\beta).
\)
After collecting the
\(\cos\theta_\alpha\cos\theta_\beta\),
\(\sin\theta_\alpha\cos\theta_\beta\),
\(\cos\theta_\alpha\sin\theta_\beta\), and
\(\sin\theta_\alpha\sin\theta_\beta\) terms, the identity follows from
\(\rho_n^2+\sin^2\widehat{\zeta}_n=1\).

Combining this identity with the Ptolemy comparison derived from
\eqref{eq:ac-37},
\[
 \mathrm{Num}-XYL_n(Q)
 \geq
 \mathrm{Num}_{\mathrm{Ptol}}
 -X_\alpha X_\beta L_n(Q),
\]
gives
\[
 \mathrm{Num}-XYL_n(Q)
 \geq
 \frac{\widetilde{\mathscr E}_{\alpha,\beta}(u)}
 {\sin(\widehat{\zeta}_n-\theta_\alpha)
  \sin(\widehat{\zeta}_n-\theta_\beta)}.
\acEquationTag{eq:ac-near-mixed-ptolemy-to-scalar}
\]
Since the denominator in
\eqref{eq:ac-near-mixed-ptolemy-to-scalar} is positive, the final
assertion follows.
\end{proof}

\begin{lemma}[One-variable minimum estimate]
\label{lem:ac-near-mixed-minimum}
Under the hypotheses of \Cref{lem:ac-near-mixed}, let
\(\widetilde{\mathscr E}_{\alpha,\beta}\) be as in
\Cref{lem:ac-near-mixed-reduction}.

Put
\[
 \overline\theta_\beta
 :=\widehat{\zeta}_n-\theta_\beta
 =\Theta_n(F_n(\beta))
 \in(0,\widehat{\zeta}_n/2].
\]

The second equality is \eqref{eq:ac-6}, applied at \(\beta\).
For \(0\leq z\leq1\), define
\[
 Q_{\alpha,\beta}(z)
 :=\cos(\alpha-\beta)
 -(1+\rho_n)\sin\alpha\sin\beta\,z^2,
\]

The set
\(
 \{z\in[0,1]:Q_{\alpha,\beta}(z)\geq\rho_n\}
\)
is a nonempty closed interval containing \(0\).  Denote its largest
element by \(z_*\).

Define
\[
 S_{\alpha,\beta}
 :=\sin\theta_\alpha\sin\theta_\beta,
 \qquad
 T_{\alpha,\beta}
 :=\sqrt{
 \sin\theta_\alpha\sin\theta_\beta
 \sin(\widehat{\zeta}_n-\theta_\alpha)
 \sin(\widehat{\zeta}_n-\theta_\beta)}
\]
and, for \(0\leq z\leq z_*\), define
\[
\begin{aligned}
 \mathscr E_{\alpha,\beta}(z)
 &:=
 \widetilde{\mathscr E}_{\alpha,\beta}
 \bigl(1-(1+\rho_n)z^2\bigr)\\
 &=
 \sin^2\widehat{\zeta}_n
 \bigl(-L_n(Q_{\alpha,\beta}(z))
 -\cos(\theta_\alpha-\overline\theta_\beta)\bigr)\\
 &\quad +(1+{\rho_n})\bigl[S_{\alpha,\beta}(1-z^2)
 +2T_{\alpha,\beta}(1-z)\bigr].
\end{aligned}                                                \acEquationTag{eq:ac-54}
\]
Thus \(\mathscr E_{\alpha,\beta}\) is the defect
\(\widetilde{\mathscr E}_{\alpha,\beta}\) expressed under the change of
variables \(s=1-(1+\rho_n)z^2\).

Then
\[
 \mathscr E_{\alpha,\beta}(z_*)\geq0
\]
and
\[
 \min_{0\leq z\leq z_*}\mathscr E_{\alpha,\beta}(z)
 \geq
 \min\left\{
 \mathscr E_{\alpha,\beta}(z_*),
 \mathscr E_{\alpha,\beta}(0)
 -(1+\rho_n)T_{\alpha,\beta}
 \right\}.                                                  \acEquationTag{eq:ac-mixed-minimum}
\]
In particular, if
\[
 \mathscr E_{\alpha,\beta}(0)
 -(1+\rho_n)T_{\alpha,\beta}\geq0,
\]
then
\(\mathscr E_{\alpha,\beta}(z)\geq0\) for every
\(z\in[0,z_*]\).
\end{lemma}

\begin{proof}

Since \(\beta\geq\alpha_n^\ast\), the definition of
\(\vartheta_n\) gives
\(\theta_\beta\in
[\widehat{\zeta}_n/2,\widehat{\zeta}_n)\).  Consequently,
\(\overline\theta_\beta\in(0,\widehat{\zeta}_n/2]\), as asserted.

The function \(Q_{\alpha,\beta}\) is continuous and nonincreasing on
\([0,1]\).  The hypotheses give
\[
 \sqrt{\frac{1-u}{1+\rho_n}}\in[0,1],
 \qquad
 Q_{\alpha,\beta}\left(
 \sqrt{\frac{1-u}{1+\rho_n}}\right)=Q\geq\rho_n.
\]
Thus its superlevel set at \(\rho_n\) is nonempty.  Continuity also
makes this set closed.

If \(0\leq z_1\leq z_2\leq1\) and
\(Q_{\alpha,\beta}(z_2)\geq\rho_n\), then monotonicity gives
\(
 Q_{\alpha,\beta}(z_1)
 \geq Q_{\alpha,\beta}(z_2)
 \geq\rho_n.
\)
The superlevel set is therefore an interval containing \(0\).
Because it is a nonempty closed subset of \([0,1]\), it is compact and
has a largest element \(z_*\).  Thus it is precisely \([0,z_*]\), as
asserted in the statement.
If \(\alpha=0\), then \(Q_{\alpha,\beta}(z)=\cos\beta\) and
\(S_{\alpha,\beta}=T_{\alpha,\beta}=0\), so
\(\mathscr E_{\alpha,\beta}\) is
constant in \(z\); its nonnegativity is included in the endpoint check
below.  We may therefore assume \(\alpha>0\) while analyzing possible
interior minima.

\smallskip
\noindent\emph{The possible interior minimum.}
To study the critical points of \(\mathscr E_{\alpha,\beta}\), define
its normalized negative derivative by
\[
 \mathscr G_{\alpha,\beta}(z)
 :=-\frac{\mathscr E_{\alpha,\beta}'(z)}{2(1+\rho_n)}.
\]
Direct differentiation gives
\[
\begin{aligned}
 \mathscr G_{\alpha,\beta}(z)
 ={}&T_{\alpha,\beta}+zS_{\alpha,\beta}
 &-\sin^2\widehat{\zeta}_n\sin\alpha\sin\beta\,z
 \left(\frac{{\sin \widehat{\zeta}_n}\,Q_{\alpha,\beta}(z)}
 {\sqrt{1-Q_{\alpha,\beta}(z)^2}}-\rho_n\right).
\end{aligned}
\]
Because
 $0<\beta-\alpha<\widehat{\zeta}_n$,
one has
\(
 \rho_n<Q_{\alpha,\beta}(0)=\cos(\beta-\alpha)<1.
\)
Since \(Q_{\alpha,\beta}\) is nonincreasing and
\(Q_{\alpha,\beta}(z)\geq\rho_n\) for \(0\leq z\leq z_*\), all
denominators in the preceding formula are nonzero on this interval.
Substitution at \(z=0\) gives
\(
 \mathscr G_{\alpha,\beta}(0)=T_{\alpha,\beta}.
\)

The identities
\[
 Q_{\alpha,\beta}'(z)
 =-2(1+\rho_n)\sin\alpha\sin\beta\,z,
\qquad
 Q_{\alpha,\beta}''(z)
 =-2(1+\rho_n)\sin\alpha\sin\beta
\]
together with
\[
 \frac{d}{dz}
 \left(\frac{Q_{\alpha,\beta}(z)}
 {\sqrt{1-Q_{\alpha,\beta}(z)^2}}\right)
 =
 \frac{Q_{\alpha,\beta}'(z)}
 {(1-Q_{\alpha,\beta}(z)^2)^{3/2}}
\]
give, after one further differentiation,
\[
\begin{aligned}
 \mathscr G_{\alpha,\beta}''(z)
 ={}&\frac{6(1+\rho_n)\sin^3\widehat{\zeta}_n
 (\sin\alpha\sin\beta)^2z}
 {(1-Q_{\alpha,\beta}(z)^2)^{5/2}}\cdot
 \left[1-Q_{\alpha,\beta}(z)^2
 -2Q_{\alpha,\beta}(z)
 \bigl(\cos(\alpha-\beta)-Q_{\alpha,\beta}(z)\bigr)\right].
\end{aligned}
\]
Finally,
\[
\begin{aligned}
 &1-Q_{\alpha,\beta}(z)^2
 -2Q_{\alpha,\beta}(z)
 \bigl(\cos(\alpha-\beta)-Q_{\alpha,\beta}(z)\bigr) =
 \bigl(Q_{\alpha,\beta}(z)-\cos(\alpha-\beta)\bigr)^2
 +\sin^2(\alpha-\beta).
\end{aligned}
\]
Therefore
\[
\begin{aligned}
 \mathscr G_{\alpha,\beta}''(z)
 ={}&\frac{6(1+\rho_n)\sin^3\widehat{\zeta}_n
 (\sin\alpha\sin\beta)^2z}
 {(1-Q_{\alpha,\beta}(z)^2)^{5/2}}
 \left[(Q_{\alpha,\beta}(z)-\cos(\alpha-\beta))^2
 +\sin^2(\alpha-\beta)\right]>0.
\end{aligned}
\]

This holds for \(0<z<z_*\).
Thus \(\mathscr G_{\alpha,\beta}\) is strictly convex and has at most
two zeros.

We now verify that
\(\mathscr E_{\alpha,\beta}(z_*)\geq0\).  Suppose first that
\(z_*=1\).

By the definition of \(Q_{\alpha,\beta}\),
\(
 Q_{\alpha,\beta}(1)=\cos D_n^\times(\alpha,\beta).
\)

Because \(z_*=1\) belongs to the superlevel set defining \(z_*\),
\(
 Q_{\alpha,\beta}(1)\geq\rho_n.
\)
It follows that
\(D_n^\times(\alpha,\beta)\leq\widehat{\zeta}_n\), while
\Cref{lem:ac-crossed-schedule} gives
\(
 \theta_\alpha+\theta_\beta\leq D_n^\times(\alpha,\beta).
\)
Consequently,
\[
 -L_n\bigl(Q_{\alpha,\beta}(1)\bigr)
 =\cos\bigl(\widehat{\zeta}_n-D_n^\times(\alpha,\beta)\bigr)
 \geq
 \cos(\widehat{\zeta}_n-\theta_\alpha-\theta_\beta)
 =\cos(\theta_\alpha-\overline\theta_\beta).
\]
The remaining terms in \eqref{eq:ac-54} vanish at \(z=1\), so
\(\mathscr E_{\alpha,\beta}(1)\geq0\).

If \(z_*<1\), continuity and the definition of \(z_*\) give
\(Q_{\alpha,\beta}(z_*)=\rho_n\).  Since
\(-L_n(\rho_n)=1\), every term in \eqref{eq:ac-54} is then
nonnegative.  Thus in both cases
\(\mathscr E_{\alpha,\beta}(z_*)\geq0\).

Since
\(\mathscr E_{\alpha,\beta}'
=-2(1+\rho_n)\mathscr G_{\alpha,\beta}\) and
\(\mathscr G_{\alpha,\beta}(0)=T_{\alpha,\beta}>0\), the only possible
interior minimum of \(\mathscr E_{\alpha,\beta}\) occurs at the first
zero \(z_c\) of \(\mathscr G_{\alpha,\beta}\), where
\(\mathscr G_{\alpha,\beta}\) changes from positive to negative.  A
second zero of \(\mathscr G_{\alpha,\beta}\), if present, produces a
local maximum of \(\mathscr E_{\alpha,\beta}\).

Convexity and
\(\mathscr G_{\alpha,\beta}(0)=T_{\alpha,\beta}\),
\(\mathscr G_{\alpha,\beta}(z_c)=0\) imply
\[
 \mathscr G_{\alpha,\beta}(z)\leq T_{\alpha,\beta}
 \left(1-\frac z{z_c}\right)
 \qquad(0\leq z\leq z_c),
\]
and hence
\[
 \mathscr E_{\alpha,\beta}(z_c)
 \geq\mathscr E_{\alpha,\beta}(0)
 -(1+{\rho_n})T_{\alpha,\beta}z_c
 \geq\mathscr E_{\alpha,\beta}(0)
 -(1+{\rho_n})T_{\alpha,\beta}.                              \acEquationTag{eq:ac-58}
\]

At the other endpoint,
\(
 \mathscr E_{\alpha,\beta}(0)
 \geq
 \mathscr E_{\alpha,\beta}(0)
 -(1+\rho_n)T_{\alpha,\beta}.
\)
Together with the estimate for
\(\mathscr E_{\alpha,\beta}(z_*)\), the classification of the possible
interior critical points, and \eqref{eq:ac-58}, this proves
\eqref{eq:ac-mixed-minimum}.

This proves the lemma.
\end{proof}

\begin{lemma}[Nonnegativity of the left-endpoint expression]
\label{lem:ac-near-mixed-slack}
Under the hypotheses and notation of
\Cref{lem:ac-near-mixed-minimum}, put
\(
 \delta_\beta:=\widehat{\zeta}_n-\beta.
\)
Define
\[
\begin{aligned}
\mathscr E_{\alpha,\beta}^{\mathrm{left}}
 &:=
 \frac{\mathscr E_{\alpha,\beta}(0)
 -(1+{\rho_n})T_{\alpha,\beta}}
 {1+{\rho_n}}\\
 &=
 (1-{\rho_n})\bigl[
 \cos(\delta_\beta+\alpha)
 -\cos(\theta_\alpha-\overline\theta_\beta)\bigr]
 +S_{\alpha,\beta}+T_{\alpha,\beta}.
\end{aligned}
\]
Then
\[
 \mathscr E_{\alpha,\beta}^{\mathrm{left}}\geq0.
\]
Equivalently,
\[
 \mathscr E_{\alpha,\beta}(0)
 -(1+\rho_n)T_{\alpha,\beta}\geq0.
\]
\end{lemma}

\noindent\emph{Proof strategy.}
Fix \(\beta\).  We prove the result by showing that
\(\alpha\mapsto\mathscr E_{\alpha,\beta}^{\mathrm{left}}\) is concave on
\([0,\alpha_n^\ast]\) and then checking its two endpoint values.
We establish concavity in two steps: first for
\(T_{\alpha,\beta}\), and then for
\(\mathscr E_{\alpha,\beta}^{\mathrm{left}}-T_{\alpha,\beta}\); the second step
reduces to the scalar estimate \eqref{eq:ac-66}.

The complete concavity and endpoint verification is given in
\Cref{app:ac-near-mixed-slack}.

\begin{proof}[Proof of \Cref{lem:ac-near-mixed}]
By \Cref{lem:ac-near-mixed-reduction}, it is enough to prove
\(\widetilde{\mathscr E}_{\alpha,\beta}(u)\geq0\).  Put
\[
 z:=\sqrt{\frac{1-u}{1+\rho_n}}.
\]

Recall from \Cref{lem:ac-near-mixed-minimum} that \(z_*\) is the
largest \(z'\in[0,1]\) satisfying
\(
 Q_{\alpha,\beta}(z')\geq\rho_n.
\)

The hypotheses \(u\geq-\rho_n\) and
\(Q=Q_{\alpha,\beta}(z)\geq\rho_n\) give
\(z\in[0,z_*]\), and \eqref{eq:ac-54} gives
\(
 \widetilde{\mathscr E}_{\alpha,\beta}(u)
 =\mathscr E_{\alpha,\beta}(z).
\)
By \Cref{lem:ac-near-mixed-slack},
\(
 \mathscr E_{\alpha,\beta}(0)
 -(1+\rho_n)T_{\alpha,\beta}\geq0.
\)
The final assertion of \Cref{lem:ac-near-mixed-minimum} therefore gives
\(\mathscr E_{\alpha,\beta}(z)\geq0\).  Hence
\(\widetilde{\mathscr E}_{\alpha,\beta}(u)\geq0\), and
\Cref{lem:ac-near-mixed-reduction} yields \(R\geq L_n(Q)\).
\end{proof}

\begin{proof}[Proof of \Cref{lem:ac-near}]
First suppose that
\(\alpha,\beta<\widehat{\zeta}_n\).  If
\(u\leq-\rho_n\), apply \Cref{lem:ac-near-direct}.  If
\(u\geq-\rho_n\) and
\(\alpha,\beta\leq\alpha_n^\ast\), apply
\Cref{lem:ac-near-low}.  If
\(u\geq-\rho_n\) and
\(\alpha,\beta\geq\alpha_n^\ast\), apply
\Cref{lem:ac-near-high}.  In the remaining case,
\(u\geq-\rho_n\), and the radial levels lie strictly on opposite sides
of \(\alpha_n^\ast\).  Interchange the two pairs if necessary so that
\(\alpha<\alpha_n^\ast<\beta\), and apply
\Cref{lem:ac-near-mixed}.  Thus \eqref{eq:ac-33} holds for every
finite-gain configuration.

It remains to allow one or both radial levels to equal
\(\widehat{\zeta}_n\), which is not covered by the finite-gain regime
lemmas.  If \(Q>\rho_n\), \Cref{cor:ac-finite-gain-closure}\textup{(i)}
with \(g:=L_n\) and \(c:=\rho_n\) extends the finite-gain estimate to
these endpoint configurations.  If \(Q=\rho_n\), then
\(L_n(\rho_n)=-1\), so the desired inequality follows from
\(R\geq-1\).  This proves \eqref{eq:ac-33} in every case.
\end{proof}

\Needspace{10\baselineskip}
\acUpperBoundaryHeading

The main conceptual result of this section is
\Cref{prop:ac-cap-dimension-reduction}.  It reduces an optimization over
two spherical caps to a one-dimensional concave maximization on an
explicit interval.
This reduction is the principal input in the proof of the upper
\((Q,R)\)-boundary for pairs with distinct anchors.

\subsection{Dimension reduction over two spherical caps}
\label{sec:ac-cap-dimension-reduction}

Throughout this subsection, assume \(n\geq2\).  We first define the
spherical caps used below, compute the maximum of \(q\cdot x\) over one
such cap when \(p\cdot q=-\rho_n\), and determine the support function
of a spherical cap.  We then prove
\Cref{prop:ac-cap-dimension-reduction}, followed by three corollaries
that record, respectively, a concave maximization in the boundary
coordinate \(q\cdot x\), the estimate obtained by removing the interval
constraint, and the exact maximizer on the constrained interval.

For any \(p\in\Sp^n\), write
\phantomsection\label{def:spherical-cap}
\[
 \mathcal C_p
 :=\{x\in\Sp^n:p\cdot x\geq\rho_n\}
 =\{x\in\Sp^n:d_n(p,x)\leq\widehat{\zeta}_n\}
\]
for the closed spherical cap centered at \(p\) with angular radius
\(\widehat{\zeta}_n\).

\paragraph{A preliminary cap computation.}

The following projection criterion describes exactly which pairs of
inner products \((p\cdot x,q\cdot x)\) can occur.  We use it first to
compute \(B_n^{\mathrm{cap}}\), and later to determine the interval in
the one-dimensional reduction over two spherical caps.

\begin{lemma}[Feasibility and the boundary of two spherical caps]
\label{lem:ac-cap-gram-boundary}
Assume \(n\geq2\).  Let
\(p,q\in\R^{n+1}\) be unit vectors with
\(p\cdot q=-\rho_n\).
\begin{enumerate}[label=\textup{(\roman*)}]
\item For \(s,t\in[-1,1]\), there exists a unit vector
\(x\in\R^{n+1}\) satisfying
\(
 p\cdot x=s,\qquad q\cdot x=t
\)
if and only if
\[
 s^2+t^2+2\rho_nst\leq1-\rho_n^2.                  \acEquationTag{eq:ac-cap-gram-feasible}
\]
\item If \(s\geq\rho_n\) and a unit vector \(x\) as in
\textup{(i)} exists, then there exists a unit vector \(x'\) satisfying
\(
 p\cdot x'=\rho_n,\qquad q\cdot x'=t.
\)
The values of \(t\) for which such a vector \(x'\) exists are exactly
\(
 -1\leq t\leq1-2\rho_n^2.
\)
\item Consequently,
\[
 B_n^{\mathrm{cap}}
 :=\max_{x\in\mathcal C_p}q\cdot x
 =\max_{y\in\mathcal C_q}p\cdot y
 =1-2\rho_n^2.                                           \acEquationTag{eq:ac-cap-bound}
\]
\end{enumerate}
\end{lemma}

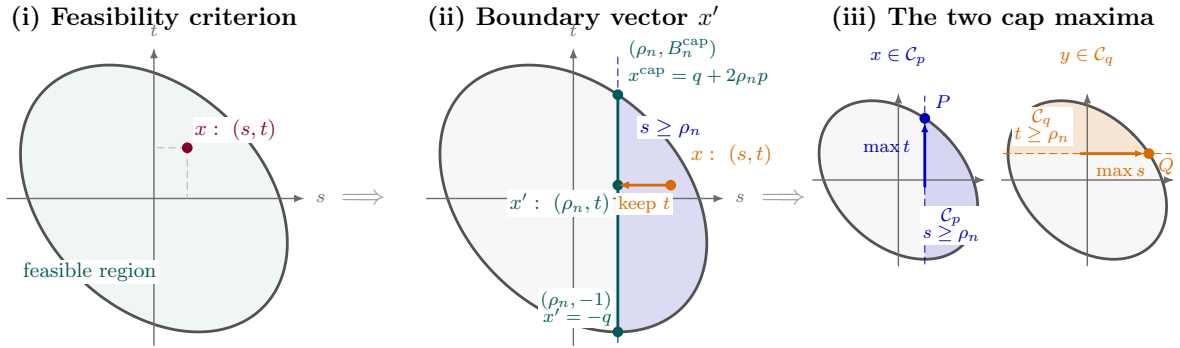
\begin{figure}[!ht]
    \centering
    \resizebox{1.0\textwidth}{!}{\colorlet{gramblue}{blue!70!black}
\colorlet{gramorange}{orange!85!black}
\colorlet{gramteal}{teal!70!black}
\colorlet{grampurple}{purple!70!black}
\colorlet{gramoutline}{black!68}
\colorlet{gramaux}{black!30}

\tikzset{
  gram axis/.style={
    -{Latex[length=1.6mm]},
    black!58,
    semithick
  },
  gram boundary/.style={
    draw=gramoutline,
    very thick
  },
  gram point/.style={
    circle,
    fill,
    inner sep=0pt,
    minimum size=4.6pt
  },
  gram statement/.style={
    draw=black!24,
    rounded corners=2pt,
    fill=black!2,
    inner sep=5pt,
    align=center
  }
}

\begin{tikzpicture}[
  line cap=round,
  line join=round,
  every node/.style={font=\small}
]
  \def\rhof{0.3333333333}
  \pgfmathsetmacro{\Bcapf}{1-2*\rhof*\rhof}
  \def\gramellipse{plot[domain=0:360,samples=180,smooth,variable=\a]
    ({
      (sqrt(1-\rhof)*cos(\a)+sqrt(1+\rhof)*sin(\a))/sqrt(2)
     },{
      (sqrt(1-\rhof)*cos(\a)-sqrt(1+\rhof)*sin(\a))/sqrt(2)
     })}

\begin{scope}[xshift=-6.35cm]
    \node[font=\normalsize\bfseries] at (0,2.72)
      {\textup{(i)} Feasibility criterion};

    \begin{scope}[x=2.02cm,y=2.02cm]
      \path[fill=gramteal!6] \gramellipse -- cycle;
      \path[gram boundary] \gramellipse;

      \draw[gram axis] (-1.10,0)--(1.13,0)
        node[right=1pt,font=\scriptsize] {\(s\)};
      \draw[gram axis] (0,-1.10)--(0,1.13)
        node[above=1pt,font=\scriptsize] {\(t\)};

      \coordinate (gramsamplei) at (.25,.38);
      \draw[gramaux,densely dashed]
        (gramsamplei)--(.25,0);
      \draw[gramaux,densely dashed]
        (gramsamplei)--(0,.38);
      \node[gram point,fill=grampurple] at (gramsamplei) {};
      \node[
        grampurple,
        above right=3pt,
        font=\footnotesize,fill=white,
        inner sep=1pt
      ] at (gramsamplei)
        {\(x:\ (s,t)\)};

      \node[
        gramteal,
        font=\footnotesize,
        fill=white,
        inner sep=1pt
      ] at (-.48,-.55)
        {feasible region};
    \end{scope}

\end{scope}

  \node[black!40,font=\large] at (-3.18,0)
    {\(\Longrightarrow\)};

\begin{scope}
    \node[font=\normalsize\bfseries] at (0,2.72)
      {\textup{(ii)} Boundary vector \(x'\)};

    \begin{scope}[x=2.02cm,y=2.02cm]
      \path[fill=black!3] \gramellipse -- cycle;

      \begin{scope}
        \clip \gramellipse -- cycle;
        \fill[gramblue!14] (\rhof,-1.15) rectangle (1.15,1.15);
      \end{scope}

      \path[gram boundary] \gramellipse;
      \draw[gram axis] (-1.10,0)--(1.13,0)
        node[right=1pt,font=\scriptsize] {\(s\)};
      \draw[gram axis] (0,-1.10)--(0,1.13)
        node[above=1pt,font=\scriptsize] {\(t\)};

      \draw[gramblue,densely dashed]
        (\rhof,-1.08)--(\rhof,1.08);
      \draw[gramteal,very thick]
        (\rhof,-1)--(\rhof,\Bcapf);

      \coordinate (gramx) at (.73,.10);
      \coordinate (gramxprime) at (\rhof,.10);
      \node[gram point,fill=gramorange] at (gramx) {};
      \node[
        gramorange,
        above right=7pt,
        font=\footnotesize
      ] at (gramx)
        {\(x:\ (s,t)\)};
      \node[gram point,fill=gramteal] at (gramxprime) {};
      \node[
        gramteal,
        below left=3pt,
        font=\footnotesize,fill=white,
        inner sep=1pt
      ] at (gramxprime)
        {\(x':\ (\rho_n,t)\)};
      \draw[-{Latex[length=1.7mm]},gramorange,very thick]
        (gramx)--(gramxprime)
        node[midway,below=3pt,font=\scriptsize,fill=white,
        inner sep=1pt]
        {keep \(t\)};

      \node[gram point,fill=gramteal] at (\rhof,\Bcapf) {};
      \node[
        gramteal,
        above right=1pt,
        font=\scriptsize,
        align=left
      ] at (\rhof,\Bcapf)
        {\((\rho_n,B_n^{\mathrm{cap}})\)\\[2pt]
         \(x^{\mathrm{cap}}=q+2\rho_np\)};

      \node[gram point,fill=gramteal] at (\rhof,-1) {};
      \node[
        gramteal,
        above left=2pt,
        font=\scriptsize,
        align=right,fill=white,
        inner sep=1pt
      ] at (\rhof,-1)
        {\((\rho_n,-1)\)\\[-2pt]\(x'=-q\)};

      \node[
        gramblue,
        font=\footnotesize,
        align=center,fill=white,
        inner sep=1pt
      ] at (.73,.53)
        {\(s\geq\rho_n\)};
    \end{scope}

\end{scope}

  \node[black!40,font=\large] at (3.18,0)
    {\(\Longrightarrow\)};

\begin{scope}[xshift=6.35cm]
    \node[font=\normalsize\bfseries] at (0,2.72)
      {\textup{(iii)} The two cap maxima};

\begin{scope}[xshift=-1.43cm,yshift=.28cm,x=1.20cm,y=1.20cm]
      \path[fill=black!3] \gramellipse -- cycle;
      \begin{scope}
        \clip \gramellipse -- cycle;
        \fill[gramblue!16] (\rhof,-1.15) rectangle (1.15,1.15);
      \end{scope}
      \path[gram boundary] \gramellipse;
      \draw[gram axis] (-1.08,0)--(1.10,0);
      \draw[gram axis] (0,-1.08)--(0,1.10);
      \draw[gramblue,densely dashed]
        (\rhof,-1.06)--(\rhof,1.06);
      \draw[-{Latex[length=1.5mm]},gramblue,very thick]
        (\rhof,-.08)--(\rhof,\Bcapf-.05)
        node[pos=.62,left=2pt,font=\scriptsize] {\(\max t\)};
      \node[gram point,fill=gramblue] at (\rhof,\Bcapf) {};
      \node[
        gramblue,
        above right=1pt,
        font=\scriptsize
      ] at (\rhof,\Bcapf)
        {\(P\)};
      \node[
        gramblue,
        font=\scriptsize,
        align=center,fill=white,
        inner sep=1pt
      ] at (.67,-.57)
        {\(\mathcal C_p\)\\[-2pt]\(s\geq\rho_n\)};
      \node[
        font=\scriptsize\bfseries,
        text=gramblue,fill=white,
        inner sep=1pt
      ] at (0,1.56)
        {\(x\in\mathcal C_p\)};
    \end{scope}

\begin{scope}[xshift=1.43cm,yshift=.28cm,x=1.20cm,y=1.20cm]
      \path[fill=black!3] \gramellipse -- cycle;
      \begin{scope}
        \clip \gramellipse -- cycle;
        \fill[gramorange!17] (-1.15,\rhof) rectangle (1.15,1.15);
      \end{scope}
      \path[gram boundary] \gramellipse;
      \draw[gram axis] (-1.08,0)--(1.10,0);
      \draw[gram axis] (0,-1.08)--(0,1.10);
      \draw[gramorange,densely dashed]
        (-1.06,\rhof)--(1.06,\rhof);
      \draw[-{Latex[length=1.5mm]},gramorange,very thick]
        (-.08,\rhof)--(\Bcapf-.05,\rhof)
        node[pos=.62,below=2pt,font=\scriptsize] {\(\max s\)};
      \node[gram point,fill=gramorange] at (\Bcapf,\rhof) {};
      \node[
        gramorange,
        below right=0.5pt,
        font=\scriptsize
      ] at (\Bcapf,\rhof+0.06)
        {\(Q\)};
      \node[
        gramorange,
        font=\scriptsize,
        align=center,fill=white,
        inner sep=1pt
      ] at (-.57,.68)
        {\(\mathcal C_q\)\\[-2pt]\(t\geq\rho_n\)};
      \node[
        font=\scriptsize\bfseries,
        text=gramorange
      ] at (0,1.56)
        {\(y\in\mathcal C_q\)};
    \end{scope}

\end{scope}

\end{tikzpicture} }
    \caption{
  The three conclusions of \Cref{lem:ac-cap-gram-boundary} for unit vectors \(p,q\) with
\(p\cdot q=-\rho_n\). Panel \textup{(i)} gives the exact feasible Gram
region for $(s,t)=(p\cdot x,q\cdot x).$
Panel \textup{(ii)} illustrates \Cref{lem:ac-cap-gram-boundary}\textup{(ii)}: whenever
\((s,t)\) is feasible and the cap condition \(s\geq\rho_n\) holds,
\((\rho_n,t)\) is also feasible. 
The possible values of \(t\) are exactly
\(-1\leq t\leq B_n^{\mathrm{cap}}\). The horizontal arrow represents
this reduction of feasible Gram data and is not a prescribed
projection of \(x\). Panel \textup{(iii)} uses two separate copies of
the feasible region for the symmetric optimizations over
\(\mathcal C_p\) and \(\mathcal C_q\); both maxima equal
\(B_n^{\mathrm{cap}}=1-2\rho_n^2\). The plotted ellipses use
\(\rho_n=1/3\); all statements and identities are symbolic.}
    \label{fig:gram_ellipse}
\end{figure}

\begin{proof}
Put \(E:=\operatorname{span}\{p,q\}\).  Since
\(p\cdot q=-\rho_n\) and \(0<\rho_n<1\), the vectors \(p,q\) are
linearly independent.  Solving the two equations
\(
 p\cdot x_E=s,\qquad q\cdot x_E=t
\)
for \(x_E\in E\) gives
\[
 x_E
 =\frac{s+\rho_nt}{1-\rho_n^2}\,p
  +\frac{\rho_ns+t}{1-\rho_n^2}\,q.                  \acEquationTag{eq:ac-cap-projection}
\]
Its squared norm is
\[
 \lVert x_E\rVert^2
 =\frac{(s+\rho_nt)s+(\rho_ns+t)t}{1-\rho_n^2}
 =\frac{s^2+t^2+2\rho_nst}{1-\rho_n^2}.
\acEquationTag{eq:ac-cap-projection-norm}
\]

If a unit vector \(x\in\R^{n+1}\) has the prescribed inner products,
then its orthogonal projection onto \(E\) is \(x_E\).  Thus
\(x=x_E+w\) for some \(w\in E^\perp\), and

\(
 1=\lVert x\rVert^2
  =\lVert x_E\rVert^2+\lVert w\rVert^2
  \geq\lVert x_E\rVert^2.
\)
This proves the necessity of \eqref{eq:ac-cap-gram-feasible}.
Conversely, assume \eqref{eq:ac-cap-gram-feasible}.  Then
\eqref{eq:ac-cap-projection-norm} gives
\(\lVert x_E\rVert\leq1\).  Because \(n\geq2\),
\(
 \dim E^\perp=n-1\geq1,
\)
so there exists \(w\in E^\perp\) with
\(
 \lVert w\rVert
 =\sqrt{1-\lVert x_E\rVert^2}.
\)
The vector \(x:=x_E+w\) is unit and has the prescribed inner products.
This proves part~\textup{(i)}.

For fixed \(t\in[-1,1]\), the function
\(s\mapsto s^2+t^2+2\rho_nst\) is nondecreasing on
\([\rho_n,1]\), because its derivative satisfies
\(
 2(s+\rho_nt)\geq2\rho_n(1+t)\geq0.
\)
Hence, if \((s,t)\) satisfies
\eqref{eq:ac-cap-gram-feasible} with \(s\geq\rho_n\), then
\(
 \rho_n^2+t^2+2\rho_n^2t
 \leq s^2+t^2+2\rho_nst
 \leq1-\rho_n^2.
\)
Part~\textup{(i)} therefore shows that \((\rho_n,t)\) is feasible.
At \(s=\rho_n\), condition \eqref{eq:ac-cap-gram-feasible} is
equivalent to
\(
 (t+1)\bigl(t-(1-2\rho_n^2)\bigr)\leq0.
\)
Since \(t\in[-1,1]\), the feasible values are exactly
\(
 -1\leq t\leq1-2\rho_n^2.
\)
This proves part~\textup{(ii)}.

If \(x\in\mathcal C_p\), then \(s:=p\cdot x\geq\rho_n\), so
part~\textup{(ii)} gives
\(
 q\cdot x\leq1-2\rho_n^2.
\)
The upper bound is attained at
\(
 x^{\mathrm{cap}}:=q+2\rho_np,
\)
because
\(
 \lVert x^{\mathrm{cap}}\rVert=1,\qquad
 p\cdot x^{\mathrm{cap}}=\rho_n,\qquad
 q\cdot x^{\mathrm{cap}}=1-2\rho_n^2.
\)
Thus the first maximum in \eqref{eq:ac-cap-bound} is
\(1-2\rho_n^2\).  Interchanging \(p\) and \(q\) proves the second and
completes part~\textup{(iii)}.
\end{proof}

\begin{lemma}[Support function of a spherical cap]
\label{lem:ac-cap-support}
For a unit vector \(q\in\R^{n+1}\) and \(w\in\R^{n+1}\), define
\[
 \mathcal S_{\rho_n}(w;q)
 :=\max_{y\in\mathcal C_q}w\cdot y.
\]
Then \(\mathcal S_{\rho_n}(0;q)=0\).  If \(w\neq0\), then
\[
 \mathcal S_{\rho_n}(w;q)=
 \begin{cases}
 \lVert w\rVert,
 &\dfrac{w}{\lVert w\rVert}\cdot q\geq\rho_n,\\[3pt]
 {\rho_n}(w\cdot q)+{\sin \widehat{\zeta}_n}
 \sqrt{\lVert w\rVert^2-(w\cdot q)^2},
 &\dfrac{w}{\lVert w\rVert}\cdot q<\rho_n.
 \end{cases}                                                \acEquationTag{eq:ac-74}
\]
\end{lemma}

\begin{proof}
The value at \(w=0\) is immediate.  Suppose \(w\neq0\), and write
\(
 w_\perp:=w-(w\cdot q)q,
 \qquad w=(w\cdot q)q+w_\perp,
 \qquad w_\perp\perp q.
\)
For \(y\in\mathcal C_q\), put \(s:=q\cdot y\).  Then
\[
 y=sq+\sqrt{1-s^2}\,\xi,
 \qquad \rho_n\leq s\leq1,
 \qquad \xi\perp q,\quad \lVert \xi\rVert=1.
\]

For fixed \(s\), the decompositions of \(w\) and \(y\) above give
\[
 w\cdot y
 =s(w\cdot q)+\sqrt{1-s^2}\,(w_\perp\cdot\xi)
 \leq s(w\cdot q)+\sqrt{1-s^2}\,\lVert w_\perp\rVert,
\]
where the inequality is Cauchy--Schwarz applied to \(w_\perp\) and
\(\xi\).
If \(w_\perp\neq0\), equality is attained by
\(\xi=w_\perp/\lVert w_\perp\rVert\); if \(w_\perp=0\), equality holds
for every admissible \(\xi\).  Hence the remaining maximization is that
of the concave function
\[
 g_w(s):=
 s(w\cdot q)+\sqrt{1-s^2}\,\lVert w_\perp\rVert,
 \qquad -1\leq s\leq1.
\]

Without the constraint \(s\geq\rho_n\), the function \(g_w\) is
maximized at
\(
 s_0=\tfrac{w\cdot q}{\lVert w\rVert},
\)
corresponding to \(y=w/\lVert w\rVert\).  If \(s_0\geq\rho_n\), this
point belongs to \(\mathcal C_q\), and the maximum is
\(\lVert w\rVert\).  If \(s_0<\rho_n\), concavity implies that \(g_w\)
is nonincreasing on \([\rho_n,1]\), so its maximum on this interval is
\(g_w(\rho_n)\).  Since
\(
 \lVert w_\perp\rVert
 =\sqrt{\lVert w\rVert^2-(w\cdot q)^2},
\)
this proves \eqref{eq:ac-74}.
\end{proof}

\Cref{fig:lemma44-cap-reduction} summarizes the reduction.  The interval
\(J_{b,c}\) and the points
\(\mathsf h_{\mathrm{free}},\mathsf h_{\mathrm{opt}}\)
appearing in panel~\textup{(c)} are defined in
\Cref{prop:ac-cap-dimension-reduction,cor:ac-cap-projected-maximizer}.

\begin{figure}[htbp]
\centering
\resizebox{0.98\textwidth}{!}{\begin{tikzpicture}[
  line cap=round,
  line join=round,
  every node/.style={font=\small},
  capone/.style={blue!65!black},
  captwo/.style={orange!85!black},
  scalar/.style={teal!65!black}
]

\begin{scope}[xshift=-5.15cm]
  \def\R{1.72}
  \def\pangle{145}
  \def\qangle{35}
  \def\capangle{70}

  \coordinate (Oa) at (0,0);
  \coordinate (Pa) at (\pangle:1.25);
  \coordinate (Qa) at (\qangle:1.25);
  \coordinate (Xa) at (190:\R);
  \coordinate (Ya) at (18:\R);

\path[fill=blue!13]
    (75:\R) arc[start angle=75,end angle=215,radius=\R] -- cycle;
  \path[fill=orange!16,fill opacity=.72]
    (-35:\R) arc[start angle=-35,end angle=105,radius=\R] -- cycle;
  \draw[black!60,semithick] (Oa) circle[radius=\R];

\draw[capone,densely dashed] (75:\R)--(215:\R);
  \draw[captwo,densely dashed] (-35:\R)--(105:\R);
  \draw[capone,very thick]
    (75:\R) arc[start angle=75,end angle=215,radius=\R];
  \draw[captwo,very thick]
    (-35:\R) arc[start angle=-35,end angle=105,radius=\R];

\draw[-{Latex[length=1.8mm]},capone,semithick] (Oa)--(Pa)
    node[pos=.72,below left=-1pt] {\(p\)};
  \draw[-{Latex[length=1.8mm]},captwo,semithick] (Oa)--(Qa)
    node[pos=.72,below right=-1pt] {\(q\)};
  \fill[capone] (Xa) circle[radius=2pt]
    node[below left=1pt] {\(x\in\mathcal C_p\)};
  \fill[captwo] (Ya) circle[radius=2pt]
    node[right=2pt] {\(y\in\mathcal C_q\)};
  \fill[black!65] (Oa) circle[radius=1.2pt];

  \node[capone,align=center,font=\footnotesize] at (-1.88,1.84)
    {\(\mathcal C_p\)\\[-2pt]\(p\cdot x\geq\rho_n\)};
  \node[captwo,align=center,font=\footnotesize] at (1.88,1.84)
    {\(\mathcal C_q\)\\[-2pt]\(q\cdot y\geq\rho_n\)};
  \node[font=\footnotesize,align=center] at (0,-2.19)
    {\textup{(a)} cap-constrained variables};
\end{scope}

\begin{scope}[xshift=0cm]
  \def\R{1.72}
  \def\capangle{70}
  \coordinate (Ob) at (0,0);
  \coordinate (Qb) at (0:1.25);
  \coordinate (Wb) at (112:1.48);
  \coordinate (Yb) at (\capangle:\R);

  \path[fill=orange!16,fill opacity=.72]
    (-\capangle:\R)
    arc[start angle=-\capangle,end angle=\capangle,radius=\R] -- cycle;
  \draw[black!60,semithick] (Ob) circle[radius=\R];
  \draw[captwo,densely dashed]
    (-\capangle:\R)--(\capangle:\R);
  \draw[captwo,very thick]
    (-\capangle:\R)
    arc[start angle=-\capangle,end angle=\capangle,radius=\R];

  \draw[-{Latex[length=1.8mm]},captwo,semithick] (Ob)--(Qb)
    node[pos=1,below right=-1pt] {\(q\)};
  \draw[-{Latex[length=2mm]},scalar,very thick] (Ob)--(Wb)
    node[pos=.66,left=2pt] {\(w_x\)};
  \draw[black!35,dashed] (Ob)--(Yb);

\draw[scalar,densely dashed]
    ($(Yb)+(22:1.02)$)--($(Yb)+(202:1.02)$);
  \fill[scalar] (Yb) circle[radius=2.2pt]
    node[right=5pt,yshift=7pt,align=left]
      {\(y_{\max}\)\\[-2pt]\(q\cdot y_{\max}=\rho_n\)};

  \node[font=\footnotesize,black!65] at (0,2.52)
    {the plane \(\operatorname{span}\{q,w_x\}\)};
  \node[font=\footnotesize,align=center] at (0,-2.19)
    {\textup{(b)} fix \(x\), then maximize \(w_x\cdot y\)};
\end{scope}

\begin{scope}[xshift=5.20cm]
  \def\zleft{-1.15}
  \def\zright{-0.05}
  \def\zfree{0.48}

  \draw[-{Latex[length=1.7mm]},black!65]
    (-1.78,0)--(1.84,0) node[right=-1pt]
    {\(\mathsf h=w_x\cdot q\)};
  \draw[-{Latex[length=1.7mm]},black!65]
    (-1.62,-.08)--(-1.62,2.02);

\draw[scalar,very thick,domain=-1.58:1.58,samples=100,smooth,variable=\z]
    plot ({\z},{.40+.25*\z+.80*sqrt(2.56-\z*\z)});

\draw[blue!16,line width=7pt] (\zleft,-.20)--(\zright,-.20);
  \draw[capone,line width=1.2pt] (\zleft,-.20)--(\zright,-.20);
  \draw[capone] (\zleft,-.27)--(\zleft,-.13);
  \draw[capone] (\zright,-.27)--(\zright,-.13);
  \node[capone,above=3pt] at ({(\zleft+\zright)/2},-.20)
    {\(J_{b,c}\)};

  \draw[orange!85!black,densely dashed]
    (\zfree,0)--(\zfree,1.742);
  \fill[orange!85!black] (\zfree,1.742) circle[radius=2.1pt]
    node[xshift=2.5pt,yshift=8pt] {\(\mathsf h_{\mathrm{free}}\)};

  \draw[capone,densely dashed]
    (\zright,0)--(\zright,1.666);
  \fill[capone] (\zright,1.666) circle[radius=2.1pt]
    node[xshift=-2.5pt,yshift=8pt] {\(\mathsf h_{\mathrm{opt}}\)};

  \draw[-{Latex[length=1.7mm]},black!55]
    (\zfree,0.8) to (\zright,0.8) node[right=14pt,font=\scriptsize]
      {\(\operatorname{proj}_{J_{b,c}}\)};

  \node[font=\footnotesize,scalar] at (.10,2.52)
    {one-variable concave objective};
  \node[font=\footnotesize,align=center] at (0,-2.19)
    {\textup{(c)} project \(\mathsf h_{\mathrm{free}}\) onto \(J_{b,c}\)};
\end{scope}

\node[font=\large,black!45] at (-2.58,0) {\(\Longrightarrow\)};
\node[font=\large,black!45] at (2.60,0) {\(\Longrightarrow\)};

\end{tikzpicture}
 }
\caption{The exact reduction in
\Cref{prop:ac-cap-dimension-reduction}.  Panel~\textup{(a)} shows
meridional sections of the two spherical caps, whose axes satisfy
\(p\cdot q=-\rho_n\).  Panel~\textup{(b)} depicts the meridional
two-plane in which the maximization over \(y\) may be carried out.
Panel~\textup{(c)} shows the remaining one-variable concave
maximization on \(J_{b,c}\) and the projection of its unconstrained
maximizer \(\mathsf h_{\mathrm{free}}\) to
\(\mathsf h_{\mathrm{opt}}\in J_{b,c}\).  The symbols in
panel~\textup{(c)} are defined in
\Cref{prop:ac-cap-dimension-reduction,cor:ac-cap-projected-maximizer}.
}
\label{fig:lemma44-cap-reduction}
\end{figure}
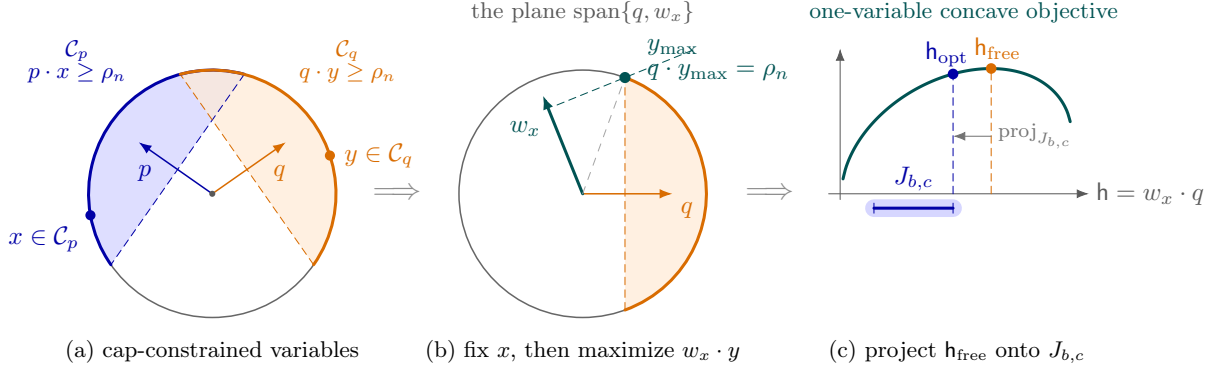
\FloatBarrier

\begin{proposition}[Exact dimension reduction over two spherical caps]
\label{prop:ac-cap-dimension-reduction}
Assume \(n\geq2\).  Let \(p,q\in\R^{n+1}\) be unit vectors satisfying
\(
 p\cdot q=-{\rho_n}.
\)
For \(a,b,c>0\), define
\[
\mathfrak{C}(a,b,c)
:=
\max_{\substack{x\in\mathcal C_p\\y\in\mathcal C_q}}
\left\{
a(q\cdot x)+b(p\cdot y)-c(x\cdot y)
\right\}.
\]

Here \(p,q,x,y\) denote vectors, whereas \(a,b,c,\nu_{b,c}\), and the scalar \(\mathsf h\) introduced below are real numbers.

Recall that
\(
 B_n^{\mathrm{cap}}:=1-2\rho_n^2,
\)
and set
\[
\begin{gathered}
 \nu_{b,c}
 :=\sqrt{b^2+c^2-2\rho_nbc},\\
 J_{b,c}:=
 \bigl[-{\rho_n}b-c B_n^{\mathrm{cap}},\,
       \min\{c-{\rho_n}b,{\rho_n}\nu_{b,c}\}\bigr].
\end{gathered}                                                   \acEquationTag{eq:ac-72}
\]
\begin{enumerate}[label=\textup{(\roman*)}]
\item The set \(J_{b,c}\) is a nonempty closed subinterval of
\([-\nu_{b,c},\nu_{b,c}]\).

\item
For every \(t\in[-1,B_n^{\mathrm{cap}}]\), the maximum in the
definition of \(\mathfrak C(a,b,c)\), subject to the additional
constraint \(q\cdot x=t\), is attained by a pair
\((x,y)\in\mathcal C_p\times\mathcal C_q\) satisfying
\(p\cdot x=\rho_n\).

\item The quantity \(\mathfrak C(a,b,c)\) equals the following
one-dimensional concave maximum:
\[
 \boxed{
 \mathfrak C(a,b,c)
 =-\frac{ab{\rho_n}}{c}
 +\max_{\mathsf h\in J_{b,c}}
 \left\{\big({\rho_n}-\tfrac{a}{c}\big)\mathsf h
 +{\sin \widehat{\zeta}_n}
 \sqrt{\nu_{b,c}^2-\mathsf h^2}\right\}.
 }                                                               \acEquationTag{eq:ac-73}
\]
\end{enumerate}

\end{proposition}

\paragraph{Geometric interpretation.}
The geometric reduction behind the formula has two stages. For fixed
\(x\), the terms involving \(y\) form the linear functional \(w_x\cdot y\),
where \(w_x:=bp-cx\).  Rotational symmetry of the spherical cap
about its axis \(q\) implies that a maximizing \(y\) may be chosen in a
meridional two-plane containing \(q\) and \(w_x\); when these vectors
are linearly independent, this plane is
\(\operatorname{span}\{q,w_x\}\).  After maximizing over \(y\), the
resulting expression depends on \(x\) only through
\(p\cdot x\) and \(q\cdot x\).  For fixed \(q\cdot x\), it is
nonincreasing as \(p\cdot x\) increases, so a maximizing \(x\) may be
chosen with \(p\cdot x=\rho_n\).  On this boundary the objective
depends only on the scalar \(q\cdot x\), or equivalently on
\(\mathsf h:=w_x\cdot q\), which ranges over the interval \(J_{b,c}\) defined
in \eqref{eq:ac-72}.  Thus the cap-constrained
maximum reduces first to the maximization of \(w_x\cdot y\) over
\(y\in\mathcal C_q\) within such a meridional two-plane, and then to a
one-variable concave maximization.

\begin{proof}[Proof of \Cref{prop:ac-cap-dimension-reduction}]
We first verify the assertions concerning the quantities in
\eqref{eq:ac-72}.  The scalar \(\nu_{b,c}\) is positive because
\(
 \nu_{b,c}^2
 =(b-\rho_nc)^2+(1-\rho_n^2)c^2>0.
\)
Since \(n\geq2\), one has
\(
 0<\rho_n\leq\tfrac13,
 \qquad
 B_n^{\mathrm{cap}}=1-2\rho_n^2>0.
\)

The left endpoint of \(J_{b,c}\) is smaller than each of its two
possible right endpoints:
\[
 (c-\rho_nb)
 -(-\rho_nb-c B_n^{\mathrm{cap}})
 =c(1+B_n^{\mathrm{cap}})>0,
\qquad
 -\rho_nb-c B_n^{\mathrm{cap}}<0
 <\rho_n\nu_{b,c}.
\]
Thus \(J_{b,c}\) is a nonempty closed interval.  To locate it relative
to \([-\nu_{b,c},\nu_{b,c}]\), compute
\[
 \nu_{b,c}^2
 -(\rho_nb+c B_n^{\mathrm{cap}})^2
 =(1-\rho_n^2)(b-2\rho_nc)^2\geq0,
\qquad
 \nu_{b,c}^2-(c-\rho_nb)^2
 =(1-\rho_n^2)b^2\geq0.
\]
Because \(\rho_nb+c B_n^{\mathrm{cap}}>0\), the first identity gives
\(
 -\nu_{b,c}
 \leq-\rho_nb-c B_n^{\mathrm{cap}}.
\)
The second gives
\(
 |c-\rho_nb|\leq\nu_{b,c},
\)
and \(0<\rho_n<1\) gives
\(\rho_n\nu_{b,c}<\nu_{b,c}\).  Both possible right endpoints in
\eqref{eq:ac-72} are therefore at most \(\nu_{b,c}\), proving
\(
 J_{b,c}\subseteq[-\nu_{b,c},\nu_{b,c}].
\)
This proves part~\textup{(i)}.

Fix \(t\in[-1,B_n^{\mathrm{cap}}]\).  By
\Cref{lem:ac-cap-gram-boundary}\textup{(ii)}, there exists
\(x\in\mathcal C_p\) satisfying \(q\cdot x=t\).  Consider all such
points \(x\).  For each of them,
the part of the objective defining \(\mathfrak C(a,b,c)\) that depends
on \(y\in\mathcal C_q\) is
\(
 b(p\cdot y)-c(x\cdot y)=(bp-cx)\cdot y.
\)
Thus, with \(w_x:=bp-cx\), the support-function formula in
\Cref{lem:ac-cap-support} shows that maximizing over
\(y\in\mathcal C_q\) gives
\(
 at+\mathcal S_{\rho_n}(w_x;q).
\)
Put \(s:=p\cdot x\).  The two quantities on which the support function
depends satisfy
\(
 w_x\cdot q=-\rho_nb-ct,\qquad
 \lVert w_x\rVert^2=b^2+c^2-2bcs.
\)
Part~\textup{(ii)} of \Cref{lem:ac-cap-gram-boundary} provides a unit
vector \(x'\in\mathcal C_p\) such that
\(
 p\cdot x'=\rho_n,\qquad q\cdot x'=t.
\)
If \(w_{x'}:=bp-cx'\), then
\[
 w_{x'}\cdot q=w_x\cdot q,\qquad
 \lVert w_{x'}\rVert^2
 =b^2+c^2-2bc\rho_n
 \geq\lVert w_x\rVert^2.
\]
For a fixed value of \(w\cdot q\), the proof of
\eqref{eq:ac-74} expresses the support value as
\[
 \max_{\rho_n\leq \sigma\leq1}
 \left\{
  \sigma(w\cdot q)
  +\sqrt{1-\sigma^2}
   \sqrt{\lVert w\rVert^2-(w\cdot q)^2}
 \right\}.
\]
Every expression inside this maximum is nondecreasing in
\(\lVert w\rVert\).  Hence
\(
 \mathcal S_{\rho_n}(w_{x'};q)
 \geq\mathcal S_{\rho_n}(w_x;q).
\)
Thus, for each feasible \(t\), the maximum may be attained with
\(p\cdot x=\rho_n\).
This proves part~\textup{(ii)}.
On this boundary, direct calculation gives
\(
 \lVert w_x\rVert^2
 =b^2+c^2-2bc\rho_n
 =\nu_{b,c}^2,
\)
and therefore \(\lVert w_x\rVert=\nu_{b,c}\).

\phantomsection\label{def:cap-height-coordinate}
Put
\(
 \mathsf h:=w_x\cdot q=-\rho_nb-c(q\cdot x)=-\rho_nb-ct.
\)
Since \(q\) is a unit vector, \(\mathsf h\) is the scalar component of
\(w_x\) in the direction of \(q\).
As
\(q\cdot x\) ranges over \([-1,B_n^{\mathrm{cap}}]\), the scalar
\(\mathsf h\)
ranges over
\(
 [-\rho_nb-c B_n^{\mathrm{cap}},\,c-\rho_nb].
\)
For \(\mathsf h<\rho_n\nu_{b,c}\), the second case of
\eqref{eq:ac-74} shows that the reduced quantity
\(at+\mathcal S_{\rho_n}(w_x;q)\) equals
\[
 -\frac{ab\rho_n}{c}
 +\big(\rho_n-\tfrac{a}{c}\big)\mathsf h
 +\sin\widehat{\zeta}_n
  \sqrt{\nu_{b,c}^2-\mathsf h^2}.                    \acEquationTag{eq:ac-cap-second-branch}
\]
This expression extends continuously to
\(\mathsf h=\rho_n\nu_{b,c}\), where the two cases of
\eqref{eq:ac-74} agree.  If the interval above extends beyond this
value, then for \(\mathsf h\geq\rho_n\nu_{b,c}\) the first case gives
\[
 -\frac{ab\rho_n}{c}-\frac ac\,\mathsf h+\nu_{b,c}.
\]
Since \(a/c>0\), this affine function is strictly decreasing in
\(\mathsf h\);
its maximum on the first-branch portion is therefore attained at
\(\mathsf h=\rho_n\nu_{b,c}\), where its value equals the value of
\eqref{eq:ac-cap-second-branch}.
Consequently, the maximum over the full \(\mathsf h\)-interval equals the
maximum of the expression in \eqref{eq:ac-cap-second-branch} over
\[
 \bigl[-\rho_nb-cB_n^{\mathrm{cap}},\,
 \min\{c-\rho_nb,\rho_n\nu_{b,c}\}\bigr]
 =J_{b,c}.
\]
This proves \eqref{eq:ac-73}.
This proves part~\textup{(iii)}.

\end{proof}

\begin{corollary}[Reduction in the variable \texorpdfstring{\(q\cdot x\)}{q dot x}]
\label{cor:ac-cap-boundary-coordinate}
Under the hypotheses and notation of
\Cref{prop:ac-cap-dimension-reduction}, for each
\(t\in[-1,B_n^{\mathrm{cap}}]\), choose a unit vector \(x_t\) such that
\[
 p\cdot x_t=\rho_n,
 \qquad
 q\cdot x_t=t,
\]
and put
\[
 w_t:=bp-cx_t,
 \qquad
 \Psi_{a,b,c}(t):=at+\mathcal S_{\rho_n}(w_t;q).
\]
Then \(\Psi_{a,b,c}\) is independent of the choice of \(x_t\), is
concave on \([-1,B_n^{\mathrm{cap}}]\), and
\[
 \mathfrak C(a,b,c)
 =\max_{-1\leq t\leq B_n^{\mathrm{cap}}}\Psi_{a,b,c}(t).
 \acEquationTag{eq:ac-cap-boundary-coordinate}
\]
\end{corollary}

\begin{proof}
The existence of \(x_t\) follows from
\Cref{lem:ac-cap-gram-boundary}\textup{(ii)}.  Moreover,
\[
 \lVert w_t\rVert^2=b^2+c^2-2\rho_nbc=\nu_{b,c}^2,
 \qquad
 w_t\cdot q=-\rho_nb-ct.
 \acEquationTag{eq:ac-cap-boundary-coordinate-data}
\]
Thus the support-function formula \eqref{eq:ac-74} shows that
\(\Psi_{a,b,c}(t)\) is independent of the chosen \(x_t\).

Maximizing first over \(y\in\mathcal C_q\) gives
\[
 \mathfrak C(a,b,c)
 =\max_{x\in\mathcal C_p}
 \bigl\{a(q\cdot x)+\mathcal S_{\rho_n}(bp-cx;q)\bigr\}.
\]
For each fixed value \(t=q\cdot x\), the fiberwise boundary conclusion
of \Cref{prop:ac-cap-dimension-reduction}\textup{(ii)} shows that replacing \(x\) by
\(x_t\) does not decrease this expression.  This proves that
the left side of \eqref{eq:ac-cap-boundary-coordinate} is at most its
right side.  Conversely, every \(x_t\) is an admissible point of
\(\mathcal C_p\), so the reverse inequality follows directly from the
definition of \(\mathfrak C(a,b,c)\).

It remains to prove concavity.  Put
\(\mathsf h_t:=-\rho_nb-ct\).  By
\eqref{eq:ac-cap-boundary-coordinate-data} and \eqref{eq:ac-74},
\[
 \Psi_{a,b,c}(t)=
 \begin{cases}
  at+\nu_{b,c},
  &\mathsf h_t\geq\rho_n\nu_{b,c},\\[2pt]
  at+\rho_n\mathsf h_t
  +\sin\widehat\zeta_n\sqrt{\nu_{b,c}^2-\mathsf h_t^2},
  &\mathsf h_t<\rho_n\nu_{b,c}.
 \end{cases}
\]
The first expression is affine, and the second is concave because
\(t\mapsto\mathsf h_t\) is affine and
\(h\mapsto\sqrt{\nu_{b,c}^2-h^2}\) is concave.  At the transition
\(\mathsf h_t=\rho_n\nu_{b,c}\), the two expressions have the same
value and the same derivative \(a\).  Since \(\mathsf h_t\) is strictly
decreasing, there is at most one transition; hence \(\Psi_{a,b,c}\) is
concave on the full interval.

\end{proof}

\begin{corollary}[Unconstrained estimate]
\label{cor:ac-cap-unconstrained-estimate}
Under the hypotheses and notation of
\Cref{prop:ac-cap-dimension-reduction}, put
\(
 \nu_{a,c}:=\sqrt{a^2+c^2-2\rho_nac}.
\)

Then
\[
 \mathfrak C(a,b,c)
 \leq
 \frac{\nu_{a,c}\nu_{b,c}-\rho_nab}{c}.              \acEquationTag{eq:ac-cap-unconstrained}
\]
\end{corollary}

\begin{proof}

By \Cref{prop:ac-cap-dimension-reduction}\textup{(i)},
\(J_{b,c}\subseteq[-\nu_{b,c},\nu_{b,c}]\), so
\eqref{eq:ac-73} gives
\[
\begin{aligned}
 \mathfrak C(a,b,c)
 &\leq-\frac{ab\rho_n}{c}
 +\max_{-\nu_{b,c}\leq\mathsf h\leq\nu_{b,c}}
 \left\{(\rho_n-a/c)\mathsf h
 +\sin\widehat{\zeta}_n
  \sqrt{\nu_{b,c}^2-\mathsf h^2}\right\}\\
 &\leq-\frac{ab\rho_n}{c}
 +\nu_{b,c}
 \sqrt{(\rho_n-a/c)^2+\sin^2\widehat{\zeta}_n}.
\end{aligned}
\]
The second inequality is Cauchy--Schwarz applied to
\[
 \bigl(\mathsf h,\sqrt{\nu_{b,c}^2-\mathsf h^2}\bigr)
 \quad\text{and}\quad
 \bigl(\rho_n-a/c,\sin\widehat{\zeta}_n\bigr);
\]
Using \(\sin^2\widehat{\zeta}_n=1-\rho_n^2\), the last square root is
\[
 \sqrt{(\rho_n-a/c)^2+1-\rho_n^2}
 =\frac{\sqrt{a^2+c^2-2\rho_nac}}{c}
 =\frac{\nu_{a,c}}{c}.
\]
Substitution proves \eqref{eq:ac-cap-unconstrained}.
\end{proof}

\begin{corollary}[Explicit projected maximizer]
\label{cor:ac-cap-projected-maximizer}
Under the hypotheses and notation of
\Cref{prop:ac-cap-dimension-reduction},
for every nonempty closed interval
\(J=[s_-,s_+]\subset\R\) and every \(t\in\R\), define
\[
 \operatorname{proj}_{J}(t)
 :=\operatorname*{arg\,min}_{\xi\in J}|t-\xi|
 =\min\{s_+,\max\{s_-,t\}\}.
\]
Define
\[
 \mathsf h_{\mathrm{free}}
 :=
 \frac{\nu_{b,c}(\rho_n-a/c)}
 {\sqrt{(\rho_n-a/c)^2+\sin^2\widehat{\zeta}_n}},
 \qquad
 \mathsf h_{\mathrm{opt}}
 :=
 \operatorname{proj}_{J_{b,c}}(\mathsf h_{\mathrm{free}}).
                                                               \acEquationTag{eq:ac-cap-maximizers}
\]
Then \(\mathsf h_{\mathrm{opt}}\) is the unique maximizer in
\eqref{eq:ac-73}, and
\[
 \boxed{\begin{aligned}
 \mathfrak C(a,b,c)
 =-\frac{ab\rho_n}{c}
 +(\rho_n-a/c)\mathsf h_{\mathrm{opt}}
 +\sin\widehat{\zeta}_n
 \sqrt{\nu_{b,c}^2-\mathsf h_{\mathrm{opt}}^2}.
 \end{aligned}}                                                \acEquationTag{eq:ac-cap-explicit}
\]
\end{corollary}

\begin{proof} Put
\[
 g_{a,b,c}(\mathsf h)
 :=(\rho_n-a/c)\mathsf h
 +\sin\widehat{\zeta}_n\sqrt{\nu_{b,c}^2-\mathsf h^2}.
\]
For \(-\nu_{b,c}<\mathsf h<\nu_{b,c}\),
\[
 g_{a,b,c}'(\mathsf h)
 =\rho_n-\frac ac
 -\frac{\sin\widehat{\zeta}_n\,\mathsf h}
 {\sqrt{\nu_{b,c}^2-\mathsf h^2}},
\qquad
 g_{a,b,c}''(\mathsf h)
 =-\frac{\sin\widehat{\zeta}_n\,\nu_{b,c}^2}
 {(\nu_{b,c}^2-\mathsf h^2)^{3/2}}<0.
\]
Moreover,
\[
 \lim_{\mathsf h\downarrow-\nu_{b,c}}g_{a,b,c}'(\mathsf h)=+\infty,
 \qquad
 \lim_{\mathsf h\uparrow\nu_{b,c}}g_{a,b,c}'(\mathsf h)=-\infty.
\]
Thus \(g_{a,b,c}'\) has a unique zero, and solving
\(g_{a,b,c}'(\mathsf h)=0\) gives
\[
 \mathsf h=
 \frac{\nu_{b,c}(\rho_n-a/c)}
 {\sqrt{(\rho_n-a/c)^2+\sin^2\widehat{\zeta}_n}}
 =\mathsf h_{\mathrm{free}}.
\]
It follows that \(g_{a,b,c}\) increases on
\([-\nu_{b,c},\mathsf h_{\mathrm{free}}]\) and decreases on
\([\mathsf h_{\mathrm{free}},\nu_{b,c}]\).  Its unique maximizer on the
nonempty closed subinterval \(J_{b,c}\) is therefore
\(
 \operatorname{proj}_{J_{b,c}}(\mathsf h_{\mathrm{free}})
 =\mathsf h_{\mathrm{opt}}.
\)
Substitution in \eqref{eq:ac-73} gives
\eqref{eq:ac-cap-explicit}.
\end{proof}

\subsection{The upper boundary for distinct anchors:
\texorpdfstring{\(Q\leq-\rho_n\)}{Q <= -rho\_n}}
\label{sec:ac-distinct-upper-boundary}

For \(Q\in[-1,-\rho_n]\), membership in the admissible region
\(\mathcal A_n\) requires \(R\leq U_n(Q)\).  The graph of \(U_n\) over
this interval is the upper curved boundary shown in panel~\textup{(b)}
of \Cref{fig:qr-distortion-region}.  We first prove the stronger affine
estimate
\[
 R\leq1+\rho_n+Q.
\]
The concavity argument in \Cref{lem:ac-far} then shows that this estimate
implies the required upper boundary.  After the dimension reduction of
\Cref{prop:ac-cap-dimension-reduction}, the proof treats the low--low,
high--high, and mixed radial regimes separately.

\begin{proposition}[Stronger affine estimate for distinct anchors]
\label{prop:ac-far-affine}

Assume \(n\geq2\).  In the common distinct-anchor notation of
\Cref{sec:ac-distinct-lower-boundary}, one has
\[
 Q\leq-{\rho_n}\quad\Longrightarrow\quad
 R\leq1+{\rho_n}+Q.                                               \acEquationTag{eq:ac-77}
\]
\end{proposition}

After the common finite-gain reduction, the proof treats the
low--low, high--high, and mixed radial regimes separately.  The
low--low argument uses the boundary-coordinate form of the cap
reduction in \Cref{cor:ac-cap-boundary-coordinate}, the mixed argument
uses \Cref{prop:ac-cap-dimension-reduction}, and the high--high argument
uses inverse gains and convexity.  The supporting results and their
roles are summarized in \Cref{tab:ac-far-proof-structure}.

\begin{table}[ht!]
\centering
\caption{Structure of the upper-boundary argument for distinct anchors.}
\label{tab:ac-far-proof-structure}
\renewcommand{\arraystretch}{1.15}
\begin{tabularx}{\textwidth}{@{}>{\raggedright\arraybackslash}p{0.25\textwidth}
                                  >{\raggedright\arraybackslash}X
                                  >{\raggedright\arraybackslash}p{0.31\textwidth}@{}}
\toprule
Result & Role or regime & Main argument \\
\midrule
\Cref{lem:ac-far-common-reduction}
&
All finite-gain regimes
&
Reduction to \eqref{eq:ac-80} and the estimates
\eqref{eq:ac-81}
\\
\Cref{lem:ac-far-low}
&
Low--low regime
&
Cap dimension reduction and endpoint maximization
\\
\Cref{lem:ac-far-high}
&
High--high regime
&
Inverse gains and convexity
\\
\Cref{lem:ac-mixed-coefficient}
&
Mixed regime: coefficient identity and bounds
&
Inverse-gain factorization and a reflection argument
\\
\Cref{lem:ac-far-mixed-unconstrained}
&
Mixed regime: unconstrained maximizer
&
Algebraic estimate for the unconstrained maximizer
\\
\Cref{lem:ac-far-mixed}
&
Mixed regime
&
Cap dimension reduction and projection to the feasible interval
\\
\bottomrule
\end{tabularx}
\end{table}
\FloatBarrier

\paragraph{Notation for this subsection.}
\phantomsection\label{def:upper-boundary-notation}
For the finite-gain comparisons in the three supporting lemmas below,
that is, for \(\alpha,\beta<\widehat{\zeta}_n\), write
\[
 r_\alpha:=r_n(\alpha),\qquad
 r_\beta:=r_n(\beta),\qquad
 \theta_\alpha:=\vartheta_n(\alpha),\qquad
 \theta_\beta:=\vartheta_n(\beta),
\]
\[
 X_\alpha:=\sqrt{1+2\rho_n r_\alpha+r_\alpha^2},\qquad
 X_\beta:=\sqrt{1+2\rho_n r_\beta+r_\beta^2},\qquad
 \Pi_{\alpha\beta}:=X_\alpha X_\beta,
\]
and
\[
 \omega_{\alpha\beta}
 :=r_\alpha r_\beta
 -\Pi_{\alpha\beta}\sin\alpha\sin\beta.
\]

Write
\[
 \mathrm{Num}:=(v_i+r_\alpha x)\cdot(v_j+r_\beta y),\qquad
 X:=\lVert v_i+r_\alpha x\rVert,\qquad
 Y:=\lVert v_j+r_\beta y\rVert,
\]
for the inner product and norms of the two unnormalized target vectors,
so that \(R=\mathrm{Num}/(XY)\).  We also use
\(B_n^{\mathrm{cap}}\) as computed in
\eqref{eq:ac-cap-bound}.  By the cell bound \eqref{eq:ac-10} and the
definition of the spherical caps,
\(
 x\in\mathcal C_{v_i},\qquad y\in\mathcal C_{v_j}.
\)
Put
\(
 \widetilde{\alpha}:=F_n(\alpha),\qquad
 \widetilde{\beta}:=F_n(\beta).
\)
Whenever \(r_\alpha>0\), define
\[
 \bar r_\alpha:=r_\alpha^{-1}=r_n(\widetilde{\alpha}),\qquad
 \bar X_\alpha
 :=\sqrt{1+2\rho_n\bar r_\alpha+\bar r_\alpha^2}.
\]
Whenever \(r_\beta>0\), define
\[
 \bar r_\beta:=r_\beta^{-1}=r_n(\widetilde{\beta}),\qquad
 \bar X_\beta
 :=\sqrt{1+2\rho_n\bar r_\beta+\bar r_\beta^2}.
\]

The equalities
\(
 \bar r_\alpha=r_n(\widetilde{\alpha}),
 \qquad
 \bar r_\beta=r_n(\widetilde{\beta}),
\)
whenever the corresponding gain is positive, are the two instances of
the reciprocal rule from \Cref{sec:critical-involution} used below.

\begin{lemma}[Common finite-gain reduction]
\label{lem:ac-far-common-reduction}
Assume \(n\geq2\).  Under the hypotheses of
\Cref{prop:ac-far-affine}, suppose that
\(\alpha,\beta<\widehat{\zeta}_n\), and retain the notation introduced
above.  Then
the affine estimate \(R\leq1+\rho_n+Q\) follows from
\[
\boxed{
 r_\alpha(v_j\cdot x)+r_\beta(v_i\cdot y)
 +\omega_{\alpha\beta}(x\cdot y)
 \leq {\rho_n}+\Pi_{\alpha\beta}
 (1+{\rho_n}+\cos\alpha\cos\beta).}                         \acEquationTag{eq:ac-80}
\]
In addition, one has the cross-term bounds
\[
 v_j\cdot x\leq B_n^{\mathrm{cap}},\qquad
 v_i\cdot y\leq B_n^{\mathrm{cap}}.                   \acEquationTag{eq:ac-81}
\]
\end{lemma}

\begin{proof}

Since \(x\in\mathcal C_{v_i}\) and \(y\in\mathcal C_{v_j}\), one has
\(
 v_i\cdot x\geq\rho_n,\qquad v_j\cdot y\geq\rho_n.
\)
Therefore \eqref{eq:chord-norm-lower-bound} in
\Cref{lem:ac-chord-geometry}, together with the definitions of
\(X_\alpha\) and \(X_\beta\), gives
\(
 X=\lVert v_i+r_\alpha x\rVert\geq X_\alpha,\qquad
 Y=\lVert v_j+r_\beta y\rVert\geq X_\beta.
\)
By the notation above,

\(R=\mathrm{Num}/(XY)\).
If \(\mathrm{Num}\leq0\), then \(R\leq0\).  Since \(Q\) is the inner
product of two unit source vectors, \(Q\geq-1\), and hence
\(
 1+\rho_n+Q\geq\rho_n>0.
\)
Thus \eqref{eq:ac-77} holds in this case.
We may therefore assume that \(\mathrm{Num}>0\).  Since
\(XY\geq\Pi_{\alpha\beta}\), it is enough to prove
\(
 \mathrm{Num}\leq\Pi_{\alpha\beta}(1+\rho_n+Q).
\)

Expanding \(\mathrm{Num}\) and \(Q\), and then rearranging, shows that
this last inequality is precisely \eqref{eq:ac-80}.

Since
\(
 x\in\mathcal C_{v_i},\qquad
 y\in\mathcal C_{v_j},\qquad
 v_i\cdot v_j=-\rho_n,
\)
one application of \eqref{eq:ac-cap-bound}, with
\((p,q)=(v_i,v_j)\), gives both
\(
 v_j\cdot x\leq B_n^{\mathrm{cap}},
 \qquad
 v_i\cdot y\leq B_n^{\mathrm{cap}},
\)
which is \eqref{eq:ac-81}.

\end{proof}

Throughout the three regimes below, every maximum over \(x,y\) is
taken over
\(
 x\in\mathcal C_{v_i},\qquad y\in\mathcal C_{v_j}.
\)

The next three lemmas verify \eqref{eq:ac-80} according to the
positions of \(\alpha\) and \(\beta\) relative to
\(\alpha_n^\ast\).

\begin{lemma}[Low--low affine estimate]
\label{lem:ac-far-low}
Assume \(n\geq2\).  Under the hypotheses of
\Cref{prop:ac-far-affine}, retain the common finite-gain notation above.
If \(Q\leq-\rho_n\) and
\(0\leq\alpha,\beta\leq\alpha_n^\ast\),
then
\[
R\leq1+\rho_n+Q.
\]
\end{lemma}

\begin{proof}

By \Cref{lem:ac-far-common-reduction}, it is enough to verify
\eqref{eq:ac-80}.  Put
\(
 \sigma:=v_j\cdot x.
\)
The reduction below expresses the relevant maximum as a one-variable
maximum over \(-1\leq\sigma\leq B_n^{\mathrm{cap}}\); we prove that it
is attained at \(\sigma=B_n^{\mathrm{cap}}\).

\smallskip
\noindent\emph{Sign of \(\omega_{\alpha\beta}\).}
The hypothesis \(Q\leq-\rho_n\) forces
\(\alpha,\beta>0\): if, for example, \(\alpha=0\), then
\(
 Q=p_0(x)\cdot p_\beta(y)=\cos\beta>0,
\)
and the case \(\beta=0\) is symmetric.  Hence
\(r_\alpha,r_\beta>0\).  Because
\(\alpha,\beta\leq\alpha_n^\ast\), the definition of
\(\vartheta_n\) gives
\(
 \theta_\alpha=\Theta_n(\alpha),\qquad
 \theta_\beta=\Theta_n(\beta).
\)

The first two identities in each row follow from
\eqref{eq:chord-geometry} in \Cref{lem:ac-chord-geometry}; the third is
the defining identity for \(\Theta_n\) from
\Cref{sec:anchored-chord-definition}:
\[
\begin{aligned}
 r_\alpha
 &=\frac{\sin\theta_\alpha}
 {\sin(\widehat{\zeta}_n-\theta_\alpha)},&
 X_\alpha
 &=\frac{\sin\widehat{\zeta}_n}
 {\sin(\widehat{\zeta}_n-\theta_\alpha)},&
 \sin\theta_\alpha&=c_n\sin\alpha,\\
 r_\beta
 &=\frac{\sin\theta_\beta}
 {\sin(\widehat{\zeta}_n-\theta_\beta)},&
 X_\beta
 &=\frac{\sin\widehat{\zeta}_n}
 {\sin(\widehat{\zeta}_n-\theta_\beta)},&
 \sin\theta_\beta&=c_n\sin\beta.
\end{aligned}
\]
Since \(c_n^2=(1+\rho_n)/2\) and
\(\sin^2\widehat{\zeta}_n=1-\rho_n^2\), substitution in
\(
 -\omega_{\alpha\beta}
 =\Pi_{\alpha\beta}\sin\alpha\sin\beta-r_\alpha r_\beta
\)
gives the exact factorization
\[
 -\omega_{\alpha\beta}
 =(1-2\rho_n)r_\alpha r_\beta=
 \frac{(1+\rho_n)(1-2\rho_n)\sin\alpha\sin\beta}
 {2\sin(\widehat{\zeta}_n-\Theta_n(\alpha))
   \sin(\widehat{\zeta}_n-\Theta_n(\beta))}>0.
\]
Consequently,
\[
 \frac{-\omega_{\alpha\beta}}{r_\alpha}
 =(1-2\rho_n)r_\beta\leq1-2\rho_n,\qquad
 \frac{-\omega_{\alpha\beta}}{r_\beta}
 =(1-2\rho_n)r_\alpha\leq1-2\rho_n,
 \acEquationTag{eq:ac-far-low-ratio-bounds}
\]
where \(0<r_\alpha,r_\beta\leq1\) follows from
\(\alpha,\beta\leq\alpha_n^\ast\).

\smallskip
\noindent\emph{Reduction to one variable.}
Apply \Cref{cor:ac-cap-boundary-coordinate} with
\[
 (p,q,a,b,c)
 =(v_i,v_j,r_\alpha,r_\beta,-\omega_{\alpha\beta}).
\]
The last coefficient is positive by the preceding factorization.  For
each \(\sigma\in[-1,B_n^{\mathrm{cap}}]\), choose \(x_\sigma\) as in
that corollary and put
\[
 w_\sigma:=r_\beta v_i+\omega_{\alpha\beta}x_\sigma,
 \qquad
 \Psi(\sigma)
 :=\Psi_{r_\alpha,r_\beta,-\omega_{\alpha\beta}}(\sigma)
 =r_\alpha\sigma
 +\mathcal S_{\rho_n}(w_\sigma;v_j).
\]
The corollary shows that \(\Psi\) is independent of the choice of
\(x_\sigma\), is concave on \([-1,B_n^{\mathrm{cap}}]\), and satisfies
\[
 \mathfrak C(r_\alpha,r_\beta,-\omega_{\alpha\beta})
 =\max_{-1\leq\sigma\leq B_n^{\mathrm{cap}}}\Psi(\sigma).
 \acEquationTag{eq:ac-far-low-coordinate-reduction}
\]
Its identities \eqref{eq:ac-cap-boundary-coordinate-data} specialize to
\[
 \lVert w_\sigma\rVert^2
 =r_\beta^2+\omega_{\alpha\beta}^2
 +2\rho_n r_\beta\omega_{\alpha\beta},
 \qquad
 w_\sigma\cdot v_j
 =-\rho_n r_\beta+\omega_{\alpha\beta}\sigma.
 \acEquationTag{eq:ac-far-low-coordinate-data}
\]

\smallskip
\noindent\emph{The endpoint maximizer.}
We now compute the derivative at the right endpoint.  At
\(\sigma=B_n^{\mathrm{cap}}=1-2\rho_n^2\), direct expansion gives
\[
 \lVert w_{B_n^{\mathrm{cap}}}\rVert^2
 -(w_{B_n^{\mathrm{cap}}}\cdot v_j)^2
 =\sin^2\widehat{\zeta}_n\,
   (r_\beta+2\rho_n\omega_{\alpha\beta})^2.
\]
The second inequality in \eqref{eq:ac-far-low-ratio-bounds} gives
\[
 r_\beta+2\rho_n\omega_{\alpha\beta}
 \geq r_\beta\bigl(1-2\rho_n(1-2\rho_n)\bigr)>0.
 \acEquationTag{eq:ac-far-low-xi-positive}
\]
Consequently,
\[
 \sqrt{\lVert w_{B_n^{\mathrm{cap}}}\rVert^2
 -(w_{B_n^{\mathrm{cap}}}\cdot v_j)^2}
 =\sin\widehat{\zeta}_n\,
   (r_\beta+2\rho_n\omega_{\alpha\beta}).
\]
Moreover
\(
 w_{B_n^{\mathrm{cap}}}\cdot v_j
 =-\rho_n r_\beta+\omega_{\alpha\beta}B_n^{\mathrm{cap}}<0,
\)
 so the second case of \eqref{eq:ac-74} applies in a left neighborhood
of \(\sigma=B_n^{\mathrm{cap}}\).

Differentiating that formula and using
\(
 \tfrac{d}{d\sigma}(w_\sigma\cdot v_j)=\omega_{\alpha\beta}
\)
gives the left derivative
\[
\begin{aligned}
 \Psi'_{-}(B_n^{\mathrm{cap}})
 &=r_\alpha+\rho_n\omega_{\alpha\beta}
 -\frac{\sin\widehat{\zeta}_n\,\omega_{\alpha\beta}
 (w_{B_n^{\mathrm{cap}}}\cdot v_j)}
 {\sqrt{\lVert w_{B_n^{\mathrm{cap}}}\rVert^2
 -(w_{B_n^{\mathrm{cap}}}\cdot v_j)^2}}\\
 &=
 \frac{\omega_{\alpha\beta}^2
 [(r_\alpha/(-\omega_{\alpha\beta})-2{\rho_n})
  (r_\beta/(-\omega_{\alpha\beta})-2{\rho_n})-1]}
 {r_\beta+2{\rho_n}\omega_{\alpha\beta}}.
\end{aligned}
\]
The two inequalities in \eqref{eq:ac-far-low-ratio-bounds} imply
\[
 \frac{r_\alpha}{-\omega_{\alpha\beta}}-2\rho_n
 \geq\frac1{1-2\rho_n}-2\rho_n
 =1+\frac{4\rho_n^2}{1-2\rho_n}>1,
 \qquad
 \frac{r_\beta}{-\omega_{\alpha\beta}}-2\rho_n
 \geq1+\frac{4\rho_n^2}{1-2\rho_n}>1.
\]
Thus the numerator and denominator in the last expression for
\(\Psi'_{-}(B_n^{\mathrm{cap}})\) are nonnegative and positive,
respectively, so \(\Psi'_{-}(B_n^{\mathrm{cap}})\geq0\).
For a concave function, every earlier secant slope is at least its
left derivative at the right endpoint.  Hence \(\Psi\) is
nondecreasing on its interval, and its maximum occurs at
\(\sigma=B_n^{\mathrm{cap}}\).

\smallskip
\noindent\emph{Verification of the scalar bound.}

Since the maximum of \(\Psi\) occurs at
\(\sigma=B_n^{\mathrm{cap}}\), and the second case of \eqref{eq:ac-74}
applies there, the endpoint identities above give
\[
\begin{aligned}
 \mathfrak C(r_\alpha,r_\beta,-\omega_{\alpha\beta})
 &=\Psi(B_n^{\mathrm{cap}})=r_\alpha B_n^{\mathrm{cap}}
 +\rho_n\bigl(-\rho_n r_\beta
  +\omega_{\alpha\beta}B_n^{\mathrm{cap}}\bigr)
 +(1-\rho_n^2)
 \bigl(r_\beta+2\rho_n\omega_{\alpha\beta}\bigr)\\
 &=B_n^{\mathrm{cap}}(r_\alpha+r_\beta)
 +\omega_{\alpha\beta}(3\rho_n-4\rho_n^3)
 \leq B_n^{\mathrm{cap}}(r_\alpha+r_\beta).
\end{aligned}                                                \acEquationTag{eq:ac-85}
\]

Here \(3\rho_n-4\rho_n^3=\rho_n(3-4\rho_n^2)>0\) and
\(\omega_{\alpha\beta}<0\), which proves the final inequality in
\eqref{eq:ac-85}.  Furthermore,
\[
 X_\alpha^2-\frac{1+\rho_n}{2}(1+r_\alpha)^2
 =\frac{1-\rho_n}{2}(1-r_\alpha)^2\geq0,
\]
and similarly for \(X_\beta,r_\beta\).  Hence
\(
 \Pi_{\alpha\beta}
 \geq\tfrac{1+\rho_n}{2}(1+r_\alpha)(1+r_\beta).
\)
Substituting this lower bound and
\(B_n^{\mathrm{cap}}=1-2\rho_n^2\), and then expanding, gives
\[
\begin{aligned}
2[{\rho_n}+(1+{\rho_n})\Pi_{\alpha\beta}
  -B_n^{\mathrm{cap}}(r_\alpha+r_\beta)]
&\geq(1-r_\alpha)(1-r_\beta)
 +{\rho_n}\bigl(4+2(r_\alpha+r_\beta)+2r_\alpha r_\beta\bigr)\\
&\quad+\rho_n^2\bigl(1+5(r_\alpha+r_\beta)+r_\alpha r_\beta\bigr)\geq0.
\end{aligned} \acEquationTag{eq:ac-86}
\]

Finally, since \(\cos\alpha\cos\beta\geq0\), the cap comparison above and
\eqref{eq:ac-85}--\eqref{eq:ac-86} give

\[
\begin{aligned}
&r_\alpha(v_j\cdot x)+r_\beta(v_i\cdot y)
 +\omega_{\alpha\beta}(x\cdot y)
 \leq\mathfrak C(r_\alpha,r_\beta,-\omega_{\alpha\beta})\\
&\quad\leq B_n^{\mathrm{cap}}(r_\alpha+r_\beta)
 \leq \rho_n+(1+\rho_n)\Pi_{\alpha\beta}
 \leq
 \rho_n+\Pi_{\alpha\beta}
 (1+\rho_n+\cos\alpha\cos\beta).
\end{aligned}
\]

This proves \eqref{eq:ac-80}, and hence the lemma.
\end{proof}

\begin{lemma}[High--high affine estimate]
\label{lem:ac-far-high}
Assume \(n\geq2\).  Under the hypotheses of
\Cref{prop:ac-far-affine}, retain the common finite-gain notation above.
If
\[
 Q\leq-\rho_n,\qquad
 \alpha_n^\ast\leq\alpha,\beta<\widehat{\zeta}_n,
\]
then \(R\leq1+\rho_n+Q\).
\end{lemma}

\begin{proof}

By \Cref{lem:ac-far-common-reduction}, it is enough to verify
\eqref{eq:ac-80}.  Put
\(
 \Pi_{\alpha\beta}^{\mathrm{HH}}:=\bar X_\alpha\bar X_\beta.
\)

The inverse gains satisfy
\(
 X_\alpha=r_\alpha\bar X_\alpha,\qquad
 X_\beta=r_\beta\bar X_\beta,
\)
and therefore
\[
 \bar r_\alpha\bar r_\beta\omega_{\alpha\beta}
 =1-\Pi_{\alpha\beta}^{\mathrm{HH}}\sin\alpha\sin\beta.
\acEquationTag{eq:ac-far-high-rescaled-coefficient}
\]

These identities give the coefficients obtained when
\eqref{eq:ac-80} is multiplied by
\(\bar r_\alpha\bar r_\beta\).
We separate the proof according to the sign of
\(\omega_{\alpha\beta}\).

\smallskip

\noindent\emph{The case \(\omega_{\alpha\beta}\geq0\).}
Since \(\bar r_\alpha\bar r_\beta>0\),
\eqref{eq:ac-far-high-rescaled-coefficient} gives
\(
 1-\Pi_{\alpha\beta}^{\mathrm{HH}}
 \sin\alpha\sin\beta\geq0.
\)
Equations \eqref{eq:ac-81} and \(x\cdot y\leq1\)
reduce \eqref{eq:ac-80}, after multiplication by
\(\bar r_\alpha\bar r_\beta\), to
\[
 B_n^{\mathrm{cap}}(\bar r_\alpha+\bar r_\beta)+1
 \leq {\rho_n}\bar r_\alpha\bar r_\beta
 +\Pi_{\alpha\beta}^{\mathrm{HH}}
  (1+{\rho_n}+\cos(\alpha-\beta)).                              \acEquationTag{eq:ac-88}
\]

Since \(\alpha,\beta\in[\alpha_n^\ast,\widehat{\zeta}_n)\),
\[
\begin{aligned}
 0\leq|\alpha-\beta|
 &\leq\widehat{\zeta}_n-\alpha_n^\ast,\\
 \cos(\alpha-\beta)
 &\geq\cos(\widehat{\zeta}_n-\alpha_n^\ast)\\
 &=\rho_n\cos\alpha_n^\ast
   +\sin\widehat{\zeta}_n\sin\alpha_n^\ast\\
 &=1-\rho_n+\rho_n\sqrt{\frac{2\rho_n}{1+\rho_n}}
 \geq1-\rho_n.
\end{aligned}
\]

The following elementary estimates depend only on
\(\bar r_\alpha,\bar r_\beta\geq0\), and therefore hold independently
of the sign of \(\omega_{\alpha\beta}\):
\[
 \Pi_{\alpha\beta}^{\mathrm{HH}}
 \geq\sqrt{(1+\bar r_\alpha^2)(1+\bar r_\beta^2)}
 \geq\bar r_\alpha+\bar r_\beta
 \acEquationTag{eq:ac-far-high-norm-comparison}
\]
They give
\[
 1+B_n^{\mathrm{cap}}(\bar r_\alpha+\bar r_\beta)
 \leq1+\bar r_\alpha+\bar r_\beta
 \leq2\Pi_{\alpha\beta}^{\mathrm{HH}}.
\]
The preceding bound
\(\cos(\alpha-\beta)\geq1-\rho_n\) also gives
\[
 \rho_n\bar r_\alpha\bar r_\beta
 +\Pi_{\alpha\beta}^{\mathrm{HH}}
 (1+\rho_n+\cos(\alpha-\beta))
 \geq2\Pi_{\alpha\beta}^{\mathrm{HH}}.
\]
Together these inequalities prove \eqref{eq:ac-88}.

\smallskip

\noindent\emph{The case \(\omega_{\alpha\beta}<0\).}
It remains to suppose that \(\omega_{\alpha\beta}<0\).  By
\eqref{eq:ac-far-high-rescaled-coefficient}, put
\[
 h_{\alpha\beta}
 :=\Pi_{\alpha\beta}^{\mathrm{HH}}\sin\alpha\sin\beta-1
 =-\bar r_\alpha\bar r_\beta\omega_{\alpha\beta}>0.
\]
Since \(\sin\alpha\sin\beta\leq1\),
\(
 h_{\alpha\beta}\leq
 \Pi_{\alpha\beta}^{\mathrm{HH}}-1
 =:h_{\alpha\beta}^{\max}.
\)
In particular, \(h_{\alpha\beta}^{\max}>0\), as required when it is
used as the positive coefficient \(c\) in
\Cref{prop:ac-cap-dimension-reduction}.

After multiplication by \(\bar r_\alpha\bar r_\beta\), the left-hand
side of \eqref{eq:ac-80} is
\(
 \bar r_\beta(v_j\cdot x)+\bar r_\alpha(v_i\cdot y)
 -h_{\alpha\beta}(x\cdot y).
\)
For
\(0\leq h\leq h_{\alpha\beta}^{\max}\),
\[
 \mathfrak F_{\alpha\beta}^{\mathrm{HH}}(h)
 :=
 \max_{\substack{x\in\mathcal C_{v_i}\\
 y\in\mathcal C_{v_j}}}
 \{\bar r_\beta(v_j\cdot x)+\bar r_\alpha(v_i\cdot y)
 -h(x\cdot y)\}.
\]
\(
 \mathfrak F_{\alpha\beta}^{\mathrm{HH}}(h)
 =\mathfrak C(\bar r_\beta,\bar r_\alpha,h)
 \quad\text{for \(p=v_i\) and \(q=v_j\)}.
\)
For each fixed \((x,y)\), the expression being maximized is affine in
\(h\).  Hence
\(\mathfrak F_{\alpha\beta}^{\mathrm{HH}}\) is convex on
\([0,h_{\alpha\beta}^{\max}]\).  We bound it at the two endpoints of
this interval and then apply convexity at \(h=h_{\alpha\beta}\).

At \(h=0\), \eqref{eq:ac-cap-bound} and
\eqref{eq:ac-far-high-norm-comparison} give
\[
 \mathfrak F_{\alpha\beta}^{\mathrm{HH}}(0)
 =B_n^{\mathrm{cap}}(\bar r_\alpha+\bar r_\beta)
 \leq\bar r_\alpha+\bar r_\beta
 \leq \Pi_{\alpha\beta}^{\mathrm{HH}}.
\]

At \(h_{\alpha\beta}^{\max}\), define the positive numbers
\[
\begin{aligned}
 A_\alpha^{\mathrm{HH}}
 &:=
 \sqrt{\bar r_\alpha^2+(h_{\alpha\beta}^{\max})^2
 -2\rho_n\bar r_\alpha h_{\alpha\beta}^{\max}},\\
 A_\beta^{\mathrm{HH}}
 &:=
 \sqrt{\bar r_\beta^2+(h_{\alpha\beta}^{\max})^2
 -2\rho_n\bar r_\beta h_{\alpha\beta}^{\max}}.
\end{aligned}
\]

Whenever
\(v_i\cdot x=v_j\cdot y=\rho_n\),
\[
 A_\alpha^{\mathrm{HH}}
 =\lVert\bar r_\alpha v_i-h_{\alpha\beta}^{\max}x\rVert,
 \qquad
 A_\beta^{\mathrm{HH}}
 =\lVert\bar r_\beta v_j-h_{\alpha\beta}^{\max}y\rVert.
\]
Taking \(p=v_i\), \(q=v_j\),
\(a=\bar r_\beta\), \(b=\bar r_\alpha\), and
\(c=h_{\alpha\beta}^{\max}\), the maximum
\(\mathfrak F_{\alpha\beta}^{\mathrm{HH}}
(h_{\alpha\beta}^{\max})\) is precisely
\(\mathfrak C(a,b,c)\), and
\(
 \nu_{a,c}=A_\beta^{\mathrm{HH}},
 \qquad
 \nu_{b,c}=A_\alpha^{\mathrm{HH}}.
\)
Therefore \Cref{cor:ac-cap-unconstrained-estimate} gives
\[
 \mathfrak F_{\alpha\beta}^{\mathrm{HH}}(h_{\alpha\beta}^{\max})
 \leq\frac{A_\alpha^{\mathrm{HH}}A_\beta^{\mathrm{HH}}
 -{\rho_n}\bar r_\alpha\bar r_\beta}{h_{\alpha\beta}^{\max}}.
 \acEquationTag{eq:ac-far-high-cap-estimate}
\]

Direct expansion gives
\[
\begin{aligned}
 \Pi_{\alpha\beta}^{\mathrm{HH}}\left(
 2h_{\alpha\beta}^{\max}
 -\frac{(A_\alpha^{\mathrm{HH}})^2}{\bar X_\alpha^2}
 -\frac{(A_\beta^{\mathrm{HH}})^2}{\bar X_\beta^2}
 \right)
 ={}&(\bar X_\alpha-\bar X_\beta)^2
 (2-\Pi_{\alpha\beta}^{\mathrm{HH}})\\
 &+2{\rho_n}(\bar X_\beta^2\bar r_\alpha
 +\bar X_\alpha^2\bar r_\beta).
\end{aligned} \acEquationTag{eq:ac-91}
\]

We now show that the right-hand side of \eqref{eq:ac-91} is
nonnegative.

If \(\Pi_{\alpha\beta}^{\mathrm{HH}}\leq2\), both terms on the
right-hand side of \eqref{eq:ac-91} are nonnegative.

Suppose now that \(\Pi_{\alpha\beta}^{\mathrm{HH}}>2\).
Since \(\alpha,\beta\geq\alpha_n^\ast\), one has
\(0<\bar r_\alpha,\bar r_\beta\leq1\), and hence
\(\Pi_{\alpha\beta}^{\mathrm{HH}}\leq2+2{\rho_n}\).  Moreover,
\[
 \frac{d}{dt}\sqrt{1+2\rho_nt+t^2}
 =\frac{t+\rho_n}{\sqrt{1+2\rho_nt+t^2}}\leq1
 \qquad(0\leq t\leq1),
\]
because
\(
 1+2\rho_nt+t^2-(t+\rho_n)^2=1-\rho_n^2>0.
\)
Thus \(t\mapsto\sqrt{1+2\rho_nt+t^2}\) is one-Lipschitz on
\([0,1]\), and
\(
 |\bar X_\alpha-\bar X_\beta|
 \leq|\bar r_\alpha-\bar r_\beta|.
\)
By symmetry, assume \(\bar r_\alpha\geq\bar r_\beta\).  Then
\[
 (\bar X_\alpha-\bar X_\beta)^2
 \leq(\bar r_\alpha-\bar r_\beta)^2
 \leq\bar r_\alpha^2\leq\bar r_\alpha
 \leq\bar X_\beta^2\bar r_\alpha
 +\bar X_\alpha^2\bar r_\beta.
\]
Since
\(-2{\rho_n}\leq2-\Pi_{\alpha\beta}^{\mathrm{HH}}<0\), the right-hand
side of \eqref{eq:ac-91} is at least
\[
 -2{\rho_n}(\bar X_\alpha-\bar X_\beta)^2
 +2{\rho_n}(\bar X_\beta^2\bar r_\alpha
 +\bar X_\alpha^2\bar r_\beta)\geq0.
\]
Thus the right-hand side of \eqref{eq:ac-91} is nonnegative in all
cases.  Since
\(\Pi_{\alpha\beta}^{\mathrm{HH}}
=\bar X_\alpha\bar X_\beta>0\),
\eqref{eq:ac-91} gives
\[
 \frac{(A_\alpha^{\mathrm{HH}})^2}{\bar X_\alpha^2}
 +\frac{(A_\beta^{\mathrm{HH}})^2}{\bar X_\beta^2}
 \leq2h_{\alpha\beta}^{\max}.
\]
Applying \(2ab\leq a^2+b^2\) directly to the two normalized ratios
gives
\[
 \frac{2A_\alpha^{\mathrm{HH}}A_\beta^{\mathrm{HH}}}
 {\Pi_{\alpha\beta}^{\mathrm{HH}}}
 =2\left(\frac{A_\alpha^{\mathrm{HH}}}{\bar X_\alpha}\right)
     \left(\frac{A_\beta^{\mathrm{HH}}}{\bar X_\beta}\right)
 \leq
 \frac{(A_\alpha^{\mathrm{HH}})^2}{\bar X_\alpha^2}
 +\frac{(A_\beta^{\mathrm{HH}})^2}{\bar X_\beta^2}
 \leq2h_{\alpha\beta}^{\max}.
\]
Thus
\(
 A_\alpha^{\mathrm{HH}}A_\beta^{\mathrm{HH}}
 \leq\Pi_{\alpha\beta}^{\mathrm{HH}}h_{\alpha\beta}^{\max}.
\)

Combining \eqref{eq:ac-far-high-cap-estimate} with the preceding
product bound gives
\(
 \mathfrak F_{\alpha\beta}^{\mathrm{HH}}
 (h_{\alpha\beta}^{\max})
 \leq\Pi_{\alpha\beta}^{\mathrm{HH}}.
\)
Since \(h_{\alpha\beta}\in[0,h_{\alpha\beta}^{\max}]\), convexity and
the bounds at \(h=0\) and \(h=h_{\alpha\beta}^{\max}\) imply
\(
 \mathfrak F_{\alpha\beta}^{\mathrm{HH}}(h_{\alpha\beta})
 \leq\Pi_{\alpha\beta}^{\mathrm{HH}}.
\)
Since \(\rho_n>0\) and
\(\cos\alpha\cos\beta\geq0\), it follows that
\[
\begin{aligned}
 {\rho_n}\bar r_\alpha\bar r_\beta
 +\Pi_{\alpha\beta}^{\mathrm{HH}}
 (1+{\rho_n}+\cos\alpha\cos\beta)
 &\geq\Pi_{\alpha\beta}^{\mathrm{HH}}\\
 &\geq\mathfrak F_{\alpha\beta}^{\mathrm{HH}}(h_{\alpha\beta})\\
 &\geq\bar r_\beta(v_j\cdot x)+\bar r_\alpha(v_i\cdot y)
 -h_{\alpha\beta}(x\cdot y).
\end{aligned}
\]
This proves \eqref{eq:ac-80} when
\(\omega_{\alpha\beta}<0\), and hence completes the high--high
regime.
\end{proof}

\paragraph{Notation for mixed radial levels.}
\phantomsection\label{def:mixed-upper-notation}

Whenever
\(
 0<\alpha\leq\alpha_n^\ast\leq\beta<\widehat{\zeta}_n,
\)
retain the standing definitions of
\(r_\alpha,\bar r_\beta,X_\alpha,\bar X_\beta\), and
\(\widetilde\beta=F_n(\beta)\), and put
\[
\begin{aligned}
 \Pi_{\alpha\beta}^{\mathrm{LH}}
 &:=X_\alpha\bar X_\beta,&
 \kappa_\beta^{\mathrm{LH}}
 &:=\sqrt{2(1-\rho_n)}\,\bar X_\beta\sin\beta-1,&
 \mu_{\alpha\beta}
 &:=r_\alpha\kappa_\beta^{\mathrm{LH}},
\end{aligned}
\]

\begin{lemma}[The mixed-level coefficient identity and bounds]
\label{lem:ac-mixed-coefficient}
Assume \(n\geq2\) and
\(0<\alpha\leq\alpha_n^\ast\leq\beta<\widehat{\zeta}_n\).
With the notation above, the following statements hold.

\begin{enumerate}[label=\textup{(\roman*)},leftmargin=*,itemsep=2pt]
\item The coefficient satisfies
\[
 \Pi_{\alpha\beta}^{\mathrm{LH}}\sin\alpha\sin\beta-r_\alpha
 =r_\alpha\kappa_\beta^{\mathrm{LH}}
 =\mu_{\alpha\beta}.
\]
\item One has
\[
 0<\kappa_\beta^{\mathrm{LH}}\leq1-2\rho_n.                \acEquationTag{eq:ac-94}
\]
\end{enumerate}
\end{lemma}

\begin{proof}
Fix distinct indices \(i,j\).
The sine-ratio formulas in \Cref{lem:ac-chord-geometry} and
\(\sin\Theta_n(\alpha)=c_n\sin\alpha\) give
\[
 \frac{X_\alpha\sin\alpha}{r_\alpha}
 =\frac{\sin\widehat{\zeta}_n\sin\alpha}
        {\sin\Theta_n(\alpha)}
 =\frac{\sin\widehat{\zeta}_n}{c_n}
 =\sqrt{2(1-\rho_n)}.
\]
Multiplying by \(\bar X_\beta\sin\beta\) and using the definition of
\(\kappa_\beta^{\mathrm{LH}}\) proves assertion \textup{(i)}.

To prove assertion \textup{(ii)}, first note that
\[
 \bar r_\beta
 =\frac{\sin\Theta_n(\widetilde{\beta})}
 {\sin(\widehat{\zeta}_n-\Theta_n(\widetilde{\beta}))},
 \qquad
 \bar X_\beta=\frac {\sin \widehat{\zeta}_n}
 {\sin(\widehat{\zeta}_n-\Theta_n(\widetilde{\beta}))}.
\]
Recall that \(v_i\cdot v_j=-\rho_n\).
The second branch of \eqref{eq:ac-5} gives
\(
 \vartheta_n(\beta)
 =\widehat{\zeta}_n-\Theta_n(\widetilde{\beta}).
\)
Consequently, \Cref{prop:ac-maximal-displacement} gives
\(
 \widehat{\zeta}_n-\Theta_n(\widetilde{\beta})
 =\vartheta_n(\beta)\leq\beta.
\)
Both angles belong to \([0,\tfrac{\pi}{2}]\), so the monotonicity of sine gives
\(
 \sin(\widehat{\zeta}_n-\Theta_n(\widetilde{\beta}))\leq\sin\beta.
\)
Therefore
\[
 \bar X_\beta\sin\beta
 =\frac{\sin\widehat{\zeta}_n\sin\beta}
 {\sin(\widehat{\zeta}_n-\Theta_n(\widetilde{\beta}))}
 \geq\sin\widehat{\zeta}_n.
\]

By the definition of \(\kappa_\beta^{\mathrm{LH}}\), this proves the
lower bound in the two-sided estimate below.

It remains to prove
\[
 \kappa_\beta^{\mathrm{LH}}\leq1-2\rho_n.
\]
For this, let
\[
 e_{ij}:=\frac{v_i-v_j}{\lVert v_i-v_j\rVert},
 \qquad
 \mathcal E_{ij}:=\Sp^{n+1}\cap e_{ij}^\perp.
\]
Since \(v_i\cdot v_j=-\rho_n\), one has
\[
 \lVert v_i-v_j\rVert=\sqrt{2(1+\rho_n)}=2c_n,
 \qquad
 v_i\cdot e_{ij}=c_n,
 \qquad
 v_j\cdot e_{ij}=-c_n.
\]

Reflection in \(e_{ij}^{\perp}\) fixes
\(\mathcal E_{ij}\) pointwise.  Moreover, it interchanges the anchors
\(v_i\) and \(v_j\), since
\[
 v_i-2(v_i\cdot e_{ij})e_{ij}=v_j.
\]
Taking inner products with \(e_{ij}\) gives
\[
 p_{\widetilde{\beta}}(v_i)\cdot e_{ij}
 =c_n\sin\widetilde{\beta}
 =\sin\Theta_n(\widetilde{\beta}),
 \qquad
 p_\beta(v_j)\cdot e_{ij}
 =-c_n\sin\beta=-\sin\Theta_n(\beta).
\]
Consequently \(p_{\widetilde{\beta}}(v_i)\) and
\(p_\beta(v_j)\) lie on opposite sides of \(\mathcal E_{ij}\), and
their geodesic distances to \(\mathcal E_{ij}\) are respectively
\(\Theta_n(\widetilde{\beta})\) and \(\Theta_n(\beta)\).
Since \(\widetilde{\beta}:=F_n(\beta)\), the defining equation
\eqref{eq:ac-2} gives
\[
 p_{\widetilde{\beta}}(v_i)\cdot p_\beta(v_j)
 =\cos\widetilde{\beta}\cos\beta
 +(v_i\cdot v_j)\sin\widetilde{\beta}\sin\beta
 =\cos\widetilde{\beta}\cos\beta
 -\rho_n\sin\widetilde{\beta}\sin\beta
 =\rho_n.
\]
Since \(\widehat{\zeta}_n=\arccos\rho_n\), it follows that
\(
 d_{n+1}\bigl(p_{\widetilde{\beta}}(v_i),p_\beta(v_j)\bigr)
 =\widehat{\zeta}_n.
\)
Since the two points lie on opposite sides of \(\mathcal E_{ij}\) and
their distance is \(\widehat{\zeta}_n<\pi\), their unique minimizing
geodesic meets \(\mathcal E_{ij}\) at a unique point \(m\).  Then
\[
 \widehat{\zeta}_n
 =d_{n+1}\bigl(p_{\widetilde{\beta}}(v_i),p_\beta(v_j)\bigr)
 =d_{n+1}\bigl(p_{\widetilde{\beta}}(v_i),m\bigr)
 +d_{n+1}\bigl(m,p_\beta(v_j)\bigr)
 \geq\Theta_n(\widetilde{\beta})+\Theta_n(\beta).
\]

\begin{figure}[!ht]
    \centering
    \resizebox{1.00\textwidth}{!}{\colorlet{anchorone}{blue!70!black}
\colorlet{anchortwo}{orange!85!black}
\colorlet{construction}{teal!70!black}
\colorlet{resultcolor}{purple!70!black}
\colorlet{auxiliary}{black!35}
\colorlet{outline}{black!70}

\tikzset{
  >={Latex[length=2mm]},
  line cap=round,
  line join=round,
  every node/.style={font=\small},
  point/.style={circle,fill,inner sep=0pt,minimum size=4.6pt},
  comparison/.style={densely dashed,black!55,line width=.8pt},
  proofbox/.style={
    draw=black!28,
    rounded corners=2pt,
    fill=black!2,
    inner sep=6pt,
    align=center,
    text width=4.15cm
  },
  endpointbox/.style={
    draw=construction!55!black,
    rounded corners=2pt,
    fill=construction!6,
    inner sep=6pt,
    align=center,
    text width=4.25cm
  },
  implication/.style={-{Latex[length=2mm]},black!55,semithick}
}

\begin{tikzpicture}[every node/.style={font=\small}]

\begin{scope}[xshift=-4.50cm,yshift=.80cm]
  \draw[outline,semithick] (0,0) circle[radius=2.05cm];
  \draw[construction,very thick] (-1.98,0)--(1.98,0);
  \node[construction,fill=white,inner sep=1.5pt] at (-.82,.24)
    {\(\mathcal E_{ij}\)};

  \coordinate (vi) at (.62,1.18);
  \coordinate (vj) at (.62,-1.18);
  \fill[anchorone] (vi) circle[radius=2.4pt]
    node[above right=2pt] {\(v_i\)};
  \fill[anchortwo] (vj) circle[radius=2.4pt]
    node[below right=2pt] {\(v_j\)};
  \draw[comparison] (vi)--(vj);
  \draw[{Latex[length=1.8mm]}-{Latex[length=1.8mm]},
        resultcolor,semithick]
    (1.03,.86) to[bend left=34] (1.03,-.86)
   ;

\node[font=\footnotesize,align=center] at (0,3.1)
    {\textup{(a)} \(v_i\) and \(v_j\) are reflection images};
\end{scope}

\begin{scope}[xshift=3.65cm,yshift=.80cm]
  \def\sphereradius{2.38}
  \def\equatorratio{0.317999364002}
  \pgfmathsetmacro{\equatorheight}{\sphereradius*\equatorratio}

\def\Ax{-0.704323380585}
  \def\Ay{ 0.521791075212}
  \def\Bx{-0.223606797750}
  \def\By{-0.457081064741}
  \def\Mx{-0.651338604685}
  \def\My{ 0.241293848514}
  \def\Azerox{-0.749417332898}
  \def\Azeroy{ 0.210546718784}
  \def\Bzerox{-0.316227766017}
  \def\Bzeroy{ 0.301680685419}

\def\Tx{ 0.011845094767}
  \def\Ty{-0.669288684468}
  \def\geodesicangle{70.528779365509}
  \def\crossingangle{21.422174014640}

\def\screeney{0.948090926280}
  \def\thetaA{19.977311514707}
  \def\thetaB{45}

\path[fill=black!1] (0,0) circle[radius=\sphereradius cm];
  \draw[outline,semithick] (0,0) circle[radius=\sphereradius cm];

\draw[construction!55,densely dashed]
    (-\sphereradius,0)
    arc[
      start angle=180,
      end angle=360,
      x radius=\sphereradius cm,
      y radius=\equatorheight cm
    ];
  \draw[construction,very thick]
    (\sphereradius,0)
    arc[
      start angle=0,
      end angle=180,
      x radius=\sphereradius cm,
      y radius=\equatorheight cm
    ];

\draw[resultcolor!20,densely dashed,line width=.55pt]
    plot[
      domain=147.079:327.079,
      samples=100,
      smooth,
      variable=\g
    ]
    ({
      \sphereradius*(\Ax*cos(\g)+\Tx*sin(\g))
     },{
      \sphereradius*(\Ay*cos(\g)+\Ty*sin(\g))
     });
  \draw[resultcolor!24,line width=.55pt]
    plot[
      domain=-32.921:147.079,
      samples=100,
      smooth,
      variable=\g
    ]
    ({
      \sphereradius*(\Ax*cos(\g)+\Tx*sin(\g))
     },{
      \sphereradius*(\Ay*cos(\g)+\Ty*sin(\g))
     });

\draw[resultcolor,very thick]
    plot[
      domain=0:\crossingangle,
      samples=45,
      smooth,
      variable=\g
    ]
    ({
      \sphereradius*(\Ax*cos(\g)+\Tx*sin(\g))
     },{
      \sphereradius*(\Ay*cos(\g)+\Ty*sin(\g))
     })
    plot[
      domain=\crossingangle:\geodesicangle,
      samples=65,
      smooth,
      variable=\g
    ]
    ({
      \sphereradius*(\Ax*cos(\g)+\Tx*sin(\g))
     },{
      \sphereradius*(\Ay*cos(\g)+\Ty*sin(\g))
     });

\draw[anchorone,densely dashed,semithick]
    plot[
      domain=0:\thetaA,
      samples=35,
      smooth,
      variable=\g
    ]
    ({
      \sphereradius*\Azerox*cos(\g)
     },{
      \sphereradius*
      (\Azeroy*cos(\g)+\screeney*sin(\g))
     });
  \draw[anchortwo,densely dashed,semithick]
    plot[
      domain=0:\thetaB,
      samples=50,
      smooth,
      variable=\g
    ]
    ({
      \sphereradius*\Bzerox*cos(\g)
     },{
      \sphereradius*
      (\Bzeroy*cos(\g)-\screeney*sin(\g))
     });

  \coordinate (A) at
    ({\sphereradius*\Ax},{\sphereradius*\Ay});
  \coordinate (B) at
    ({\sphereradius*\Bx},{\sphereradius*\By});
  \coordinate (m) at
    ({\sphereradius*\Mx},{\sphereradius*\My});
  \coordinate (A0) at
    ({\sphereradius*\Azerox},{\sphereradius*\Azeroy});
  \coordinate (B0) at
    ({\sphereradius*\Bzerox},{\sphereradius*\Bzeroy});

  \fill[anchorone] (A) circle[radius=2.5pt]
    node[above left=4pt] {\(A\)};
  \fill[anchortwo] (B) circle[radius=2.5pt]
    node[below right=2pt,fill=white,
        inner sep=1pt] {\(B\)};
  \fill[resultcolor] (m) circle[radius=2.2pt]
    node[below right=2pt,fill=white,
        inner sep=1pt] {\(m\in\mathcal E_{ij}\)};
  \fill[anchorone!75] (A0) circle[radius=1.8pt]
    node[below left=2pt,fill=white,
        inner sep=1pt] {\(A_0\)};
  \fill[anchortwo!85] (B0) circle[radius=1.8pt]
    node[above right=1pt] {\(B_0\)};

  \node[anchorone,font=\footnotesize,anchor=east,fill=white,
        inner sep=1pt]
    at (-1.86,.93) {\(\Theta_n(\widetilde\beta)\)};
  \node[anchortwo,font=\footnotesize,anchor=west]
    at (-.42,-.28) {\(\Theta_n(\beta)\)};
  \node[
    construction,
    fill=white,
    inner sep=1.5pt,
    font=\footnotesize
  ] at (1.25,-.88)
    {\(\mathcal E_{ij}\)};
  \node[
    resultcolor,
    fill=white,
    inner sep=1pt,
    font=\scriptsize
  ] at (1.28,1.55)
    {supporting great circle};

  \node[
    draw=black!25,
    fill=white,
    rounded corners=2pt,
    inner sep=5pt,
    align=center,
    text width=7.0cm
  ] at (0,-3.65)
  {\textcolor{anchorone}{\(A=p_{\widetilde\beta}(v_i)\), \quad
     \(A\cdot e_{ij}=\sin\Theta_n(\widetilde\beta)>0\)}
   \\[2pt]
   \textcolor{anchortwo}{\(B=p_\beta(v_j)\), \quad
     \(B\cdot e_{ij}=-\sin\Theta_n(\beta)<0\)}
      \\[2pt]
      {\(\displaystyle
     \widehat{\zeta}_n=d(A,B)
     =d(A,m)+d(m,B)\)\\[-1pt]
   \(\displaystyle
     \geq\Theta_n(\widetilde\beta)+\Theta_n(\beta).\)}
    };

\node[font=\footnotesize,align=center] at (0,3.10)
    {\textup{(b)} the crossing argument involving \(\mathcal E_{ij}\)};
\end{scope}

\end{tikzpicture} }
    \caption{In panel \textup{(a)}, reflection fixes the great hypersphere \(\mathcal E_{ij}\) pointwise and interchanges the anchors \(v_i\) and \(v_j\), where \(\lVert v_i-v_j\rVert=2c_n\). In panel \textup{(b)}, the points \(A=p_{\widetilde\beta}(v_i)\) and \(B=p_\beta(v_j)\) lie on opposite sides of \(\mathcal E_{ij}\), so their minimizing geodesic meets \(\mathcal E_{ij}\) at a crossing point \(m\). Let \(A_0\) and \(B_0\) be respective nearest points of \(A\) and \(B\) on \(\mathcal E_{ij}\). Then \(d_{n+1}(A,m)\geq d_{n+1}(A,A_0) =\Theta_n(\widetilde\beta)\) and \(d_{n+1}(B,m)\geq d_{n+1}(B,B_0) =\Theta_n(\beta)\). Hence \(d_{n+1}(A,B)\geq \Theta_n(\widetilde\beta)+\Theta_n(\beta)\).
}
    \label{fig:reflection_geometry}
\end{figure}

Using \(\sin\Theta_n(\beta)=c_n\sin\beta\),
\(\sin\widehat{\zeta}_n/c_n=\sqrt{2(1-\rho_n)}\), the monotonicity of
sine on \([0,\widehat{\zeta}_n]\subset[0,\tfrac{\pi}{2}]\), and
\(
\widehat{\zeta}_n\geq
\Theta_n(\widetilde{\beta})+\Theta_n(\beta),
\)
we obtain
\[
 \bar X_\beta\sin\beta
 =\frac{\sin\widehat{\zeta}_n\sin\beta}
         {\sin(\widehat{\zeta}_n-\Theta_n(\widetilde{\beta}))}
 =\sqrt{2(1-\rho_n)}\,
   \frac{\sin\Theta_n(\beta)}
        {\sin(\widehat{\zeta}_n-\Theta_n(\widetilde{\beta}))}
 \leq\sqrt{2(1-\rho_n)}.
\]
Together with the lower bound above, this gives
\[
 \sqrt{2(1-\rho_n)}\sin\widehat{\zeta}_n-1
 \leq\kappa_\beta^{\mathrm{LH}}\leq1-2\rho_n.
\]
The lower bound is positive.  Indeed, the function
\(
 f(\rho):=2(1-\rho)(1-\rho^2)
\)
is decreasing on \((0,1/3]\), because
\(
 f'(\rho)=-2(1+2\rho-3\rho^2)<0,
\)
and
\(
 f(\tfrac{1}{3})=\tfrac{32}{27}>1.
\)
Thus \(2(1-\rho_n)(1-\rho_n^2)>1\), completing the proof.
\end{proof}

The case in which the constrained one-variable maximizer equals the
unconstrained maximizer will require the following algebraic estimate.

\begin{lemma}[Scalar estimate for the unconstrained maximizer]
\label{lem:ac-far-mixed-unconstrained}

Let \(\rho,r,t,\kappa\) be real numbers satisfying
\[
 0<\rho\leq\frac13,\qquad
 0<r\leq1,\qquad
 t>0,\qquad
 0<\kappa\leq1-2\rho,
\]
and put
\(
 \mu:=r\kappa.
\)

Then
\[
 1-2\rho\mu>0.                                           \acEquationTag{eq:ac-mixed-positive-factor}
\]
Moreover, if
\[
 \frac{
 \sqrt{1+\mu^2-2\rho\mu}\,(t-\rho\kappa)}
 {\sqrt{t^2+\kappa^2-2\rho t\kappa}}
 \leq
 \rho+\mu(1-2\rho^2),                                  \acEquationTag{eq:ac-mixed-left-endpoint}
\]
then
\[
 \sqrt{1+\mu^2-2\rho\mu}\,
 \sqrt{t^2+\kappa^2-2\rho t\kappa}
 \leq\kappa\sqrt{(1+2\rho r+r^2)(1+2\rho t+t^2)}+\rho t.
 \acEquationTag{eq:ac-mixed-free-bound}
\]
\end{lemma}

The proof is given in
Appendix~\ref{app:ac-unconstrained-maximizer-estimate}.

\begin{lemma}[Mixed-level affine estimate]
\label{lem:ac-far-mixed}
Assume \(n\geq2\).  Under the hypotheses of
\Cref{prop:ac-far-affine}, retain the common finite-gain notation above.
Suppose
\[
 Q\leq-\rho_n,\qquad
 \alpha,\beta<\widehat{\zeta}_n,\qquad
 \min\{\alpha,\beta\}\leq\alpha_n^\ast
 \leq\max\{\alpha,\beta\}.
\]

Then \(R\leq1+\rho_n+Q\).
\end{lemma}

\begin{proof}
By \Cref{lem:ac-far-common-reduction}, it is enough to verify
\eqref{eq:ac-80}.  By symmetry, suppose
\(\alpha\leq\alpha_n^\ast\leq\beta\).

\smallskip
\noindent\emph{The coefficient after inverse-gain rescaling.}

The inequality \(Q\leq-\rho_n\) rules out \(\alpha=0\).
Thus the radial levels satisfy
\(
 0<\alpha\leq\alpha_n^\ast\leq\beta<\widehat{\zeta}_n.
\)

Use the notation for mixed radial levels introduced immediately before
\Cref{lem:ac-mixed-coefficient}.  That lemma gives

\[
 \Pi_{\alpha\beta}^{\mathrm{LH}}\sin\alpha\sin\beta-r_\alpha
 =r_\alpha\kappa_\beta^{\mathrm{LH}}
 =\mu_{\alpha\beta},
 \qquad
 0<\kappa_\beta^{\mathrm{LH}}\leq1-2\rho_n.
\]

The inverse-gain quantities satisfy
\[
 X_\beta=r_\beta\bar X_\beta,\qquad
 \bar r_\beta\Pi_{\alpha\beta}
 =\Pi_{\alpha\beta}^{\mathrm{LH}},\qquad
 \bar r_\beta\omega_{\alpha\beta}
 =-r_\alpha\kappa_\beta^{\mathrm{LH}}.
\]

Indeed, the first identity follows from
\(\bar r_\beta=r_\beta^{-1}\) and the positivity of both sides; the
second follows from the definitions
of the two normalization products.

The third follows from the definition of
\(\omega_{\alpha\beta}\) in the
\hyperref[def:upper-boundary-notation]{standing notation for this
subsection}, the second identity above, and
\Cref{lem:ac-mixed-coefficient}\textup{(i)}.
Consequently, after multiplying \eqref{eq:ac-80} by
\(\bar r_\beta>0\), it is enough to prove

\[
 \mathfrak F_{\alpha\beta}^{\mathrm{LH}}:=
 \max_{\substack{x\in\mathcal C_{v_i}\\
                  y\in\mathcal C_{v_j}}}
\left\{
 r_\alpha\bar r_\beta(v_j\cdot x)+(v_i\cdot y)
 -\mu_{\alpha\beta}(x\cdot y)
 \right\}
 \leq \Pi_{\alpha\beta}^{\mathrm{LH}}.
 \acEquationTag{eq:ac-far-mixed-cap-bound}
\]

The maximum above is the quantity
\(\mathfrak C(a,b,c)\) in
\Cref{prop:ac-cap-dimension-reduction} under the substitution

\(
 a=r_\alpha\bar r_\beta,\qquad
 b=1,\qquad
 c=\mu_{\alpha\beta}.
\)
For this substitution, the generic norm quantities from
\Cref{prop:ac-cap-dimension-reduction,cor:ac-cap-unconstrained-estimate}
become
\[
\begin{aligned}
 \nu_{1,\mu_{\alpha\beta}}
 &=\sqrt{1+\mu_{\alpha\beta}^2
          -2\rho_n\mu_{\alpha\beta}},\\
 \nu_{r_\alpha\bar r_\beta,\mu_{\alpha\beta}}
 &=r_\alpha
 \sqrt{\bar r_\beta^2+(\kappa_\beta^{\mathrm{LH}})^2
       -2\rho_n\bar r_\beta\kappa_\beta^{\mathrm{LH}}}.
\end{aligned}
\]
Moreover,
\[
 \rho_n-\frac ac
 =\rho_n-\frac{\bar r_\beta}{\kappa_\beta^{\mathrm{LH}}},
\]
and, since \(\kappa_\beta^{\mathrm{LH}}>0\),
\[
 \sqrt{(\rho_n-a/c)^2+\sin^2\widehat{\zeta}_n}
 =
 \frac{
 \sqrt{\bar r_\beta^2+(\kappa_\beta^{\mathrm{LH}})^2
       -2\rho_n\bar r_\beta\kappa_\beta^{\mathrm{LH}}}}
 {\kappa_\beta^{\mathrm{LH}}}.
\]
The generic unconstrained maximizer in
\Cref{cor:ac-cap-projected-maximizer} is therefore
\[
 \mathsf h_{\mathrm{free}}
 =
 \frac{
 \sqrt{1+\mu_{\alpha\beta}^2-2\rho_n\mu_{\alpha\beta}}\,
 (\rho_n\kappa_\beta^{\mathrm{LH}}-\bar r_\beta)}
 {\sqrt{\bar r_\beta^2+(\kappa_\beta^{\mathrm{LH}})^2
       -2\rho_n\bar r_\beta\kappa_\beta^{\mathrm{LH}}}}.
\]
Define
\(
 \mathsf h_{\mathrm{opt}}
 :=\operatorname{proj}_{J_{1,\mu_{\alpha\beta}}}
 (\mathsf h_{\mathrm{free}}).
\)

When \(v_i\cdot x=v_j\cdot y=\rho_n\), the two square roots above are
\[
 \sqrt{1+\mu_{\alpha\beta}^2-2\rho_n\mu_{\alpha\beta}}
 =\lVert v_i-\mu_{\alpha\beta}x\rVert,
\]
and
\[
 \sqrt{\bar r_\beta^2+(\kappa_\beta^{\mathrm{LH}})^2
       -2\rho_n\bar r_\beta\kappa_\beta^{\mathrm{LH}}}
 =\lVert\bar r_\beta v_j-\kappa_\beta^{\mathrm{LH}}y\rVert.
\]

Recall that
\[
 J_{1,\mu_{\alpha\beta}}
 =
 \bigl[-\rho_n-\mu_{\alpha\beta}B_n^{\mathrm{cap}},\,
 \min\{\mu_{\alpha\beta}-\rho_n,\,
       \rho_n
       \sqrt{1+\mu_{\alpha\beta}^2
             -2\rho_n\mu_{\alpha\beta}}\}\bigr].
\]
We show that \(\mathsf h_{\mathrm{free}}\) is strictly smaller than the
second entry in the minimum defining the right endpoint.  Since
\(\bar r_\beta/\kappa_\beta^{\mathrm{LH}}>0\), one has
\[
 \rho_n-\frac{\bar r_\beta}{\kappa_\beta^{\mathrm{LH}}}<\rho_n.
\]
If
\(\rho_n-\bar r_\beta/\kappa_\beta^{\mathrm{LH}}\leq0\), then
\[
 \mathsf h_{\mathrm{free}}\leq0
 <\rho_n
 \sqrt{1+\mu_{\alpha\beta}^2-2\rho_n\mu_{\alpha\beta}}.
\]
If \(0<\rho_n-\bar r_\beta/\kappa_\beta^{\mathrm{LH}}<\rho_n\), then
\[
 \frac{\mathsf h_{\mathrm{free}}}
 {\sqrt{1+\mu_{\alpha\beta}^2-2\rho_n\mu_{\alpha\beta}}}
 =\frac{\rho_n-\bar r_\beta/\kappa_\beta^{\mathrm{LH}}}
 {\sqrt{(\rho_n-\bar r_\beta/\kappa_\beta^{\mathrm{LH}})^2
 +\sin^2\widehat{\zeta}_n}}
 <\rho_n.
\]
Thus
\[
 \mathsf h_{\mathrm{free}}
 <\rho_n
 \sqrt{1+\mu_{\alpha\beta}^2-2\rho_n\mu_{\alpha\beta}}.
\]
Consequently, the nearest-point projection formula in
\Cref{cor:ac-cap-projected-maximizer} leaves exactly the three
possibilities
\[
 \mathsf h_{\mathrm{opt}}=\mu_{\alpha\beta}-\rho_n,\qquad
 \mathsf h_{\mathrm{opt}}
 =-\rho_n-\mu_{\alpha\beta}B_n^{\mathrm{cap}},\qquad
 \mathsf h_{\mathrm{opt}}=\mathsf h_{\mathrm{free}}.
\]
Apply \Cref{lem:ac-far-mixed-unconstrained} with
\[
 \rho=\rho_n,\qquad
 r=r_\alpha,\qquad
 t=\bar r_\beta,\qquad
 \kappa=\kappa_\beta^{\mathrm{LH}},\qquad
 \mu=\mu_{\alpha\beta}.
\]
Its first conclusion gives
\(
 1-2\rho_n\mu_{\alpha\beta}>0.
\)

\smallskip
\noindent\emph{The case
\(\mathsf h_{\mathrm{opt}}=\mu_{\alpha\beta}-{\rho_n}\).}
The projected-maximizer formula in
\Cref{cor:ac-cap-projected-maximizer}, equivalently
\eqref{eq:ac-cap-explicit}, gives
\(
 \mathfrak F_{\alpha\beta}^{\mathrm{LH}}
 =B_n^{\mathrm{cap}}
 +r_\alpha(\rho_n\kappa_\beta^{\mathrm{LH}}-\bar r_\beta).
\)
Since \(B_n^{\mathrm{cap}}\leq1\), \(\kappa_\beta^{\mathrm{LH}}\leq1\), and
\(\bar r_\beta\geq0\),
\(
 \mathfrak F_{\alpha\beta}^{\mathrm{LH}}
 \leq1+\rho_n r_\alpha.
\)
Moreover,
\(
 X_\alpha^2-(1+\rho_n r_\alpha)^2
 =(1-\rho_n^2)r_\alpha^2\geq0,
\)
and \(\bar X_\beta\geq1\).  Hence
\(
 1+\rho_n r_\alpha
 \leq X_\alpha
 \leq X_\alpha\bar X_\beta
 =\Pi_{\alpha\beta}^{\mathrm{LH}}.
\)

\smallskip
\noindent\emph{The case
\(\mathsf h_{\mathrm{opt}}
=-{\rho_n}-\mu_{\alpha\beta}B_n^{\mathrm{cap}}\).}

Since \(1-2\rho_n\mu_{\alpha\beta}>0\), the identity
\[
 1+\mu_{\alpha\beta}^2-2\rho_n\mu_{\alpha\beta}
 -({\rho_n}+\mu_{\alpha\beta}B_n^{\mathrm{cap}})^2
 = \sin^2\widehat{\zeta}_n
   (1-2{\rho_n}\mu_{\alpha\beta})^2
\]

together with the projected-maximizer formula
\eqref{eq:ac-cap-explicit} from
\Cref{cor:ac-cap-projected-maximizer} give
\(
 \mathfrak F_{\alpha\beta}^{\mathrm{LH}}
 =B_n^{\mathrm{cap}}(1+r_\alpha\bar r_\beta)
 -\mu_{\alpha\beta}(3\rho_n-4\rho_n^3).
\)
Since \(B_n^{\mathrm{cap}}\leq1\),
\(\mu_{\alpha\beta}>0\), and
\(
 3\rho_n-4\rho_n^3>0,
\)
the right-hand side is at most
\(1+r_\alpha\bar r_\beta\).  Finally,
\[
\begin{aligned}
 \Pi_{\alpha\beta}^{\mathrm{LH}}
 &\geq
 \sqrt{(1+r_\alpha^2)(1+\bar r_\beta^2)}
 \geq1+r_\alpha\bar r_\beta,
\end{aligned}
\]
where the second inequality follows after squaring from
\((r_\alpha-\bar r_\beta)^2\geq0\).

\smallskip
\noindent\emph{The case
\(\mathsf h_{\mathrm{opt}}=\mathsf h_{\mathrm{free}}\).}
The projected-maximizer formula in
\Cref{cor:ac-cap-projected-maximizer} gives
\[
\mathfrak F_{\alpha\beta}^{\mathrm{LH}}
 =
 \frac{
 \sqrt{\bar r_\beta^2+(\kappa_\beta^{\mathrm{LH}})^2
       -2\rho_n\bar r_\beta\kappa_\beta^{\mathrm{LH}}}\,
 \sqrt{1+\mu_{\alpha\beta}^2
       -2\rho_n\mu_{\alpha\beta}}
 -\rho_n\bar r_\beta}
 {\kappa_\beta^{\mathrm{LH}}}.                         \acEquationTag{eq:ac-100}
\]

The equality
\(\mathsf h_{\mathrm{opt}}=\mathsf h_{\mathrm{free}}\) implies that
\(\mathsf h_{\mathrm{free}}\in J_{1,\mu_{\alpha\beta}}\).
Its left-endpoint inequality, together with the displayed formula for
\(\mathsf h_{\mathrm{free}}\), is
\[
 \frac{
 \sqrt{1+\mu_{\alpha\beta}^2-2\rho_n\mu_{\alpha\beta}}\,
 (\bar r_\beta-\rho_n\kappa_\beta^{\mathrm{LH}})}
 {\sqrt{\bar r_\beta^2+(\kappa_\beta^{\mathrm{LH}})^2
       -2\rho_n\bar r_\beta\kappa_\beta^{\mathrm{LH}}}}
 \leq
 \rho_n+\mu_{\alpha\beta}(1-2\rho_n^2).
\]
This is the hypothesis
\eqref{eq:ac-mixed-left-endpoint} under the substitution above.
The second conclusion of
\Cref{lem:ac-far-mixed-unconstrained} therefore gives
\[
\begin{aligned}
 &\sqrt{1+\mu_{\alpha\beta}^2-2\rho_n\mu_{\alpha\beta}}\,
  \sqrt{\bar r_\beta^2+(\kappa_\beta^{\mathrm{LH}})^2
        -2\rho_n\bar r_\beta\kappa_\beta^{\mathrm{LH}}}\\
 &\qquad\leq
 \kappa_\beta^{\mathrm{LH}}\Pi_{\alpha\beta}^{\mathrm{LH}}
 +\rho_n\bar r_\beta.
\end{aligned}
\]
Dividing by \(\kappa_\beta^{\mathrm{LH}}>0\) and using
\eqref{eq:ac-100} gives
\(
 \mathfrak F_{\alpha\beta}^{\mathrm{LH}}
 \leq\Pi_{\alpha\beta}^{\mathrm{LH}}.
\)

Thus \eqref{eq:ac-far-mixed-cap-bound} holds for all three possible
values of \(\mathsf h_{\mathrm{opt}}\).  For every
\(x\in\mathcal C_{v_i}\) and \(y\in\mathcal C_{v_j}\), it follows that
\[
\begin{aligned}
 &r_\alpha\bar r_\beta(v_j\cdot x)+(v_i\cdot y)
 -r_\alpha\kappa_\beta^{\mathrm{LH}}(x\cdot y)\\
 &\qquad\leq\mathfrak F_{\alpha\beta}^{\mathrm{LH}}
 \leq\Pi_{\alpha\beta}^{\mathrm{LH}}\\
 &\qquad\leq
 \rho_n\bar r_\beta
 +\Pi_{\alpha\beta}^{\mathrm{LH}}
 (1+\rho_n+\cos\alpha\cos\beta).
\end{aligned}
\]

This is the scaled form of \eqref{eq:ac-80}.  Hence
\eqref{eq:ac-80} holds in the mixed case.  By
\Cref{lem:ac-far-common-reduction},
\(R\leq1+\rho_n+Q\), as required.
\end{proof}

\begin{proof}[Proof of \Cref{prop:ac-far-affine}]
Suppose first that \(\alpha,\beta<\widehat{\zeta}_n\).
Apply \Cref{lem:ac-far-low,lem:ac-far-high,lem:ac-far-mixed}
according as both levels lie below \(\alpha_n^\ast\), both lie above it,
or they lie on opposite sides.

If at least one level equals \(\widehat{\zeta}_n\), the case
\(Q<-\rho_n\) follows from
\Cref{cor:ac-finite-gain-closure}\textup{(ii)} with
\(g(s):=1+\rho_n+s\) and \(c:=-\rho_n\).
For \(Q=-\rho_n\), the desired estimate is \(R\leq1\), which is
automatic.
\end{proof}

\begin{corollary}[Upper boundary for distinct anchors]
\label{lem:ac-far}
Assume \(n\geq2\).  In the common distinct-anchor notation of
\Cref{sec:ac-distinct-lower-boundary}, one has
\[
 Q\leq-{\rho_n}\quad\Longrightarrow\quad
 R\leq U_n(Q)
 =-{\rho_n}Q+{\sin \zeta_n}\sqrt{1-Q^2}.                    \acEquationTag{eq:ac-76}
\]
\end{corollary}

\begin{proof}
On the interval \(s\in[-1,-{\rho_n}]\), the upper boundary function
\(
 U_n(s)=-{\rho_n}s+{\sin \widehat{\zeta}_n}\sqrt{1-s^2}
\)
is concave; here \(\sin\widehat{\zeta}_n=\sin\zeta_n\).  Its endpoint
values \(U_n(-1)={\rho_n}\) and \(U_n(-{\rho_n})=1\) therefore give
\(
 U_n(s)\geq1+{\rho_n}+s
 \qquad(-1\leq s\leq-{\rho_n}).
\)
The conclusion now follows immediately from
\Cref{prop:ac-far-affine}.
\end{proof}

Together, the same-anchor estimate \Cref{lem:ac-same-anchor}, the
lower-boundary result \Cref{lem:ac-near}, and the upper-boundary result
\Cref{lem:ac-far} control every active-band comparison for \(n\geq2\);
the middle \(Q\)-range was handled at the beginning of
\Cref{sec:anchored-chord-estimates}.
\Cref{sec:anchored-chord-assembly} treats the identity band, the
antipodal extension, and the case \(n=1\).

\acAssemblyHeading

Three tasks remain.  We first extend the anchored--chord estimates to
comparisons involving the identity band and the antipodal extension.
We then treat the case \(n=1\), using the lower-boundary
estimate of \Cref{lem:ac-near}; the spherical-cap and upper-boundary
estimates of \Cref{sec:anchored-chord-upper-boundary} are not needed.
Finally, we assemble all cases and use the north-pole fiber, together
with the known lower bound, to prove sharpness
for the remaining cases \(n\geq2\).

\subsection{Identity band and antipodal extension}
\label{sec:assembly-identity-antipodal}

\begin{lemma}[Identity-band estimate]
\label{lem:ac-identity-band}

Every comparison between two pairs of \(\mathcal R_n^+\) in which at
least one pair belongs to the identity band in \eqref{eq:ac-8} has
defect at most \(\zeta_n\).

\end{lemma}

\begin{proof}

For an anchored--chord pair
\((p_\alpha(x),H_{\alpha,i}(x))\), the equatorial identification in
\Cref{sec:anchored-chord-definition} and \eqref{eq:ac-11} give
\[
 p_\alpha(x)\cdot H_{\alpha,i}(x)
 =\sin\alpha\,(x\cdot H_{\alpha,i}(x))\geq0,
\]
so \(d_{n+1}(p_\alpha(x),H_{\alpha,i}(x))\leq\tfrac{\pi}{2}\).
For an identity-band pair \((p_\beta(y),y)\),
\(\beta\geq\widehat{\zeta}_n\), and
\(d_{n+1}(p_\beta(y),y)=\tfrac{\pi}{2}-\beta\).
Applying the triangle inequality twice in \(\Sp^{n+1}\) therefore
bounds the defect of these two pairs by
\[
 d_{n+1}(p_\alpha(x),H_{\alpha,i}(x))
 +d_{n+1}(p_\beta(y),y)
 \leq\pi-\beta\leq\zeta_n.
\]
For two identity-band pairs at radial levels
\(\beta_1,\beta_2\), the same inequality gives defect at most
\(\pi-\beta_1-\beta_2\leq2\zeta_n-\pi\leq\zeta_n\).
\end{proof}

\subsection{The case \texorpdfstring{\(n=1\)}{n=1}}
\label{sec:assembly-n-one}

\begin{proposition}[Optimality for \(n=1\)]
\label{prop:ac-n-one}
The relation of \Cref{def:anchored-chord-relation}, specialized to
\(n=1\), is a correspondence
\[
 \mathcal R_1\subseteq\Sp^2\times\Sp^1
\]
with
\[
 \dis(\mathcal R_1)=\zeta_1=\frac{2\pi}{3}.
\]
In particular, \(\mathcal R_1\) is an optimal correspondence between
\(\Sp^2\) and \(\Sp^1\).
\end{proposition}

\begin{proof}
Here
\[
 \rho_1=\frac12,\qquad \widehat{\zeta}_1=\frac\pi3,
 \qquad \zeta_1=\frac{2\pi}{3}=2\widehat{\zeta}_1.
\]
By \Cref{prop:ac-basic}, \(\mathcal R_1\) is a correspondence.  The three
anchors form a regular triangle in the circle \(\Sp(W_3)\).  Each closed
spherical Voronoi cell \(M_i\) is the circular arc of radius
\(\widehat{\zeta}_1\) centered at \(v_i\), and hence has diameter
\(2\widehat{\zeta}_1=\zeta_1\).
For \(x\in M_i\), the normalized positive combination
\(H_{\alpha,i}(x)\) lies on the shorter circular arc from \(v_i\) to
\(x\); consequently,
\[
 H_{\alpha,i}(x)\in M_i.
\]

Consider two pairs in the anchored--chord part of \(\mathcal R_1^+\), and
denote their source and target distances by \(d_{\mathrm{src}}\) and
\(d_{\mathrm{tar}}\), respectively.
If their radial
parameters are \(\alpha,\beta\leq\widehat{\zeta}_1\), then the path through the north
pole gives
\[
 d_{\mathrm{src}}\leq\alpha+\beta
 \leq2\widehat{\zeta}_1=\zeta_1.
\]
If the pairs use the same anchor, the preceding inclusion places both targets
in the same cell, so \(d_{\mathrm{tar}}\leq\zeta_1\); thus
\(|d_{\mathrm{src}}-d_{\mathrm{tar}}|\leq\zeta_1\).  Suppose that their
anchors are distinct.  When
\(d_{\mathrm{src}}\geq\widehat{\zeta}_1\), the bounds
\(d_{\mathrm{src}}\leq\zeta_1\) and \(d_{\mathrm{tar}}\leq\pi\) give
\[
 d_{\mathrm{src}}-d_{\mathrm{tar}}\leq\zeta_1,
 \qquad
 d_{\mathrm{tar}}-d_{\mathrm{src}}
 \leq\pi-\widehat{\zeta}_1=\zeta_1.
\]
When \(d_{\mathrm{src}}<\widehat{\zeta}_1\), put
\(Q:=\cos d_{\mathrm{src}}>\rho_1\).
\Cref{lem:ac-near} gives
\[
 R\geq-\rho_1\cos d_{\mathrm{src}}
   -\sin\widehat{\zeta}_1\sin d_{\mathrm{src}}
   =\cos(\zeta_1+d_{\mathrm{src}}),
\]
where \(R\) is the target inner product.  Since
\(0<\zeta_1+d_{\mathrm{src}}<\pi\), it follows that
\(d_{\mathrm{tar}}\leq\zeta_1+d_{\mathrm{src}}\); the reverse
difference is already controlled by \(d_{\mathrm{src}}\leq\zeta_1\).
Hence all
comparisons between two anchored--chord pairs have defect at most
\(\zeta_1\).

\Cref{lem:ac-identity-band} now gives
\(\dis(\mathcal R_1^+)\leq\zeta_1\).  By \eqref{eq:ac-9} and
\Cref{lem:relation-helmet-trick},
\(\dis(\mathcal R_1)=\dis(\mathcal R_1^+)\leq\zeta_1\).
At the north pole, one source point is related to all three anchors,
two of which are separated by \(\tfrac{2\pi}{3}\); hence
\(\dis(\mathcal R_1)=\zeta_1\).
Equation \eqref{eq:gh-correspondence-formula} gives
\(2\dgh(\Sp^1,\Sp^2)\leq\dis(\mathcal R_1)=\zeta_1\), while
\Cref{prop:adjacent-lower-bound} gives the reverse inequality.  Thus
\(\mathcal R_1\) is optimal.

\end{proof}

\subsection{Proof of
\texorpdfstring{\Cref{thm:main}}{the consecutive-sphere theorem}}
\label{sec:assembly-main-theorem}
The same-anchor estimate, the lower and upper \((Q,R)\)-boundary
estimates for distinct anchors, the identity-band and antipodal-extension
argument, and the separate treatment of \(n=1\) now cover all cases
required for the proof.

\begin{proof}[Proof of \Cref{thm:main}]
The case \(n=1\) is \Cref{prop:ac-n-one}.  Hence assume \(n\geq2\).
By \Cref{prop:ac-basic}, \(\mathcal R_n\) is a correspondence.

By \eqref{eq:ac-9} and \Cref{lem:relation-helmet-trick},
\(\dis(\mathcal R_n)=\dis(\mathcal R_n^+)\).
\Cref{lem:ac-identity-band} handles every comparison in
\(\mathcal R_n^+\) involving the identity band, so only two
anchored--chord pairs remain.

Take two anchored--chord pairs,
\((p_\alpha(x),H_{\alpha,i}(x))\) and
\((p_\beta(y),H_{\beta,j}(y))\), and define the corresponding source and
target inner products by
\[
 Q:=p_\alpha(x)\cdot p_\beta(y),\qquad
 R:=H_{\alpha,i}(x)\cdot H_{\beta,j}(y).
\]
The four possible configurations are as follows.

\begin{itemize}[leftmargin=*,itemsep=0.6em]
\item \emph{Same anchor.}
If they use the same anchor,
\Cref{lem:ac-same-anchor} gives defect at most \(\widehat{\zeta}_n<\zeta_n\).
\item \emph{Distinct anchors with \(Q\geq\rho_n\).}
If their anchors are distinct and \(Q\geq {\rho_n}\), \Cref{lem:ac-near} is equivalent
to
\[
 d_n(H_{\alpha,i}(x),H_{\beta,j}(y))
 \leq d_{n+1}(p_\alpha(x),p_\beta(y))+\zeta_n.
\]
The reverse inequality follows because the source distance is at most
\(\widehat{\zeta}_n<\zeta_n\), whereas the target distance is
nonnegative.
\item \emph{Distinct anchors with \(Q\leq-\rho_n\).}
If \(Q\leq-{\rho_n}\), \Cref{lem:ac-far} is equivalent to the reverse
nontrivial inequality
\[
 d_n(H_{\alpha,i}(x),H_{\beta,j}(y))
 \geq d_{n+1}(p_\alpha(x),p_\beta(y))-\zeta_n.
\]
For the remaining inequality, the source distance is at least
\(\zeta_n\), while the target distance is at most \(\pi\), so its excess
over the source is at most \(\pi-\zeta_n=\widehat{\zeta}_n<\zeta_n\).

\item \emph{Distinct anchors with \(-\rho_n<Q<\rho_n\).}
Finally, if \(-{\rho_n}<Q<{\rho_n}\), then the source distance lies strictly between
\(\widehat{\zeta}_n\) and \(\zeta_n\); since every target distance lies in \([0,\pi]\),
both distortion inequalities follow directly.
\end{itemize}
Hence \(\dis(\mathcal R_n^+)\leq\zeta_n\), and therefore
\(\dis(\mathcal R_n)\leq\zeta_n\).

At the north pole one source point is related to all \(n+2\) simplex
vertices, and two distinct vertices are at distance \(\zeta_n\).  Therefore
\(\dis(\mathcal R_n)\geq\zeta_n\), proving \acSharpDistortionEquation.
By \eqref{eq:gh-correspondence-formula},
\(\dis(\mathcal R_n)=\zeta_n\) gives
\(2\dgh(\Sp^n,\Sp^{n+1})\leq\zeta_n\), whereas
\Cref{prop:adjacent-lower-bound} gives
\(2\dgh(\Sp^n,\Sp^{n+1})\geq\zeta_n\).  Hence
\acMainResultEquation{} holds.

\end{proof}

\paragraph{Acknowledgments.} D.~Kim gratefully acknowledges support from the National Research Foundation of Korea (NRF) through grants funded by the Korean government (MSIT, RS-2025-00515946; MOE, RS-2025-25397599). S.~Lim gratefully acknowledges support from the National Research Foundation of Korea (NRF) through grants funded by the Korean government (MSIT, RS-2025-23324186). F.~M\'emoli gratefully acknowledges support from the National Science Foundation through grants CCF-2523653 and DMS-2524362. During the development and preparation of this work, the authors used OpenAI’s ChatGPT as an assistive tool for exploratory calculations, identifying possible gaps, ambiguities, and imprecisions in proofs, and improving the clarity and style of the text. The authors take full responsibility for the contents of the manuscript.

\clearpage
\appendix
\numberwithin{equation}{section}
\section{Estimates for pairs with a common anchor}
\label{app:ac-common-anchor-estimates}

This section contains the scalar estimates used in the proof of the
common-anchor comparison in \Cref{lem:ac-common-anchor-pointwise}.

\subsection{Auxiliary half-angle inequalities}
\label{app:ac-half-angle-inequalities}

\begin{lemma}[Auxiliary half-angle inequalities]
\label{lem:ac-half-angle-inequalities}
For every \(n\geq2\), one has
\[
 s_n(1+\sqrt {\rho_n})\geq c_n,\qquad
 4s_n(1-s_n)\geq\frac1{c_n(1+c_n)},\qquad
 \sin \widehat{\zeta}_n\geq c_n,
\]
and
\[
 \rho_n\sqrt {\rho_n}
 +(1-\rho_n)\sin \widehat{\zeta}_n\geq c_n.
\]
\end{lemma}

\begin{proof}[Proof of \Cref{lem:ac-half-angle-inequalities}]
Recall that
\[
 0<\rho_n\leq\frac13,\qquad
 \frac1{\sqrt3}\leq s_n<\frac1{\sqrt2}
 <c_n\leq\sqrt{\frac23}.
\]
First,
\[
 s_n(c_n+s_n)\geq2s_n^2=1-\rho_n\geq\frac23
 >\frac1{\sqrt3}\geq\sqrt{\rho_n}.
\]

Multiplying the preceding inequality
\(s_n(c_n+s_n)\geq\sqrt{\rho_n}\) by the positive quantity
\(\sqrt{\rho_n}/(c_n+s_n)\) gives
\[
 s_n\sqrt{\rho_n}
 \geq\frac{\rho_n}{c_n+s_n}=c_n-s_n.
\]
This is equivalent to
\(s_n(1+\sqrt{\rho_n})\geq c_n\), proving the first inequality.

The function \(\tau\mapsto4\tau(1-\tau)\) is decreasing on
\([1/\sqrt3,1/\sqrt2]\), while
\(\tau\mapsto1/[\tau(1+\tau)]\) is decreasing on
\([1/\sqrt2,\sqrt{2/3}]\).  Since
\(s_n\in[1/\sqrt3,1/\sqrt2)\) and
\(c_n\in(1/\sqrt2,\sqrt{2/3}]\), respectively, these two
monotonicity statements give
\[
 4s_n(1-s_n)\geq2(\sqrt2-1)
 \geq\frac1{c_n(1+c_n)}.
\]
Furthermore,
\[
 \sin^2\widehat{\zeta}_n-c_n^2
 =1-\rho_n^2-\frac{1+\rho_n}{2}
 =\frac{(1+\rho_n)(1-2\rho_n)}2\geq0,
\]
which proves the third inequality.

For the last inequality, define
\[
 f(\rho):=\rho\sqrt\rho+(1-\rho)\sqrt{1-\rho^2}
 -\sqrt{\frac{1+\rho}{2}},
 \qquad 0<\rho\leq\frac13.
\]
Then
\[
 f'(\rho)=\frac32\sqrt\rho-\sqrt{1-\rho^2}
 -\frac{\rho(1-\rho)}{\sqrt{1-\rho^2}}
 -\frac1{4\sqrt{(1+\rho)/2}}<0.
\]
Indeed,
\[
 \frac32\sqrt\rho\leq\frac{\sqrt3}{2}
 <\frac{2\sqrt2}{3}\leq\sqrt{1-\rho^2},
\]
and the remaining two terms in \(f'(\rho)\) are negative.  Moreover,
\(
 9\sqrt3\,f(1/3)=3+4\sqrt6-9\sqrt2>0.
\)
For the final strict inequality, both sides of
\(3+4\sqrt6>9\sqrt2\) are positive, and
 \(
 (3+4\sqrt6)^2-(9\sqrt2)^2=24\sqrt6-57>0,
 \)
because
\(
 (24\sqrt6)^2=3456>3249=57^2.
\)
Since \(f\) is decreasing and \(0<\rho_n\leq1/3\), it follows that
\(
 f(\rho_n)\geq f(\tfrac{1}{3})>0,
\)
which is the last asserted half-angle inequality.
\end{proof}

\subsection{Endpoint reduction for the comparison inner product}
\label{app:ac-common-anchor-endpoint-reduction}

\begin{proof}[Proof of \Cref{lem:ac-common-anchor-scalar}]

\noindent\emph{No interior minimum.}
Differentiation gives
\[
 E_\alpha'(a)
 =\frac{\Gamma_\alpha(a)}
 {(1+2r_\alpha a+r_\alpha^2)^{3/2}},
\]
where
\[
\begin{aligned}
 \Gamma_\alpha(a)
 &:=g_{\alpha,0}+g_{\alpha,1}a
 -3r_\alpha^2\sin\alpha(1-s_n)a^2,\\
 g_{\alpha,0}
 &:=r_\alpha^3c_n\cos\alpha
 +\sin\alpha\bigl(s_n-r_\alpha^2(1-s_n)\bigr),\\
 g_{\alpha,1}
 &:=r_\alpha^2c_n\cos\alpha
 +r_\alpha\sin\alpha
 \bigl(3s_n-2-2r_\alpha^2(1-s_n)\bigr).
\end{aligned}                                                     \acEquationTag{eq:ac-22}
\]
The denominator in the formula for \(E_\alpha'(a)\) is positive on
\([\rho_n,1]\), and the leading coefficient of \(\Gamma_\alpha\) is
negative. If \(g_{\alpha,0}>0\), the discriminant
 \(
 g_{\alpha,1}^2
 +12r_\alpha^2\sin\alpha(1-s_n)g_{\alpha,0}
 \)
is positive.  Thus the two roots of \(\Gamma_\alpha\) are real, and
their product is negative.  Exactly one root is positive, so on the
positive half-line the sign of \(\Gamma_\alpha\) can change only from
\(+\) to \(-\).

If \(g_{\alpha,0}=0\), then
 \(
 \Gamma_\alpha(a)
 =a\bigl(g_{\alpha,1}
 -3r_\alpha^2\sin\alpha(1-s_n)a\bigr),
 \)
so on the positive half-line its sign either is always negative or
changes once from \(+\) to \(-\).  In either case,
\(E_\alpha\) has no interior minimum.

It remains to consider \(g_{\alpha,0}<0\).  The first term in the
formula for \(g_{\alpha,0}\) is nonnegative, so
\(
 s_n-r_\alpha^2(1-s_n)<0.
\)
Moving that first term of \(g_{\alpha,0}<0\) to the other side and dividing by
\(r_\alpha^2>0\) gives
\[
 r_\alpha c_n\cos\alpha
 <\sin\alpha\left(1-s_n-\frac{s_n}{r_\alpha^2}\right).
\]
Consequently,
\[
\begin{aligned}
 \frac{g_{\alpha,1}}{r_\alpha}+\sin\alpha
 &=r_\alpha c_n\cos\alpha
   +\sin\alpha\bigl(3s_n-1-2r_\alpha^2(1-s_n)\bigr)\\
 &<\sin\alpha\left(
   2s_n-\frac{s_n}{r_\alpha^2}-2r_\alpha^2(1-s_n)\right)<0.
\end{aligned}
\]
The last inequality follows from
\(s_n-r_\alpha^2(1-s_n)<0\), which implies
\(2r_\alpha^2(1-s_n)>2s_n\).  Hence
\(g_{\alpha,1}<0\), and every coefficient of
\(\Gamma_\alpha\) is negative.  Thus \(E_\alpha'(a)<0\) for
\(a\in(\rho_n,1)\). We have again excluded an interior minimum. Since \(E_\alpha\) is continuous on \([\rho_n,1]\), it follows that
\[
 E_\alpha(a)\geq
 \min\{E_\alpha(\rho_n),E_\alpha(1)\}.                       \acEquationTag{eq:ac-24}
\]

\smallskip
\noindent\emph{The endpoint \(a=1\).}
At \(a=1\),
 \(
 E_\alpha(1)=c_n\cos\alpha+s_n\sin\alpha
 =\cos(\alpha-\widehat{\zeta}_n/2)\geq c_n.
 \)

\smallskip
\noindent\emph{The endpoint \(a=\rho_n\).} To realize this scalar endpoint in the common-anchor geometry, fix
distinct indices \(i,j\) and put \(x:=-v_j\).  The boundary
characterization in \Cref{lem:ac-simplex-cell} gives
\(x\in M_i\) and \(v_i\cdot x=\rho_n\).  We may consequently write
\[
 x=\rho_n v_i+w\sin\widehat{\zeta}_n,
 \qquad w\perp v_i,\qquad \lVert w\rVert=1.
\]
Put \(\theta_\alpha:=\vartheta_n(\alpha)\). For the chosen point \(x=-v_j\), we have \(v_i\cdot x=\rho_n\).
By \Cref{prop:ac-maximal-displacement},
\(
 d_n\bigl(v_i,H_{\alpha,i}(x)\bigr)
 =\vartheta_n(\alpha)=\theta_\alpha.
\)

Applying
\eqref{eq:chord-angular-form} to
\[
 x=\rho_n v_i+w\sin\widehat{\zeta}_n
 \quad\text{and}\quad
 H_{\alpha,i}(x)
 =\frac{v_i+r_\alpha x}{\lVert v_i+r_\alpha x\rVert}
\]
gives
\(
 H_{\alpha,i}(x)
 =v_i\cos\theta_\alpha+w\sin\theta_\alpha.
\)
Using
\[
 p_\alpha(x)
 =e_0\cos\alpha
  +\bigl(\rho_n v_i
             +w\sin\widehat{\zeta}_n\bigr)\sin\alpha
\]
and
\[
 \iota_{v_i}H_{\alpha,i}(x)
 =(c_ne_0+s_nv_i)\cos\theta_\alpha
  +w\sin\theta_\alpha,
\]

Taking the inner product of the preceding two vectors and comparing
with the definition of \(E_\alpha\) in
\Cref{lem:ac-common-anchor-scalar}, we obtain
\[
 p_\alpha(x)\cdot\iota_{v_i}H_{\alpha,i}(x)
 =E_\alpha(\rho_n)
 =\bigl(c_n\cos\alpha+s_n\rho_n\sin\alpha\bigr)\cos\theta_\alpha
 +\sin \widehat{\zeta}_n\sin\alpha\sin\theta_\alpha.
 \acEquationTag{eq:ac-26}
\]

\smallskip
\noindent\emph{Low radial levels
\(\bigl(0\leq\alpha\leq\alpha_n^\ast\bigr)\).}
If \(0\leq\alpha\leq \alpha_n^\ast\), then
\(\sin\theta_\alpha=\sin\Theta_n(\alpha)=c_n\sin\alpha\).  Moreover,
\(\theta_\alpha=\Theta_n(\alpha)\leq
\Theta_n(\alpha_n^\ast)=\widehat{\zeta}_n/2\), and hence
\(\cos\theta_\alpha\geq c_n\).

Moreover,
 \(
 \theta_\alpha
 =\arcsin(c_n\sin\alpha)\leq\arcsin(\sin\alpha)=\alpha,
 \)
because \(0<c_n\leq1\) and
\(0\leq\alpha\leq\alpha_n^\ast<\tfrac{\pi}{2}\).  Hence
\(\cos\theta_\alpha\geq\cos\alpha\).  Substituting
\(\sin\theta_\alpha=c_n\sin\alpha\) into \eqref{eq:ac-26} and using
this last inequality gives \[
E_\alpha({\rho_n})
\geq c_n-c_n(1-{\sin \widehat{\zeta}_n})\sin^2\alpha
      +s_n{\rho_n}\sin\alpha\cos\theta_\alpha
\geq c_n.
\]

The last step follows from
\(\cos\theta_\alpha\geq c_n\), \(0\leq\sin\alpha\leq1\), and
\[
 1-{\sin \widehat{\zeta}_n}
 =\frac{\rho_n^2}{1+{\sin \widehat{\zeta}_n}}
 \leq \rho_n^2
 \leq s_n{\rho_n}.
\]
Indeed,
\[
 s_n\rho_n\sin\alpha\cos\theta_\alpha
 \geq c_ns_n\rho_n\sin\alpha
 \geq c_ns_n\rho_n\sin^2\alpha
 \geq c_n(1-\sin\widehat{\zeta}_n)\sin^2\alpha.
\]

\smallskip
\noindent\emph{High radial levels
\(\bigl(\alpha_n^\ast\leq\alpha<\widehat{\zeta}_n\bigr)\).} Suppose \(\alpha_n^\ast\leq\alpha<\widehat{\zeta}_n\).
As a function of
\(\theta_\alpha\in[\widehat{\zeta}_n/2,\widehat{\zeta}_n]\), the
right-hand side of \eqref{eq:ac-26} is positive: both coefficients of
\(\cos\theta_\alpha\) and \(\sin\theta_\alpha\) are positive, and both
trigonometric functions are nonnegative on this interval.  Its second
derivative with respect to \(\theta_\alpha\) is the negative of the
same expression.  It is therefore concave, and its minimum on the
closed interval occurs at
\(\widehat{\zeta}_n/2\) or \(\widehat{\zeta}_n\).
At \(\widehat{\zeta}_n/2\) it is
\[
 c_n\left[c_n\cos\alpha+(1-{\rho_n}+{\rho_n}s_n)\sin\alpha\right],               \acEquationTag{eq:ac-29}
\]
and at \(\widehat{\zeta}_n\) it is
\[
 c_n{\rho_n}\cos\alpha+(s_n\rho_n^2+\sin^2 \widehat{\zeta}_n)\sin\alpha.                          \acEquationTag{eq:ac-30}
\]

Each of \eqref{eq:ac-29} and \eqref{eq:ac-30}, viewed as a function of
\(\alpha\), is positive on
\([\alpha_n^\ast,\widehat{\zeta}_n]\), and its second derivative is its
negative.  Hence each function is concave, so its minimum on this
closed interval occurs at an endpoint.  The formulas, initially
obtained for \(\alpha<\widehat{\zeta}_n\), extend continuously to
\(\alpha=\widehat{\zeta}_n\).  It remains to evaluate them at
\(\alpha=\alpha_n^\ast\) and
\(\alpha=\widehat{\zeta}_n\); the four inequalities in
\Cref{lem:ac-half-angle-inequalities} provide the required bounds.

\smallskip
\noindent\emph{The four high-radial-level endpoint checks.}
We now spell out the four endpoint verifications.  Since
\[
 \sin\alpha_n^\ast=\frac{s_n}{c_n},\qquad
 \cos\alpha_n^\ast=\frac{\sqrt{\rho_n}}{c_n},
 \qquad \sin \widehat{\zeta}_n=2c_ns_n,\qquad \cos \widehat{\zeta}_n=\rho_n,
\]
the value in \eqref{eq:ac-29} at \(\alpha=\alpha_n^\ast\) is
\[
 c_n\sqrt{\rho_n}+s_n(1-\rho_n+\rho_ns_n)
 \geq c_n\sqrt{\rho_n}+s_n(1-\rho_n)\geq c_n.
\]
Here the last inequality is the first inequality in
\Cref{lem:ac-half-angle-inequalities}, multiplied by
\(1-\sqrt{\rho_n}\).  At \(\alpha=\widehat{\zeta}_n\), the value in
\eqref{eq:ac-29} is
\[
 c_n^2\bigl[1+4s_n^3(1-s_n)\bigr]\geq c_n,
\]
where the displayed bound is equivalent to the second inequality in
\Cref{lem:ac-half-angle-inequalities}, using
\(1-c_n=s_n^2/(1+c_n)\).

At \(\alpha=\alpha_n^\ast\), the value in \eqref{eq:ac-30} is
\[
 \rho_n\sqrt{\rho_n}+(1-\rho_n)\sin \widehat{\zeta}_n
 +\frac{s_n^2\rho_n^2}{c_n}\geq c_n
\]
by the last inequality in
\Cref{lem:ac-half-angle-inequalities}.  Finally, its value at
\(\alpha=\widehat{\zeta}_n\) is
\[
 \sin \widehat{\zeta}_n+\rho_n^2\bigl[c_n-\sin \widehat{\zeta}_n(1-s_n)\bigr]
 \geq\sin \widehat{\zeta}_n\geq c_n,
\]
because
\(c_n-\sin \widehat{\zeta}_n(1-s_n)=c_n[1-2s_n(1-s_n)]>0\).
The final inequality is the third one in
\Cref{lem:ac-half-angle-inequalities}.
Thus the continuous extensions of \eqref{eq:ac-29}--\eqref{eq:ac-30}
are at least \(c_n\) at both endpoints.  Equations
\eqref{eq:ac-24}, \eqref{eq:ac-26}, and
\eqref{eq:ac-29}--\eqref{eq:ac-30}, together with
\Cref{lem:ac-half-angle-inequalities} and these endpoint calculations,
prove
\[
 \min_{\rho_n\leq a\leq1}E_\alpha(a)
 =\min\{E_\alpha(\rho_n),E_\alpha(1)\}
 \geq c_n.
\]
\end{proof}

\section{Algebraic calculations for the low radial range}
\label{app:ac-low-radial-algebra}

This section contains the coefficient calculation used in the exact
factorization and the proof of the polynomial estimate used in the
low radial range.

\subsection{Coefficient verification for the factorization}
\label{app:ac-low-factorization}

We verify the coefficient identity used in
\eqref{eq:ac-45}. Retain the notation from the proof of
\Cref{lem:ac-near-low} and, for this calculation only, abbreviate
\[
 \rho:=\rho_n,\qquad
 k:=r_{\alpha\beta}^{\times},\qquad
 s:=r_{\alpha\beta}^{+},\qquad
 \chi:=\chi_n.
\]

We begin by expanding the polynomial defined in
\Cref{lem:ac-low-polynomial-sign}.  Its four summands satisfy
\[
\begin{aligned}
 -2(1-\rho)\bigl(k(1+z)+2\bigr)s
 &=-2(1-\rho)(k+2)s-2(1-\rho)ksz,\\
 k^3(1+z)\bigl[1+(4\rho-5)z^2\bigr]
 &=k^3+k^3z+k^3(4\rho-5)z^2+k^3(4\rho-5)z^3,\\
 4k^2(1-\chi z^2)&=4k^2-4k^2\chi z^2,\\
 4k\bigl[2-\rho(2-\rho)(1+z)\bigr]
 &=4k\bigl[2-\rho(2-\rho)\bigr]
   -4k\rho(2-\rho)z.
\end{aligned}
\]
Consequently, if
\(
 P_\rho(k,s,z)=p_0+p_1z+p_2z^2+p_3z^3,
\)
then
\[
\begin{aligned}
 p_0:={}&-2(1-\rho)(k+2)s+k^3+4k^2
         +4k\bigl[2-\rho(2-\rho)\bigr],\\
 p_1:={}&-2(1-\rho)ks+k^3-4k\rho(2-\rho),\\
 p_2:={}&k^3(4\rho-5)-4k^2\chi,\\
 p_3:={}&k^3(4\rho-5).
\end{aligned}
\]

We next expand the left-hand side of \eqref{eq:ac-45}.  Put
\[
 a:=1+\rho s+k^2,
 \qquad
 b:=k^2+2k-(1-\rho)s.
\]
The formulas preceding \eqref{eq:ac-45}, together with
\(\sin^2\widehat{\zeta}_n=1-\rho^2\), become
\[
\begin{aligned}
 X_\alpha^2X_\beta^2
 &=a^2+(1-\rho^2)(s^2-4k^2),\\
 X_\alpha X_\beta\cos\eta
 &=a-2(1-\rho^2)k^2z^2,\\
 \frac{\mathcal G}{1+\rho}
 &=b-2kz-\chi k^2z^2.
\end{aligned}
\]
Define the normalized left-hand side
\[
 \mathscr V(z)
 :=
 \frac{
 \sin^2\widehat{\zeta}_n
 (X_\alpha X_\beta)^2\sin^2\eta-\mathcal G^2
 }{k(1+\rho)^2}.
\]
Since
\(
(X_\alpha X_\beta)^2\sin^2\eta
=X_\alpha^2X_\beta^2-(X_\alpha X_\beta\cos\eta)^2,
\)
substitution gives
\[
\begin{aligned}
 \mathscr V(z)
={}&
 \frac{1-\rho^2}{k(1+\rho)^2}
 \left[
 a^2+(1-\rho^2)(s^2-4k^2)
 -\bigl(a-2(1-\rho^2)k^2z^2\bigr)^2
 \right]\\
 &-\frac1k\bigl(b-2kz-\chi k^2z^2\bigr)^2\\
={}&\frac1k\Bigl\{
 (1-\rho)^2
 \bigl[s^2-4k^2+4ak^2z^2
       -4(1-\rho^2)k^4z^4\bigr]\\
 &\hspace{31mm}
 -\bigl(b-2kz-\chi k^2z^2\bigr)^2
 \Bigr\}.
\end{aligned}
\]
Here the second equality uses
\((1-\rho^2)^2/(1+\rho)^2=(1-\rho)^2\).
The remaining square expands as
\[
\begin{aligned}
 \bigl(b-2kz-\chi k^2z^2\bigr)^2
 ={}&b^2-4kbz
 +(4k^2-2\chi k^2b)z^2\\
 &+4\chi k^3z^3+\chi^2k^4z^4.
\end{aligned}
\]
Thus
\(
\mathscr V(z)=q_0+q_1z+q_2z^2+q_3z^3+q_4z^4,
\)
where
\[
\begin{aligned}
 q_0&:=\frac{(1-\rho)^2(s^2-4k^2)-b^2}{k},\\
 q_1&:=4b,\\
 q_2&:=4(1-\rho)^2ak-4k+2\chi kb,\\
 q_3&:=-4\chi k^2,\\
 q_4&:=-k^3\bigl[4(1-\rho)^2(1-\rho^2)+\chi^2\bigr].
\end{aligned}
\]

It remains to compare these five coefficients with
\(p_0,\ldots,p_3\).  For the constant and linear coefficients, direct
substitution of \(b=k^2+2k-(1-\rho)s\) gives
\[
\begin{aligned}
 q_0
 &=2(1-\rho)(k+2)s-k^3-4k^2
   -4k\bigl[2-\rho(2-\rho)\bigr]
 =-p_0,\\
 q_1
 &=4k^2+8k-4(1-\rho)s
 =p_0-p_1.
\end{aligned}
\]
For the quadratic coefficient, use
\[
 2\rho(1-\rho)-\chi=-1,
 \qquad
 4(1-\rho)^2+2\chi=6-4\rho.
\]
After substituting \(a\) and \(b\) in \(q_2\) and collecting the terms
containing \(ks\), \(k\), \(k^2\), and \(k^3\), respectively, we obtain
\[
\begin{aligned}
 q_2
 &=-2(1-\rho)ks-4k\rho(2-\rho)
   +4\chi k^2+(6-4\rho)k^3\\
 &=p_1-p_2.
\end{aligned}
\]
The cubic coefficient satisfies immediately
\[
 q_3=-4\chi k^2=p_2-p_3.
\]
Finally,
\[
 4(1-\rho)^2(1-\rho^2)+\chi^2=5-4\rho,
\]
and hence
\[
 q_4=k^3(4\rho-5)=p_3.
\]
We have therefore proved
\[
\begin{aligned}
 \mathscr V(z)
 &=-p_0+(p_0-p_1)z+(p_1-p_2)z^2
   +(p_2-p_3)z^3+p_3z^4\\
 &=(z-1)(p_0+p_1z+p_2z^2+p_3z^3)\\
 &=(z-1)P_\rho(k,s,z).
\end{aligned}
\]
Multiplying by \(k(1+\rho)^2\) and restoring the original notation
gives
\[
 \sin^2\widehat{\zeta}_n(X_\alpha X_\beta)^2\sin^2\eta-\mathcal G^2
 =r_{\alpha\beta}^{\times}(1+\rho_n)^2(z-1)
   P_{\rho_n}
   (r_{\alpha\beta}^{\times},r_{\alpha\beta}^{+},z),
\]
which is \eqref{eq:ac-45}.

\subsection{Sign of \texorpdfstring{\(P_{\rho_n}\)}{P}}
\label{app:ac-low-polynomial-sign}

\begin{proof}[Proof of \Cref{lem:ac-low-polynomial-sign}]
The polynomial is affine and strictly decreasing in
\(r_{\alpha\beta}^{+}\):
\(
 \partial_{r_{\alpha\beta}^{+}}P_{\rho_n}
 =-2(1-\rho_n)
 \bigl(r_{\alpha\beta}^{\times}(1+z)+2\bigr)<0.
\)
Moreover, the arithmetic--geometric mean inequality gives
\(
 r_{\alpha\beta}^{+}\geq2r_{\alpha\beta}^{\times}.
\)

We divide the proof according to the sign of
\(
 r_{\alpha\beta}^{\times}(1-\chi_nz^2)+2(\rho_n-z).
\)
When this expression is nonpositive, the arithmetic--geometric mean
bound \(r_{\alpha\beta}^{+}\geq2r_{\alpha\beta}^{\times}\), together
with the monotonicity of \(P_{\rho_n}\) in
\(r_{\alpha\beta}^{+}\), directly gives the required sign.  When it is
positive, the hypothesis \eqref{eq:ac-low-poly-hyp} gives a sharper
lower bound for \(r_{\alpha\beta}^{+}\); monotonicity then reduces the
conclusion to the positivity of the polynomial
\(\mathcal B_{\rho_n}\) defined below.

First suppose
\[
 r_{\alpha\beta}^{\times}(1-\chi_nz^2)+2(\rho_n-z)\leq0.
\acEquationTag{eq:ac-low-poly-first-subcase}
\]

Since \(P_{\rho_n}\) is decreasing in
\(r_{\alpha\beta}^{+}\), we have
\[
\begin{aligned}
 P_{\rho_n}
 (r_{\alpha\beta}^{\times},r_{\alpha\beta}^{+},z)
 &\leq
 P_{\rho_n}
 (r_{\alpha\beta}^{\times},2r_{\alpha\beta}^{\times},z)\\
 &=
 r_{\alpha\beta}^{\times}
 \bigl(2\rho_n+r_{\alpha\beta}^{\times}(1+z)\bigr)\Bigl[
 r_{\alpha\beta}^{\times}(1-\chi_nz^2)+2(\rho_n-z)\\
 &\hspace{43mm}
 -2(1-\rho_n)z
 \bigl(1+r_{\alpha\beta}^{\times}(2-\rho_n)z\bigr)
 \Bigr]\leq0.
\end{aligned}
\]

It remains to consider
\[
 r_{\alpha\beta}^{\times}(1-\chi_nz^2)+2(\rho_n-z)>0.
\acEquationTag{eq:ac-low-poly-second-subcase}
\]
The hypothesis \eqref{eq:ac-low-poly-hyp} gives
\[
 r_{\alpha\beta}^{+}>
 \frac{(r_{\alpha\beta}^{\times})^2(1-\chi_nz^2)
 +2r_{\alpha\beta}^{\times}(1-z)}{1-\rho_n}.
\]

Since \(P_{\rho_n}\) is strictly decreasing in
\(r_{\alpha\beta}^{+}\), the preceding strict lower bound gives
\[
P_{\rho_n}
 (r_{\alpha\beta}^{\times},r_{\alpha\beta}^{+},z)
<
 P_{\rho_n}\left(
 r_{\alpha\beta}^{\times},
 \frac{(r_{\alpha\beta}^{\times})^2(1-\chi_nz^2)
 +2r_{\alpha\beta}^{\times}(1-z)}{1-\rho_n},
 z\right).
\]
Substitution in the displayed formula for \(P_{\rho_n}\)
and collection of like powers give
\[
P_{\rho_n}\left(
 r_{\alpha\beta}^{\times},
 \frac{(r_{\alpha\beta}^{\times})^2(1-\chi_nz^2)
 +2r_{\alpha\beta}^{\times}(1-z)}{1-\rho_n},
 z\right)
=-r_{\alpha\beta}^{\times}
 \mathcal B_{\rho_n}(r_{\alpha\beta}^{\times},z),
\]
where
\[
\begin{aligned}
 \mathcal B_{\rho_n}(r_{\alpha\beta}^{\times},z):={}&
 (1+z)(r_{\alpha\beta}^{\times})^2
 [1+(4\rho_n^2-8\rho_n+3)z^2]
 +4r_{\alpha\beta}^{\times}(1-z^2)\\
 &+4\rho_n(2-\rho_n)(1+z)-8z.
\end{aligned}
\]

Since \(0<\rho_n\leq1/2\),
\(
 4\rho_n^2-8\rho_n+3
 =(2\rho_n-1)(2\rho_n-3)\geq0.
\)
Thus the terms in \(\mathcal B_{\rho_n}\) containing
\(r_{\alpha\beta}^{\times}\) are nonnegative.

For every \(z\in[0,1]\),
\(
 \mathcal B_{\rho_n}(r_{\alpha\beta}^{\times},z)
 \geq4\rho_n(2-\rho_n)(1+z)-8z.
\)

The affine function
\(
 z\longmapsto4\rho_n(2-\rho_n)(1+z)-8z
\)
is decreasing on \([0,\rho_n]\).  Therefore, if \(z\leq\rho_n\),
\[
 \mathcal B_{\rho_n}(r_{\alpha\beta}^{\times},z)
 \geq4\rho_n(2-\rho_n)(1+\rho_n)-8\rho_n
 =4\rho_n^2(1-\rho_n)>0.
\]

Suppose now that \(z>\rho_n\).

Inequality \eqref{eq:ac-low-poly-second-subcase} gives
\(
 r_{\alpha\beta}^{\times}(1-\chi_nz^2)
 >2(z-\rho_n)>0.
\)
Since \(r_{\alpha\beta}^{\times}>0\), it follows first that
\(1-\chi_nz^2>0\).  Division by this positive quantity then gives
\[
 r_{\alpha\beta}^{\times}>
 \frac{2(z-\rho_n)}{1-\chi_nz^2}.
\]
Moreover,
\[
 \partial_{r_{\alpha\beta}^{\times}}
 \mathcal B_{\rho_n}(r_{\alpha\beta}^{\times},z)
 ={}2(1+z)r_{\alpha\beta}^{\times}
 \bigl[1+(4\rho_n^2-8\rho_n+3)z^2\bigr]
 +4(1-z^2)>0.
\]
Hence
\(r_{\alpha\beta}^{\times}\mapsto
\mathcal B_{\rho_n}(r_{\alpha\beta}^{\times},z)\)
is strictly increasing.

Since
\[
 r_{\alpha\beta}^{\times}>
 \frac{2(z-\rho_n)}{1-\chi_nz^2},
\]
monotonicity gives
\[
\mathcal B_{\rho_n}(r_{\alpha\beta}^{\times},z)
\geq
\mathcal B_{\rho_n}\left(
  \frac{2(z-\rho_n)}{1-\chi_nz^2},z\right).
\]

To verify the required value, multiply first by the positive quantity
\((1-\chi_nz^2)^2/4\).  Direct substitution in the definition of
\(\mathcal B_{\rho_n}\) gives
\[
\begin{aligned}
 &\frac{(1-\chi_nz^2)^2}{4}
 \mathcal B_{\rho_n}\left(
   \frac{2(z-\rho_n)}{1-\chi_nz^2},z\right)\\
 &\quad=(1+z)(z-\rho_n)^2
       [1+(4\rho_n^2-8\rho_n+3)z^2]\\
 &\qquad\quad
       +2(z-\rho_n)(1-z^2)(1-\chi_nz^2)\\
 &\qquad\quad
       +[\rho_n(2-\rho_n)(1+z)-2z](1-\chi_nz^2)^2.
\end{aligned}
\]
Substituting \(\chi_n=1+2\rho_n-2\rho_n^2\), expanding, and collecting
powers of \(z\), the right-hand side factors as
\[
\begin{aligned}
 z^2(1-\rho_n)^2
 [1+(1-2\rho_n^2)z]\,
 [1+3z^2-2\rho_n(2-\rho_n)z(1+z)].
\end{aligned}
\]
Since
\(
 1+(1-2\rho_n^2)z=-(2\rho_n^2z-z-1),
\)
division by \((1-\chi_nz^2)^2/4\) yields
\[
\begin{aligned}
 \mathcal B_{\rho_n}\left(
   \frac{2(z-\rho_n)}{1-\chi_nz^2},z\right)
 ={}&\frac{-4z^2(1-\rho_n)^2(2\rho_n^2z-z-1)}
 {(1-\chi_nz^2)^2}\\
 &\quad\cdot
 [1+3z^2-2\rho_n(2-\rho_n)z(1+z)].
\end{aligned}                                                \acEquationTag{eq:ac-50}
\]

Here \(2\rho_n^2z-z-1<0\).  Since
\(\rho\mapsto2\rho(2-\rho)\) is increasing on \([0,1/2]\) and
\(0<\rho_n\leq1/2\),
\(
 2\rho_n(2-\rho_n)\leq\tfrac32,
\)
the remaining factor in \eqref{eq:ac-50} satisfies
\[
\begin{aligned}
 1+3z^2-2\rho_n(2-\rho_n)z(1+z)
 &\geq1-\frac32z+\frac32z^2
 &=\frac58+\frac32\left(z-\frac12\right)^2>0.
\end{aligned}
\]
Thus
\(\mathcal B_{\rho_n}(r_{\alpha\beta}^{\times},z)>0\), and consequently
\(
 P_{\rho_n}
 (r_{\alpha\beta}^{\times},r_{\alpha\beta}^{+},z)<0
\)
in the subcase \eqref{eq:ac-low-poly-second-subcase}.

Together with the argument under
\eqref{eq:ac-low-poly-first-subcase}, this proves the lemma.
\end{proof}

\section{The mixed lower-bound endpoint estimate}
\label{app:ac-near-mixed-slack}

This section contains the concavity and endpoint verification used in
\Cref{lem:ac-near-mixed-slack}.

\begin{proof}[Proof of \Cref{lem:ac-near-mixed-slack}]
Fix \(\beta\).  We carry out the calculation described after the statement of the lemma.

\smallskip
\noindent\emph{Preliminary angle bound and the endpoint expression.}
The complementary angles satisfy
\[
 \overline\theta_\beta\geq \delta_\beta.                  \acEquationTag{eq:ac-59}
\]

By \Cref{prop:ac-maximal-displacement},
\(\theta_\beta=\vartheta_n(\beta)\leq\beta\).  Hence
\(
 \overline\theta_\beta
 =\widehat{\zeta}_n-\theta_\beta
 \geq\widehat{\zeta}_n-\beta
 =\delta_\beta,
\)
which proves \eqref{eq:ac-59}.
The displayed formula for
\(\mathscr E_{\alpha,\beta}^{\mathrm{left}}\) follows by substituting
\(z=0\) in \eqref{eq:ac-54} and using
\(\sin^2\widehat{\zeta}_n=(1-{\rho_n})(1+{\rho_n})\).

\smallskip
\noindent\emph{Concavity of the \(T_{\alpha,\beta}\)-term.}
Write the \(T_{\alpha,\beta}\)-term as

\[
 \sqrt{\sin(\widehat{\zeta}_n-\overline\theta_\beta)
 \sin\overline\theta_\beta}\,\psi_n(\theta_\alpha),
 \qquad
 \psi_n(\xi)
 :=\sqrt{\sin\xi\sin(\widehat{\zeta}_n-\xi)}.
\]
For \(0<\xi\leq \widehat{\zeta}_n/2\),
\[
 \psi_n'(\xi)
 =\frac{\sin(\widehat{\zeta}_n-2\xi)}{2\psi_n(\xi)}\geq0,
\]
and

\[
 4\psi_n(\xi)^3\psi_n''(\xi)
 =-4\sin\xi\sin(\widehat{\zeta}_n-\xi)\cos(\widehat{\zeta}_n-2\xi)
  -\sin^2(\widehat{\zeta}_n-2\xi)<0.
\]
Indeed,
\(0\leq \widehat{\zeta}_n-2\xi
<\widehat{\zeta}_n<\tfrac{\pi}{2}\).

Since \(\psi_n(0)=0\) and \(\psi_n\) is continuous at \(0\),
monotonicity and concavity extend to
\([0,\widehat{\zeta}_n/2]\).

By the derivative identities in \eqref{eq:theta-derivatives},
\(\Theta_n\) is increasing and concave on
\([0,\alpha_n^\ast]\).

For \(\alpha_0,\alpha_1\in[0,\alpha_n^\ast]\) and
\(t\in[0,1]\), concavity of \(\Theta_n\) and monotonicity of
\(\psi_n\) give the first inequality below, while concavity of
\(\psi_n\) gives the second:
\[
\begin{aligned}
 \psi_n\!\left(\Theta_n(t\alpha_0+(1-t)\alpha_1)\right)
 &\geq
 \psi_n\!\left(t\Theta_n(\alpha_0)
 +(1-t)\Theta_n(\alpha_1)\right)\\
 &\geq
 t\psi_n(\Theta_n(\alpha_0))
 +(1-t)\psi_n(\Theta_n(\alpha_1)).
\end{aligned}
\]
Thus
\(\psi_n\circ\Theta_n\) is concave.  Since the coefficient multiplying
this composition in the displayed formula above is nonnegative,
\(\alpha\mapsto T_{\alpha,\beta}\) is concave as well.

\smallskip
\noindent\emph{Concavity of
\(\mathscr E_{\alpha,\beta}^{\mathrm{left}}-T_{\alpha,\beta}\).}
We reduce the sign of its second derivative to a scalar estimate.

\smallskip
\noindent\textbf{Reduction to the scalar estimate.}
We differentiate the remaining part of
\(\mathscr E_{\alpha,\beta}^{\mathrm{left}}\) directly.  Put
\[
 \Xi_n(\alpha):=\cos\theta_\alpha
 -\frac{s_n^2}{\cos^3\theta_\alpha}.
\]
Using \eqref{eq:theta-derivatives} and the definition of
\(\Xi_n(\alpha)\), we obtain
\[
\begin{aligned}
 \cos(\theta_\alpha-\overline\theta_\beta)(\theta_\alpha')^2
 +\sin(\theta_\alpha-\overline\theta_\beta)\theta_\alpha''
 =\sin\overline\theta_\beta\sin\theta_\alpha
 +\cos\overline\theta_\beta\,\Xi_n(\alpha)
\end{aligned}
\]
and hence

\[
 \partial_\alpha^2\bigl(\mathscr E_{\alpha,\beta}^{\mathrm{left}}-T_{\alpha,\beta}\bigr)
 =-(1-{\rho_n})[\cos(\delta_\beta+\alpha)-\cos\overline\theta_\beta\,\Xi_n(\alpha)]
  +\sin\theta_\alpha(\sin\overline\theta_\beta
  -\sin\widehat{\zeta}_n\cos\overline\theta_\beta).
\]

Since
\(0\leq\overline\theta_\beta\leq\widehat{\zeta}_n/2\),
\[
 \tan\overline\theta_\beta
 \leq\tan\frac{\widehat{\zeta}_n}{2}
 =\frac{\sin\widehat{\zeta}_n}{1+\rho_n}
 <\sin\widehat{\zeta}_n.
\]

Since \(\cos\overline\theta_\beta>0\), multiplication by this
quantity gives
\(
 \sin\overline\theta_\beta
 -\sin\widehat{\zeta}_n\cos\overline\theta_\beta < 0.
\)

Hence the summand
\(
 \sin\theta_\alpha
 \bigl(\sin\overline\theta_\beta
 -\sin\widehat{\zeta}_n\cos\overline\theta_\beta\bigr)
\)
in
\(\partial_\alpha^2
(\mathscr E_{\alpha,\beta}^{\mathrm{left}}-T_{\alpha,\beta})\)
is nonpositive.  It remains to prove
\[
 \cos(\delta_\beta+\alpha)
 -\cos\overline\theta_\beta\,\Xi_n(\alpha)\geq0.
\]
If \(\Xi_n(\alpha)<0\), this follows immediately because both cosine
factors are positive.  Suppose therefore that
\(\Xi_n(\alpha)\geq0\), and put
\(
 \tau_\alpha:=\tan\alpha,\qquad
 \tau_n^\ast:=\tan(\widehat{\zeta}_n-\alpha_n^\ast).
\)

The identities
\[
 \sin\alpha
 =\frac{\tau_\alpha}{\sqrt{1+\tau_\alpha^2}},
 \qquad
 \cos\alpha
 =\frac{1}{\sqrt{1+\tau_\alpha^2}}
\]
and \(\theta_\alpha=\Theta_n(\alpha)\) give
\[
 \sin\theta_\alpha
 =\frac{c_n\tau_\alpha}{\sqrt{1+\tau_\alpha^2}},
 \qquad
 \cos\theta_\alpha
 =\frac{\sqrt{1+s_n^2\tau_\alpha^2}}
 {\sqrt{1+\tau_\alpha^2}}.
\]
Substitution into the definition of \(\Xi_n\) yields
\[
 \Xi_n(\alpha)=
 \frac{c_n^2(1-s_n^2\tau_\alpha^4)}
 {\sqrt{1+\tau_\alpha^2}(1+s_n^2\tau_\alpha^2)^{3/2}},       \acEquationTag{eq:ac-65}
\]
because the numerator after passage to a common denominator is
\(
 (1+s_n^2\tau_\alpha^2)^2
 -s_n^2(1+\tau_\alpha^2)^2
 =c_n^2(1-s_n^2\tau_\alpha^4).
\)

It remains to prove the scalar estimate
\[
 \Xi_n(\alpha)\leq\cos\alpha-\tau_n^\ast\sin\alpha.          \acEquationTag{eq:ac-66}
\]

\smallskip
\noindent\textbf{The case \(n\geq2\).}
Assume \(n\geq2\), equivalently
\(\rho_n\leq1/3\).  If
\(\tau_\alpha=0\), then \(\alpha=0\), and \eqref{eq:ac-66} reduces to
\(c_n^2\leq1\).  Assume therefore that \(\tau_\alpha>0\).  Multiplying the
right-hand side minus the left-hand side of \eqref{eq:ac-66} by the
positive factor \(\sqrt{1+\tau_\alpha^2}\) gives
\[
\begin{aligned}
&\sqrt{1+\tau_\alpha^2}
\bigl[\cos\alpha-\tau_n^\ast\sin\alpha-\Xi_n(\alpha)\bigr]\\
&\quad=
\underbrace{\tau_\alpha\left[
s_n^2(\tau_\alpha^{-1}+c_n^2\tau_\alpha^3)-\tau_n^\ast\right]}_
{\text{first summand}}\\
&\qquad\quad+
\underbrace{c_n^2(1-s_n^2\tau_\alpha^4)
\left[1-(1+s_n^2\tau_\alpha^2)^{-3/2}\right]}_
{\text{second summand}}.
\end{aligned}
\]

We first establish the bound
\(
 \tau_n^\ast<\tfrac12.
\)

Regard \(\rho\) as a continuous variable in \((0,1/3]\), and put
\[
 \widehat\zeta(\rho):=\arccos\rho,\qquad
 c_\rho:=\sqrt{\frac{1+\rho}{2}},\qquad
 s_\rho:=\sqrt{\frac{1-\rho}{2}},
\]
\[
 \alpha^\ast(\rho):=
 \arcsin\left(\frac{s_\rho}{c_\rho}\right).
\]
Direct differentiation gives
\[
 \frac{d}{d\rho}
 \bigl(\widehat\zeta(\rho)-\alpha^\ast(\rho)\bigr)
 =\frac{1-2c_\rho\sqrt\rho}
 {4c_\rho^2s_\rho\sqrt\rho}>0.
\]
Indeed,
\(2c_\rho\sqrt\rho=\sqrt{2\rho(1+\rho)}<1\) for
\(0<\rho\leq1/3\).  Therefore
\(\tan(\widehat\zeta(\rho)-\alpha^\ast(\rho))\) is increasing on this
interval, and at its right endpoint it equals
\[
 \tan\bigl(\widehat\zeta(1/3)-\alpha^\ast(1/3)\bigr)
 =\frac{2\sqrt2-1}{2\sqrt2+1}<\frac12.
\]
Taking \(\rho=\rho_n\) proves the stated bound.

We now show that the first summand is nonnegative.  Weighted AM--GM,
with weights \(3/4\) and \(1/4\), gives
\(
 \tau_\alpha^{-1}+c_n^2\tau_\alpha^3
 \geq \tfrac{4(c_n^2)^{1/4}}{3^{3/4}}
\).  Consequently,
\[
 s_n^2(\tau_\alpha^{-1}+c_n^2\tau_\alpha^3)
 \geq\frac{4s_n^2(c_n^2)^{1/4}}{3^{3/4}}
 \geq\frac{4\,2^{1/4}}9.
\]
For the second inequality, the function
\[
 \rho\longmapsto
 \frac{1-\rho}{2}
 \left(\frac{1+\rho}{2}\right)^{1/4}
\]
is decreasing for \(\rho>0\), since its logarithmic derivative is
\[
 -\frac{1}{1-\rho}+\frac{1}{4(1+\rho)}<0.
\]
Hence
\[
 s_n^2(c_n^2)^{1/4}
 \geq\frac13\left(\frac23\right)^{1/4}
\]
when \(\rho_n\leq1/3\).  Moreover,
\(
 \tfrac{4\,2^{1/4}}9>\tfrac12
\)
because the equivalent inequality \(8\,2^{1/4}>9\) follows after
taking fourth powers from \(8192>6561\).
Together with \(\tau_n^\ast<1/2\), these estimates show that
the first summand is nonnegative.

The second summand is also nonnegative: the assumption
\(\Xi_n(\alpha)\geq0\) and \eqref{eq:ac-65} give
\(1-s_n^2\tau_\alpha^4\geq0\), while
\(1-(1+s_n^2\tau_\alpha^2)^{-3/2}\geq0\).

\smallskip
\noindent\textbf{The case \(n=1\).}
Here \({\rho_1}=1/2\), and
\[
 c_1^2=\frac34,\qquad s_1^2=\frac14,\qquad
 0\leq \tau_\alpha\leq\frac1{\sqrt2},\qquad
 \tau_1^\ast=4\sqrt2-3\sqrt3<\frac12.
\]

Specializing \eqref{eq:ac-65} to \(n=1\) gives
\[
 \sqrt{1+\tau_\alpha^2}\,\Xi_1(\alpha)
 =
 \frac{3(1-\tau_\alpha^4/4)}
 {4(1+\tau_\alpha^2/4)^{3/2}}
 \leq
 \frac{3}{4(1+\tau_\alpha^2/4)^{3/2}},
\]
where the inequality uses \(1-\tau_\alpha^4/4\leq1\).  We now prove
the scalar estimate
\[
 \frac{3}{4(1+\tau_\alpha^2/4)^{3/2}}
 \leq1-\frac{\tau_\alpha}{2}.
\]

Bernoulli's inequality gives
\(
 (1+\tau_\alpha^2/4)^{3/2}
 \geq1+\tfrac{3\tau_\alpha^2}{8}
\).  Therefore
\[
 \frac{3}{4(1+\tau_\alpha^2/4)^{3/2}}
 \leq
 \frac{3}{4(1+3\tau_\alpha^2/8)}.
\]
It is enough to prove
\(
 \tfrac{3}{4(1+3\tau_\alpha^2/8)}
 \leq1-\tfrac{\tau_\alpha}{2}
\).
Multiplication by the positive denominator
\(4(1+3\tau_\alpha^2/8)\) shows that this inequality is equivalent to
\[
 4-8\tau_\alpha+6\tau_\alpha^2-3\tau_\alpha^3\geq0
 \qquad\left(0\leq \tau_\alpha\leq\frac1{\sqrt2}\right).
\]
The polynomial is decreasing on this interval, because its derivative
\(-8+12\tau_\alpha-9\tau_\alpha^2\) is always negative, and its value at the right endpoint
is \((28-19\sqrt2)/4>0\).
Combining the preceding estimate with the specialized form of
\eqref{eq:ac-65} gives
\[
 \sqrt{1+\tau_\alpha^2}\,\Xi_1(\alpha)
 \leq
 \frac{3}{4(1+\tau_\alpha^2/4)^{3/2}}
 \leq1-\frac{\tau_\alpha}{2}
 \leq1-\tau_1^\ast\tau_\alpha
 =\sqrt{1+\tau_\alpha^2}
 \bigl(\cos\alpha-\tau_1^\ast\sin\alpha\bigr).
\]
The penultimate inequality uses
\(\tau_1^\ast<1/2\) and \(\tau_\alpha\geq0\), and the final equality uses
\[
 \cos\alpha=\frac1{\sqrt{1+\tau_\alpha^2}},
 \qquad
 \sin\alpha=\frac{\tau_\alpha}{\sqrt{1+\tau_\alpha^2}}.
\]
Dividing by the positive factor \(\sqrt{1+\tau_\alpha^2}\) proves
\eqref{eq:ac-66} for \(n=1\).

\smallskip
\noindent\textbf{Return to concavity.}
We now return to the displayed second-derivative formula.
Equations \eqref{eq:ac-59} and \eqref{eq:ac-66}, together with
\(0\leq \delta_\beta\leq \widehat{\zeta}_n-\alpha_n^\ast\) give
\[
 \cos\overline\theta_\beta\,\Xi_n(\alpha)
 \leq\cos\delta_\beta(\cos\alpha-\tau_n^\ast\sin\alpha)
 \leq\cos\delta_\beta
 (\cos\alpha-\tan\delta_\beta\sin\alpha)
 =\cos(\delta_\beta+\alpha).
\]

The inequalities
\(
 \sin\overline\theta_\beta
 -\sin\widehat{\zeta}_n\cos\overline\theta_\beta<0
\)
and
\(
 \cos(\delta_\beta+\alpha)
 -\cos\overline\theta_\beta\,\Xi_n(\alpha)\geq0
\)
in the displayed formula for the second derivative show that
\[
 \partial_\alpha^2
 \bigl(\mathscr E_{\alpha,\beta}^{\mathrm{left}}
       -T_{\alpha,\beta}\bigr)\leq0
 \qquad(0<\alpha<\alpha_n^\ast).
\]
Hence
\(\mathscr E_{\alpha,\beta}^{\mathrm{left}}-T_{\alpha,\beta}\) is
concave on \((0,\alpha_n^\ast)\).  Because it is continuous on
\([0,\alpha_n^\ast]\), the concavity inequality extends to the
endpoints by taking limits.  Thus it is concave on
\([0,\alpha_n^\ast]\).  Since
\(\alpha\mapsto T_{\alpha,\beta}\) is also concave,
\(\alpha\mapsto\mathscr E_{\alpha,\beta}^{\mathrm{left}}\) is concave
on this interval.

\smallskip
\noindent\emph{Endpoint values.}
At the two endpoints,
\[
 \mathscr E_{0,\beta}^{\mathrm{left}}
 =(1-{\rho_n})
 (\cos\delta_\beta-\cos\overline\theta_\beta)\geq0,
\]

while, at \(\alpha=\alpha_n^\ast\), one has
\(\theta_\alpha=\Theta_n(\alpha_n^\ast)
=\widehat{\zeta}_n/2\).  Moreover,
\(0\leq\overline\theta_\beta\leq \widehat{\zeta}_n/2\), so
\(\sin(\widehat{\zeta}_n-\overline\theta_\beta)
\geq\sin\overline\theta_\beta\).  Therefore
\[
\begin{aligned}
 S_{\alpha,\beta}+T_{\alpha,\beta}
 &=s_n\left[\sin(\widehat{\zeta}_n-\overline\theta_\beta)
 +\sqrt{\sin(\widehat{\zeta}_n-\overline\theta_\beta)
 \sin\overline\theta_\beta}\right]\\
 &\geq s_n\left[\sin(\widehat{\zeta}_n-\overline\theta_\beta)
 +\sin\overline\theta_\beta\right]\\
 &=2s_n^2\cos(\widehat{\zeta}_n/2-\overline\theta_\beta)\\
 &=(1-{\rho_n})
 \cos(\widehat{\zeta}_n/2-\overline\theta_\beta),
\end{aligned}
\]
and hence
\[
 \mathscr E_{\alpha_n^\ast,\beta}^{\mathrm{left}}
 \geq(1-{\rho_n})\cos(\delta_\beta+\alpha_n^\ast)>0.
\]
Here \(0\leq \delta_\beta+\alpha_n^\ast
\leq \widehat{\zeta}_n<\tfrac{\pi}{2}\), which proves the
strict positivity.

\smallskip
\noindent\textbf{Conclusion.}
Concavity and the endpoint values prove
\(\mathscr E_{\alpha,\beta}^{\mathrm{left}}\geq0\).  By its
definition, this is equivalent to
\(
 \mathscr E_{\alpha,\beta}(0)
 \geq(1+\rho_n)T_{\alpha,\beta}.
\)
This proves the lemma.
\end{proof}

\section{The mixed upper-boundary estimate}
\label{app:ac-mixed-scalar-estimate}

This section contains the proof of the scalar estimate used when the
unconstrained maximizer belongs to the interval in
\Cref{cor:ac-cap-projected-maximizer}.

\subsection{The unconstrained-maximizer estimate}
\label{app:ac-unconstrained-maximizer-estimate}

\begin{proof}[Proof of \Cref{lem:ac-far-mixed-unconstrained}]

Since \(\mu=r\kappa\) and \(0<r\leq1\), one has
\(
 0<\mu\leq\kappa\leq1-2\rho.
\)
Therefore, for \(0<\rho\leq1/3\),
\[
 2\rho\mu\leq2\rho(1-2\rho)<1,
\]
which proves \eqref{eq:ac-mixed-positive-factor}.

Put
\[
 \Pi:=\sqrt{(1+2\rho r+r^2)(1+2\rho t+t^2)}
\]
and consider the difference of the squares of the two sides of
\eqref{eq:ac-mixed-free-bound}:
\[
 \Delta
 :=(\kappa\Pi+\rho t)^2
 -(t^2+\kappa^2-2\rho t\kappa)
  (1+\mu^2-2\rho\mu).
\]
Its exact expansion is
\[
 \Delta=(1-\kappa^2)(\mu^2-t^2)
 +2\rho\kappa(1+\kappa)(1+rt)(\mu+t)
 +\rho^2t^2+2\rho t\kappa\Pi.
 \acEquationTag{eq:ac-101}
\]

To prove \(\Delta\geq0\), it suffices to establish
\[
 (1-\kappa)(t-\mu)\leq2\rho\kappa.                       \acEquationTag{eq:ac-102}
\]
Indeed, the first two terms of \eqref{eq:ac-101} then combine as
\[
 (1+\kappa)(\mu+t)
 \bigl[2\rho\kappa(1+rt)-(1-\kappa)(t-\mu)\bigr]\geq0,
\]
and the remaining terms are nonnegative.

If \(t\leq\rho\kappa\), then
\[
 (1-\kappa)(t-\mu)
 \leq(1-\kappa)t
 \leq t
 \leq\rho\kappa
 \leq2\rho\kappa,
\]
which proves \eqref{eq:ac-102}.  Suppose instead that
\(t>\rho\kappa\).  The hypothesis
\eqref{eq:ac-mixed-left-endpoint} is now an inequality between positive
quantities, so it may be squared.

Moreover,
\(
 t^2+\kappa^2-2\rho t\kappa
 =(t-\rho\kappa)^2+(1-\rho^2)\kappa^2>0.
\)
After multiplying the difference between the squared right- and
left-hand sides of \eqref{eq:ac-mixed-left-endpoint} by
\(t^2+\kappa^2-2\rho t\kappa\), one obtains the
factorization below.

\[
\begin{aligned}
 0\leq{} &(t^2+\kappa^2-2\rho t\kappa)
 \left[
 \bigl(\rho+\mu(1-2\rho^2)\bigr)^2
 -\frac{(1+\mu^2-2\rho\mu)(t-\rho\kappa)^2}
  {t^2+\kappa^2-2\rho t\kappa}
 \right]\\
 &\quad=(1-\rho^2)
 [t(1-2\rho\mu)+\mu\kappa]\\
 &\qquad{}\times
 \{\kappa[2\rho+\mu(1-4\rho^2)]
  -t(1-2\rho\mu)\}.
\end{aligned}
\]

The factors
\[
 1-\rho^2
 \qquad\text{and}\qquad
 t(1-2\rho\mu)+\mu\kappa
\]
are positive: the first because \(0<\rho\leq1/3\), and the second
because \(t,\mu,\kappa>0\) and
\eqref{eq:ac-mixed-positive-factor} holds.  Hence the remaining factor
is nonnegative, which gives
\[
 t\leq
 \frac{\kappa[2\rho+\mu(1-4\rho^2)]}
 {1-2\rho\mu}.                                             \acEquationTag{eq:ac-103}
\]
Since \(0<\kappa<1\), the affine function
\(
 t\longmapsto2\rho\kappa-(1-\kappa)(t-\mu)
\)
is decreasing.  Substituting the upper bound \eqref{eq:ac-103} into
this function gives
\[
 2\rho\kappa-(1-\kappa)(t-\mu)
 \geq
 \frac{\mathscr P(\mu)}{1-2\rho\mu},
\]
where
\[
 \mathscr P(\varepsilon)
 :=
 2\rho\kappa^2
 +\varepsilon[(1-\kappa)^2-4\rho^2\kappa^2]
 -2\rho(1-\kappa)\varepsilon^2.
\]
This polynomial is concave on \([0,\kappa]\), and
\[
 \mathscr P(0)=2\rho\kappa^2,\qquad
 \mathscr P(\kappa)
 =\kappa[(1-\kappa)^2+2\rho(1-2\rho)\kappa^2]\geq0.
\]
Since \(0<\mu\leq\kappa\), it follows that
\(\mathscr P(\mu)\geq0\).  This proves \eqref{eq:ac-102} and hence
\(\Delta\geq0\).

Both
\(\kappa\Pi+\rho t\) and
\[
 \sqrt{1+\mu^2-2\rho\mu}\,
 \sqrt{t^2+\kappa^2-2\rho t\kappa}
\]
are positive.  Thus \(\Delta\geq0\) implies
\eqref{eq:ac-mixed-free-bound}.
\end{proof}
 
\bibliographystyle{alpha}

\begin{thebibliography}{CCSG{\etalchar{+}}09}

\bibitem[AAC{\etalchar{+}}26]{arya-GW}
Shreya Arya, Arnab Auddy, Ranthony~A. Clark, Sunhyuk Lim, Facundo M{\'e}moli,
  and Daniel Packer.
\newblock The {Gromov--Wasserstein} distance between spheres.
\newblock {\em Foundations of Computational Mathematics}, 26(1):75--130, 2026.

\bibitem[ABC{\etalchar{+}}26]{adams2022gromov}
Henry Adams, Johnathan Bush, Nate Clause, Florian Frick, Mario G{\'o}mez,
  Michael Harrison, R.~Amzi Jeffs, Evgeniya Lagoda, Sunhyuk Lim, Facundo
  M{\'e}moli, Michael Moy, Nikola Sadovek, Matt Superdock, Daniel Vargas,
  Qingsong Wang, and Ling Zhou.
\newblock Gromov-{H}ausdorff distances, {B}orsuk-{U}lam theorems, and
  {V}ietoris-{R}ips complexes.
\newblock {\em Algebraic \& Geometric Topology}, 2026.
\newblock To appear.

\bibitem[BBI01]{burago2022course}
Dmitri Burago, Yuri Burago, and Sergei Ivanov.
\newblock {\em A Course in Metric Geometry}, volume~33 of {\em Graduate Studies
  in Mathematics}.
\newblock American Mathematical Society, 2001.

\bibitem[CCSG{\etalchar{+}}09]{chazal2009gromov}
Fr{\'e}d{\'e}ric Chazal, David Cohen-Steiner, Leonidas~J Guibas, Facundo
  M{\'e}moli, and Steve~Y Oudot.
\newblock {Gromov--Hausdorff} stable signatures for shapes using persistence.
\newblock {\em Computer Graphics Forum}, 28(5):1393--1403, 2009.

\bibitem[CM10]{carlson-memoli-clustering}
Gunnar Carlsson and Facundo M{\'e}moli.
\newblock Characterization, stability and convergence of hierarchical
  clustering methods.
\newblock {\em Journal of Machine Learning Research}, 11(47):1425--1470, 2010.

\bibitem[Dek95]{dekster1995jung}
Boris~V. Dekster.
\newblock The {J}ung theorem for spherical and hyperbolic spaces.
\newblock {\em Acta Mathematica Hungarica}, 67(4):315--331, 1995.

\bibitem[DS81]{dubins1981equidiscontinuity}
Lester Dubins and Gideon Schwarz.
\newblock Equidiscontinuity of {B}orsuk-{U}lam functions.
\newblock {\em Pacific Journal of Mathematics}, 95(1):51--59, 1981.

\bibitem[Edw75]{edwards1975structure}
David~A. Edwards.
\newblock The structure of superspace.
\newblock In {\em Studies in Topology}, pages 121--133. Academic Press, 1975.

\bibitem[Gro99]{gromov1999metric}
Mikhail Gromov.
\newblock {\em Metric Structures for {R}iemannian and Non-{R}iemannian Spaces},
  volume 152 of {\em Progress in Mathematics}.
\newblock Birkh\"auser, 1999.
\newblock Based on the 1981 French original, with appendices by M. Katz, P.
  Pansu, and S. Semmes.

\bibitem[HJ23]{harrison2023quantitative}
Michael Harrison and R.~Amzi Jeffs.
\newblock Quantitative upper bounds on the {G}romov-{H}ausdorff distance
  between spheres.
\newblock {\em arXiv preprint arXiv:2309.11237}, 2023.

\bibitem[LMO24]{lim2020vietoris}
Sunhyuk Lim, Facundo M{\'e}moli, and Osman~Berat Okutan.
\newblock Vietoris-{R}ips persistent homology, injective metric spaces, and the
  filling radius.
\newblock {\em Algebraic \& Geometric Topology}, 24(2):1019--1100, 2024.

\bibitem[LMS23]{lim2021gromov}
Sunhyuk Lim, Facundo M{\'e}moli, and Zane Smith.
\newblock The {G}romov--{H}ausdorff distance between spheres.
\newblock {\em Geometry \& Topology}, 27(9):3733--3800, 2023.

\bibitem[M{\'e}m11]{memoli-2011-GW}
Facundo M{\'e}moli.
\newblock {Gromov--Wasserstein} distances and the metric approach to object
  matching.
\newblock {\em Foundations of Computational Mathematics}, 11(4):417--487, 2011.

\bibitem[M{\'e}m12]{memoli-properties}
Facundo M{\'e}moli.
\newblock Some properties of {G}romov--{H}ausdorff distances.
\newblock {\em Discrete \& Computational Geometry}, 48(2):416--440, 2012.

\bibitem[MS05]{memoli-sapiro}
Facundo M{\'e}moli and Guillermo Sapiro.
\newblock A theoretical and computational framework for isometry invariant
  recognition of point cloud data.
\newblock {\em Foundations of Computational Mathematics}, 5(3):313--347, 2005.

\bibitem[MS24]{memoli2024embedding}
Facundo M{\'e}moli and Zane~T. Smith.
\newblock Embedding--projection correspondences for the estimation of the
  {G}romov--{H}ausdorff distance.
\newblock {\em arXiv preprint arXiv:2407.03295}, 2024.

\bibitem[MSZ24]{zhou-cup-product}
Facundo M{\'e}moli, Anastasios Stefanou, and Ling Zhou.
\newblock Persistent cup product structures and related invariants.
\newblock {\em Journal of Applied and Computational Topology}, 8:93--148, 2024.

\bibitem[MZ25]{memoli2019persistent}
Facundo M{\'e}moli and Ling Zhou.
\newblock Persistent homotopy groups of metric spaces.
\newblock {\em Journal of Topology and Analysis}, 17(5):1481--1542, 2025.

\bibitem[RM26]{rodriguezmartin2024novel}
Sa{\'u}l Rodr{\'i}guez~Mart{\'i}n.
\newblock Some novel constructions of {G}romov--{H}ausdorff-optimal
  correspondences between spheres.
\newblock {\em Experimental Mathematics}, 2026.
\newblock Published online 11 April 2026.

\bibitem[San46]{santalo1946convex}
L.~A. Santal{\'o}.
\newblock Convex regions on the \(n\)-dimensional spherical surface.
\newblock {\em Annals of Mathematics}, 47(3):448--459, 1946.

\bibitem[Sch17]{schmiedl2017computational}
Felix Schmiedl.
\newblock Computational aspects of the {G}romov--{H}ausdorff distance and its
  application in non-rigid shape matching.
\newblock {\em Discrete \& Computational Geometry}, 57(4):854--880, 2017.

\bibitem[Stu23]{sturm-GW}
Karl-Theodor Sturm.
\newblock {\em The Space of Spaces: Curvature Bounds and Gradient Flows on the
  Space of Metric Measure Spaces}, volume 290, no. 1443 of {\em Memoirs of the
  American Mathematical Society}.
\newblock American Mathematical Society, Providence, RI, 2023.

\end{thebibliography}
\newcommand{\etalchar}[1]{$^{#1}$}

\end{document}